\documentclass[12pt]{article}
\usepackage[]{amsmath,amssymb}
\usepackage{amscd}
\usepackage{latexsym}
\usepackage{cite}
\usepackage{amsthm}

\newtheorem{definition}{Definition}[section]
\newtheorem{theorem}[definition]{Theorem}
\newtheorem{lemma}[definition]{Lemma}
\newtheorem{corollary}[definition]{Corollary}
\newtheorem{remark}[definition]{Remark}
\newtheorem{example}[definition]{Example}

\newtheorem{note}[definition]{Note}
\newtheorem{assumption}[definition]{Assumption}
\newtheorem{proposition}[definition]{Proposition}

\def\F{\mathbb F}

\begin{document}
\title{\bf  
Lowering-Raising triples 
and $U_q(\mathfrak{sl}_2)$
}
\author{
Paul Terwilliger}
\date{}

\maketitle
\begin{abstract}
We introduce the notion of a lowering-raising (or LR) 
triple of linear transformations
on a nonzero finite-dimensional vector space. 
We show how to normalize an LR triple,
and classify up to isomorphism the normalized LR triples.
We  describe the LR triples
using various maps,
such as the
reflectors, 
the inverters, the 
 unipotent maps, and the rotators.
We relate the  LR triples
to the equitable presentation of 
the quantum algebra $U_q(\mathfrak{sl}_2)$ and Lie algebra
 $\mathfrak{sl}_2$.

 \end{abstract}

\section{Introduction}

\noindent For the quantum algebra
$U_q(\mathfrak{sl}_2)$, the equitable presentation 
was introduced in
\cite{equit} and further investigated in
\cite{fduq},
\cite{ba}.
For the Lie algebra $\mathfrak{sl}_2$,
the equitable presentation
was introduced in
\cite{ht}
and comprehensively studied in
\cite{bt}. 
These equitable presentations have been
related to
Leonard pairs
\cite{alnajjar},
\cite{alnajjar2},
tridiagonal pairs
\cite{bockting},
Leonard triples
\cite{gao},
\cite{huang},
the universal Askey-Wilson algebra
\cite{uawe},
the tetrahedron algebra
\cite{hartwig},
\cite{ht},
\cite{3ptsl2},
the $q$-tetrahedron algebra
\cite{irt},
\cite{qtet},
and distance-regular graphs
\cite{boyd}.
See also
\cite{alnajjar3},
\cite{neubauer},
\cite{uqsl2hat},
\cite{tersym}.
\\

\noindent 
From the equitable point of view,
consider a finite-dimensional irreducible 
module for 
 $U_q(\mathfrak{sl}_2)$ or
 $\mathfrak{sl}_2$.
In 
\cite[Lemma~7.3]{fduq} and
\cite[Section~8]{bt}
we encounter three nilpotent linear transformations
of the module,
with each transformation acting as a lowering map and raising map in
multiple ways.
In order to describe this situation more precisely,
we now introduce the notion of
a lowering-raising (or LR) 
triple of linear transformations.
\medskip

\noindent 
An LR triple is described as follows
(formal definitions begin in Section 2).
Fix an integer $d\geq 0$.
Let $\mathbb F$ denote a field, and let $V$ denote a vector
space over $\mathbb F$ with dimension $d+1$. By a decomposition of
$V$ we mean a sequence  $\lbrace V_i\rbrace_{i=0}^d$
of one-dimensional subspaces whose direct sum is $V$.
Let  $\lbrace V_i\rbrace_{i=0}^d$
denote a decomposition of $V$. A linear transformation
$A \in {\rm End}(V)$ is said to lower
$\lbrace V_i\rbrace_{i=0}^d$ whenever
$AV_i = V_{i-1} $ for $1 \leq i \leq d$ and
$AV_0 = 0$.
The map $A$ is said to raise
$\lbrace V_i\rbrace_{i=0}^d$ whenever
$AV_i = V_{i+1} $ for $0 \leq i \leq d-1$ and
$AV_d = 0$.
An ordered pair of elements $A,B$ in
${\rm End}(V)$ is called lowering-raising (or LR) whenever
there exists  a decomposition of $V$ that is lowered by
$A$ and raised by $B$.
A 3-tuple of elements $A,B,C$ in
${\rm End}(V)$ is called an LR triple whenever
any two of $A,B,C$ form an LR pair on $V$. 
The LR triple $A,B,C$ is said to be over $\mathbb F$ and
have diameter $d$. 
\medskip

\noindent In this paper we obtain three main results,
which are summarized as follows:
(i) we show how to normalize an LR triple, and
 classify up to isomorphism the
normalized LR triples; (ii) 
we describe the LR triples using
various maps, such as
the reflectors, the inverters, 
the unipotent maps, and the rotators;
(iii) we relate the
LR triples to the equitable presentations of
$U_q(\mathfrak{sl}_2)$ and
$\mathfrak{sl}_2$.
\medskip

\noindent We now describe our results in more detail.
We set the stage with some general remarks; the assertions
therein will be established in the main body of the paper.
Let the integer $d$ and the vector space 
$V$ be as above, and assume for the moment that $d=0$. Then
$A,B,C \in {\rm End}(V)$ form an LR triple if and only if
each of $A,B,C$ is zero; this LR triple is called
trivial.
Until further notice, assume that $d\geq 0$ and
let $A,B,C$ denote an LR triple on $V$.
As we describe this LR triple, we will use the following 
notation. Observe that any permutation of 
$A,B,C$ is an LR triple on $V$.
For any object $f$ that we associate with $A,B,C$
let $f'$ (resp. $f''$) denote the corresponding object for
the LR triple $B,C,A$ (resp. $C,A,B$).
Since $A,B$ is an LR pair on $V$, there exists a
decomposition $\lbrace V_i\rbrace_{i=0}^d$ of $V$ that 
is lowered by $A$ and raised by $B$. This decomposition is
uniquely determined by $A,B$ and called the
$(A,B)$-decomposition of $V$. For $0 \leq i \leq d$ we have
$
A^{d-i}V=
V_0+V_1+ \cdots + V_i$
and
$
B^{d-i}V=
V_d+V_{d-1}+ \cdots + V_{d-i}$.
\medskip

\noindent We now introduce the parameter array of $A,B,C$.
For $1 \leq i \leq d$ we have
$AV_i=V_{i-1}$ and $BV_{i-1}=V_i$. Therefore, $V_i$ is invariant
under $BA$ and the corresponding eigenvalue is a nonzero
scalar in $\mathbb F$.
Denote this eigenvalue by $\varphi_i$. For notational convenience
define $\varphi_0=0$ and $\varphi_{d+1}=0$.
We call the sequence
\begin{eqnarray*}
(\lbrace \varphi_i \rbrace_{i=1}^d;
\lbrace \varphi'_i \rbrace_{i=1}^d;
\lbrace \varphi''_i \rbrace_{i=1}^d)
\end{eqnarray*}
the parameter array of $A,B,C$.
\medskip

\noindent We now introduce the idempotent data of $A,B,C$.
For $0 \leq i \leq d$ define
$E_i \in {\rm End}(V)$ such that $(E_i-I)V_i=0$ and
$E_iV_j=0$ for $0 \leq j \leq d$, $j \not=i$. Thus $E_i$
is the projection from $V$ onto $V_i$. Note that
$V_i = E_iV$.
We have
\begin{eqnarray*}
E_i = \frac{A^{d-i}B^d A^i}{\varphi_1 \cdots \varphi_d},
\qquad \qquad 
E_i = \frac{B^{i}A^d B^{d-i}}{\varphi_1 \cdots \varphi_d}.
\end{eqnarray*}
We call the sequence
\begin{eqnarray*}
(
\lbrace E_i\rbrace_{i=0}^d;
\lbrace E'_i\rbrace_{i=0}^d;
\lbrace E''_i\rbrace_{i=0}^d)
\end{eqnarray*}
the idempotent data of $A,B,C$.
\medskip

\noindent We now introduce the Toeplitz data of $A,B,C$.
A basis $\lbrace v_i \rbrace_{i=0}^d$ of $V$
is called an $(A,B)$-basis whenever
$v_i \in V_i$ for $0 \leq i \leq d$ and
$Av_i=v_{i-1}$ for $1 \leq i \leq d$.
Let $\lbrace u_i \rbrace_{i=0}^d$ denote a $(C,B)$-basis of $V$
and let $\lbrace v_i \rbrace_{i=0}^d$ denote a $(C,A)$-basis
of $V$ such that
$u_0=v_0$.
Let $T$
denote the
transition matrix from
$\lbrace u_i \rbrace_{i=0}^d$ to
$\lbrace v_i \rbrace_{i=0}^d$. Then $T$  has the
 form
\begin{eqnarray*}
T =	\left(
	\begin{array}{ c c cc c c }
	\alpha_0 & \alpha_1   &  \cdot  & \cdot&  \cdot  & \alpha_d  \\
	 & \alpha_0  & \alpha_1  & \cdot& \cdot &  \cdot     \\
	 &   & \alpha_0   & \cdot & \cdot& \cdot  \\
	   &   &  & \cdot  & \cdot & \cdot \\
	     &  & &  & \cdot & \alpha_1 \\
	      {\bf 0}  &&  & &   & \alpha_0  \\
	      \end{array}
	      \right),
\end{eqnarray*}
where $\alpha_i \in \mathbb F$ for
$0 \leq i \leq d$ and $\alpha_0=1$.
A matrix of the above form is said to be upper triangular
and Toeplitz, with parameters
$\lbrace \alpha_i\rbrace_{i=0}^d$.
The matrix $T^{-1}$ is upper triangular and Toeplitz; let
$\lbrace \beta_i\rbrace_{i=0}^d$ denote its parameters.
We call the sequence
\begin{eqnarray*}
(
\lbrace \alpha_i \rbrace_{i=0}^d,
\lbrace \beta_i \rbrace_{i=0}^d;
\lbrace \alpha'_i \rbrace_{i=0}^d,
\lbrace \beta'_i \rbrace_{i=0}^d;
\lbrace \alpha''_i \rbrace_{i=0}^d,
\lbrace \beta''_i \rbrace_{i=0}^d 
)
\end{eqnarray*}
the Toeplitz data of $A,B,C$.
\medskip

\noindent We now introduce the trace data of $A,B,C$. 
For $0 \leq i \leq d$ let
$a_i$ denote the trace of $CE_i$.
We have $\sum_{i=0}^d a_i=0$.
If $A,B,C$ is trivial then
$a_0=0$. If $A,B,C$ is nontrivial then
$a_i = \alpha'_1(\varphi''_{d-i+1}-\varphi''_{d-i})$
and 
$a_i = \alpha''_1(\varphi'_{d-i+1}-\varphi'_{d-i})$
for $0 \leq i \leq d$.
We call the sequence
\begin{eqnarray*}
(\lbrace a_i \rbrace_{i=0}^d;
\lbrace a'_i \rbrace_{i=0}^d;
\lbrace a''_i \rbrace_{i=0}^d
)
\end{eqnarray*}
 the trace data of $A,B,C$.
\medskip

\noindent    
With respect to an $(A,B)$-basis of $V$,
the matrices representing $A,B,C$ are
\begin{eqnarray*}
&&A:\; 
\left(
\begin{array}{ c c cc c c }
 0 & 1   &   &&   & \bf 0  \\
  & 0  &  1  &&  &      \\ 
   &   &  0   & \cdot &&  \\
     &   &  & \cdot  & \cdot & \\
       &  & &  & \cdot & 1 \\
        {\bf 0}  &&  & &   &  0  \\
	\end{array}
	\right),
	\qquad \qquad
	B:\;
	\left(
	\begin{array}{ c c cc c c }
	0 &   &   &&   & \bf 0  \\
	\varphi_1 & 0  &   &&  &      \\
	 &  \varphi_2 & 0   & &&  \\
	   &   & \cdot & \cdot  & & \\
	     &  & & \cdot & \cdot & \\
	      {\bf 0}  &&  & &  \varphi_d & 0  \\
	      \end{array}
	      \right),
\\
&&C:\; 
\left(
\begin{array}{ c c c c c c }
 a_0 & \varphi'_d/\varphi_1   &   &&   & \bf 0  \\
 \varphi''_d & a_1 &  \varphi'_{d-1}/\varphi_2  &&  &      \\ 
   & \varphi''_{d-1}  & a_2  
   & \cdot &&  \\
     &   & \cdot & \cdot  & \cdot & \\
       &  & & \cdot & \cdot & \varphi'_1/\varphi_d \\
        {\bf 0}  &&  & &  \varphi''_1 &  a_d \\
	\end{array}
	\right).
\end{eqnarray*}
Assume for the moment that $A,B,C$ is nontrivial.
Then $A,B,C$ is determined up to isomorphism 
by its parameter array and any one of
\begin{eqnarray*}
a_0, a'_0, a''_0;
\quad \qquad 
a_d, a'_d, a''_d;
\quad \qquad 
\alpha_1, \alpha'_1, \alpha''_1;
\quad \qquad 
\beta_1, \beta'_1, \beta''_1.
\end{eqnarray*}
We often put the emphasis on $\alpha_1$, and call
this the first Toeplitz number of $A,B,C$.
In Propositions
\ref{prop:AlphaRecursion}--\ref{prop:AlphaRecursion3}
we obtain some recursions that 
give the Toeplitz data of $A,B,C$  in terms
of 
its parameter array and first Toeplitz number.
\medskip

\noindent We now introduce the bipartite condition.
The LR triple $A,B,C$ is said to be bipartite whenever
$a_i = a'_i = a''_i = 0$ for $0 \leq i \leq d$.
Assume for the moment that $A,B,C$ is not bipartite.
Then $A,B,C$ is nontrivial, and each of
\begin{eqnarray*}
\alpha_1, \qquad
\alpha'_1, \qquad
\alpha''_1, \qquad
\beta_1, \qquad
\beta'_1, \qquad
\beta''_1
\end{eqnarray*}
is nonzero.
Until further notice assume that $A,B,C$ is bipartite.
Then $d=2m$ is even.
Moreover for $0 \leq i \leq d$, each of
\begin{eqnarray*}
\alpha_i, \qquad
\alpha'_i, \qquad
\alpha''_i, \qquad
\beta_i, \qquad
\beta'_i, \qquad
\beta''_i
\end{eqnarray*}
is zero if $i$ is odd and nonzero if $i$ is even.
There exists a direct sum $V=V_{\rm out} + V_{\rm in}$
such that $V_{\rm out}$ is equal to  each of
\begin{eqnarray*}
\sum_{j=0}^m E_{2j} V,
\qquad \quad 
 \sum_{j=0}^m E'_{2j} V,
\qquad \quad 
\sum_{j=0}^m E''_{2j} V
\end{eqnarray*}
and $V_{\rm in}$ is equal to each of
\begin{eqnarray*}
\sum_{j=0}^{m-1} E_{2j+1} V,
\qquad \quad 
\sum_{j=0}^{m-1} E'_{2j+1} V,
\qquad \quad 
\sum_{j=0}^{m-1} E''_{2j+1} V.
\end{eqnarray*}
The dimensions of $V_{\rm out}$  and
 $V_{\rm in}$ are $m+1$ and $m$, respectively. 
We have
\begin{eqnarray*}
&& AV_{\rm out} = V_{\rm in},
 \qquad \qquad 
 BV_{\rm out} = V_{\rm in},
 \qquad \qquad 
 CV_{\rm out} = V_{\rm in},
\\
&&
 AV_{\rm in} \subseteq V_{\rm out},
 \qquad \qquad 
 BV_{\rm in} \subseteq V_{\rm out},
 \qquad \qquad 
 CV_{\rm in} \subseteq V_{\rm out}.
\end{eqnarray*}
Define
\begin{eqnarray}
\label{eq:6listIntro}
A_{\rm out}, 
\qquad 
A_{\rm in},
\qquad
B_{\rm out}, 
\qquad 
B_{\rm in},
\qquad 
C_{\rm out},
\qquad 
C_{\rm in}
\end{eqnarray}
 in ${\rm End}(V)$ as follows.
The map $A_{\rm out}$ acts on
 $V_{\rm out}$ as $A$, and on
 $V_{\rm in}$ as zero.
The map $A_{\rm in}$ acts on
 $V_{\rm in}$ as $A$, and on
 $V_{\rm out}$ as zero.
The other maps in
(\ref{eq:6listIntro}) are similarly defined.
By construction
\begin{equation*}
A = A_{\rm out}+ 
 A_{\rm in}, 
 \qquad \qquad
B = B_{\rm out}+ 
 B_{\rm in}, 
 \qquad \qquad
C = C_{\rm out}+ 
 C_{\rm in}.
\end{equation*}
We are done assuming that $A,B,C$ is bipartite.
\medskip

\noindent We now introduce the equitable condition.
The LR triple $A,B,C$ is said to be equitable whenever
$\alpha_i = \alpha'_i=\alpha''_i$ for $0 \leq i \leq d$.
In this case
$\beta_i = \beta'_i=\beta''_i$ for $0 \leq i \leq d$.
Assume for the moment that $A,B,C$ is trivial.
Then $A,B,C$ is equitable.
Next assume that $A,B,C$ is nonbipartite.
Then $A,B,C$ is equitable if and only if
$\alpha_1 = \alpha'_1=\alpha''_1$.
In this case
$\varphi_i = \varphi'_i = \varphi''_i$ for $1 \leq i \leq d$,
and
$a_i = a'_i = a''_i = \alpha_1(\varphi_{d-i+1}-\varphi_{d-i})$
for $0 \leq i \leq d$.
Next assume that $A,B,C$ is bipartite and nontrivial.
Then $A,B,C$ is equitable if and only if
$\alpha_2 = \alpha'_2=\alpha''_2$.
In this case
$
\varphi_{i-1}\varphi_i = 
\varphi'_{i-1}\varphi'_i = 
\varphi''_{i-1}\varphi''_i$
for $2 \leq i \leq d$.
We are done with our general remarks.
\medskip

\noindent 
Concerning the normalization of LR triples,
we now define what
it means for $A,B,C$ to be normalized.
Assume for the moment that $A,B,C$  is trivial. Then
$A,B,C$ is normalized.
Next assume that $A,B,C$ is nonbipartite. Then
$A,B,C$ is normalized whenever $\alpha_1=\alpha'_1=\alpha''_1=1$.
Next assume that $A,B,C$ is bipartite and nontrivial. Then 
$A,B,C$ is normalized whenever $\alpha_2= \alpha'_2=\alpha''_2=1$.
If $A,B,C$ is normalized then $A,B,C$ is equitable.
We now explain how to normalize $A,B,C$.
Assume for the moment that $A,B,C$ is trivial. Then there is nothing to do.
Next assume that $A,B,C$ is nonbipartite.
Then there exists a unique sequence
$\alpha, \beta,
\gamma$ of nonzero scalars in $\mathbb F$ such that $\alpha A, \beta B, \gamma C$
is normalized. Next assume that $A,B,C$ is bipartite and nontrivial.
Then there exists a unique 
sequence $\alpha, \beta, \gamma$ of nonzero scalars in $\mathbb F$
such that 
\begin{eqnarray*}
 \alpha A_{\rm out}
+A_{\rm in},
\qquad
\beta  B_{\rm out}+
B_{\rm in},
\qquad
 \gamma C_{\rm out}+
C_{\rm in}
\end{eqnarray*}
is normalized.
\medskip

\noindent We now describe our 
classification up to isomorphism of the normalized LR triples
over $\mathbb F$. Up to isomorphism there exists a unique normalized 
LR triple over $\mathbb F$ with diameter $d=0$, and this LR
triple is trivial.
Up to isomorphism there exists a unique normalized LR triple over
$\mathbb F$ with diameter $d=1$, and this is given in
Lemma
\ref{lem:1or2}. For $d\geq 2$, we display nine families of
normalized 
  LR triples over
$\mathbb F$ that have diameter $d$, denoted
\begin{eqnarray*}
&&
{\rm NBWeyl}^+_d(\mathbb F;j,q),
\qquad \quad
{\rm NBWeyl}^-_d(\mathbb F;j,q),
\qquad \quad
{\rm NBWeyl}^-_d(\mathbb F;t),
\\
&&
{\rm NBG}_d(\mathbb F;q),
\qquad \quad
{\rm NBG}_d(\mathbb F;1),
\\
&&
{\rm NBNG}_d(\mathbb F;t),
\\
&&
{\rm B}_d(\mathbb F;t,\rho_0,\rho'_0,\rho''_0),
\qquad \quad 
{\rm B}_d(\mathbb F;1,\rho_0,\rho'_0,\rho''_0),
\qquad \quad 
{\rm B}_2(\mathbb F;\rho_0,\rho'_0,\rho''_0).
\end{eqnarray*}
We show that each 
normalized LR triple over $\mathbb F$ with diameter $d$ is isomorphic to
exactly one of these examples.
\medskip

\noindent We now describe the LR triples using various maps.
Let $A,B,C$ denote an LR triple on $V$.
We show that there exists a unique antiautomorphism $\dagger$ of 
${\rm End}(V)$ that sends $A\leftrightarrow B$.
We call $\dagger$ the $(A,B)$-reflector.
Assume for the moment that $A,B,C$ is equitable and nonbipartite.
We show that $\dagger$ fixes $C$.
Next assume that $A,B,C$ is equitable, bipartite, and nontrivial.
We show that $\dagger$ sends
$A_{\rm out}\leftrightarrow B_{\rm in}$ and
$B_{\rm out}\leftrightarrow A_{\rm in}$. We also show that
$\dagger$ sends
each of $C_{\rm out},C_{\rm in}$ 
to a scalar multiple of the other.
Define
\begin{eqnarray*}
\Psi = \sum_{i=0}^d \frac{\varphi_1 \varphi_2 \cdots \varphi_i}{\varphi_d
\varphi_{d-1}\cdots \varphi_{d-i+1}} E_i.
\end{eqnarray*}
We call $\Psi$ the $(A,B)$-inverter.
We show that the following three LR pairs are mutually isomorphic:
\begin{eqnarray*}
A, \Psi^{-1}B\Psi
\qquad \qquad
B,A
\qquad \qquad
\Psi A \Psi^{-1}, B.
\end{eqnarray*}
Define
\begin{eqnarray*}
{\mathbb A} = \sum_{i=0}^d E_{d-i} E''_{i},
\qquad \qquad 
{\mathbb B} = \sum_{i=0}^d E'_{d-i} E_{i},
\qquad \qquad 
{\mathbb C} = \sum_{i=0}^d E''_{d-i} E'_i.
\end{eqnarray*}
We call
$\mathbb A,
\mathbb B,
\mathbb C$
the unipotent maps for $A,B,C$. 
We show that
\begin{eqnarray*}
\mathbb A = \sum_{i=0}^d \alpha'_i A^i,
\qquad \qquad 
\mathbb B = \sum_{i=0}^d \alpha''_i B^i,
\qquad \qquad 
\mathbb C = \sum_{i=0}^d \alpha_i C^i
\end{eqnarray*}
and
\begin{eqnarray*}
\mathbb A^{-1} = \sum_{i=0}^d \beta'_i A^i,
\qquad \qquad 
\mathbb B^{-1} = \sum_{i=0}^d \beta''_i B^i,
\qquad \qquad 
\mathbb C^{-1} = \sum_{i=0}^d \beta_i C^i.
\end{eqnarray*}
By a rotator for $A,B,C$ we mean
an element
$R \in {\rm End}(V)$ such that
for $0 \leq i \leq d$,
\begin{eqnarray*}
E_i R  =  R E'_i, \qquad \qquad
E'_i R = R E''_i, \qquad \qquad
E''_i R  = R E_i.
\end{eqnarray*}
Let $\mathcal R$ denote the  set of rotators for
$A,B,C$. Note that $\mathcal R$ is a subspace of
the $\mathbb F$-vector space ${\rm End}(V)$.
 We obtain the following basis
for 
$\mathcal R$. Assume for the moment that
$A,B,C$ is trivial. Then $\mathcal R ={\rm End}(V)$ has a basis consisting
of the identity element.
Next assume that $A,B,C$
is nonbipartite. Then $\mathcal R$ has a basis
$\Omega$ such that
\begin{eqnarray*}
&&
\Omega = \mathbb B 
\Biggl(\sum_{i=0}^d \frac{\varphi_1 \cdots \varphi_i}
{\varphi_d \cdots \varphi_{d-i+1}} E_i\Biggr) \mathbb A
= 
\mathbb C 
\Biggl(\sum_{i=0}^d \frac{\varphi_1 \cdots \varphi_i}
{\varphi_d \cdots \varphi_{d-i+1}} E'_i\Biggr) \mathbb B 
\\
&& \qquad \qquad \qquad 
=
\mathbb A 
\Biggl(\sum_{i=0}^d \frac{\varphi_1 \cdots \varphi_i}
{\varphi_d \cdots \varphi_{d-i+1}} E''_i\Biggr) \mathbb C.
\end{eqnarray*}
Next assume that $A,B,C$ is bipartite and nontrivial.
Then $\mathcal R$ has a basis $\Omega_{\rm out}$,
 $\Omega_{\rm in}$ such that
\begin{eqnarray*}
&&
\Omega_{\rm out} = 
\mathbb B 
\Biggl(\sum_{j=0}^{d/2} \frac{\varphi_1 \varphi_2\cdots \varphi_{2j}}
{\varphi_d \varphi_{d-1}\cdots \varphi_{d-2j+1}} E_{2j}\Biggr) \mathbb A
= 
\mathbb C 
\Biggl(\sum_{j=0}^{d/2} \frac{\varphi_1 \varphi_2 \cdots \varphi_{2j}}
{\varphi_d \varphi_{d-1}\cdots \varphi_{d-2j+1}} E'_{2j}\Biggr) \mathbb B 
\\
&& \qquad \qquad \qquad 
=
\mathbb A 
\Biggl(\sum_{j=0}^{d/2} \frac{\varphi_1 \varphi_2 \cdots \varphi_{2j}}
{\varphi_d \varphi_{d-1}\cdots \varphi_{d-2j+1}} E''_{2j}\Biggr) \mathbb C
\end{eqnarray*}
and
\begin{eqnarray*}
&&\Omega_{\rm in} = 
\mathbb B 
\Biggl(\sum_{j=0}^{d/2-1} \frac{\varphi_2 \varphi_3\cdots \varphi_{2j+1}}
{\varphi_{d-1} \varphi_{d-2}\cdots \varphi_{d-2j}} E_{2j+1}\Biggr) \mathbb A
= 
\mathbb C 
\Biggl(\sum_{j=0}^{d/2-1} \frac{\varphi_2 \varphi_3 \cdots \varphi_{2j+1}}
{\varphi_{d-1} \varphi_{d-2}\cdots \varphi_{d-2j}} E'_{2j+1}\Biggr) \mathbb B 
\nonumber
\\
&& \qquad \qquad \qquad 
=
\mathbb A 
\Biggl(\sum_{j=0}^{d/2-1} \frac{\varphi_2 \varphi_3\cdots \varphi_{2j+1}}
{\varphi_{d-1} \varphi_{d-2} \cdots \varphi_{d-2j}} E''_{2j+1}\Biggr)
\mathbb C.
\end{eqnarray*}
\noindent We now briefly relate the LR triples to 
the equitable presentations of
$U_q(\mathfrak{sl}_2)$ and
$\mathfrak{sl}_2$.
Adjusting the equitable presentation of
$U_q(\mathfrak{sl}_2)$ in two ways, 
we obtain an
algebra
$U^R_q(\mathfrak{sl}_2)$ called the
reduced 
$U_q(\mathfrak{sl}_2)$ algebra,
and an
algebra
$U^E_q(\mathfrak{sl}_2)$ called the extended
$U_q(\mathfrak{sl}_2)$ algebra.
Let $A,B,C$ denote an LR triple on $V$.
After imposing some minor restrictions on
its parameter array, we use  $A,B,C$ 
to construct a module on $V$
for 
$U_q(\mathfrak{sl}_2)$ or
$U^R_q(\mathfrak{sl}_2)$ or
$U^E_q(\mathfrak{sl}_2)$ or
$\mathfrak{sl}_2$. Each construction involves
the equitable presentation.
\medskip

\noindent This paper is organized as follows.
In Section 2 we review some basic concepts
and explain our notation.
In Sections 3--10 
we develop a theory of LR pairs that will be applied to
LR triples later in the paper. 
In Section 11 we classify a type of finite sequence
said to be constrained, for use in our LR triple classification
later in the paper.
Section 12 is about upper triangular Toeplitz matrices.
In Section 13 we introduce the LR triples,
and discuss their parameter array, idempotent data,
Toeplitz data, and trace data. In Sections 14, 15
we obtain some equations relating the parameter
array, Toeplitz data, and trace data. We also
introduce the LR triples of Weyl and $q$-Weyl type.
Sections 16--18 are about the bipartite, equitable, and normalized
LR triples, respectively. In Sections 19, 20 we compare the
structure of a bipartite and nonbipartite LR triple, using
the notions of an idempotent centralizer and double lowering space.
Sections 21--23 are about the unipotent maps, rotators, and
reflectors, respectively.
In Sections 24--30 we classify up to isomorphism the normalized 
LR triples. Section 31 is about the Toeplitz data,
and how the unipotent maps are related to the
exponential function and quantum exponential function.
In Section 32 we display some relations that are satisfied
by an LR triple. In Section 33 we relate the LR triples to the
equitable presentations of 
$U_q(\mathfrak{sl}_2)$ and
$\mathfrak{sl}_2$. Section 34 contains three characterizations
of an LR triple. Sections 35, 36 are appendices that contain
some matrix representations
of an LR triple.

\bigskip

 \section{Preliminaries}
\noindent We now begin our formal argument.
In this section we review some basic concepts
and explain our notation.
We will be discussing algebras and Lie algebras. An algebra
without the Lie prefix is 
meant to be associative
and have a 1. A subalgebra has the same 1 as the parent algebra.
Recall the ring of integers 
$\mathbb Z = \lbrace 0,\pm 1,\pm 2,\ldots\rbrace$.
Throughout the paper we fix an integer $d\geq 0$.
For a sequence $\lbrace u_i\rbrace_{i=0}^d$, 
we call $u_i$ the {\it $i$-component} or {\it $i$-coordinate} of
 the sequence.
By the
{\it inversion} of 
 $\lbrace u_i\rbrace_{i=0}^d$
we mean the sequence
$\lbrace u_{d-i}\rbrace_{i=0}^d$.
Let $\mathbb F$ denote a field.
Let $V$ denote a vector space over $\mathbb F$ with dimension
$d+1$. 
Let 
${\rm End}(V)$ denote the $\mathbb F$-algebra
consisting of the $\mathbb F$-linear maps from $V$ to $V$.
Let ${\rm Mat}_{d+1}(\mathbb F)$ denote the $\mathbb F$-algebra consisting of
the $d+1$ by $d+1$ matrices that have all entries in $\mathbb F$.
We index the rows and columns by $0,1,\ldots, d$.
Let $\lbrace v_i\rbrace_{i=0}^d$
denote a basis for $V$. For $A \in {\rm End}(V)$
and $M\in 
{\rm Mat}_{d+1}(\mathbb F)$, we say that {\it $M$ represents $A$
with respect to 
 $\lbrace v_i\rbrace_{i=0}^d$} whenever 
 $Av_j = \sum_{i=0}^d M_{ij}v_i$ for $0 \leq j \leq d$.
Suppose we are given two bases for $V$,
denoted
$\lbrace u_i\rbrace_{i=0}^d$
and
$\lbrace v_i\rbrace_{i=0}^d$. By the {\it transition matrix
from
$\lbrace u_i\rbrace_{i=0}^d$
to $\lbrace v_i\rbrace_{i=0}^d$} we mean
the matrix $S \in  
{\rm Mat}_{d+1}(\mathbb F)$ such that
$v_j = \sum_{i=0}^d S_{ij}u_i$ for $0 \leq j \leq d$.
Let $S$ denote the transition matrix from
$\lbrace u_i\rbrace_{i=0}^d$
to $\lbrace v_i\rbrace_{i=0}^d$.
Then $S^{-1}$ exists and equals the
transition matrix from
$\lbrace v_i\rbrace_{i=0}^d$
to $\lbrace u_i\rbrace_{i=0}^d$.
Let $\lbrace w_i\rbrace_{i=0}^d$ denote a basis for $V$
and let $H$ denote the transition matrix from
$\lbrace v_i\rbrace_{i=0}^d$
to $\lbrace w_i\rbrace_{i=0}^d$.
Then $S H$ is the transition matrix from 
$\lbrace u_i\rbrace_{i=0}^d$
to $\lbrace w_i\rbrace_{i=0}^d$.
Let
 $A \in {\rm End}(V)$ and let
$M \in {\rm Mat}_{d+1}(\mathbb F)$ represent
$A$ with respect to 
$\lbrace u_i\rbrace_{i=0}^d$. Then
$S^{-1}MS$ represents $A$ with respect to
$\lbrace v_i\rbrace_{i=0}^d$.
Define a matrix ${\bf Z} \in
{\rm Mat}_{d+1}(\mathbb F)$ with $(i,j)$-entry $\delta_{i+j,d}$
for $0 \leq i,j\leq d$. For example if $d=3$,
\begin{eqnarray*}
{\bf Z}=
\left(
\begin{array}{ c c cc}
 0 & 0   & 0  &1  \\
 0  & 0  &  1  &0    \\ 
 0  &  1  &  0   & 0\\
  1    &  0 & 0 & 0 \\
	\end{array}
	\right).
\end{eqnarray*}
\noindent Note that ${\bf Z}^2=I$. Let
$\lbrace v_i\rbrace_{i=0}^d$ denote a basis for $V$
and consider the inverted basis
$\lbrace v_{d-i}\rbrace_{i=0}^d$. Then ${\bf Z}$
is the transition matrix from
$\lbrace v_i\rbrace_{i=0}^d$ to
$\lbrace v_{d-i}\rbrace_{i=0}^d$.
\medskip

\noindent 
By a {\it  decomposition of $V$} we mean a sequence
$\lbrace V_i\rbrace_{i=0}^d$ of one dimensional subspaces of
$V$ such that $V= \sum_{i=0}^d V_i$ (direct sum).
Given a decomposition $\lbrace V_i\rbrace_{i=0}^d$ of $V$,
for  notational
convenience define $V_{-1}=0$ and
$V_{d+1}=0$.
Let $\lbrace v_i \rbrace_{i=0}^d$ denote a basis for $V$.
For $0 \leq i \leq d$ let $V_i$ denote the
span of $v_i$. Then the sequence
 $\lbrace V_i \rbrace_{i=0}^d$ is a decomposition of $V$,
said to be  
 {\it induced} by the basis
 $\lbrace v_i \rbrace_{i=0}^d$.
Let 
 $\lbrace u_i \rbrace_{i=0}^d$ and 
 $\lbrace v_i \rbrace_{i=0}^d$ denote bases for 
$V$. Then the following are equivalent:
(i)
 the transition matrix from $\lbrace u_i \rbrace_{i=0}^d$ to 
 $\lbrace v_i \rbrace_{i=0}^d$ is diagonal;
 (ii) $\lbrace u_i \rbrace_{i=0}^d$ and
 $\lbrace v_i \rbrace_{i=0}^d$ induce the same decomposition of $V$.
\medskip

\noindent 
Let $\lbrace V_i\rbrace_{i=0}^d$ denote a decomposition of $V$.
For $0 \leq i \leq d$ define $E_i
 \in {\rm End}(V)$ such that
$(E_i - I)V_i = 0$ and
$E_iV_j=0$ for $0 \leq j \leq d$, $j\not=i$.
We call $E_i$ the $i$th {\it primitive idempotent}
for 
 $\lbrace V_i \rbrace_{i=0}^d$.
We have
(i) $E_i E_j = \delta_{i,j}E_i$ $(0 \leq i,j\leq d)$;
(ii) $I = \sum_{i=0}^d E_i$;
(iii) $V_i = E_iV$ $(0 \leq i \leq d)$;
(iv) ${\rm rank}(E_i) = 1 ={\rm tr}(E_i)$ $(0 \leq i \leq d)$,
where tr means trace.
We call $\lbrace E_i \rbrace_{i=0}^d$  the
{\it idempotent sequence} for 
 $\lbrace V_i \rbrace_{i=0}^d$.
Note that
$\lbrace E_{d-i}\rbrace_{i=0}^d$ is the idempotent
sequence for the decomposition
$\lbrace V_{d-i}\rbrace_{i=0}^d$.
\medskip

\noindent Let $\lbrace v_i \rbrace_{i=0}^d$ denote
a basis for $V$.
Let
$\lbrace V_i \rbrace_{i=0}^d$ denote the induced decomposition
of $V$, with idempotent sequence
$\lbrace E_i \rbrace_{i=0}^d$.
For $0 \leq r \leq d$ consider the matrix in
${\rm Mat}_{d+1}(\F)$ that represents $E_r$ with
respect to
$\lbrace v_i \rbrace_{i=0}^d$.
This matrix has $(r,r)$-entry 1 and all other entries 0.

\begin{lemma}
\label{lem:EiMeaning}
Let 
 $A \in {\rm End}(V)$.
Let
$\lbrace V_i \rbrace_{i=0}^d$ denote a decomposition of
$V$ with idempotent sequence
$\lbrace E_i \rbrace_{i=0}^d$.
Consider a basis for $V$ that induces
$\lbrace V_i \rbrace_{i=0}^d$.
 Let $M \in
{\rm Mat}_{d+1}(\F)$ represent 
$A$ with respect to  this basis.
Then for $0 \leq r,s\leq d$ the 
entry $M_{r,s}=0$ if and only if
$E_r A E_s = 0$.
\end{lemma}
\begin{proof} Represent
$A, E_r,E_s$ by matrices
with respect to the given basis.
\end{proof}

\noindent
By a {\it flag on $V$} we mean a sequence $\lbrace U_i \rbrace_{i=0}^d$
of subspaces of $V$ such that 
$U_i$ has dimension $i+1$ for
$0 \leq i \leq d$ and
$U_{i-1} \subseteq U_i$ for
$1 \leq i \leq d$.
For a flag 
$\lbrace U_i \rbrace_{i=0}^d$ on $V$ we have 
$U_d=V$.
Let 
$\lbrace V_i\rbrace_{i=0}^d$ denote a decomposition of $V$.
For $0 \leq i \leq d$ define $U_i = V_0 + \cdots + V_i$.
Then the sequence 
$\lbrace U_i \rbrace_{i=0}^d$ is a flag on $V$. 
This  flag is said to be {\it induced} by the decomposition
$\lbrace V_i\rbrace_{i=0}^d$.
Let $\lbrace u_i\rbrace_{i=0}^d$ denote a basis of $V$.
This basis induces a decomposition of $V$, which 
in turn
induces a flag on $V$. This flag is said to be {\it induced}
by the basis 
$\lbrace u_i\rbrace_{i=0}^d$.
Let $\lbrace u_i\rbrace_{i=0}^d$ and 
 $\lbrace v_i\rbrace_{i=0}^d$ denote bases of $V$.
Then the following are equivalant: 
(i) the transition matrix from
$\lbrace u_i\rbrace_{i=0}^d$ to
 $\lbrace v_i\rbrace_{i=0}^d$ is upper triangular;
(ii) 
$\lbrace u_i\rbrace_{i=0}^d$ and
 $\lbrace v_i\rbrace_{i=0}^d$ induce the same flag on $V$.
\medskip

\noindent 
Suppose we are given two flags on $V$, denoted
$\lbrace U_i \rbrace_{i=0}^d$ and
$\lbrace U'_i \rbrace_{i=0}^d$.
These flags are called {\it opposite} whenever
$U_i \cap U'_j = 0$ if $i+j<d$ $(0 \leq i,j\leq d)$.
The following are equivalent: (i) 
$\lbrace U_i \rbrace_{i=0}^d$ and
$\lbrace U'_i \rbrace_{i=0}^d$ are opposite;
(ii) there exists a decomposition $\lbrace V_i \rbrace_{i=0}^d$
of $V$
that induces 
$\lbrace U_i \rbrace_{i=0}^d$ and whose inversion
induces 
$\lbrace U'_i \rbrace_{i=0}^d$. In this case
$V_i = U_i \cap U'_{d-i}$ for $0 \leq i \leq d$.
\medskip

\noindent
Let
$\lbrace V_i\rbrace_{i=0}^d$ denote a decomposition of $V$.
For $A \in {\rm End}(V)$, we say that $A$ {\it lowers}
$\lbrace V_i\rbrace_{i=0}^d$ whenever
$A V_i = V_{i-1} $ for $1 \leq i \leq d$ and
$AV_0 = 0$.

\begin{lemma}
\label{lem:LowerRaise}
Let $\lbrace V_i\rbrace_{i=0}^d$ denote
a decomposition of $V$, with idempotent
sequence
 $\lbrace E_i\rbrace_{i=0}^d$.
For $A \in 
{\rm End}(V)$ the following are equivalent:
\begin{enumerate}
\item[\rm (i)] $A$ lowers
$\lbrace V_i\rbrace_{i=0}^d$;
\item[\rm (ii)] $E_i A E_j =
\begin{cases}
\not=0 &  {\mbox{\rm if $j-i=1$}}; \\
0 & {\mbox{\rm if $j-i\not=1$}}
\end{cases}
\qquad \qquad (0 \leq i,j \leq d)$.
\end{enumerate}
\end{lemma}
\begin{proof} Use
Lemma \ref{lem:EiMeaning}.
\end{proof}

\noindent 
Let
$\lbrace V_i\rbrace_{i=0}^d$ denote a decomposition of $V$
and let
$A \in {\rm End}(V)$.
Assume that 
$\lbrace V_i\rbrace_{i=0}^d$ is lowered by $A$.
Then $V_i = A^{d-i}V_d$ for $0 \leq i \leq d$.
Moreover $A^{d+1}=0$. 
For $0 \leq i \leq d$ the subspace
$V_0 + \cdots + V_i$ is the kernel of $A^{i+1}$ and
equal to $A^{d-i}V$.
In particular, $V_0$ is the kernel of $A$ and
equal to $A^dV$.
The sequences
$\lbrace {\rm ker}\,A^{i+1} \rbrace_{i=0}^d$ and
$\lbrace A^{d-i}V\rbrace_{i=0}^d$  both
equal  the
 flag on $V$  induced by
$\lbrace V_i\rbrace_{i=0}^d$.
We say that $A$ {\it raises}
$\lbrace V_i\rbrace_{i=0}^d$ whenever
$A V_i = V_{i+1} $ for $0 \leq i \leq d-1$ and
$AV_d = 0$. Note that $A$ raises
$\lbrace V_i\rbrace_{i=0}^d$ if and only if $A$
lowers the inverted decomposition
$\lbrace V_{d-i}\rbrace_{i=0}^d$.

\begin{lemma}
\label{lem:RaiseLower}
Let $\lbrace V_i\rbrace_{i=0}^d$ denote
a decomposition of $V$, with idempotent
sequence
 $\lbrace E_i\rbrace_{i=0}^d$.
For $A \in 
{\rm End}(V)$ the following are equivalent:
\begin{enumerate}
\item[\rm (i)] $A$ raises
$\lbrace V_i\rbrace_{i=0}^d$;
\item[\rm (ii)] $E_i A E_j =
\begin{cases}
\not=0 &  {\mbox{\rm if $i-j=1$}}; \\
0 & {\mbox{\rm if $i-j\not=1$}}
\end{cases}
\qquad \qquad (0 \leq i,j \leq d)$.
\end{enumerate}
\end{lemma}
\begin{proof} Apply
Lemma
\ref{lem:LowerRaise}
to the decomposition
$\lbrace V_{d-i}\rbrace_{i=0}^d$.
\end{proof}

\begin{definition}
\label{def:Nil}
\rm
An element $A \in 
	      {\rm End}(V)$ will be called {\it Nil} whenever
	      $A^{d+1}=0$ and $A^d \not=0$.
\end{definition}

\begin{lemma} 
\label{lem:NilRec}
For 
$A \in 
 {\rm End}(V)$ the following are equivalent:
\begin{enumerate}
\item[\rm (i)] $A$ is Nil;
\item[\rm (ii)] there exists a decomposition of $V$ that
is lowered by $A$;
\item[\rm (iii)] there exists a decomposition of $V$ that
is raised by $A$;
\item[\rm (iv)]  for $0 \leq i \leq d$ the kernel of
$A^{i+1}$ is $A^{d-i}V$;
\item[\rm (v)]  the kernel of
$A$ is $A^dV$;
\item[\rm (vi)] the sequence
$\lbrace {\rm ker}\,A^{i+1} \rbrace_{i=0}^d$ is a flag on $V$.
\end{enumerate}
\end{lemma}
\begin{proof} 
${\rm (i)}\Rightarrow {\rm (ii)}$ 
By assumption there 
exists $v \in V$ such that $A^dv\not=0$.
By assumption $A^{d+1}v=0$.
Define $v_i=A^{d-i}v$ for $0 \leq i \leq d$.
Then $Av_i= v_{i-1}$ for
$1 \leq i \leq d$ and $Av_0=0$.
By these comments, for $0 \leq i \leq d$ the vector
$v_i$ is in the kernel of $A^{i+1}$ and
not in the kernel of $A^i$. Therefore
 $\lbrace v_i \rbrace_{i=0}^d$
are linearly independent, and hence form a basis
for $V$. By construction the induced decomposition
 of $V$ is lowered by $A$.
\\
\noindent 
${\rm (ii)}\Leftrightarrow {\rm (iii)}$
A decomposition of $V$ is raised by $A$ if and only if
its inversion is lowered by $A$.
\\
\noindent 
${\rm (ii)}\Rightarrow {\rm (iv)}$
By the comments above Lemma
\ref{lem:RaiseLower}.
\\
\noindent 
${\rm (iv)}\Rightarrow {\rm (v)}$ Clear.
\\
\noindent 
${\rm (v)}\Rightarrow {\rm (i)}$ 
Observe that $A^{d+1}V=A(A^dV)=0$, so
$A^{d+1}=0$.
The map 
$A$ is not invertible, so $A$ has nonzero kernel.
This kernel is $A^dV$, so
 $A^dV\not=0$. Therefore $A^d\not=0$.
So $A$ is Nil by Definition
\ref{def:Nil}.
\\
\noindent 
${\rm (ii)}\Rightarrow {\rm (vi)}$
By the comments above Lemma
\ref{lem:RaiseLower}.
\\
\noindent 
${\rm (vi)}\Rightarrow {\rm (i)}$  
For $0 \leq i \leq d$ let $U_i$ denote the
kernel of $A^{i+1}$. 
By assumption $\lbrace U_i\rbrace_{i=0}^d$ is a flag on $V$.
We have $U_d=V$,
so $A^{d+1}=0$. We have $U_{d-1}\not=V$, so $A^d\not=0$.
Therefore $A$ is Nil by Definition
\ref{def:Nil}.
\end{proof}

\noindent We emphasize a point from
 Lemma
\ref{lem:NilRec}. 
For a Nil element $A \in 
	      {\rm End}(V)$ 
the sequence $\lbrace A^{d-i}V\rbrace_{i=0}^d$ is a flag on 
$V$.

\section{LR pairs}

\noindent In this paper, our main topic 
is the notion of an LR triple.
As a warmup,
we first consider the
notion of an LR pair.

\medskip

\noindent Throughout this section $V$ denotes a vector
space over $\mathbb F$ with dimension $d+1$.

\begin{definition}
\label{def:lr}
\rm
An ordered pair $A,B$ of elements in ${\rm End}(V)$
is called {\it lowering-raising} (or {\it  LR})
whenever there exists a decomposition
of $V$ that
is lowered by $A$ and raised by $B$.
We refer to such a pair as an 
 {\it LR pair on $V$}.
This LR pair is said to be {\it over $\mathbb F$}.
We call $V$ the {\it underlying vector space}.
We call $d$ the {\it diameter} of the pair.
\end{definition}

\begin{lemma} 
Let $A,B$ denote an LR pair on $V$. Then 
$B,A$ is an LR pair on $V$. 
\end{lemma}

\begin{lemma} 
\label{lem:ABdecIndNil}
Let $A,B$ denote an LR pair on $V$. Then 
each of $A,B$ is  Nil.
\end{lemma}

\noindent We mention a very special case.

\begin{example} 
\label{def:triv}
\rm Assume that $d=0$.
Then $A,B \in {\rm End}(V)$ form
an LR pair if and only if $A=0$ and $B=0$.
This LR pair will be called {\it trivial}.
\end{example}

\noindent 
Let $A,B$ denote an LR pair on $V$. By Definition
\ref{def:lr},
there exists a decomposition $\lbrace V_i\rbrace_{i=0}^d$
of $V$ that is lowered by $A$ and raised by $B$.
We have
$V_i = A^{d-i}V_d = B^i V_0$
for
 $0 \leq i \leq d$. 
Moreover 
$V_0=A^dV$ and $V_d=B^dV$.
Therefore $V_i = A^{d-i}B^dV= B^iA^dV$
for $0 \leq i \leq d$.
The decomposition
$\lbrace V_i\rbrace_{i=0}^d$ is uniquely determined by
$A,B$;  we call 
$\lbrace V_i\rbrace_{i=0}^d$ the {\it $(A,B)$-decomposition of
$V$}. Its inversion
$\lbrace V_{d-i}\rbrace_{i=0}^d$ is the
$(B,A)$-decomposition of $V$.
\medskip

\begin{definition}
\label{def:ABE}
\rm Let $A,B$ denote an LR pair on $V$.
By the {\it idempotent sequence} 
for $A,B$ we mean
the idempotent sequence for the $(A,B)$-decomposition of $V$.
\end{definition}

\noindent We have some comments.

\begin{lemma} 
\label{lem:Ebackward}
Let $A,B$ denote an LR pair on $V$, with
idempotent sequence $\lbrace E_i \rbrace_{i=0}^d$.
Then the LR pair $B,A$ has idempotent sequence
$\lbrace E_{d-i} \rbrace_{i=0}^d$.
\end{lemma}

\begin{lemma}
\label{lem:ABdecInd}
Let $A,B$ denote an LR pair on $V$.
The $(A,B)$-decomposition of $V$ induces the flag 
$\lbrace A^{d-i}V\rbrace_{i=0}^d$.
The $(B,A)$-decomposition of $V$ induces the flag 
$\lbrace B^{d-i}V\rbrace_{i=0}^d$.
The flags 
$\lbrace A^{d-i}V\rbrace_{i=0}^d$ and
 $\lbrace B^{d-i}V\rbrace_{i=0}^d$ are 
opposite.
 \end{lemma}

\begin{lemma}
\label{lem:alphaBeta}
Let $A,B$ denote an LR pair on $V$.
For nonzero $\alpha, \beta \in \mathbb F$
the pair $\alpha A, \beta B$ is an LR pair
on $V$. The 
$(\alpha A, \beta B)$-decomposition of
$V$ is equal to the
$(A,B)$-decomposition of $V$.
Moreover the
idempotent sequence for
$\alpha A, \beta B$ is equal to
the idempotent sequence for
$A, B$.
\end{lemma}

\begin{lemma}
\label{lem:ABAaction}
Let $A,B$ denote an LR pair on $V$.
For $0 \leq r,s\leq d$, consider the action of
the map $A^rB^dA^s$  on the $(A,B)$-decomposition of $V$.
The map sends the $s$-component onto the $(d-r)$-component.
The map sends all other components to zero.
\end{lemma}


\begin{lemma}
\label{lem:ABAbasis}
Let $A,B$ denote an LR pair on $V$. Then the following
 is a basis for the $\mathbb F$-vector space 
${\rm End}(V)$:
\begin{eqnarray}
\label{eq:ABAbasis}
A^r B^d A^s \qquad 0 \leq r,s\leq d.
\end{eqnarray}
\end{lemma}
\begin{proof} 
The dimension of ${\rm End}(V)$ is $(d+1)^2$.
The list 
(\ref{eq:ABAbasis}) contains
$(d+1)^2$ elements, and these are linearly independent
by Lemma
\ref{lem:ABAaction}. The result follows.
\end{proof}

\begin{corollary}
\label{cor:ABgen} 
 Let $A,B$ denote an LR pair on $V$. Then
the $\mathbb F$-algebra 
${\rm End}(V)$ is generated by $A,B$.
\end{corollary}
\begin{proof} By
 Lemma
\ref{lem:ABAbasis}.
\end{proof}

 \begin{lemma}
\label{lem:paPre} 
 Let $A,B$ denote an LR pair on $V$. Let
 $\lbrace V_i\rbrace_{i=0}^d$ denote the $(A,B)$-decomposition
 of 
 $V$. Then the following {\rm (i)--(iv)} hold.
 \begin{enumerate}
 \item[\rm (i)]
 For $0 \leq i \leq d$ the subspace
 $V_i$ is invariant under $AB$ and $BA$.
 \item[\rm (ii)]
  The map $BA$ is zero on $V_0$.
 \item[\rm (iii)]
 The map $AB$ is zero on $V_d$.
 \item[\rm (iv)]
For $1 \leq i \leq d$, the eigenvalue of
 $AB$ on $V_{i-1}$ is nonzero and equal to the eigenvalue
 of $BA$ on $V_i$.
 \end{enumerate}
\end{lemma}
\begin{proof} (i)--(iii) The decomposition $\lbrace V_i\rbrace_{i=0}^d$
is lowered by $A$ and raised by $B$.
\\
\noindent (iv) Pick $0 \not= u \in V_{i-1}$  and
$0 \not=v \in V_i$. There exist nonzero $r,s \in \mathbb F$
such that $Av=r u$ and $Bu=sv$. The scalar $rs$
is the eigenvalue of $AB$ on $V_{i-1}$, and
the eigenvalue of $BA$ on $V_i$.
\end{proof}

 \begin{definition}
\label{def:pa} 
 \rm  Let $A,B$ denote an LR pair on $V$. Let
 $\lbrace V_i\rbrace_{i=0}^d$ denote the $(A,B)$-decomposition
 of 
 $V$. For $1 \leq i \leq d$ let $\varphi_i$
 denote the eigenvalue referred to in Lemma
\ref{lem:paPre}(iv). 
 Thus $0 \not=\varphi_i\in \mathbb F$.
The sequence $\lbrace \varphi_i\rbrace_{i=1}^d$ is called the
 {\it parameter sequence} for $A,B$.
 For notational convenience define $\varphi_0=0$ and $\varphi_{d+1}=0$.
 \end{definition}

\begin{lemma} 
\label{lem:BAvAB}
Let
$A,B$ denote an LR pair on $V$, with parameter sequence
$\lbrace \varphi_i\rbrace_{i=1}^d$. Then the LR pair
 $B,A$ has parameter sequence
$\lbrace \varphi_{d-i+1}\rbrace_{i=1}^d$.
\end{lemma}
\begin{proof} Use Lemma
\ref{lem:paPre}  and Definition
\ref{def:pa}.
\end{proof}

\noindent Here is an example of an LR pair.

\begin{example} 
\label{ex:LR}
\rm
Let $\lbrace \varphi_i\rbrace_{i=1}^d$ denote a sequence of
nonzero scalars in $\mathbb F$.
Let $\lbrace v_i\rbrace_{i=0}^d$ denote a basis for $V$.
Define $A \in {\rm End}(V)$ such that
$Av_i = \varphi_i v_{i-1}$ for $1 \leq i \leq d$ and
$Av_0=0$. Define
 $B \in {\rm End}(V)$ such that
 $Bv_{i} =   v_{i+1}$ for $0 \leq i \leq d-1$ and
 $Bv_d=0$. 
Then the pair $A,B$ is an LR pair on $V$, with
parameter sequence
$\lbrace \varphi_i \rbrace_{i=1}^d$.
The $(A,B)$-decomposition of $V$ is induced by the
basis $\lbrace v_i \rbrace_{i=0}^d$.
\end{example}

\noindent Let $A,B$ denote an LR pair on $V$,
with idempotent sequence $\lbrace E_i\rbrace_{i=0}^d$.
Our next goal is to obtain each $E_i$ in terms of
$A,B$.

\begin{lemma}
\label{lem:AiBione}
Let $A,B$ denote an LR pair on $V$, with
parameter sequence $\lbrace \varphi_i\rbrace_{i=1}^d $
and idempotent sequence $\lbrace E_i\rbrace_{i=0}^d$.
Then
\begin{eqnarray}
&&
\label{eq:AiBione}
AB = \sum_{j=0}^{d-1} E_j \varphi_{j+1},
\qquad \qquad 
\label{eq:BiAione}
BA = \sum_{j=1}^d E_j \varphi_{j}.
\end{eqnarray}
\end{lemma} 
\begin{proof} Use
Definitions
\ref{def:ABE},
\ref{def:pa}.
\end{proof}

\noindent The following result is a generalization of
 Lemma
\ref{lem:AiBione}.

\begin{lemma}
\label{lem:AiBi}
Let $A,B$ denote an LR pair on $V$, with
parameter sequence $\lbrace \varphi_i\rbrace_{i=1}^d $
and idempotent sequence $\lbrace E_i\rbrace_{i=0}^d$.
Then for $0 \leq r \leq d$, 
\begin{eqnarray}
&&
\label{eq:AiBi}
A^rB^r = \sum_{j=0}^{d-r} E_j \varphi_{j+1}\varphi_{j+2} \cdots \varphi_{j+r},
\\
&&
\label{eq:BiAi}
B^rA^r = \sum_{j=r}^{d} E_j \varphi_{j}\varphi_{j-1} \cdots \varphi_{j-r+1}.
\end{eqnarray}
\end{lemma}
\begin{proof} 
To verify 
(\ref{eq:AiBi}), note that for $0 \leq i \leq d$, the two sides agree on
component $i$ of the $(A,B)$-decomposition of $V$.
Line
(\ref{eq:BiAi}) is similarly verified.
\end{proof}

\begin{lemma}
\label{lem:Eform}
Let $A,B$ denote an LR pair on $V$,
with idempotent sequence
$\lbrace E_i\rbrace_{i=0}^d$. Then 
for $0 \leq i \leq d$,
\begin{eqnarray}
\label{eq:TwoE}
E_i = \frac{A^{d-i}B^d A^i}{\varphi_1 \cdots \varphi_d},
\qquad \qquad 
E_i = \frac{B^{i}A^d B^{d-i}}{\varphi_1 \cdots \varphi_d},
\end{eqnarray}
where $\lbrace \varphi_j \rbrace_{j=1}^d$ is the 
parameter sequence for $A,B$.
\end{lemma}
\begin{proof} 
To obtain the formula on the left in
(\ref{eq:TwoE}),
in the equation
$A^{d-i}B^d A^i=A^{d-i}B^{d-i}B^iA^i$, evaluate
the right-hand side using
Lemma
\ref{lem:AiBi} and simplify the result using
$E_r E_s = \delta_{r,s}E_r$ $(0 \leq r,s\leq d)$.
The formula on the right in
(\ref{eq:TwoE}) is similarly obtained.
\end{proof}

\begin{lemma}
\label{lem:zeroprod}
Let $A,B$ denote an LR pair on $V$, with
idempotent sequence 
$\lbrace E_i \rbrace_{i=0}^d$.
Then for $0 \leq i<j\leq d$ the following are zero:
\begin{eqnarray*}
A^j E_i,
\qquad \qquad 
E_j A^{d-i},
\qquad \qquad 
E_i B^j,
\qquad \qquad 
B^{d-i} E_j.
\end{eqnarray*}
\end{lemma}
\begin{proof}
By 
(\ref{eq:TwoE}) together with $A^{d+1}=0$ and $B^{d+1}=0$.
\end{proof}

\begin{lemma} 
\label{lem:varphiTrace}
Let $A,B$ denote an LR pair on $V$, with
parameter sequence 
$\lbrace \varphi_i \rbrace_{i=1}^d$  and
idempotent sequence 
$\lbrace E_i \rbrace_{i=0}^d$.
Then for $0 \leq i \leq d$,
\begin{eqnarray}
{\rm tr}(AB E_i)=\varphi_{i+1},
\qquad \qquad 
{\rm tr}(BAE_i)=\varphi_{i}.
\label{eq:traceE}
\end{eqnarray}
\end{lemma}
\begin{proof}  
In the equation on the left in
(\ref{eq:AiBione}),
  multiply each side on the right by $E_i$
to get $ABE_i = \varphi_{i+1} E_i$. In this equation, take
the trace of each side, and recall that $E_i$ has trace 1.
This gives the
equation on the left in
(\ref{eq:traceE}). The other equation
in (\ref{eq:traceE})
is similarly verified.
\end{proof}

\noindent Let $A,B$ denote an LR pair on $V$. We now
describe a set of bases for $V$, called
$(A,B)$-bases.

\begin{definition}
\label{def:ABbasis}
\rm
Let $A,B$ denote an LR pair on $V$. 
 Let $\lbrace V_i \rbrace_{i=0}^d$ denote
the 
$(A,B)$-decomposition of $V$.
A basis 
$\lbrace v_i \rbrace_{i=0}^d$ for $V$
is called an {\it $(A,B)$-basis} whenever:
\begin{enumerate}
\item[\rm (i)]
 $ v_i \in V_i$ for $0 \leq i \leq d$;
\item[\rm (ii)] 
 $Av_i = v_{i-1}$ for $1 \leq i \leq d$.
\end{enumerate}
\end{definition}

\begin{lemma}
\label{lem:BAbasisU}
Let $A,B$ denote an LR pair on $V$. 
Let $\lbrace v_i \rbrace_{i=0}^d$ denote an $(A,B)$-basis
for $V$, and let $\lbrace v'_i\rbrace_{i=0}^d$ denote
any vectors in $V$.
Then the following are equivalent:
\begin{enumerate}
\item[\rm (i)] 
$\lbrace v'_i\rbrace_{i=0}^d$ is an $(A,B)$-basis for $V$;
\item[\rm (ii)] there exists $0 \not=\zeta \in \mathbb F$
such that $v'_i = \zeta v_i $ for $0 \leq i \leq d$.
\end{enumerate}
\end{lemma}
\begin{proof} Use Definition
\ref{def:ABbasis}.
\end{proof}

\begin{lemma}
\label{lem:ABmatrix}
Let $A,B$ denote an LR pair on $V$, with
parameter sequence $\lbrace \varphi_i\rbrace_{i=1}^d$.
Let $\lbrace v_i \rbrace_{i=0}^d$ denote a basis for $V$.
Then the following are equivalent:
\begin{enumerate}
\item[\rm (i)] 
$\lbrace v_i \rbrace_{i=0}^d$ is an
$(A,B)$-basis for $V$;
\item[\rm (ii)]
with respect to 
$\lbrace v_i \rbrace_{i=0}^d$ the
matrices representing
$A$ and $B$ are
\begin{eqnarray}
\label{eq:ABrep}
A:\; 
\left(
\begin{array}{ c c cc c c }
 0 & 1   &   &&   & \bf 0  \\
  & 0  &  1  &&  &      \\ 
   &   &  0   & \cdot &&  \\
     &   &  & \cdot  & \cdot & \\
       &  & &  & \cdot & 1 \\
        {\bf 0}  &&  & &   &  0  \\
	\end{array}
	\right),
	\qquad \qquad
	B:\;
	\left(
	\begin{array}{ c c cc c c }
	0 &   &   &&   & \bf 0  \\
	\varphi_1 & 0  &   &&  &      \\
	 &  \varphi_2 & 0   & &&  \\
	   &   & \cdot & \cdot  & & \\
	     &  & & \cdot & \cdot & \\
	      {\bf 0}  &&  & &  \varphi_d & 0  \\
	      \end{array}
	      \right).
	      \end{eqnarray}
\end{enumerate}
\end{lemma}
\begin{proof}
${\rm (i)}\Rightarrow {\rm (ii)}$ 
Use Definitions
\ref{def:pa},
\ref{def:ABbasis}.
\\
${\rm (ii)}\Rightarrow {\rm (i)}$ 
Let $\lbrace V_i \rbrace_{i=0}^d$ denote the
decomposition of $V$ induced by 
$\lbrace v_i \rbrace_{i=0}^d$.
By (\ref{eq:ABrep}),  
 $\lbrace V_i \rbrace_{i=0}^d$ is lowered by
 $A$ and raised by $B$.
Therefore 
 $\lbrace V_i \rbrace_{i=0}^d$ is the $(A,B)$-decomposition
 of $V$.
Now by Definition
\ref{def:ABbasis},
$\lbrace v_i \rbrace_{i=0}^d$ is an
$(A,B)$-basis for $V$.
\end{proof}

\begin{lemma}
\label{lem:5Char}
Let $A,B$ denote an LR pair on $V$,
with parameter sequence $\lbrace \varphi_i \rbrace_{i=1}^d$.
Let $\lbrace v_i\rbrace_{i=0}^d$ denote any vectors in $V$.
Then the following are equivalent:
\begin{enumerate}
\item[\rm (i)] 
$\lbrace v_i\rbrace_{i=0}^d$ is an $(A,B)$-basis for $V$;
\item[\rm (ii)] $0 \not= v_0 \in A^dV$ and $Bv_i = \varphi_{i+1}v_{i+1}$ for
$0 \leq i \leq d-1$;
\item[\rm (iii)] there exists $0 \not= \eta \in A^dV$ such that
$v_i = (\varphi_1 \varphi_2 \cdots \varphi_i)^{-1}B^i \eta$ for
$0 \leq i \leq d$;
\item[\rm (iv)] $0 \not=v_d \in B^dV$ and $Av_i = v_{i-1}$ for
$1 \leq i \leq d$;
\item[\rm (v)] there exists $0 \not= \xi \in B^dV$ such that
$v_i = A^{d-i} \xi$ for
$0 \leq i \leq d$.
\end{enumerate}
\end{lemma}
\begin{proof} Use
Lemma \ref{lem:ABmatrix}.
\end{proof}

\noindent Let $A,B$ denote an LR pair on $V$.
By an {\it inverted $(A,B)$-basis for $V$}
we mean the inversion of an $(A,B)$-basis for $V$.

\begin{lemma}
\label{lem:invABbasis}
Let $A,B$ denote an LR pair on $V$. 
Let $\lbrace V_i\rbrace_{i=0}^d$ denote the
$(A,B)$-decomposition of $V$.
A basis $\lbrace v_i \rbrace_{i=0}^d$ for $V$ is
an inverted $(A,B)$-basis if and only if both
\begin{enumerate}
\item[\rm (i)] $v_i \in V_{d-i}$ for $0 \leq i \leq d$;
\item[\rm (ii)] $Av_i = v_{i+1}$ for $0 \leq i \leq d-1$.
\end{enumerate}
\end{lemma}
\begin{proof} By Definition
\ref{def:ABbasis}
and the meaning of inversion.
\end{proof}

\begin{lemma}
\label{lem:invABmatrix}
Let $A,B$ denote an LR pair on $V$, with
parameter sequence $\lbrace \varphi_i\rbrace_{i=1}^d$.
Let $\lbrace v_i \rbrace_{i=0}^d$ denote a basis
for $V$. Then the following are equivalent:
\begin{enumerate}
\item[\rm (i)]
$\lbrace v_i \rbrace_{i=0}^d$ is an
inverted $(A,B)$-basis for $V$;
\item[\rm (ii)] with respect to
$\lbrace v_i \rbrace_{i=0}^d$ 
 the matrices representing $A$ and $B$ are
\begin{eqnarray}
\label{eq:invABrep}
	A:\;
	\left(
	\begin{array}{ c c cc c c }
	0 &   &   &&   & \bf 0  \\
	1 & 0  &   &&  &      \\
	 &  1 & 0   & &&  \\
	   &   & \cdot & \cdot  & & \\
	     &  & & \cdot & \cdot & \\
	      {\bf 0}  &&  & &  1 & 0  \\
	      \end{array}
	      \right),
	      \qquad \qquad
B:\; 
\left(
\begin{array}{ c c cc c c }
 0 & \varphi_d   &   &&   & \bf 0  \\
  & 0  &  \varphi_{d-1}  &&  &      \\ 
   &   &  0   & \cdot &&  \\
     &   &  & \cdot  & \cdot & \\
       &  & &  & \cdot & \varphi_1 \\
        {\bf 0}  &&  & &   &  0  \\
	\end{array}
	\right).
	      \end{eqnarray}
\end{enumerate}
\end{lemma}
\begin{proof} By Lemma
\ref{lem:ABmatrix} and
the meaning of inversion.
\end{proof}

\begin{lemma}
\label{lem:5CharInv}
Let $A,B$ denote an LR pair on $V$,
with parameter sequence $\lbrace \varphi_i \rbrace_{i=1}^d$.
Let $\lbrace v_i\rbrace_{i=0}^d$ denote any vectors in $V$.
Then the following are equivalent:
\begin{enumerate}
\item[\rm (i)] 
$\lbrace v_i\rbrace_{i=0}^d$ is an inverted $(A,B)$-basis for $V$;
\item[\rm (ii)] $0 \not=v_0 \in B^dV$ and $Av_i = v_{i+1}$ for
$0 \leq i \leq d-1$;
\item[\rm (iii)] there exists $0 \not= \xi \in B^dV$ such that
$v_i = A^i \xi$ for
$0 \leq i \leq d$;
\item[\rm (iv)] $0 \not=v_d \in A^dV$ and 
$Bv_i = \varphi_{d-i+1}v_{i-1}$
for
$1 \leq i \leq d$;
\item[\rm (v)] there exists $0 \not= \eta \in A^dV$ such that
$v_i = (\varphi_1 \varphi_{2} \cdots \varphi_{d-i})^{-1}B^{d-i} \eta$ for
$0 \leq i \leq d$.
\end{enumerate}
\end{lemma}
\begin{proof} By
Lemma
\ref{lem:5Char}
and the meaning of inversion.
\end{proof}

\noindent Let $A,B$ denote an LR pair on $V$.
We now consider a
$(B,A)$-basis for $V$.

\begin{lemma}
\label{lem:BAbasis}
Let $A,B$ denote an LR pair on $V$. 
Let $\lbrace V_i\rbrace_{i=0}^d$ denote the
$(A,B)$-decomposition of $V$.
A basis $\lbrace v_i \rbrace_{i=0}^d$ for $V$ is
a $(B,A)$-basis if and only if both
\begin{enumerate}
\item[\rm (i)] $v_i \in V_{d-i}$ for $0 \leq i \leq d$;
\item[\rm (ii)] $Bv_i = v_{i-1}$ for $1 \leq i \leq d$.
\end{enumerate}
\end{lemma}
\begin{proof} Apply Definition
\ref{def:ABbasis}
to the LR pair $B,A$.
\end{proof}

\begin{lemma}
\label{lem:BAbasisMat}
Let $A,B$ denote an LR pair on $V$, with
parameter sequence $\lbrace \varphi_i\rbrace_{i=1}^d$.
Let $\lbrace v_i \rbrace_{i=0}^d$ denote a basis for $V$.
Then the following are equivalent:
\begin{enumerate}
\item[\rm (i)]
$\lbrace v_i \rbrace_{i=0}^d$ is a $(B,A)$-basis for $V$;
\item[\rm (ii)]
with respect to
$\lbrace v_i \rbrace_{i=0}^d$ the matrices representing
 $A$ and $B$ are
\begin{eqnarray}
\label{eq:BArep}
A:\;
	\left(
	\begin{array}{ c c cc c c }
	0 &   &   &&   & \bf 0  \\
	\varphi_d & 0  &   &&  &      \\
	 &  \varphi_{d-1} & 0   & &&  \\
	   &   & \cdot & \cdot  & & \\
	     &  & & \cdot & \cdot & \\
	      {\bf 0}  &&  & &  \varphi_1 & 0  \\
	      \end{array}
	      \right),
	      \qquad \qquad
B:\; 
\left(
\begin{array}{ c c cc c c }
 0 & 1 &   &&   & \bf 0  \\
  & 0  &  1 &&  &      \\ 
   &   &  0   & \cdot &&  \\
     &   &  & \cdot  & \cdot & \\
       &  & &  & \cdot & 1 \\
        {\bf 0}  &&  & &   &  0  \\
	\end{array}
	\right).
\end{eqnarray}
\end{enumerate}
\end{lemma}
\begin{proof} Apply
Lemma
\ref{lem:ABmatrix} to the LR pair $B,A$.
\end{proof}

\begin{lemma}
\label{lem:5CharBA}
Let $A,B$ denote an LR pair on $V$,
with parameter sequence $\lbrace \varphi_i \rbrace_{i=1}^d$.
Let $\lbrace v_i\rbrace_{i=0}^d$ denote any vectors in $V$.
Then the following are equivalent:
\begin{enumerate}
\item[\rm (i)] 
$\lbrace v_i\rbrace_{i=0}^d$ is a $(B,A)$-basis for $V$;
\item[\rm (ii)] $0 \not=v_0 \in B^dV$ and $Av_i = \varphi_{d-i}
 v_{i+1}$ for
$0 \leq i \leq d-1$;
\item[\rm (iii)] there exists $0 \not= \xi \in B^dV$ such that
$v_i = (\varphi_d \varphi_{d-1} \cdots \varphi_{d-i+1})^{-1} A^i
\xi$ for
$0 \leq i \leq d$;
\item[\rm (iv)] $0 \not=v_d \in A^dV$ and $Bv_i = 
v_{i-1}$ for
$1 \leq i \leq d$;
\item[\rm (v)] there exists $0 \not= \eta \in A^dV$ such that
$v_i = B^{d-i} \eta$ for
$0 \leq i \leq d$.
\end{enumerate}
\end{lemma}
\begin{proof} Apply Lemma
\ref{lem:5Char} to the LR pair $B,A$.
\end{proof}

\noindent Let $A,B$ denote an LR pair on $V$.
We now consider an inverted $(B,A)$-basis for $V$.

\begin{lemma}
Let $A,B$ denote an LR pair on $V$. 
Let $\lbrace V_i\rbrace_{i=0}^d$ denote the
$(A,B)$-decomposition of $V$.
A basis $\lbrace v_i \rbrace_{i=0}^d$ for $V$ is
an inverted $(B,A)$-basis if and only if both
\begin{enumerate}
\item[\rm (i)] $v_i \in V_{i}$ for $0 \leq i \leq d$;
\item[\rm (ii)] $Bv_i = v_{i+1}$ for $0 \leq i \leq d-1$.
\end{enumerate}
\end{lemma}
\begin{proof} Apply Lemma
\ref{lem:invABbasis}
to the LR pair $B,A$.
\end{proof}

\begin{lemma}
\label{lem:BAmatrixInv}
Let $A,B$ denote an LR pair on $V$, with
parameter sequence $\lbrace \varphi_i\rbrace_{i=1}^d$.
Let $\lbrace v_i \rbrace_{i=0}^d$ denote a basis for 
$V$. Then the following are equivalent:
\begin{enumerate}
\item[\rm (i)] 
$\lbrace v_i \rbrace_{i=0}^d$ is an inverted 
 $(B,A)$-basis for $V$;
\item[\rm (ii)] 
with respect to
 $\lbrace v_i \rbrace_{i=0}^d$ 
 the matrices representing  $A$ and $B$ are
\begin{eqnarray}
\label{eq:invBArep}
A:\; 
\left(
\begin{array}{ c c cc c c }
 0 & \varphi_1   &   &&   & \bf 0  \\
  & 0  &  \varphi_2  &&  &      \\ 
   &   &  0   & \cdot &&  \\
     &   &  & \cdot  & \cdot & \\
       &  & &  & \cdot & \varphi_d \\
        {\bf 0}  &&  & &   &  0  \\
	\end{array}
	\right),
	\qquad \qquad
	B:\;
	\left(
	\begin{array}{ c c cc c c }
	0 &   &   &&   & \bf 0  \\
	1 & 0  &   &&  &      \\
	 &  1 & 0   & &&  \\
	   &   & \cdot & \cdot  & & \\
	     &  & & \cdot & \cdot & \\
	      {\bf 0}  &&  & &  1 & 0  \\
	      \end{array}
	      \right).
	      \end{eqnarray}
\end{enumerate}
\end{lemma}
\begin{proof} Apply Lemma
\ref{lem:invABmatrix}
to the LR pair $B,A$.
\end{proof}

\begin{lemma}
\label{lem:5CharBAinv}
Let $A,B$ denote an LR pair on $V$,
with parameter sequence $\lbrace \varphi_i \rbrace_{i=1}^d$.
Let $\lbrace v_i\rbrace_{i=0}^d$ denote any vectors in $V$.
Then the following are equivalent:
\begin{enumerate}
\item[\rm (i)] 
$\lbrace v_i\rbrace_{i=0}^d$ is an inverted $(B,A)$-basis for $V$;
\item[\rm (ii)] $0 \not=v_0 \in A^dV$ and $Bv_i = v_{i+1}$
for $0 \leq i \leq d-1$;
\item[\rm (iii)] there exists $0 \not= \eta \in A^dV$ such that
$v_i = B^{i} \eta$ for
$0 \leq i \leq d$;
\item[\rm (iv)] $0 \not=v_d \in B^dV$ and $Av_i = 
\varphi_i v_{i-1}$ for
$1 \leq i \leq d$;
\item[\rm (v)] there exists $0 \not= \xi \in B^dV$ such that
$v_i = (\varphi_d \varphi_{d-1} \cdots \varphi_{i+1})^{-1}
A^{d-i} \xi$ for
$0 \leq i \leq d$.
\end{enumerate}
\end{lemma}
\begin{proof} Apply Lemma
\ref{lem:5CharInv}
to the LR pair $B,A$.
\end{proof}

\noindent Let $A,B$ denote an LR pair on $V$. Earlier we used
$A,B$
to obtain four bases for $V$.
We now consider some transitions between these bases.

\begin{lemma}
\label{lem:sometrans}
Let $A,B$ denote an LR pair on $V$, with
parameter sequence $\lbrace \varphi_i\rbrace_{i=1}^d$.
\begin{enumerate}
\item[\rm (i)]
Let $\lbrace v_i \rbrace_{i=0}^d$ denote
an $(A,B)$-basis for $V$. Then 
the sequence
$\lbrace \varphi_1 \varphi_2\cdots \varphi_i v_i\rbrace_{i=0}^d$
is an inverted $(B,A)$-basis for $V$.
\item[\rm (ii)]
Let $\lbrace v_i \rbrace_{i=0}^d$ denote
an inverted $(B,A)$-basis for $V$. Then 
$\lbrace (\varphi_1 \varphi_2\cdots \varphi_i )^{-1}v_i\rbrace_{i=0}^d$
is an $(A,B)$-basis for $V$.
\item[\rm (iii)]
Let $\lbrace v_i \rbrace_{i=0}^d$ denote
a $(B,A)$-basis for $V$. Then 
the sequence
$\lbrace (\varphi_1 \varphi_{2}\cdots \varphi_{d-i})^{-1}
v_i\rbrace_{i=0}^d$
is an inverted  $(A,B)$-basis for $V$.
\item[\rm (iv)]
Let $\lbrace v_i \rbrace_{i=0}^d$ denote
an inverted $(A,B)$-basis for $V$. Then 
$\lbrace \varphi_1 \varphi_{2}\cdots \varphi_{d-i} 
v_i\rbrace_{i=0}^d$
is a $(B,A)$-basis for $V$.
\end{enumerate}
\end{lemma}
\begin{proof} (i), (ii) Compare
Lemma \ref{lem:5Char}(iii) and Lemma
\ref{lem:5CharBAinv}(iii). \\
\noindent (iii), (iv) Compare
Lemma \ref{lem:5CharInv}(v) and Lemma
\ref{lem:5CharBA}(v).
\end{proof}

\noindent We now discuss isomorphisms for
LR pairs.

\begin{definition}
\label{def:isoLRP}
\rm
Let $A,B$ denote an LR pair on $V$.
Let $V'$ denote a vector space over $\mathbb F$ with
dimension $d+1$, and let $A',B'$ denote an LR pair on
$V'$. By an {\it isomorphism of LR pairs from $A,B$ to
$A',B'$} we mean an $\mathbb F$-linear bijection $\sigma :V \to V'$
such that $\sigma A = A' \sigma $ and
$\sigma B = B' \sigma$. The LR pairs
$A,B$ and $A',B'$ are called {\it isomorphic} whenever
there exists an isomorphism of LR pairs from
$A,B$ to $A',B'$.
\end{definition}

\noindent We now classify the LR pairs up to isomorphism.

\begin{proposition}
\label{prop:LRpairClass}
Consider the map which sends an LR pair to its parameter sequence.
This map induces a bijection between the following two sets:
\begin{enumerate}
\item[\rm (i)] the isomorphism classes of LR pairs over $\mathbb F$
that have diameter $d$;
\item[\rm (ii)] the sequences $\lbrace \varphi_i \rbrace_{i=1}^d$
of nonzero scalars in $\mathbb F$.
\end{enumerate}
\end{proposition}
\begin{proof} By Example
\ref{ex:LR} and Lemma
\ref{lem:BAmatrixInv}.
\end{proof}

\noindent We have some comments about Definition
\ref{def:isoLRP}.

\begin{lemma}
\label{lem:isoMove}
Referring to Definition
\ref{def:isoLRP}, let $\lbrace E_i \rbrace_{i=0}^d$
and
 $\lbrace E'_i \rbrace_{i=0}^d$
denote the idempotent sequences for $A,B$ and $A',B'$ respectively.
Let $\sigma$ denote an 
isomorphism of LR pairs from $A,B$ to $A',B'$.
Then 
 $\sigma E_i = E'_i \sigma $ for $0 \leq i \leq d$.
\end{lemma}
\begin{proof} Use Lemma
\ref{lem:Eform}.
\end{proof}

\begin{lemma}
\label{lem:isoFix}
Let $A,B$ denote an LR pair on $V$.
For nonzero $\sigma \in {\rm End}(V)$ the following are equivalent:
\begin{enumerate}
\item[\rm (i)] $\sigma$ is an isomorphism of LR pairs from
$A,B$ to $A,B$;
\item[\rm (ii)] $\sigma$ commutes with $A$ and $B$;
\item[\rm (iii)] $\sigma$  commutes with everything in
${\rm End}(V)$;
\item[\rm (iv)] there exists $0 \not=\zeta \in \mathbb F$
such that $\sigma = \zeta I$.
\end{enumerate}
\end{lemma}
\begin{proof}
${\rm (i)}\Rightarrow {\rm (ii)}$  By Definition
\ref{def:isoLRP}.
\\
\noindent 
${\rm (ii)}\Rightarrow {\rm (iii)}$ By
Corollary
\ref{cor:ABgen}.
\\
\noindent
${\rm (iii)}\Rightarrow {\rm (iv)}$ By
linear algebra.
\\
\noindent
${\rm (iv)}\Rightarrow {\rm (i)}$ By
 Definition
\ref{def:isoLRP}.
\end{proof}

\noindent We have some comments about Lemma
\ref{lem:alphaBeta}.

\begin{lemma}
\label{lem:alphaBetacom}
Let $A,B$ denote an LR pair on $V$,
with parameter sequence $\lbrace \varphi_i \rbrace_{i=1}^d$. 
For nonzero $\alpha, \beta \in \mathbb F$ the LR pair
$\alpha A,\beta B$ has parameter sequence $\lbrace \alpha \beta \varphi_i
\rbrace_{i=1}^d$.
\end{lemma}
\begin{proof} Use
Definition \ref{def:pa}.
\end{proof}

\begin{lemma}
Let $A,B$ denote a nontrivial LR pair over $\mathbb F$.
For nonzero $\alpha, \beta \in \mathbb F$ the following
are equivalent:
\begin{enumerate}
\item[\rm (i)] the LR pairs $A,B$ and $\alpha A, \beta B$
are isomorphic;
\item[\rm (ii)] $\alpha \beta = 1$.
\end{enumerate}
\end{lemma}
\begin{proof} Use
Proposition
\ref{prop:LRpairClass} and
Lemma
\ref{lem:alphaBetacom}.
\end{proof}

\begin{lemma} Let $A,B$ denote an LR pair on $V$.
Pick nonzero $\alpha, \beta \in \mathbb F$.
\begin{enumerate}
\item[\rm (i)] Let $\lbrace v_i \rbrace_{i=0}^d$
denote an $(A,B)$-basis for $V$. Then 
the sequence $\lbrace \alpha^{-i} v_i \rbrace_{i=0}^d$
is an $(\alpha A, \beta B)$-basis for $V$.
\item[\rm (ii)] Let $\lbrace v_i \rbrace_{i=0}^d$
denote an inverted $(A,B)$-basis for $V$. Then 
the sequence $\lbrace \alpha^{i} v_i \rbrace_{i=0}^d$
is an inverted $(\alpha A, \beta B)$-basis for $V$.
\item[\rm (iii)] Let $\lbrace v_i \rbrace_{i=0}^d$
denote a $(B,A)$-basis for $V$. Then 
the sequence $\lbrace \beta^{-i} v_i \rbrace_{i=0}^d$
is a $(\beta B,\alpha A)$-basis for $V$.
\item[\rm (iv)] Let $\lbrace v_i \rbrace_{i=0}^d$
denote an inverted $(B,A)$-basis for $V$. Then 
the sequence $\lbrace \beta^{i} v_i \rbrace_{i=0}^d$
is an inverted $(\beta B, \alpha A)$-basis for $V$.
\end{enumerate}
\end{lemma}
\begin{proof} To obtain part (i) use
Lemma
\ref{lem:alphaBeta}
and
Definition
\ref{def:ABbasis}. Parts (ii)--(iv) are similarly
obtained.
\end{proof}

\begin{definition}
\label{def:FlagLR}
\rm
Let 
$\lbrace U_i\rbrace_{i=0}^d$ denote a flag on $V$.
An element $A \in 
	      {\rm End}(V)$ 
is said to {\it lower} 
$\lbrace U_i\rbrace_{i=0}^d$ whenever
$AU_i = U_{i-1}$ for $1 \leq i \leq d$
and $AU_0=0$. The map $A$ is said to
{\it raise} 
$\lbrace U_i\rbrace_{i=0}^d$ whenever
$U_i + AU_i = U_{i+1}$ for $0 \leq i \leq d-1$.
\end{definition}

\begin{lemma}
\label{lem:Ainduce}
Let $\lbrace V_i\rbrace_{i=0}^d$ denote
a decomposition of $V$ that is lowered by
$A \in {\rm End}(V)$.
\begin{enumerate}
\item[\rm (i)] The flag induced by 
$\lbrace V_i\rbrace_{i=0}^d$ is lowered by $A$.
\item[\rm (ii)] The flag induced by 
$\lbrace V_{d-i}\rbrace_{i=0}^d$ is raised by $A$.
\end{enumerate}
\end{lemma}
\begin{proof} (i) 
Let $\lbrace U_i\rbrace_{i=0}^d$ denote the flag on $V$ that
is induced by
$\lbrace V_i\rbrace_{i=0}^d$.
Then $U_i= V_0 + \cdots + V_i$ for
$0 \leq i \leq d$. By assumption
$AV_i = V_{i-1} $ for $1 \leq i \leq d$ and
$AV_0=0$. Therefore 
$AU_i = U_{i-1}$ for
$1 \leq i \leq d$ and $AU_0=0$.
In other words, the flag
$\lbrace U_i\rbrace_{i=0}^d$ is lowered by $A$.
\\
\noindent (ii) Let 
 $\lbrace U_i\rbrace_{i=0}^d$ denote the flag on $V$ that
is induced by
 $\lbrace V_{d-i}\rbrace_{i=0}^d$.
Then $U_i = V_{d-i}+ \cdots + V_d$ for
$0 \leq i \leq d$. For $0 \leq i \leq d-1$,
$AU_i = 
 V_{d-i-1}+ \cdots + V_{d-1}$.
By these comments $U_i + AU_i = U_{i+1}$.
Therefore 
 $\lbrace U_i\rbrace_{i=0}^d$ is raised by $A$.
\end{proof}

\begin{lemma} 
\label{lem:LRFlag}
Let $A,B$ denote an LR pair on $V$.
\begin{enumerate}
\item[\rm (i)] The flag 
$\lbrace A^{d-i}V\rbrace_{i=0}^d$ is lowered by $A$
and raised by $B$.
\item[\rm (ii)] The flag 
$\lbrace B^{d-i}V\rbrace_{i=0}^d$ is raised by $A$
and lowered by $B$.
\end{enumerate}
\end{lemma}
\begin{proof}
The $(B,A)$-decomposition of $V$ is the inversion of
the 
 $(A,B)$-decomposition of $V$.
The 
 $(A,B)$-decomposition of $V$ is lowered by $A$.
The 
 $(B,A)$-decomposition of $V$ is lowered by $B$.
By Lemma
\ref{lem:ABdecInd},
the 
 $(A,B)$-decomposition of $V$ induces the flag
$\lbrace A^{d-i}V\rbrace_{i=0}^d$, and
the 
 $(B,A)$-decomposition of $V$ induces the flag
$\lbrace B^{d-i}V\rbrace_{i=0}^d$.
The result follows in view of Lemma
\ref{lem:Ainduce}.
\end{proof}

\begin{lemma}
\label{lem:oppF}
Let $\lbrace U_i \rbrace_{i=0}^d$ and
 $\lbrace U'_i \rbrace_{i=0}^d$ denote flags
 on $V$. Then the following are equivalent:
 \begin{enumerate}
\item[\rm (i)] 
 $\lbrace U_i \rbrace_{i=0}^d$ and
 $\lbrace U'_i \rbrace_{i=0}^d$ are opposite;
\item[\rm (ii)] there exists
	      $A \in {\rm End}(V)$ 
that lowers 
 $\lbrace U_i \rbrace_{i=0}^d$ and
raises $\lbrace U'_i \rbrace_{i=0}^d$.
\end{enumerate}
\noindent
Assume that {\rm (i), (ii)} hold, and
define $V_i = U_i \cap U'_{d-i}$ for $0 \leq i \leq d$.
Then the decomposition $\lbrace V_i \rbrace_{i=0}^d$
is lowered by $A$.
\end{lemma}
\begin{proof} 
${\rm (i)}\Rightarrow {\rm (ii)}$ 
Define
 $V_i = U_i \cap U'_{d-i}$ for $0 \leq i \leq d$.
Then $\lbrace V_i \rbrace_{i=0}^d$ is a decomposition of $V$
that induces 
$\lbrace U_i \rbrace_{i=0}^d$ and whose
inversion induces 
$\lbrace U'_i \rbrace_{i=0}^d$.
	      Let $A \in {\rm End}(V)$ 
lower 
 $\lbrace V_i \rbrace_{i=0}^d$.
By Lemma
\ref{lem:Ainduce}, $A$ lowers
 $\lbrace U_i \rbrace_{i=0}^d$
and raises
 $\lbrace U'_i \rbrace_{i=0}^d$.
\\
${\rm (ii)}\Rightarrow {\rm (i)}$ 
We display a decomposition
$\lbrace W_i \rbrace_{i=0}^d$ of
$V$ that induces
 $\lbrace U_i \rbrace_{i=0}^d$ and
 whose inversion
 induces 
 $\lbrace U'_i \rbrace_{i=0}^d$.
Define $W_i = A^{d-i} U'_0$ for $0 \leq i \leq d$.
We show that
 $\lbrace W_i \rbrace_{i=0}^d$ is a decomposition
 of $V$.
Note that $U'_0$
has dimension one, so $W_i$ has dimension at most one
for $0 \leq i \leq d$.
Using the assumption that $A$ raises 
 $\lbrace U'_i \rbrace_{i=0}^d$, we obtain
 $U'_j = W_d +W_{d-1} + \cdots + W_{d-j}$ for
 $0 \leq j \leq d$. Setting $j=d$ we find
$V = \sum_{i=0}^d W_i$. The dimension of $V$ is $d+1$.
Therefore the sum
$V = \sum_{i=0}^d W_i$
is direct, and $W_i$ has dimension
one for $0 \leq i \leq d$. In other words
 $\lbrace W_i \rbrace_{i=0}^d$ is a decomposition of $V$.
By construction, the inverted decomposition
$\lbrace W_{d-i} \rbrace_{i=0}^d$ induces
 $\lbrace U'_i \rbrace_{i=0}^d$.
By assumption $A$ lowers the flag
$\lbrace U_{i} \rbrace_{i=0}^d$.
By Definition 
\ref{def:FlagLR},
$\lbrace U_{i} \rbrace_{i=0}^d$
is the unique flag on $V$ that
is lowered by $A$.
By the definition of
$\lbrace W_i \rbrace_{i=0}^d$ we
find $AW_i = W_{i-1}$
for $1 \leq i \leq d$.
Also $AW_0=A^{d+1}V=0$ since $A$ is Nil. 
Consequently $A$ lowers
$\lbrace W_i \rbrace_{i=0}^d$.
By Lemma
\ref{lem:Ainduce},
$A$ lowers the flag induced by
$\lbrace W_{i} \rbrace_{i=0}^d$.
By these comments
$\lbrace W_{i} \rbrace_{i=0}^d$ induces
  $\lbrace U_i \rbrace_{i=0}^d$.
We have shown that the decomposition 
 $\lbrace W_{i} \rbrace_{i=0}^d$ induces
 $\lbrace U_i \rbrace_{i=0}^d$ and the
 inverted decomposition 
$\lbrace W_{d-i} \rbrace_{i=0}^d$ 
induces $\lbrace U'_i \rbrace_{i=0}^d$.
Therefore
 $\lbrace U_i \rbrace_{i=0}^d$ and
 $\lbrace U'_i \rbrace_{i=0}^d$
are opposite.
\\
\noindent 
Assume that (i), (ii) hold. Recall
from the proof of
${\rm (ii)}\Rightarrow {\rm (i)}$ that
the decomposition
 $\lbrace W_{i} \rbrace_{i=0}^d$ induces
 $\lbrace U_i \rbrace_{i=0}^d$ and 
the inverted decomposition
$\lbrace W_{d-i} \rbrace_{i=0}^d$
induces $\lbrace U'_i \rbrace_{i=0}^d$.
Therefore 
$W_i = U_{i} \cap U'_{d-i} = V_i$ for $0 \leq i \leq d$.
Recall also from the proof of
${\rm (ii)}\Rightarrow {\rm (i)}$ that
 $A$ lowers $\lbrace W_i \rbrace_{i=0}^d$.
So $A$ lowers $\lbrace V_i \rbrace_{i=0}^d$.
\end{proof}

\begin{proposition}
\label{prop:LRchar}
Let $A,B \in 
	      {\rm End}(V)$. Then $A,B$ is an LR pair on $V$ if and only if
	     the following {\rm (i)--(iii)} hold:
\begin{enumerate}
\item[\rm (i)] $A$ and $B$ are Nil;
\item[\rm (ii)] 
the flag $\lbrace A^{d-i}V\rbrace_{i=0}^d$ is raised by $B$;
\item[\rm (iii)] 
the flag $\lbrace B^{d-i}V\rbrace_{i=0}^d$ is raised by $A$.
\end{enumerate}
\end{proposition}
\begin{proof}
First assume that $A,B$ is an LR pair on $V$. Then
condition (i) holds by Lemma
\ref{lem:ABdecIndNil}, and conditions (ii), (iii) hold
by Lemma
\ref{lem:LRFlag}. Conversely, assume that 
the conditions (i)--(iii) hold. 
To show that $A,B$ is an LR pair on $V$,
we display a decomposition $\lbrace V_i\rbrace_{i=0}^d$
of $V$ that is lowered by $A$ and raised by $B$.
By construction and since $A$ is Nil,
the flag 
$\lbrace A^{d-i}V\rbrace_{i=0}^d$ is lowered by
$A$.
By assumption the flag 
$\lbrace B^{d-i}V\rbrace_{i=0}^d$ is raised by $A$.
Now by Lemma
\ref{lem:oppF} the flags
$\lbrace A^{d-i}V\rbrace_{i=0}^d$ and
$\lbrace B^{d-i}V\rbrace_{i=0}^d$ are opposite.
Define $V_i = A^{d-i}V \cap B^iV$ for
$0 \leq i \leq d$. By 
 Lemma
\ref{lem:oppF} the decomposition 
$\lbrace V_i\rbrace_{i=0}^d$ is lowered by $A$.
Interchanging the roles of $A,B$ in the argument so far,
we see that
$\lbrace V_i\rbrace_{i=0}^d$ is raised by $B$.
We have shown that $A,B$ is an LR pair on $V$.
\end{proof}

\noindent We define some matrices for later use.

\begin{definition}
\label{def:Dmat}
\rm
Let $A,B$ denote an LR pair on $V$, with
parameter sequence $\lbrace \varphi_i \rbrace_{i=1}^d$.
Define a diagonal matrix $D
 \in {\rm Mat}_{d+1}(\mathbb F)$ with $(i,i)$-entry
 $\varphi_1 \varphi_2 \cdots \varphi_i$
 for $0 \leq i \leq d$.
\end{definition}

\begin{lemma}
\label{lem:Dmeaning}
Let $A,B$ denote an LR pair on $V$, with
parameter sequence $\lbrace \varphi_i \rbrace_{i=1}^d$.
For $M \in 
 {\rm Mat}_{d+1}(\mathbb F)$ the following are equivalent:
\begin{enumerate}
\item[\rm (i)]
$M$ is the transition matrix from an $(A,B)$-basis of $V$
to an inverted $(B,A)$-basis of $V$;
\item[\rm (ii)] there exists $0 \not=\zeta \in \mathbb F$
such that $M=\zeta D$.
\end{enumerate}
\end{lemma}
\begin{proof} 
Use Lemma
\ref{lem:sometrans}(i).
\end{proof}

\begin{definition}
\label{def:tauMat}
Let $\tau$ denote the matrix in
 ${\rm Mat}_{d+1}(\mathbb F)$ that has $(i-1,i)$-entry 1 for
 $1 \leq i \leq d$, and
 all other entries $0$. Thus
\begin{eqnarray*}
\tau =	\left(
	\begin{array}{ c c cc c c }
	0 &  1  &    & &   & {\bf 0}  \\
	 &  0  & 1  & &  &      \\
	 &   &  0   & \cdot & &   \\
	   &   &  &  \cdot & \cdot &  \\
	     &  & &  &     \cdot   & 1 \\
	      {\bf 0}  &&  & &   & 0  \\
	      \end{array}
	      \right),
\qquad \qquad
{\bf Z}\tau {\bf Z} = 
  \left(
	\begin{array}{ c c cc c c }
	0 &    &    & &   & {\bf 0}  \\
	 1 &  0  &   & &  &      \\
	 & 1  &  0   &  & &   \\
	   &   & \cdot  &  \cdot &  \\
	     &  & &  \cdot &     \cdot   &  \\
	      {\bf 0}  &&  & & 1  & 0  \\
	      \end{array}
	      \right).
\end{eqnarray*}
\end{definition}

\noindent Let $A,B$ denote an LR pair on $V$, with
parameter sequence $\lbrace \varphi_i \rbrace_{i=1}^d$.
In lines
(\ref{eq:ABrep}),
(\ref{eq:invABrep}),
(\ref{eq:BArep}),
(\ref{eq:invBArep})
we encountered some matrices that had
$\lbrace \varphi_i \rbrace_{i=1}^d$ among the entries.
We now express these
matrices in terms of $\bf Z$, $D$, $\tau$.

\begin{lemma} Referring to Definitions
\ref{def:Dmat}
and
\ref{def:tauMat},
\begin{eqnarray*}
&&
D^{-1}\tau D =	\left(
	\begin{array}{ c c cc c c }
	0 &  \varphi_1  &    & &   & {\bf 0}  \\
	 &  0  & \varphi_2  & &  &      \\
	 &   &  0   & \cdot & &   \\
	   &   &  &  \cdot & \cdot &  \\
	     &  & &  &     \cdot   & \varphi_d \\
	      {\bf 0}  &&  & &   & 0  \\
	      \end{array}
	      \right),
\quad 
{\bf Z}D{\bf Z} \tau {\bf Z} D^{-1} {\bf Z} =	\left(
	\begin{array}{ c c cc c c }
	0 &  \varphi_d  &    & &   & {\bf 0}  \\
	 &  0  & \varphi_{d-1}  & &  &      \\
	 &   &  0   & \cdot & &   \\
	   &   &  &  \cdot & \cdot &  \\
	     &  & &  &     \cdot   & \varphi_1\\
	      {\bf 0}  &&  & &   & 0  \\
	      \end{array}
	      \right),
\\
&&
{\bf Z}D^{-1}\tau D {\bf Z} = 
  \left(
	\begin{array}{ c c cc c c }
	0 &    &    & &   & {\bf 0}  \\
	 \varphi_d &  0  &   & &  &      \\
	 & \varphi_{d-1}  &  0   &  & &   \\
	   &   & \cdot  &  \cdot &  \\
	     &  & &  \cdot &     \cdot   &  \\
	      {\bf 0}  &&  & & \varphi_1  & 0  \\
	      \end{array}
	      \right),
\quad 
D{\bf Z}\tau {\bf Z} D^{-1} = 
  \left(
	\begin{array}{ c c cc c c }
	0 &    &    & &   & {\bf 0}  \\
	 \varphi_1 &  0  &   & &  &      \\
	 & \varphi_{2}  &  0   &  & &   \\
	   &   & \cdot  &  \cdot &  \\
	     &  & &  \cdot &     \cdot   &  \\
	      {\bf 0}  &&  & & \varphi_d  & 0  \\
	      \end{array}
	      \right).
\end{eqnarray*}
\end{lemma}
\begin{proof} Matrix multiplication.
\end{proof}

\noindent In Section 12 we will
consider some powers of the matrix $\tau$
from Definition
\ref{def:tauMat}. We now compute the
entries of these powers.

\begin{lemma}
\label{lem:tauPower}
Referring to
Definition
\ref{def:tauMat},
for $0 \leq r \leq d$ the matrix
$\tau^r$ has $(i,j)$-entry
\begin{eqnarray*}
(\tau^r)_{i,j} = 
\begin{cases}
1 &  {\mbox{\rm if $j-i=r$}}; \\
0 & {\mbox{\rm if $j-i\not=r$}}
\end{cases}
\qquad \qquad (0 \leq i,j \leq d).
\end{eqnarray*}
Moreover $\tau^{d+1}=0$.
\end{lemma}
\begin{proof}
Matrix multiplication.
\end{proof}

\noindent Let $A,B$ denote an LR pair on $V$, with
parameter sequence $\lbrace \varphi_i \rbrace_{i=1}^d$.
As we proceed, we will encounter the case in which 
the $\lbrace \varphi_i \rbrace_{i=1}^d$ satisfy a linear recurrence.
We now consider this case.

\begin{lemma}
\label{lem:ABrelations}
Let $A,B$ denote an LR pair on $V$, with
parameter sequence $\lbrace \varphi_i \rbrace_{i=1}^d$.
Pick an integer $r$ $(1 \leq r \leq d+1)$.
Let $x$ and 
 $\lbrace y_i \rbrace_{i=0}^{r}$
 denote scalars in 
$\mathbb F$.
Then the following are equivalent:
\begin{enumerate}
\item[\rm (i)] $ x A^{r-1} =\sum_{i=0}^{r} y_i A^i B A^{r-i}$;
\item[\rm (ii)] $ x B^{r-1} = \sum_{i=0}^{r} y_i B^{r-i} A B^i$;
\item[\rm (iii)] 
$x = y_0 \varphi_{i} + 
y_1 \varphi_{i+1} + 
\cdots + y_{r}\varphi_{i+r}$ for $0 \leq i \leq d-r+1$.
\end{enumerate}
\end{lemma}
\begin{proof}
Represent $A$ and $B$ by matrices as in
(\ref{eq:ABrep}).
\end{proof}

\section{LR pairs of Weyl and $q$-Weyl type}

In this section we investigate two families
of LR pairs, said to have Weyl type and $q$-Weyl type.
We begin with an example that illustrates
Lemma
\ref{lem:ABrelations}. 
 Throughout this section $V$ denotes a vector
space over $\mathbb F$ with dimension $d+1$.

\begin{example}
\label{ex:Weyl}
\rm
Let $A,B$ denote an LR pair on $V$, with
parameter sequence $\lbrace \varphi_i \rbrace_{i=1}^d$.
 By Lemma
\ref{lem:ABrelations},
\begin{eqnarray}
AB-BA = I     
\label{eq:Weyl1}
\end{eqnarray}
if and only if 
\begin{eqnarray}
\varphi_{i+1}-\varphi_i = 1
\qquad \qquad  (0 \leq i \leq d).
\label{eq:Weyl2}
\end{eqnarray}
\end{example}

\begin{note}\rm
The equation 
(\ref{eq:Weyl1}) is called the {\it Weyl relation}.
\end{note}

\begin{definition}
\label{def:Weyl}
\rm
The LR pair $A,B$ in Example
\ref{ex:Weyl} is said to have
{\it Weyl type} whenever it satisfies
the equivalent conditions 
{\rm (\ref{eq:Weyl1})}, 
{\rm (\ref{eq:Weyl2})}.
\end{definition}

\begin{note}\rm
Referring to Example
\ref{ex:Weyl},
assume that $A,B$ has Weyl type.
Then the LR pair $B,-A$ has Weyl type.
\end{note}

\begin{lemma}
\label{lem:pPrime}
Referring to Example
\ref{ex:Weyl},
assume that $A,B$ has Weyl type.
Then {\rm (i), (ii)} hold below.
\begin{enumerate}
\item[\rm (i)]
  $\varphi_i = i$ for $1 \leq i \leq d$. 
\item[\rm (ii)] 
The integer $p=d+1$ is prime and
${\rm Char}(\mathbb F)=p$.
\end{enumerate}
\end{lemma}
\begin{proof}
By (\ref{eq:Weyl2}) and $\varphi_0=0$
we obtain $\varphi_i = i$
for $1\leq i\leq d+1$. We have $\varphi_{d+1}=0$,
so $d+1=0$ in the field $\mathbb F$.
For $1 \leq i \leq d$ we have
$\varphi_i\not=0$, so
$i \not=0$ in 
 $\mathbb F$. The results follow.
 \end{proof}

\begin{lemma}
\label{ex:Weylback}
Assume that $p=d+1$ is prime and 
${\rm Char}(\mathbb F) = p$.
Define
$\varphi_i = i$ for $1 \leq i \leq d$.
Then $\lbrace \varphi_i \rbrace_{i=1}^d$ are nonzero;
let $A,B$ denote the LR pair over $\mathbb F$
that has parameter sequence
$\lbrace \varphi_i \rbrace_{i=1}^d$.
Then $A,B$ has Weyl type.
\end{lemma}
\begin{proof}
One checks
that condition
(\ref{eq:Weyl2})
is satisfied.
\end{proof}

\begin{lemma}
\label{lem:WeylChar}
Assume that $p=d+1$ is prime and
${\rm Char}(\mathbb F)=p$.
 Then for
$A,B \in {\rm End}(V)$ the following are equivalent:
\begin{enumerate}
\item[\rm (i)] neither of $A,B$ is invertible and $AB-BA=I$;
\item[\rm (ii)] $A,B$ is an LR pair on $V$ that has Weyl type.
\end{enumerate}
\end{lemma}
\begin{proof}
${\rm (i)}\Rightarrow {\rm (ii)}$ 
Since $A$ is not invertible, 
there exists $0 \not=\eta \in V$
such that $A\eta=0$.
Define $v_i = B^i\eta$
for $0 \leq i \leq d$.
By construction, $Av_0=0$ and
$Bv_{i-1} = v_i$ for $1 \leq i \leq d$.
Using $AB-BA=I$ and induction on $i$, we obtain
$Av_i=i v_{i-1}$ for $1 \leq i \leq d$.
Note that $i\not=0$ in $\mathbb F$ for $1 \leq i \leq d$.
By these comments, for $0 \leq i \leq d$
the vector $v_i$ is in the kernel of $A^{i+1}$
and not in the kernel of $A^i$.
Therefore $\lbrace v_i \rbrace_{i=0}^d$ are
linearly independent and hence form a basis for $V$.
By construction 
 the induced decomposition
 $\lbrace V_i \rbrace_{i=0}^d$ 
is lowered by $A$. Therefore $A$ is Nil.
Replacing $A,B$ by $B,-A$ in the above argument,
we see that $B$ is Nil. Now
$Bv_d=B^{d+1}\eta=0$.
Now by construction $B$ raises the decomposition 
 $\lbrace V_i \rbrace_{i=0}^d$.
We have shown that the decomposition
 $\lbrace V_i \rbrace_{i=0}^d$ is lowered by $A$ and
 raised by $B$. Therefore $A,B$ is an LR pair on $V$.
This LR pair has Weyl type by Definition
\ref{def:Weyl} and since $AB-BA=I$.
\\
${\rm (ii)}\Rightarrow {\rm (i)}$ The elements
$A,B$ are not invertible, since they are Nil by Lemma
\ref{lem:ABdecIndNil}.
By 
Definition
\ref{def:Weyl}
we have $AB-BA=I$.
\end{proof}

\noindent Later in the paper we will use
the following curious fact about LR pairs
of Weyl type.

\begin{lemma}
\label{lem:curious}
Assume $d\geq 2$.
Let $A,B$ denote an LR pair on $V$ that has Weyl type.
Define $C=-A-B$. Then the pairs $B,C$ and $C,A$ are
LR pairs on $V$ that have Weyl type.
\end{lemma}
\begin{proof}
By Definition \ref {def:Weyl},
$AB-BA=I$. 
By Lemma
\ref{lem:pPrime}(ii),
$p=d+1$ is prime and
${\rm Char}(\mathbb F)=p$.
We show that $B,C$ is an LR pair on $V$
that has Weyl type. To do this we apply
Lemma
\ref{lem:WeylChar} to the pair $B,C$.
Using $AB-BA=I$ and the definition of $C$, we find
$BC-CB=I$. The map $B$ is not invertible
since $B$ is Nil. We show that $C$ is not invertible.
By Lemma
\ref{lem:BAmatrixInv},
with respect to an inverted $(B,A)$-basis for
$V$ the element $A+B$ is represented by
\begin{equation}
\label{eq:matrixEIG}
\left(
\begin{array}{ c c cc c c }
 0 & 1  &   &&   & \bf 0  \\
  1& 0  &   2&&  &      \\ 
   & 1  &  0   & \cdot &&  \\
     &   & \cdot & \cdot  & \cdot & \\
       &  & & \cdot & \cdot & d\\
        {\bf 0}  &&  & & 1 &  0  \\
	\end{array}
	\right).
\end{equation}
By our assumption $d\geq 2$,
the prime $p=d+1$ is odd. Therefore  $d$ is even.
Define
\begin{eqnarray*}
z_{2i}= \frac{(-1)^i}{ 2^{i}i!}  \qquad \qquad 
(0 \leq i \leq d/2)
\end{eqnarray*}
and 
$z_{2i+1}= 0$ for $0 \leq i < d/2$.
The matrix (\ref{eq:matrixEIG}) times the
column vector
$(z_0, z_1, \ldots, z_d)^t$ is zero.
Therefore 
the matrix (\ref{eq:matrixEIG}) is not invertible.
Therefore $A+B$ is not invertible, so $C$
is not invertible.
Applying Lemma
\ref{lem:WeylChar} to the pair $B,C$
we find that $B,C$ is an LR pair
on $V$ that has Weyl type. One similarly shows
that $C,A$ is an LR pair on $V$ that has Weyl type.
\end{proof}

\noindent Here is another example of an LR pair.
\begin{example}
\label{ex:qWeyl}
\rm
Let $A,B$ denote an LR pair on $V$, with
parameter sequence $\lbrace \varphi_i \rbrace_{i=1}^d$.
Pick a nonzero $q \in \mathbb F$ such that
$q^2 \not=1$. By Lemma
\ref{lem:ABrelations},
\begin{eqnarray}
\frac{qAB-q^{-1}BA}{q-q^{-1}}= I    
\label{eq:qWeyl1}
\end{eqnarray}
if and only if 
\begin{eqnarray}
\frac{q\varphi_{i+1}-q^{-1}\varphi_i}{q-q^{-1}}= 1
\qquad \qquad (0 \leq i \leq d).
\label{eq:qWeyl2}
\end{eqnarray}
\end{example}

\begin{note}\rm
The equation 
(\ref{eq:qWeyl1}) is called the {\it $q$-Weyl relation}.
\end{note}

\begin{definition}
\label{def:qWeyl}
\rm
The LR pair $A,B$ in
Example \ref{ex:qWeyl}
is said to have {\it $q$-Weyl type}
whenever it satisfies the equivalent conditions
(\ref{eq:qWeyl1}),
(\ref{eq:qWeyl2}).
\end{definition}

\begin{note}
\label{note:minusq}
\rm 
Referring to Example
\ref{ex:qWeyl},
assume that $A,B$ has
$q$-Weyl type. Then $A,B$ has $(-q)$-Weyl type.
Moreover the LR pair $B,A$ has $(q^{-1})$-Weyl type.
\end{note}

\begin{lemma}
\label{lem:stand}
Referring to Example
\ref{ex:qWeyl},
assume that $A,B$ has
$q$-Weyl type. 
 Then {\rm (i)--(v)} hold below.
\begin{enumerate}
\item[\rm (i)] $d\geq 1$.
\item[\rm (ii)]
 $\varphi_i = 1-q^{-2i}$ for $1 \leq i \leq d$. 
\item[\rm (iii)]
Assume that
${\rm Char}(\mathbb F)\not=2$ and $d$ is odd.
Then $q$ is
a primitive $(2d+2)$-root of unity.
\item[\rm (iv)]
Assume that
${\rm Char}(\mathbb F)\not=2$ and $d$ is even.
Then 
$q$ becomes a primitive $(2d+2)$-root of unity, after replacing
$q$ by $-q$ if necessary.
\item[\rm (v)]
Assume that
${\rm Char}(\mathbb F)=2$.
Then $d$ is even. Moreover $q$ is a primitive $(d+1)$-root of
unity.
\end{enumerate}
\end{lemma}
\begin{proof} 
By (\ref{eq:qWeyl2}) and $\varphi_0=0$
along with  induction on $i$,
we obtain $\varphi_i = 1-q^{-2i}$ for
$1 \leq i \leq d+1$. 
We have $\varphi_{d+1}=0$,
so 
$q^{2d+2}=1$.
For $1 \leq i \leq d$ we have
$\varphi_i\not=0$, so
$q^{2i}\not=1$.
The results follow.
\end{proof}

\begin{definition}\rm
\label{def:qstand}
For $q \in \mathbb F$ the ordered pair $d,q$ will be
called {\it standard} whenever the following
{\rm (i)--(iii)} hold.
\begin{enumerate}
\item[\rm (i)]
  $d\geq 1$.
\item[\rm (ii)] Assume ${\rm Char}(\mathbb F)\not=2$. Then  $q$ is
a primitive $(2d+2)$-root of unity.
\item[\rm (iii)] Assume ${\rm Char}(\mathbb F)=2$.  Then  $d$ is even,
and $q$ is a primitive $(d+1)$-root of unity.
\end{enumerate}
\end{definition}

\begin{note}\rm
Referring to Definition
\ref{def:qstand}, assume that $d,q$ is standard.
Then $q$ is nonzero and $q^2\not=1$.
\end{note}

\begin{lemma}
\label{lem:Mq}
Referring to Example
\ref{ex:qWeyl},
assume that $A,B$ has $q$-Weyl type.
Then $d,q$ is standard or 
$d,-q$ is standard.
\end{lemma}
\begin{proof}
Use Lemma
\ref{lem:stand}
and 
Definition
\ref{def:qstand}.
\end{proof}

\noindent For the rest of this section, the
following assumption is in effect.

\begin{assumption}
\label{def:sqRoot}
\rm
Fix $q \in \mathbb F$
and assume that $d,q$ is standard.
  We fix a square root $q^{1/2}$ in
the algebraic closure  $\overline {\mathbb F}$.
\end{assumption}

\begin{lemma}
\label{ex:exceptional}
With reference to Assumption
\ref{def:sqRoot}, define
$\varphi_i = 1-q^{-2i}$ for $1 \leq i \leq d$.
Then $\lbrace \varphi_i \rbrace_{i=1}^d$ are nonzero;
let $A,B$ denote the LR pair over $\mathbb F$
that has parameter sequence
$\lbrace \varphi_i \rbrace_{i=1}^d$.
Then $A,B$ has $q$-Weyl type.
\end{lemma}
\begin{proof} Condition
(\ref{eq:qWeyl2})
 is readily checked.
\end{proof}

\begin{lemma}
\label{lem:qWeylABextend}
With reference to
Assumption 
\ref{def:sqRoot},
 for $A,B \in 
	      {\rm End}(V)$ 
the following are equivalent:
\begin{enumerate}
\item[\rm (i)] neither of $A,B$ is  invertible and
\begin{eqnarray}
\frac{qAB-q^{-1}BA}{q-q^{-1}}= I;
\label{eq:ABqWeyl}
\end{eqnarray}
\item[\rm (ii)] $A,B$ is an LR pair on $V$ that has
$q$-Weyl type.
\end{enumerate}
\end{lemma}
\begin{proof}
The proof is similar to the proof of
Lemma
\ref{lem:WeylChar}. For the sake of completeness
we give the details.
\\
\noindent 
${\rm (i)}\Rightarrow {\rm (ii)}$ 
For $1 \leq i \leq d$ define 
$\varphi_i = 1-q^{-2i}$
and note that $\varphi_i \not=0$.
Since $A$ is not invertible, there exists $0 \not=\eta \in V$
such that $A\eta=0$.
Define $v_i = B^i\eta$
for $0 \leq i \leq d$.
By construction, $Av_0=0$ and
$Bv_{i-1} = v_i$ for $1 \leq i \leq d$.
Using (\ref{eq:ABqWeyl})
and induction on $i$, we obtain
$Av_i=\varphi_i v_{i-1}$ for $1 \leq i \leq d$.
By these comments, for $0 \leq i \leq d$
the vector $v_i$ is in the kernel of
$A^{i+1}$ and not in the kernel of
$A^i$. Therefore
$\lbrace v_i \rbrace_{i=0}^d$  are linearly
independent and hence form a basis for $V$.
By construction the induced decomposition
 $\lbrace V_i \rbrace_{i=0}^d$ 
is lowered by $A$. Therefore $A$ is Nil.
Replacing $A,B,q$ by $B,A,q^{-1}$ in the above argument,
we see that $B$ is Nil. Now
$Bv_d=B^{d+1}\eta=0$.
Now by construction $B$ raises the decomposition 
 $\lbrace V_i \rbrace_{i=0}^d$.
We have shown that the decomposition
 $\lbrace V_i \rbrace_{i=0}^d$ is lowered by $A$ and
 raised by $B$. Therefore $A,B$ is an LR pair on $V$.
This LR pair has $q$-Weyl type by Definition
\ref{def:qWeyl} and 
(\ref{eq:ABqWeyl}).
\\
${\rm (ii)}\Rightarrow {\rm (i)}$ The elements
$A,B$ are not invertible since they are Nil
by Lemma
\ref{lem:ABdecIndNil}.
By Definition
\ref{def:qWeyl}
the pair $A,B$ satisfies
(\ref{eq:ABqWeyl}).
\end{proof}

\noindent 
With reference to Assumption
\ref{def:sqRoot},
let $A,B$ denote an LR pair on $V$ that has $q$-Weyl type.
Later in the paper we will need the eigenvalues of
$qA+q^{-1}B$. Our next goal is to compute these eigenvalues.

\begin{lemma} 
\label{lem:bMatrixPre}
Pick a nonzero $b \in \mathbb F $ such that
$b^i \not=1$ for $1 \leq i \leq d$. 
For the
tridiagonal matrix 
\begin{eqnarray}
\left(
\begin{array}{ c c cc c c }
 0 & b^d-1   &   &&   & \bf 0  \\
 b-1 & 0  & b^d-b  &&  &      \\ 
   & b^2-1  &  0   & \cdot &&  \\
     &   & \cdot & \cdot  & \cdot & \\
       &  & & \cdot & \cdot & b^d-b^{d-1} \\
        {\bf 0}  &&  & & b^d-1  &  0  \\
	\end{array}
	\right)
\label{eq:Deltamat}
\end{eqnarray}
the roots of the characteristic polynomial are
\begin{eqnarray}
\label{eq:evListPre}
 b^j - b^{d-j} \qquad \qquad j=0,1,\ldots, d.
\end{eqnarray}
For $0 \leq j \leq d$ we give a (column) eigenvector for
the matrix {\rm (\ref{eq:Deltamat})} 
and 
eigenvalue $b^j-b^{d-j}$. This eigenvector has $i$th coordinate
\begin{eqnarray*}
 {}_3\phi_2 \Biggl(
\genfrac{}{}{0pt}{}
 {b^{-i}, \;b^{j-d},\;-b^{-j}}
 {0, \;\;b^{-d}}
 \;\Bigg\vert \; b,\;b\Biggr)
 \end{eqnarray*}
for $0 \leq i \leq d$.
We follow the standard notation for basic hypergeometric series 
{\rm \cite[p.~4]{GR}}.
\end{lemma}
\begin{proof} See
\cite[Example 5.9]{terPA}.
\end{proof}

\begin{definition}
\label{def:thetaI}
\rm
With reference to Assumption
\ref{def:sqRoot},
define
\begin{equation}
\theta_j = 
q^{j+1/2} + q^{-j-1/2}
 \qquad \qquad 
(0 \leq j \leq d).
\label{eq:thetaCalc}
\end{equation}
\end{definition}

\begin{lemma}
With reference to Assumption
\ref{def:sqRoot} and
Definition
\ref{def:thetaI}, the following {\rm (i)--(iv)} hold.
\begin{enumerate}
\item[\rm (i)] $\theta_j = -\theta_{d-j}$ for
$0 \leq j \leq d$.
\item[\rm (ii)] Assume that $d=2m$ is even. Then $\theta_m =0$.
\item[\rm (iii)] Assume that ${\rm Char}(\mathbb F)\not=2$.
Then $\lbrace \theta_j \rbrace_{j=0}^d$ are mutually
distinct.
\item[\rm (iv)] Assume that ${\rm Char}(\mathbb F)=2$, 
so that $d=2 m$ is even.
Then $\lbrace \theta_j \rbrace_{j=0}^m$ are mutually
distinct.
\end{enumerate}
\end{lemma}
\begin{proof} 
Use Definition
\ref{def:thetaI} and
the restrictions on $q$ given in
Definition
\ref{def:qstand}.
\end{proof}

\begin{lemma}
\label{lem:qAqiB}
With reference to Assumption
\ref{def:sqRoot},
let $A,B$ denote an LR pair on $V$ that has $q$-Weyl type.
Then for $qA+q^{-1}B$
the roots of the characteristic polynomial
are $\lbrace \theta_j \rbrace_{j=0}^d$.
\end{lemma}
\begin{proof} 
Let $H$
 denote the matrix 
 in 
${\rm Mat}_{d+1}(\mathbb F)$ 
that represents $qA+q^{-1}B$
with respect to an inverted $(B,A)$-basis for $V$.
The entries of $H$ are obtained using
Lemma
\ref{lem:BAmatrixInv}.
Let $\Delta$ denote the matrix
(\ref{eq:Deltamat}),
with $b=q^{-1}$.
For $1 \leq i \leq d$ define
$n_i = q^{1/2}(q^{-i}-1)$ and note that $n_i \not=0$.
Define
a diagonal matrix $N \in 
{\rm Mat}_{d+1}(\overline {\mathbb F})$ with $(i,i)$-entry
$n_1 n_2 \cdots n_i$ for $0 \leq i \leq d$. Note
that $N$ is invertible. 
By matrix multiplication
$H = q^{-1/2} N^{-1} \Delta N$.
One checks that
\begin{eqnarray*}
\theta_j = q^{-1/2} (b^{j}-b^{d-j})  \qquad \qquad (0 \leq j \leq d).
\end{eqnarray*}
By these comments and Lemma
\ref{lem:bMatrixPre}, for the matrix $H$ the
roots of the characteristic polynomial are
$\lbrace \theta_j \rbrace_{j=0}^d$.
The result follows.
\end{proof}

\section{The dual space $V^*$}

Recall our vector space $V$ over $\mathbb F$ with dimension
$d+1$. Let $V^*$ denote the vector space over $\mathbb F$
consisting of the $\mathbb F$-linear maps from $V$ to $\mathbb F$.
We call $V^*$ the {\it dual space} for $V$.
The vector spaces $V$ and $V^*$ have the same dimension $d+1$.
There exists a bilinear form $(\,,\,): V \times V^* \to \mathbb F$
such that $(u,f)= f(u)$ for all $u \in V$ and
$f \in V^*$. This bilinear form is nondegenerate in the sense of
\cite[Section~11]{3ptsl2}.
We view $(V^*)^*=V$. Nonempty subsets $X \subseteq V$ and
$Y \subseteq V^*$ are called {\it orthogonal}
whenever $(x,y)=0$ for all $x \in X$
and $y \in Y$. For a subspace $U$ of $V$ (resp. $V^*$) let
$U^\perp$ denote the set of vectors in $V^*$ (resp. $V$)
that are orthogonal to everything in $U$. The subspace $U^\perp$
is called the {\it orthogonal complement of $U$}. 
Note that
${\rm dim}(U) + 
{\rm dim}(U^\perp) = d+1$.

\medskip

\noindent 
A basis $\lbrace v_i\rbrace_{i=0}^d$ of
$V$ and a basis
$\lbrace v'_i\rbrace_{i=0}^d$ of
$V^*$ are called {\it dual}
whenever $(v_i,v'_j)=\delta_{i,j}$ 
for $0 \leq i,j\leq d$.
Each basis of $V$ (resp. $V^*$) is dual to
a unique basis of $V^*$ (resp. $V$). 
Let $\lbrace u_i\rbrace_{i=0}^d$ (resp.
$\lbrace v_i\rbrace_{i=0}^d$)
denote a basis of 
$V$, and let  
$\lbrace u'_i\rbrace_{i=0}^d$ (resp.
$\lbrace v'_i\rbrace_{i=0}^d$) denote the dual basis
of $V^*$. Then 
the following matrices are
transpose: (i) the transition matrix from
$\lbrace u_i\rbrace_{i=0}^d$ to 
$\lbrace v_i\rbrace_{i=0}^d$;
(ii) the transition matrix from
$\lbrace v'_i\rbrace_{i=0}^d$ to 
$\lbrace u'_i\rbrace_{i=0}^d$.
\medskip

\noindent A decomposition $\lbrace V_i\rbrace_{i=0}^d$ of
$V$ and a decomposition
$\lbrace V'_i\rbrace_{i=0}^d$ of
$V^*$ are called {\it dual}
whenever $V_i, V'_j$ are orthogonal
for all $i,j$ $(0 \leq i,j\leq d)$ such that $i\not=j$.
Let $\lbrace v_i \rbrace_{i=0}^d$ denote a basis of
$V$ and let
 $\lbrace v'_i \rbrace_{i=0}^d$ denote the dual basis of $V^*$.
 Then the following are dual:
(i) the decomposition of $V$ induced by
 $\lbrace v_i \rbrace_{i=0}^d$;
(ii) the decomposition of $V^*$ induced by
 $\lbrace v'_i \rbrace_{i=0}^d$.
Each decomposition of $V$ (resp. $V^*$) is dual to
a unique decomposition of $V^*$ (resp. $V$). 

\medskip

\noindent 
A flag $\lbrace U_i\rbrace_{i=0}^d$ on
$V$ and a flag
$\lbrace U'_i\rbrace_{i=0}^d$ on
$V^*$ are called {\it dual}
whenever $U_i, U'_j$ are orthogonal 
for all $i,j$ $(0 \leq i,j\leq d)$ such that $i+j=d-1$.
In this case
$U_i, U'_j$ are orthogonal complements
for all $i,j$ $(0 \leq i,j\leq d)$ such that $i+j=d-1$.
Each flag on $V$ (resp. $V^*$) is dual to
a unique flag on $V^*$ (resp. $V$). 

\begin{lemma} 
\label{lem:flagdual}
Let $\lbrace V_i\rbrace_{i=0}^d$ denote a decomposition
of $V$ and let
$\lbrace V'_i\rbrace_{i=0}^d$ denote the dual decomposition
of $V^*$. Then the following are dual:
\begin{enumerate}
\item[\rm (i)] the flag on $V$ induced by 
$\lbrace V_i\rbrace_{i=0}^d$;
\item[\rm (ii)] the flag on $V^*$ induced by 
$\lbrace V'_{d-i}\rbrace_{i=0}^d$.
\end{enumerate}
\end{lemma}
\begin{proof}
Let $\lbrace U_i \rbrace_{i=0}^d$ and
$\lbrace U'_i \rbrace_{i=0}^d$ denote the flags
from (i) and (ii), respectively.
For $0 \leq i \leq d$ we have
$U_i = 
V_0+\cdots + V_i$ and
$U'_i = 
 V'_{d-i} + \cdots + V'_{d}$.
For $0 \leq i,j\leq d$ such that $i+j=d-1$, the
subspace
 $U'_j = 
 V'_{i+1} + \cdots + V'_{d}$ is orthogonal to 
$U_i$. The result follows.
\end{proof}

\noindent For $\mathbb F$-algebras $\mathcal A$ and
$\mathcal A'$, a map $\sigma: \mathcal A \to \mathcal A'$
is called an {\it $\mathbb F$-algebra antiisomorphism}
whenever $\sigma$ is an isomorphism of $\mathbb F$-vector
spaces and $(ab)^\sigma= b^\sigma a^\sigma $ for all
$a,b\in \mathcal A$. 
By an {\it antiautomorphism} of $\mathcal A$
we mean an $\mathbb F$-algebra antiisomorphism
 $\sigma:\mathcal A\to \mathcal A$.
For
	       $X \in {\rm End}(V)$ there exists 
	      a unique element of
	      ${\rm End}(V^*)$, denoted
$\tilde X$, such that $(Xu,v)=(u,\tilde X v)$
for all $u \in V$ and $v \in V^*$. The map $\tilde X$
is called the {\it adjoint of $X$}. The adjoint map
	      ${\rm End}(V)\to 
	      {\rm End}(V^*)$, $X \mapsto \tilde X$ 
	      is an $\mathbb F$-algebra antiisomorphism.
Let $\lbrace v_i\rbrace_{i=0}^d$ denote a basis of
$V$ and let  
$\lbrace v'_i\rbrace_{i=0}^d$ denote the dual basis
of $V^*$. Then for $X \in
{\rm End}(V)$ the following matrices are
transpose: (i) the matrix representing $X$ with
respect to 
$\lbrace v_i\rbrace_{i=0}^d$;
(ii) the matrix representing $\tilde X$ with
respect to 
$\lbrace v'_i\rbrace_{i=0}^d$.
\medskip

\noindent
Let $\lbrace V_i\rbrace_{i=0}^d$ denote a decomposition
of
$V$ and let  
$\lbrace V'_i\rbrace_{i=0}^d$ denote the dual decomposition 
of $V^*$. 
Let
$\lbrace E_i\rbrace_{i=0}^d$ denote the idempotent sequence
for 
 $\lbrace V_i\rbrace_{i=0}^d$.
Then 
$\lbrace \tilde E_i\rbrace_{i=0}^d$ is the idempotent sequence
for 
 $\lbrace V'_i\rbrace_{i=0}^d$.

\begin{lemma} 
\label{lem:Adual}
Let 
$\lbrace V_i\rbrace_{i=0}^d$ denote a decomposition
of $V$ and let
$\lbrace V'_i\rbrace_{i=0}^d$ denote the dual decomposition
of $V^*$. Then for $A \in {\rm End}(V)$,
\begin{enumerate}
\item[\rm (i)] $A$ lowers
$\lbrace V_i\rbrace_{i=0}^d$ if and only if
$\tilde A$ raises
$\lbrace V'_i\rbrace_{i=0}^d$;
\item[\rm (ii)] $A$ raises
$\lbrace V_i\rbrace_{i=0}^d$ if and only if
$\tilde A$ lowers
$\lbrace V'_i\rbrace_{i=0}^d$.
\end{enumerate}
\end{lemma}
\begin{proof}
(i) We invoke
Lemmas
\ref{lem:LowerRaise},
\ref{lem:RaiseLower}.
Let $\lbrace E_i\rbrace_{i=0}^d$ denote
the idempotent sequence for
$\lbrace V_i\rbrace_{i=0}^d$.
Then
$\lbrace \tilde E_i\rbrace_{i=0}^d$ is the
idempotent sequence
for 
$\lbrace V'_i\rbrace_{i=0}^d$.
Recall that the adjoint map is an antiisomorphism.
So for $0 \leq i,j\leq d$,
$E_iAE_j=0$ if and only if
$\tilde E_j \tilde A \tilde E_i=0$.
Consequently
Lemma
\ref{lem:LowerRaise}(ii) holds for
 $\lbrace E_i\rbrace_{i=0}^d$ and $A$, if and only if
Lemma
\ref{lem:RaiseLower}(ii)
holds for  
 $\lbrace \tilde E_i\rbrace_{i=0}^d$ and $\tilde A$.
The result now follows in view of Lemmas
\ref{lem:LowerRaise},
\ref{lem:RaiseLower}.
\\
\noindent (ii) Similar to the proof of (i) above.
\end{proof}

\begin{lemma}
\label{lem:Nildual}
For $A \in {\rm End}(V)$, $A$ is Nil if and only if $\tilde A$
is Nil. In this case the following flags are dual:
\begin{eqnarray}
\lbrace A^{d-i}V\rbrace_{i=0}^d, \qquad \qquad
\lbrace {\tilde A}^{d-i}V^*\rbrace_{i=0}^d.
\label{eq:twoflags}
\end{eqnarray}
\end{lemma}
\begin{proof} The adjoint map is an antiisomorphism.
So for $0 \leq i \leq d+1$,
$A^i = 0$ if and only if $\tilde A^i= 0$.
Therefore,
$A$ is Nil if and only if $\tilde A$
is Nil. In this case, 
the flags 
(\ref{eq:twoflags}) are dual since
for $0 \leq i,j\leq d$ such that $i+j=d-1$,
\begin{eqnarray*}
(A^{d-i}V,\tilde A^{d-j}V^*) = 
(A^{d-j} A^{d-i}V,V^*) = (A^{d+1}V,V^*) = (0,V^*) =  0.
\end{eqnarray*}
\end{proof}

\begin{lemma}
\label{lem:LRDdual}
Let $A,B$ denote an LR pair on $V$.  Then the following
{\rm (i)--(iii)} hold.
\begin{enumerate}
\item[\rm (i)] 
The pair $\tilde A, \tilde B$ is an LR pair on $V^*$.
\item[\rm (ii)]
The $(A,B)$-decomposition of $V$
is dual to the 
$(\tilde B,\tilde A)$-decomposition of $V^*$.
\item[\rm (iii)]
The $(B,A)$-decomposition of $V$
is dual to the 
$(\tilde A,\tilde B)$-decomposition of $V^*$.
\end{enumerate}
\end{lemma}
\begin{proof} 
Let $\lbrace V_i \rbrace_{i=0}^d$ denote
the $(A,B)$-decomposition of $V$.
Let $\lbrace V'_i \rbrace_{i=0}^d$ denote
the dual decomposition of $V^*$.
By construction
 $\lbrace V_i \rbrace_{i=0}^d$ is lowered by $A$ and
 raised by $B$.
By Lemma
\ref{lem:Adual},
 $\lbrace V'_i \rbrace_{i=0}^d$ is raised by $\tilde A$ and
lowered by $\tilde B$.
Therefore,
the decomposition
$\lbrace V'_{d-i} \rbrace_{i=0}^d$ is lowered by $\tilde A$ and
raised by $\tilde B$.
The results follow.
\end{proof}

\begin{lemma}
\label{lem:LRdual}
Let $A,B$ denote an LR pair on $V$, with idempotent sequence
$\lbrace E_i\rbrace_{i=0}^d$. Then
the LR pair $\tilde A, \tilde B$ has idempotent sequence
$\lbrace \tilde E_{d-i}\rbrace_{i=0}^d$.
\end{lemma}
\begin{proof}
By the 
assertion above
Lemma
\ref{lem:Adual}, along with
   Lemma
\ref{lem:LRDdual}(iii).
\end{proof}

\begin{lemma}
\label{lem:LRdualPar}
Let $A,B$ denote an LR pair on $V$, with parameter sequence
$\lbrace \varphi_i\rbrace_{i=1}^d$. Then
the LR pair $\tilde A, \tilde B$ has parameter sequence
$\lbrace \varphi_{d-i+1}\rbrace_{i=1}^d$.
\end{lemma}
\begin{proof}
An element in ${\rm End}(V)$ has the same trace as its
adjoint. 
Let $\lbrace \tilde \varphi_i \rbrace_{i=1}^d$ denote
the parameter sequence for
 $\tilde A, \tilde B$. For $1 \leq i \leq d$
 we show that $\tilde \varphi_i = \varphi_{d-i+1}$.
By Lemmas
\ref{lem:varphiTrace},
\ref{lem:LRdual} and since ${\rm tr}(XY)= {\rm tr}(YX)$,
\begin{eqnarray*}
\tilde \varphi_i = {\rm tr}(\tilde B \tilde A \tilde E_{d-i})
= 
 {\rm tr}(E_{d-i} A B)
= 
 {\rm tr}(AB E_{d-i})
=
\varphi_{d-i+1}.
\end{eqnarray*}
The result follows.
\end{proof}

\begin{lemma}
\label{lem:LRBdual}
Let $A,B$ denote an LR pair on $V$.
Then the following {\rm (i)--(iv)} hold.
\begin{enumerate}
\item[\rm (i)]
For an $(A,B)$-basis of $V$,
its dual is an
inverted $(\tilde A,\tilde B)$-basis of $V^*$.
\item[\rm (ii)]
For an inverted $(A,B)$-basis of $V$,
its dual is 
a $(\tilde A,\tilde B)$-basis of $V^*$.
\item[\rm (iii)]
For a $(B,A)$-basis of $V$,
its dual is an
inverted $(\tilde B,\tilde A)$-basis of $V^*$.
\item[\rm (iv)]
For an inverted $(B,A)$-basis of $V$,
its dual is a
$(\tilde B,\tilde A)$-basis of $V^*$.
\end{enumerate}
\end{lemma}
\begin{proof} (i) Let
$\lbrace v_i \rbrace_{i=0}^d$ denote an $(A,B)$-basis of $V$.
Let $\lbrace v'_i \rbrace_{i=0}^d$ denote the dual
basis of $V^*$.
With respect to 
$\lbrace v_i \rbrace_{i=0}^d$ the matrices representing
$A,B$ are given in
(\ref{eq:ABrep}). For these matrices
the transpose 
represents 
$\tilde A, \tilde B$ with respect to
$\lbrace v'_i \rbrace_{i=0}^d$.
Applying Lemma
\ref{lem:invABmatrix} to the LR pair $\tilde A,\tilde B$ and
using Lemma
\ref{lem:LRdualPar},
we
see that 
$\lbrace v'_i \rbrace_{i=0}^d$ is an inverted 
$(\tilde A, \tilde B)$-basis of $V^*$.
\\
\noindent (ii)--(iv) Similar to the proof of (i) above.
\end{proof}

\begin{lemma}
\label{lem:LRDiso}
A given LR pair $A,B$ on $V$ is
 isomorphic to
the LR pair $\tilde B,\tilde A$ on $V^*$.
\end{lemma}
\begin{proof} Let $\lbrace \varphi_i\rbrace_{i=1}^d$ denote
the parameter sequence of $A,B$.
By Lemma
\ref{lem:BAvAB}
the LR pair 
$B,A$ has parameter sequence
$\lbrace \varphi_{d-i+1}\rbrace_{i=1}^d$.
Now by
Lemma
\ref{lem:LRdualPar} 
the LR pair $\tilde B,\tilde A$
has parameter sequence
$\lbrace \varphi_{i}\rbrace_{i=1}^d$.
The LR pairs
$A,B$ and
$\tilde B,\tilde A$ have the same parameter sequence,
so they are isomorphic by
Proposition
\ref{prop:LRpairClass}.
\end{proof}

\section{The reflector $\dagger$}

\noindent Throughout this section the following notation
is in effect.
Let $V$ denote a vector space over $\mathbb F$ with
dimension $d+1$. Let $A,B$ denote an LR pair
on $V$. We discuss a certain antiautomorphism
$\dagger$ of ${\rm End}(V)$ called the reflector for
$A,B$.

\begin{proposition}
\label{prop:antiaut}
There exists a unique antiautomorphism
$\dagger$ of ${\rm End}(V)$ such that
$A^\dagger = B$ and
$B^\dagger=A$. Moreover
$(X^{\dagger})^{\dagger} = X$ for all
$X \in {\rm End}(V)$.
\end{proposition}
\begin{proof} We first show that $\dagger$  exists.
By Lemma
\ref{lem:LRDiso} there exists an isomorphism
 $\sigma$ of LR pairs
from
$A,B$ to  $\tilde B,\tilde A$.
Thus 
$\sigma:V\to V^*$ is an $\mathbb F$-linear bijection
such that $\sigma A = \tilde B \sigma$
and $\sigma B = \tilde A \sigma$.
By construction the map
${\rm End}(V)\to {\rm End}(V^*)$, $X\mapsto \sigma X\sigma^{-1}$
is an $\mathbb F$-algebra isomorphism that sends
$A\mapsto \tilde B$ and
$B\mapsto \tilde A$. 
Recall that the adjoint map 
${\rm End}(V)\to {\rm End}(V^*)$, $X\mapsto \tilde X$
is an $\mathbb F$-algebra antiisomorphism.
By these comments
the composition
\begin{equation*}
\dagger:\quad \begin{CD} 
{\rm End}(V)  @>>  {\rm adj} >  
 {\rm End}(V^*)  @>> X\mapsto \sigma^{-1}X\sigma > {\rm End}(V) 
                  \end{CD}
\end{equation*}
is an antiautomorphism of ${\rm End}(V)$ such that
$A^\dagger=B$ and $B^\dagger=A$. We have shown that
$\dagger$ exists. We now show that $\dagger$ is
unique. Let $\mu$ denote an antiautomorphism
of ${\rm End}(V)$ such that $A^\mu=B$
and $B^\mu=A$. We show that $\dagger = \mu$.
The composition
\begin{equation}
\label{eq:compdt}
 \begin{CD} 
{\rm End}(V)  @>>  \dagger >  
 {\rm End}(V)  @>> \mu^{-1} > {\rm End}(V) 
                  \end{CD}
\end{equation}
 is an $\mathbb F$-algebra isomorphism
that fixes each of $A,B$. By this and
Corollary
\ref{cor:ABgen}, 
 the
map
(\ref{eq:compdt}) fixes everything in
${\rm End}(V)$ and is therefore the identity map.
Consequently $\dagger = \mu$.
We have shown that $\dagger$ is unique.
To obtain the last assertion of the lemma,
note that $\dagger^{-1}$ is an antiautomorphism
of ${\rm End}(V)$ that sends $A\leftrightarrow B$.
Therefore $\dagger = \dagger^{-1}$ by the uniqueness
of $\dagger$. Consequently $(X^\dagger)^{\dagger}=X$
for all $X \in {\rm End}(V)$.
\end{proof}

\begin{definition}\rm
\label{def:REFdag}
By the {\it reflector} for $A,B$ (or the
{\it $(A,B)$-reflector}) we mean
the  antiautomorphism $\dagger$
 from
Proposition
\ref{prop:antiaut}.
\end{definition}

\begin{lemma}
\label{lem:trivDag}
Assume that $A,B$ is trivial. Then the
$(A,B)$-reflector fixes everything in 
${\rm End}(V)$.
\end{lemma}
\begin{proof} 
By assumption $d=0$,
so the identity $I$ is a basis
for the $\mathbb F$-vector space
${\rm End}(V)$.   
The 
$(A,B)$-reflector is $\mathbb F$-linear
and fixes $I$. The result follows.
\end{proof}

\begin{lemma}
\label{lem:daggerE}
Let  $\lbrace E_i \rbrace_{i=0}^d$ denote
the idempotent sequence for $A,B$.
Then the $(A,B)$-reflector
fixes $E_i$ for $0 \leq i \leq d$.
\end{lemma} 
\begin{proof}
Referring to
(\ref{eq:TwoE}),
for the equation on the left
apply
 $\dagger$ to each side and evaluate the result
using the equation on the right.
\end{proof}

\begin{lemma} The $(A,B)$-reflector is the same as the
$(B,A)$-reflector.
\end{lemma}
\begin{proof} By Proposition
\ref{prop:antiaut} and Definition
\ref{def:REFdag}.
\end{proof}

\begin{lemma} Let $\tilde \dagger$ denote the
reflector for the LR pair $\tilde A, \tilde B$.
Then the following diagram commutes:

\begin{equation*}
\begin{CD}
{\rm End}(V) @>{\rm adj}> >
                {\rm End}(V^*) 
           \\ 
          @V\dagger VV                     @VV\tilde \dagger V \\
               {\rm End}(V) @>>{\rm adj}> 
              {\rm End}(V^*) 
                   \end{CD}
\end{equation*}
\end{lemma}
\begin{proof} 
Recall from Corollary \ref{cor:ABgen} that ${\rm End}(V)$ is
generated by $A,B$.
Chase $A$ and $B$ around the diagram,
using Proposition
\ref{prop:antiaut} and Definition
\ref{def:REFdag}. The result follows.
\end{proof}

\section{The inverter $\Psi$}

\noindent Throughout this section the following notation
is in effect.
Let $V$ denote a vector space over $\mathbb F$
with dimension $d+1$. Let $A,B$ denote an
LR pair on $V$, with parameter sequence
$\lbrace \varphi_i \rbrace_{i=1}^d$ and
idempotent sequence $\lbrace E_i \rbrace_{i=0}^d$.
 We
discuss a map 
$\Psi \in
 {\rm End}(V)$ called the inverter for $A,B$. 
The name is motivated by Proposition
\ref{prop:invmap}
below. 

\begin{definition}
\label{def:REF}
\rm Define
\begin{eqnarray}
\Psi = \sum_{i=0}^d \frac{\varphi_1 \varphi_2 \cdots \varphi_i}{\varphi_d
\varphi_{d-1}\cdots \varphi_{d-i+1}} E_i.
\label{eq:REF}
\end{eqnarray}
We call $\Psi$ the {\it inverter} for $A,B$ or
the {\it $(A,B)$-inverter}.
\end{definition}

\begin{lemma}
\label{lem:InvTriv}
Assume that $A,B$ is trivial. Then $\Psi = I$.
\end{lemma}
\begin{proof} In Definition 
\ref{def:REF} set $d=0$ and note that $E_0=I$.
\end{proof}

\begin{lemma}
\label{lem:refInv}
The map $\Psi$ is invertible, and 
\begin{eqnarray}
\Psi^{-1} = \sum_{i=0}^d \frac{\varphi_d \varphi_{d-1} \cdots 
\varphi_{d-i+1}}{\varphi_1
\varphi_{2}\cdots \varphi_{i}} E_i.
\label{eq:REFinv}
\end{eqnarray}
\end{lemma}
\begin{proof} Use the fact that
 $E_i E_j = \delta_{i,j}E_i$ for $0 \leq i,j\leq d$ 
and $I = \sum_{i=0}^d E_i$.
\end{proof}

\begin{lemma} 
\label{lem:coefDup}
Referring to the sum
{\rm (\ref{eq:REF})}, for $0 \leq i \leq d$
the coefficients of $E_i$ and $E_{d-i}$ are the same;
in other words
\begin{eqnarray}
\label{eq:CoefSame}
\frac{\varphi_1 \varphi_2 \cdots \varphi_i}
{\varphi_d\varphi_{d-1}\cdots \varphi_{d-i+1}} = 
\frac{\varphi_1 \varphi_2 \cdots \varphi_{d-i}}
{\varphi_d\varphi_{d-1}\cdots \varphi_{i+1}}.
\end{eqnarray}
\end{lemma}
\begin{proof} Line 
(\ref{eq:CoefSame}) is readily checked.
\end{proof}

\begin{lemma}
\label{lem:refCom}
For $0 \leq i \leq d$,
\begin{eqnarray}
\Psi E_i = E_i \Psi = \frac{\varphi_1 \varphi_2 \cdots \varphi_i}
{\varphi_d\varphi_{d-1}\cdots \varphi_{d-i+1}} E_i.
\label{eq:Psicom}
\end{eqnarray}
\end{lemma}
\begin{proof} Use Definition
\ref{def:REF}.
\end{proof}

\begin{corollary}
\label{cor:PsiAct}
For $0 \leq i \leq d$ the following hold on $E_iV$:
\begin{eqnarray}
\label{eq:PsionEiV}
\Psi = 
 \frac{\varphi_1 \varphi_2 \cdots \varphi_i}
{\varphi_d\varphi_{d-1}\cdots \varphi_{d-i+1}} I,
\qquad \qquad
\Psi^{-1} = 
 \frac{\varphi_d \varphi_{d-1} \cdots \varphi_{d-i+1}}
{\varphi_1\varphi_2\cdots \varphi_{i}} I.
\end{eqnarray}
\end{corollary}
\begin{proof} Use
Lemma
\ref{lem:refCom}.
\end{proof}

\begin{lemma}
\label{lem:decompstab}
The map $\Psi$ fixes the $(A,B)$-decomposition of
$V$ and the $(B,A)$-decomposition of $V$.
\end{lemma}
\begin{proof}
The sequence $\lbrace E_iV\rbrace_{i=0}^d$ is the $(A,B)$-decomposition of
$V$.
The sequence $\lbrace E_{d-i}V\rbrace_{i=0}^d$ is the $(B,A)$-decomposition of
$V$.
 By Corollary
\ref{cor:PsiAct}, $\Psi E_iV=E_iV$ for
$0 \leq i \leq d$. The result follows.
\end{proof}

\begin{lemma}
\label{lem:flagstab}
The map $\Psi$ fixes each of the following flags:
\begin{eqnarray*}
\lbrace A^{d-i}V\rbrace_{i=0}^d,
\qquad \qquad 
\lbrace B^{d-i}V\rbrace_{i=0}^d.
\end{eqnarray*}
\end{lemma}
\begin{proof} 
The flag on the left (resp. right) is induced by the
 $(A,B)$-decomposition (resp. 
$(B,A)$-decomposition) of $V$.
The result follows in view of
Lemma
\ref{lem:decompstab}.
\end{proof}

\begin{lemma}
The map $\Psi$ commutes with $AB$ and $BA$.
\end{lemma}
\begin{proof}
By Lemma
\ref{lem:AiBione},
$E_i$ commutes with $AB$ and $BA$ for $0 \leq i \leq d$.
The result follows in view of Definition
\ref{def:REF}.
\end{proof}

\begin{lemma}
\label{lem:psiFix}
The map $\Psi$ is fixed by the $(A,B)$-reflector
$\dagger$ from Proposition
\ref{prop:antiaut}
and Definition
\ref{def:REFdag}.
\end{lemma}
\begin{proof}
By Lemma
\ref{lem:daggerE}
and 
Definition
\ref{def:REF}.
\end{proof}

\begin{lemma} 
\label{lem:ABBA}
The following maps are inverse:
\begin{enumerate}
\item[\rm (i)] the inverter for the LR pair $A,B$;
\item[\rm (ii)] the inverter for the LR pair $B,A$.
\end{enumerate}
\end{lemma}
\begin{proof} Use Lemmas
\ref{lem:Ebackward},
\ref{lem:BAvAB} along with
Lemma
\ref{lem:coefDup}.
\end{proof}

\begin{lemma} For nonzero $\alpha, \beta$ in $\mathbb F$
the following maps are the same:
\begin{enumerate}
\item[\rm (i)] the inverter for the LR pair $A,B$;
\item[\rm (ii)] the inverter for the LR pair $\alpha A, \beta B$.
\end{enumerate}
\end{lemma}
\begin{proof} Use Lemmas
\ref{lem:alphaBeta},
\ref{lem:alphaBetacom}.
\end{proof}

\begin{lemma}
\label{lem:reflectAdj}
The following maps are inverse:
\begin{enumerate}
\item[\rm (i)] the inverter for the LR pair $\tilde A, \tilde B$;
\item[\rm (ii)] the adjoint of the inverter for the LR pair $A,B$.
\end{enumerate}
\end{lemma}
\begin{proof}
Use Lemmas
\ref{lem:LRdual},
\ref{lem:LRdualPar}
along with Lemmas
\ref{lem:refInv},
\ref{lem:coefDup}.
\end{proof}

\noindent We turn our attention to the maps $\Psi A \Psi^{-1}$
and $\Psi^{-1}B\Psi$. We first consider how these
maps act on $E_iV$ for $0 \leq i \leq d$.

\begin{lemma}
\label{lem:PsiABact}
The following {\rm (i), (ii)} hold.
\begin{enumerate}
\item[\rm (i)]
$\Psi A \Psi^{-1}$ is zero on $E_0V$. Moreover 
 for $1 \leq i \leq d$ and on $E_iV$,
\begin{eqnarray}
\Psi A \Psi^{-1} = \frac{\varphi_{d-i+1}}{\varphi_i} A.
\label{eq:A1}
\end{eqnarray}
\item[\rm (ii)]
$\Psi^{-1} B \Psi$ is zero on $E_dV$.
Moreover 
 for $0 \leq i \leq d-1$ and on $E_iV$,
\begin{eqnarray}
\Psi^{-1} B \Psi = \frac{\varphi_{d-i}}{\varphi_{i+1}} B.
\label{eq:B1}
\end{eqnarray}
\end{enumerate}
\end{lemma}
\begin{proof} The decomposition
$\lbrace E_iV\rbrace_{i=0}^d$ is lowered by $A$ and
raised by $B$. The results follow from this
and Corollary
\ref{cor:PsiAct}.
\end{proof}

\begin{corollary} The $(A,B)$-decomposition of
$V$ is lowered by $\Psi A \Psi^{-1}$ and
raised by $\Psi^{-1}B\Psi$.
\end{corollary} 
\begin{proof}
By Lemma
\ref{lem:PsiABact},
and since
the $(A,B)$-decomposition of $V$
is equal to $\lbrace E_iV\rbrace_{i=0}^d$.
\end{proof}

\begin{lemma}
\label{lem:PsiTable}
In the table below, we give the matrices
that represent
 $\Psi A \Psi^{-1} $ and
$\Psi^{-1}B \Psi$ with respect to various 
bases for $V$.

     \bigskip
{\small
\centerline{
\begin{tabular}[t]{c|cc}
 {\rm type of basis} & {\rm matrix rep. of $\Psi A \Psi^{-1}$}
&
 {\rm matrix rep. of $\Psi^{-1} B \Psi$}
 \\
 \hline
 $(A,B)$ &
$
\left(
\begin{array}{ c c cc c c }
 0 & \varphi_d/\varphi_1   &   &&   & \bf 0  \\
  & 0  &  \varphi_{d-1}/\varphi_2  &&  &      \\ 
   &   &  0   & \cdot &&  \\
     &   &  & \cdot  & \cdot & \\
       &  & &  & \cdot & \varphi_1/\varphi_d \\
        {\bf 0}  &&  & &   &  0  \\
	\end{array}
	\right)
$
 &
$
	\left(
	\begin{array}{ c c cc c c }
	0 &   &   &&   & \bf 0  \\
	\varphi_d & 0  &   &&  &      \\
	 &  \varphi_{d-1} & 0   & &&  \\
	   &   & \cdot & \cdot  & & \\
	     &  & & \cdot & \cdot & \\
	      {\bf 0}  &&  & &  \varphi_1 & 0  \\
	      \end{array}
	      \right)
$
\\
\\
{\rm inv. $(A,B)$}
&	
$	
\left(
	\begin{array}{ c c cc c c }
	0 &   &   &&   & \bf 0  \\
	\varphi_1/\varphi_d & 0  &   &&  &      \\
	 &  \varphi_2/\varphi_{d-1} & 0   & &&  \\
	   &   & \cdot & \cdot  & & \\
	     &  & & \cdot & \cdot & \\
	      {\bf 0}  &&  & &  \varphi_d/ \varphi_1 & 0  \\
	      \end{array}
	      \right)
$
&
$
\left(
\begin{array}{ c c cc c c }
 0 & \varphi_1   &   &&   & \bf 0  \\
  & 0  &  \varphi_{2}  &&  &      \\ 
   &   &  0   & \cdot &&  \\
     &   &  & \cdot  & \cdot & \\
       &  & &  & \cdot & \varphi_d \\
        {\bf 0}  &&  & &   &  0  \\
	\end{array}
	\right)
$
\\
\\
{\rm  $(B,A)$}
&
$
\left(
	\begin{array}{ c c cc c c }
	0 &   &   &&   & \bf 0  \\
	\varphi_1 & 0  &   &&  &      \\
	 &  \varphi_{2} & 0   & &&  \\
	   &   & \cdot & \cdot  & & \\
	     &  & & \cdot & \cdot & \\
	      {\bf 0}  &&  & &  \varphi_d & 0  \\
	      \end{array}
	      \right)
$&$
\left(
\begin{array}{ c c cc c c }
 0 & \varphi_1/\varphi_d &   &&   & \bf 0  \\
  & 0  &  \varphi_2/\varphi_{d-1} &&  &      \\ 
   &   &  0   & \cdot &&  \\
     &   &  & \cdot  & \cdot & \\
       &  & &  & \cdot & \varphi_d/\varphi_1 \\
        {\bf 0}  &&  & &   &  0  \\
	\end{array}
	\right)
$
\\
\\
{\rm inv. $(B,A)$}
&
$
\left(
\begin{array}{ c c cc c c }
0 & \varphi_d   &   &&   & \bf 0  \\
 & 0  &  \varphi_{d-1}  &&  &      \\ 
  &   &  0   & \cdot &&  \\
    &   &  & \cdot  & \cdot & \\
      &  & &  & \cdot & \varphi_1 \\
       {\bf 0}  &&  & &   &  0  \\
       \end{array}
       \right)
$
&
$
       \left(
       \begin{array}{ c c cc c c }
       0 &   &   &&   & \bf 0  \\
       \varphi_d/\varphi_1 & 0  &   &&  &      \\
	&  \varphi_{d-1}/\varphi_2 & 0   & &&  \\
	  &   & \cdot & \cdot  & & \\
	    &  & & \cdot & \cdot & \\
	     {\bf 0}  &&  & &  \varphi_1/\varphi_d & 0  \\
	     \end{array}
	     \right)
$
   \end{tabular}}
     \bigskip
}
\end{lemma}
\begin{proof} Use Lemmas
\ref{lem:ABmatrix},
\ref{lem:invABmatrix},
\ref{lem:BAbasisMat},
\ref{lem:BAmatrixInv}
along with Lemma
\ref{lem:PsiABact}.
\end{proof}

\noindent Recall the reflector antiautomorphism $\dagger$ from
Proposition
\ref{prop:antiaut}
and
Definition
\ref{def:REFdag}.
\begin{lemma}
\label{lem:newLRP}
 The following {\rm (i)--(vii)} hold:
\begin{enumerate}
\item[\rm (i)]
the ordered pair
$A,\Psi^{-1}B\Psi$ is an LR pair on $V$;
\item[\rm (ii)]
the 
$(A,\Psi^{-1}B\Psi)$-decomposition of $V$ is equal to the
$(A,B)$-decomposition of $V$;
\item[\rm (iii)]
  the idempotent
sequence of
$A,\Psi^{-1}B\Psi$ is equal to the idempotent sequence of
$A,B$;
\item[\rm (iv)] an
$(A,\Psi^{-1}B\Psi)$-basis of $V$ is the same thing
as an $(A,B)$-basis of $V$;
\item[\rm (v)] 
the LR pair $A,\Psi^{-1}B\Psi$ has parameter sequence
$\lbrace\varphi_{d-i+1} \rbrace_{i=1}^d$;
\item[\rm (vi)] 
for the LR pair $A,\Psi^{-1}B\Psi $ the reflector is
the composition
\begin{equation*}
 \begin{CD} 
{\rm End}(V)  @>>  {\dagger} >  
 {\rm End}(V)  @>> X\mapsto \Psi^{-1}X\Psi > {\rm End}(V); 
                  \end{CD}
\end{equation*}
\item[\rm (vii)] 
for the LR pair $A,\Psi^{-1}B\Psi$ the inverter is
$\Psi^{-1}$.
\end{enumerate}
\end{lemma}
\begin{proof} (i), (ii) 
The $(A,B)$-decomposition of $V$ is lowered by
$A$ and raised by
$\Psi^{-1}B\Psi$.
\\
\noindent (iii) By 
(ii) above and Definition
\ref{def:ABE}.
\\
\noindent (iv) 
 By (ii) above and
 Definition
\ref{def:ABbasis}.
\\
\noindent (v) Consider the matrices that
represent
$A$ and $\Psi^{-1}B \Psi$ with respect to an
 $(A,B)$-basis of $V$. For $A$ this matrix is
given in
Lemma \ref{lem:ABmatrix}.
For $\Psi^{-1} B\Psi $ this matrix is given in
Lemma
\ref{lem:PsiTable}.
\\
\noindent (vi)
Use Lemma
\ref{lem:psiFix}.
\\
\noindent (vii) By  (iii), (v) above and
line (\ref{eq:REFinv}).
\end{proof}

\begin{lemma}
\label{lem:five2}
 The following {\rm (i)--(vii)} hold:
\begin{enumerate}
\item[\rm (i)]
the ordered pair
$\Psi A\Psi^{-1},B$ is an LR pair on $V$;
\item[\rm (ii)]
the 
$(\Psi A\Psi^{-1},B)$-decomposition of $V$ is equal to the
$(A,B)$-decomposition of $V$;
\item[\rm (iii)]
  the idempotent
sequence of
$\Psi A\Psi^{-1}, B$ is equal to the idempotent sequence of
$A,B$;
\item[\rm (iv)] a
$(B,\Psi A\Psi^{-1})$-basis of $V$ is the same thing
as a $(B,A)$-basis of $V$;
\item[\rm (v)] 
the LR pair $\Psi A\Psi^{-1},B$ has parameter sequence
$\lbrace\varphi_{d-i+1} \rbrace_{i=1}^d$;
\item[\rm (vi)]
for the LR pair $\Psi A\Psi^{-1},B$ the reflector is the
composition
\begin{equation*}
 \begin{CD} 
{\rm End}(V)  @>>  {\dagger} >  
 {\rm End}(V)  @>> X\mapsto \Psi X \Psi^{-1} > {\rm End}(V); 
                  \end{CD}
\end{equation*}
\item[\rm (vii)] 
for the LR pair $\Psi A\Psi^{-1},B$ the inverter is
$\Psi^{-1}$.
\end{enumerate}
\end{lemma}
\begin{proof} Similar to the proof of
Lemma
\ref{lem:newLRP}.
\end{proof}

\begin{proposition}
\label{prop:invmap}
The following three LR pairs are mutually isomorphic:
\begin{eqnarray}
A, \Psi^{-1}B\Psi
\qquad \qquad
B,A
\qquad \qquad
\Psi A \Psi^{-1}, B.
\label{eq:threeLRP}
\end{eqnarray}
\end{proposition}
\begin{proof} The LR pairs
(\ref{eq:threeLRP}) have the same parameter sequence
by 
Lemmas
\ref{lem:BAvAB},
\ref{lem:newLRP}(v),
\ref{lem:five2}(v).
The result follows in view of Proposition
\ref{prop:LRpairClass}.
\end{proof}

\begin{lemma}
For  $\sigma \in {\rm End}(V)$ the
following are equivalent:
\begin{enumerate}
\item[\rm (i)] $\sigma$ is an isomorphism of LR pairs
from $A,\Psi^{-1}B\Psi$ to $B,A$;
\item[\rm (ii)] $\sigma$ is an isomorphism of LR pairs
from $B,A$ to $\Psi A \Psi^{-1},B$;
\item[\rm (iii)] $\sigma$ sends each $(A,B)$-basis of $V$
to a $(B,A)$-basis of $V$.
\end{enumerate}
\end{lemma}
\begin{proof}
${\rm (i)}\Rightarrow {\rm (iii)}$
The map $\sigma $ sends each $(A,\Psi^{-1}B\Psi)$-basis
of $V$ to a $(B,A)$-basis of $V$. Also by Lemma
\ref{lem:newLRP}(iv),
an 
 $(A,\Psi^{-1}B\Psi)$-basis of $V$ is the same thing
as an $(A,B)$-basis of $V$.
\\
\noindent
${\rm (iii)}\Rightarrow {\rm (i)}$
The matrix that represents $A$ (resp. $\Psi^{-1}B\Psi$)
with respect to an $(A,B)$-basis of $V$ is equal to
the matrix that represents $B$ (resp. $A$) with
respect to a $(B,A)$-basis of $V$. 
\\
${\rm (ii)}\Rightarrow {\rm (iii)}$
The map $\sigma $ is an isomorphism of LR pairs
from
$A,B$ to $B,\Psi A \Psi^{-1}$.
So $\sigma $ sends each $(A,B)$-basis
of $V$ to a $(B,\Psi A \Psi^{-1})$-basis of $V$. Also by Lemma
\ref{lem:five2}(iv), a 
 $(B,\Psi A \Psi^{-1})$-basis of $V$ is the same thing
as a $(B,A)$-basis of $V$.
\\
\noindent 
${\rm (iii)}\Rightarrow {\rm (ii)}$
The matrix that represents $B$ (resp. $A$)
with respect to an $(A,B)$-basis of $V$ is equal to
the matrix that represents
$\Psi A\Psi^{-1}$ 
 (resp. $B$) with
respect to a $(B,A)$-basis of $V$. 
\end{proof}

\begin{lemma}
For  $\sigma \in {\rm End}(V)$ the
following are equivalent:
\begin{enumerate}
\item[\rm (i)] $\sigma$ is an isomorphism of LR pairs
from $A,\Psi^{-1}B\Psi$ to
 $\Psi A \Psi^{-1},B$;
\item[\rm (ii)] there exists $0 \not=\zeta \in \mathbb F$
such that  $\sigma=\zeta \Psi$.
\end{enumerate}
\end{lemma}
\begin{proof}
${\rm (i)}\Rightarrow {\rm (ii)}$
Using 
Definition
\ref{def:isoLRP},
we find that 
$ \Psi^{-1}\sigma $ commutes with
each of $A,\Psi^{-1}B\Psi$.
Now by Lemma
\ref{lem:isoFix} (applied to the LR pair
$A,\Psi^{-1}B\Psi$)
there exists $0 \not=\zeta \in \mathbb F$
such that 
$\Psi^{-1}\sigma = \zeta I$.
Therefore $\sigma=\zeta \Psi$.
\\
${\rm (ii)}\Rightarrow {\rm (i)}$
It suffices to show that $\Psi$ is an isomorphism
 of LR pairs
from $A,\Psi^{-1}B\Psi$ to
 $\Psi A \Psi^{-1},B$.
Since $\Psi$ is invertible the
map $\Psi:V\to V$ is an $\mathbb F$-linear
bijection. Observe that
$\Psi A = (\Psi A \Psi^{-1})\Psi$
and 
$\Psi (\Psi^{-1} B \Psi) = B \Psi$.
Now by 
Definition
\ref{def:isoLRP}, $\Psi$ is an isomorphism
 of LR pairs
from $A,\Psi^{-1}B\Psi$ to
 $\Psi A \Psi^{-1},B$.
\end{proof}

\begin{lemma}
The LR pairs
{\rm (\ref{eq:threeLRP})}
all have the same inverter.
\end{lemma}
\begin{proof} By Lemmas
\ref{lem:ABBA},
\ref{lem:newLRP}(vii),
\ref{lem:five2}(vii).
\end{proof}

\section{The outer and inner part}

\noindent Throughout this section the following
assumptions are in effect. We assume
that $d=2m$ is even. Let $V$ denote
a vector space over $\mathbb F$ with dimension
$d+1$.  Let $A,B$ denote an
LR pair on $V$, 
 with parameter sequence $\lbrace \varphi_i\rbrace_{i=1}^d$,
idempotent sequence $\lbrace E_i \rbrace_{i=0}^d$, and
inverter $\Psi$.
Note that $\lbrace E_iV\rbrace_{i=0}^d$ is the
$(A,B)$-decomposition of $V$, which is lowered by $A$
and raised by $B$.

\begin{definition}
\label{def:VoutVin}
\rm Define
\begin{eqnarray*}
V_{\rm out} = \sum_{j=0}^{m} E_{2j}V,
\qquad \qquad 
V_{\rm in} = \sum_{j=0}^{m-1} E_{2j+1}V.
\end{eqnarray*}
\end{definition}

\begin{lemma}
\label{lem:LRPV0V1}
We have
\begin{eqnarray}
\label{eq:LRPV0V1}
 V= V_{\rm out}+V_{\rm in} \qquad \quad  {\mbox{\rm (direct sum).}} 
\end{eqnarray}
Moreover
\begin{eqnarray*}
 {\rm dim}(V_{\rm out}) = m+1,
\qquad \qquad 
 {\rm dim}(V_{\rm in}) = m.
\end{eqnarray*}
\end{lemma}
\begin{proof} 
Since $\lbrace E_iV\rbrace_{i=0}^d$ is a decomposition of $V$.
\end{proof}

\begin{lemma}
\label{lem:LRPtrivialT}
 We have $V_{\rm out}\not=0$ and
 $V_{\rm in}\not=V$.
Moreover the following are equivalent:
{\rm (i)} 
$A,B$ is trivial;
{\rm (ii)} 
 $V_{\rm out}=V$; {\rm (iii)}
 $V_{\rm in}=0$.
\end{lemma}
\begin{proof} 
Use Example
\ref{def:triv}
and
Lemma
\ref{lem:LRPV0V1}.
\end{proof}

\begin{lemma}
\label{lem:EoutEin}
The following {\rm (i)--(iii)} hold:
\begin{enumerate}
\item[\rm (i)] for even $i$ $(0 \leq i \leq d)$,
the map $E_i$ leaves
$V_{\rm out}$ invariant, 
and is zero on $V_{\rm in}$;
\item[\rm (ii)] for odd $i$ $(0 \leq i \leq d)$,
the map $E_i$ leaves
$V_{\rm in}$ invariant, 
and is zero on $V_{\rm out}$;
\item[\rm (iii)] each of
$V_{\rm out}$, 
$V_{\rm in}$ is invariant under $\Psi$.
\end{enumerate}
\end{lemma}
\begin{proof} (i), (ii) Use Definition
\ref{def:VoutVin}.
\\
\noindent (iii) By (i), (ii) above and
Definition
\ref{def:REF}.
\end{proof}

\begin{lemma}
\label{lem:ABaction}
Referring to Definition
\ref{def:VoutVin},
\begin{eqnarray*}
AV_{\rm out} = V_{\rm in},
  \qquad 
 AV_{\rm in} \subseteq V_{\rm out},
\qquad  
BV_{\rm out} =V_{\rm in},
  \qquad 
 BV_{\rm in} \subseteq V_{\rm out}.
\end{eqnarray*}
Moreover
\begin{eqnarray*}
 A^2V_{\rm out} \subseteq V_{\rm out},
 \qquad 
 A^2V_{\rm in} \subseteq V_{\rm in},
\qquad 
 B^2V_{\rm out} \subseteq V_{\rm out},
 \qquad 
 B^2V_{\rm in} \subseteq V_{\rm in}.
\end{eqnarray*}
\end{lemma}
\begin{proof} 
By Definition
\ref{def:VoutVin} and the construction.
\end{proof}

\begin{definition}
\label{def:OUTERINNER}
\rm
Referring to Definition
\ref{def:VoutVin},
the subspace
$V_{\rm out}$
(resp. $V_{\rm in}$) will be called  the {\it outer part}
(resp. {\it inner part}) of $V$ with
respect to $A,B$.
\end{definition}

\begin{lemma}
\label{lem:InOutswap}
The outer part of $V$ with respect to $A,B$
coincides with the
outer part of $V$ with respect to $B,A$.
Moreover, the inner part of $V$ with respect to $A,B$
coincides with the
inner part of $V$ with respect to $B,A$.
\end{lemma}
\begin{proof}
By 
Lemma \ref{lem:Ebackward}
and
Definition
\ref{def:VoutVin}, along with the assumption that $d$ is even.
\end{proof}

\begin{lemma}
\label{lem:ANZ}
Assume that $A,B$ is nontrivial. Then $A$ and $B$ are  nonzero on
both
 $V_{\rm out}$ and 
 $V_{\rm in}$. 
\end{lemma}
\begin{proof} By Definition
\ref{def:VoutVin}
and the construction.
\end{proof}

\begin{definition}
\label{def:Ainout}
\rm
Using the LR pair $A,B$ we define
\begin{eqnarray}
A_{\rm out}, 
\qquad 
A_{\rm in},
\qquad
B_{\rm out}, 
\qquad 
B_{\rm in}
\label{eq:LRP6list}
\end{eqnarray}
 in ${\rm End}(V)$ as follows.
The map $A_{\rm out}$ (resp. $B_{\rm out}$) acts on
 $V_{\rm out}$ as $A$ (resp. $B$), and on
 $V_{\rm in}$ as zero.
The map $A_{\rm in}$ (resp. $B_{\rm in}$) acts on
 $V_{\rm in}$ as $A$ (resp. $B$), and on
 $V_{\rm out}$ as zero.
By construction
\begin{equation*}
A = A_{\rm out}+ 
 A_{\rm in}, 
 \qquad \qquad
B = B_{\rm out}+ 
 B_{\rm in}.
\end{equation*}
\end{definition}


\begin{lemma}
\label{lem:AiAoLI}
Assume that $A,B$ is nontrivial.
  Then
\begin{enumerate}
\item[\rm (i)] the maps
 $A_{\rm out}, A_{\rm in}$
are linearly independent over $\mathbb F$;
\item[\rm (ii)] the maps
 $B_{\rm out}, B_{\rm in}$
are linearly independent over $\mathbb F$.
\end{enumerate}
\end{lemma}
\begin{proof} (i)
Suppose we are given $r,s \in
\mathbb F$
such that
 $rA_{\rm out}+ sA_{\rm in}=0$. In this equation
 apply each side to
 $V_{\rm out}$,
to find $rA=0$ on $V_{\rm out}$. 
By Lemma
\ref{lem:ANZ} 
$A\not=0$ on $V_{\rm out}$. Therefore $r=0$.
One similarly shows that $s=0$.
\\
\noindent (ii) Similar to the proof of (i) above.
\end{proof}

\begin{definition}
\label{def:PoutPin}
 Define
\begin{eqnarray*}
\Psi_{\rm out} = \sum_{j=0}^{m}
\frac{\varphi_1 \varphi_2 \cdots \varphi_{2j}}
{\varphi_d\varphi_{d-1}\cdots \varphi_{d-2j+1}} E_{2j},
\qquad \qquad 
\Psi_{\rm in} = \sum_{j=0}^{m-1}
\frac{\varphi_2 \varphi_3 \cdots \varphi_{2j+1}}
{\varphi_{d-1}\varphi_{d-2}\cdots \varphi_{d-2j}} E_{2j+1}.
\end{eqnarray*}
\end{definition}

\begin{lemma}
The following {\rm (i)--(iv)} hold:
\begin{enumerate}
\item[\rm (i)]
the subspace $V_{\rm out}$ is invariant
under
$\Psi_{\rm out}$;
\item[\rm (ii)]
$\Psi_{\rm out}$ is zero on
$V_{\rm in}$;
\item[\rm (iii)]
the subspace $V_{\rm in}$ is invariant
under
$\Psi_{\rm in}$;
\item[\rm (iv)]
$\Psi_{\rm in}$ is zero on
$V_{\rm out}$.
\end{enumerate}
\end{lemma}
\begin{proof} By Lemma
\ref{lem:EoutEin}(i),(ii) and
Definition
\ref{def:PoutPin}.
\end{proof}

\noindent The following two propositions are obtained by routine computation.

\begin{proposition}
\label{lem:LRPA2B2C2Out}
 The elements $A^2, B^2$ act on
$V_{\rm out}$ as an LR pair.
For this LR pair,
\begin{enumerate}
\item[\rm (i)]
the diameter is $m$;
\item[\rm (ii)]
the parameter sequence  is
$\lbrace \varphi_{2j-1}\varphi_{2j}\rbrace_{j=1}^m$;
\item[\rm (iii)] the idempotent sequence is
given by the 
 actions of
$
\lbrace E_{2j}\rbrace_{j=0}^m$ on
$V_{\rm out}$;
\item[\rm (iv)] the inverter is
equal to the action of
$\Psi_{\rm out}$
 on $V_{\rm out}$.
\end{enumerate}
\end{proposition}

\begin{proposition}
\label{lem:LRPA2B2C2In}
Assume that $A,B$ is
 nontrivial.
Then 
 $A^2, B^2$ act on
$V_{\rm in}$ as an LR pair. For this LR pair,
\begin{enumerate}
\item[\rm (i)]
the diameter is $m-1$;
\item[\rm (ii)]
the
 parameter sequence is
$\lbrace \varphi_{2j}\varphi_{2j+1}\rbrace_{j=1}^{m-1}$;
\item[\rm (iii)]
 the idempotent sequence is given by the
actions of 
$\lbrace E_{2j+1}\rbrace_{j=0}^{m-1}$ on $V_{\rm in}$;
\item[\rm (iv)]
the inverter is equal to the action of
$\Psi_{\rm in}$
on $V_{\rm in}$.
\end{enumerate}
\end{proposition}

\begin{lemma}
\label{lem:psiFixOutIn}
The maps $\Psi_{\rm out}$, $\Psi_{\rm in}$ are fixed by the $(A,B)$-reflector
$\dagger$ from Proposition
\ref{prop:antiaut}
and Definition
\ref{def:REFdag}.
\end{lemma}
\begin{proof}
By Lemma
\ref{lem:daggerE}
and 
Definition
\ref{def:PoutPin}.
\end{proof}

\begin{lemma} Assume  that $A,B$ is nontrivial. Then
\begin{eqnarray*}
\Psi = \Psi_{\rm out} + \frac{\varphi_1}{\varphi_{d}} \Psi_{\rm in}.
\end{eqnarray*}
\end{lemma}
\begin{proof} Compare Definitions
\ref{def:REF},
\ref{def:PoutPin}.
\end{proof}

\begin{definition}\rm
We call $\Psi_{\rm out}$ (resp.
$\Psi_{\rm in}$)
the {\it outer inverter} (resp. {\it inner inverter})
for the LR pair $A,B$.
\end{definition}

\section{The projector $J$}

\noindent Throughout this section the following
assumptions are in effect. We assume
that $d=2m$ is even. Let $V$ denote
a vector space over $\mathbb F$ with dimension
$d+1$.  Let $A,B$ denote an
LR pair on $V$, 
 with parameter sequence $\lbrace \varphi_i\rbrace_{i=1}^d$,
idempotent sequence $\lbrace E_i \rbrace_{i=0}^d$, and
inverter $\Psi$. Recall the subspaces 
$V_{\rm out}$ and $V_{\rm in}$ from
Definition
\ref{def:VoutVin}.

\begin{definition}
\label{def:J}
\rm
Define $J \in {\rm End}(V)$ such that
$(J-I)V_{\rm out} = 0$ and $J V_{\rm in} = 0$.
Referring to
(\ref{eq:LRPV0V1}), the map $J$ (resp. $I-J$) acts as the projection
from $V$ onto $V_{\rm out}$ (resp. $V_{\rm in}$).
We call $J$ (resp. $I-J$) the {\it outer projector}
(resp. {\it inner projector}) for the LR pair $A,B$.
By the {\it projector} for $A,B$ we mean the outer projector.
\end{definition}

\begin{lemma}
\label{lem:Jtriv}
The map $J\not=0$.
If $A,B$ is trivial then
$J=I$. If $A,B$ is nontrivial
then $J, I$ are linearly independent over $\mathbb F$.
\end{lemma}
\begin{proof} 
Use (\ref{eq:LRPV0V1}) and
Lemma
\ref{lem:LRPtrivialT}.
\end{proof}

\begin{lemma}
\label{lem:Jfacts}
The following {\rm (i)--(v)} hold:
\begin{enumerate}
\item[\rm (i)] $J= \sum_{j=0}^{d/2} E_{2j}$;
\item[\rm (ii)] $J^2=J$;
\item[\rm (iii)] for even $i$ $(0 \leq i \leq d)$,
 $E_iJ=JE_i=E_i$;
\item[\rm (iv)] for odd $i$ $(0 \leq i \leq d)$,
$E_iJ=JE_i=0$;
\item[\rm (v)] $V_{\rm out} = JV$  and $V_{\rm in} = (I-J)V$.
\end{enumerate}
\end{lemma}
\begin{proof} (i) For the given equation the
two sides agree on $E_iV$ for $0 \leq i \leq d$.
\\
\noindent (ii)--(iv) Use (i) above and $E_rE_s = \delta_{r,s}E_r $
for $0 \leq r,s\leq d$.
\\
\noindent (v) By Definition
\ref{def:J}.
\end{proof}

\begin{lemma} 
\label{lem:rankandtrace}
For the map $J$ (resp. $I-J$) the rank and trace are equal to
$m+1$ (resp. $m$).
\end{lemma}
\begin{proof} By 
Lemma
\ref{lem:LRPV0V1},
Definition
\ref{def:J},
and linear algebra.
\end{proof}

\begin{lemma}
\label{lem:JRef}
The map $J$ is fixed by the $(A,B)$-reflector $\dagger$
from Proposition
\ref{prop:antiaut}
and Definition
\ref{def:REFdag}.
\end{lemma}
\begin{proof} By
Lemmas
\ref{lem:daggerE},
\ref{lem:Jfacts}(i).
\end{proof}

\begin{lemma}
The following maps are the same:
\begin{enumerate}
\item[\rm (i)] the projector for the LR pair $A,B$;
\item[\rm (ii)] the projector for the LR pair $B,A$.
\end{enumerate}
\end{lemma}
\begin{proof}
By Lemma
\ref{lem:InOutswap} and Definition
\ref{def:J}.
\end{proof}

\begin{lemma}
For nonzero  $\alpha, \beta \in \mathbb F$
the following maps are the same:
\begin{enumerate}
\item[\rm (i)] the projector for the LR pair $A,B$;
\item[\rm (ii)] the projector for the LR pair $\alpha A,\beta B$.
\end{enumerate}
\end{lemma}
\begin{proof}
By Lemma
\ref{lem:alphaBeta}
and Lemma
\ref{lem:Jfacts}(i).
\end{proof}

\begin{lemma}
The following maps are the same:
\begin{enumerate}
\item[\rm (i)] the projector for the LR pair $\tilde A,\tilde B$;
\item[\rm (ii)] the adjoint of the projector for  $A,B$.
\end{enumerate}
\end{lemma}
\begin{proof}
By 
Lemma \ref{lem:LRdual} and
Lemma
\ref{lem:Jfacts}(i).
\end{proof}

\begin{lemma}
\label{lem:INOutFacts}
Referring to Definition 
\ref{def:Ainout}
the following {\rm (i)--(iii)} hold:
\begin{enumerate}
\item[\rm (i)] $A_{\rm out} = AJ=(I-J)A$ and 
 $B_{\rm out} = BJ=(I-J)B$;
\item[\rm (ii)] $A_{\rm in} = JA = A(I-J)$ 
and $B_{\rm in} = JB = B(I-J)$;
\item[\rm (iii)] $A = AJ+JA$ 
and $B = BJ+JB$.
\end{enumerate}
\end{lemma}
\begin{proof} (i), (ii) For each given equation the
two sides agree on $V_{\rm out}$ and $V_{\rm in}$.
\\
\noindent (iii) By (i) above.
\end{proof}

\begin{lemma} $J$ commutes with each of $A^2, B^2, AB, BA$.
\end{lemma}
\begin{proof} Use Lemma
\ref{lem:INOutFacts}(iii).
\end{proof}

\begin{lemma}
Referring to Definition
\ref{def:PoutPin}
the following {\rm (i), (ii)} hold.
\begin{enumerate}
\item[\rm (i)]
$ \Psi_{\rm out} = 
J \Psi = \Psi J$.
\item[\rm (ii)] For $A,B$ nontrivial,
\begin{eqnarray*}
 \varphi_1/\varphi_d \Psi_{\rm in} = (I-J) \Psi = \Psi (I-J).
\end{eqnarray*}
\end{enumerate}
\end{lemma}
\begin{proof}
Use Definitions
\ref{def:REF},
\ref{def:PoutPin},
\ref{def:J}.
\end{proof}

\begin{lemma}
\label{lem:sigmaJ}
Let $V'$ denote a vector space over $\mathbb F$
with dimension $d+1$, and let $A',B'$ denote
an LR pair on $V'$. Let $J'$ denote the
projector for $A',B'$.
Let $\sigma $ denote an isomorphism of LR pairs from
$A,B$ to $A',B'$. Then $\sigma J = J' \sigma$.
\end{lemma}
\begin{proof} By Lemma
\ref{lem:isoMove}
and 
Lemma
\ref{lem:Jfacts}(i).
\end{proof}

\section{Similarity and bisimilarity}

In this section we describe two equivalence relations
for LR pairs, called similarity and bisimilarity.
Let $V$ denote a vector space over $\mathbb F$
with dimension $d+1$.

\begin{definition}\rm
Let $A,B$ and $A',B'$ denote LR pairs on $V$.
These LR pairs will be called {\it associates} whenever
there exist nonzero $\alpha, \beta $  in $ \mathbb F$
such that  $A'=\alpha A$ and $B'=\beta B$.
Associativity is
an equivalence relation.
\end{definition}

\begin{lemma}
\label{lem:twoview}
Let $A,B$ denote an LR pair on $V$.
Let $V'$ denote a vector space over $\mathbb F$
with dimension $d+1$, and let
$A',B'$ denote an LR pair on $V'$.
Let $\sigma :V\to V'$ denote an $\mathbb F$-linear bijection.
Then for nonzero $\alpha, \beta $ in $\mathbb F$ the following
{\rm (i)--(iii)} are equivalent:
\begin{enumerate}
\item[\rm (i)] $\sigma$ is an isomorphism of LR pairs from
$\alpha A,\beta B$ to $A',B'$;
\item[\rm (ii)] $\sigma$ is an isomorphism of LR pairs from
$A,B$ to $A'/\alpha,B'/\beta$;
\item[\rm (iii)]  $\alpha \sigma A  = A' \sigma$ and
 $\beta \sigma B  = B' \sigma$.
\end{enumerate}
\end{lemma}
\begin{proof} By 
Definition
\ref{def:isoLRP},  assertions (i), (iii) are equivalent
and
assertions (ii), (iii) are equivalent.
\end{proof}

\begin{lemma}
\label{lem:assocIso}
Let $A,B$ and $A',B'$ denote LR pairs over $\mathbb F$.
Then the following are equivalent:
\begin{enumerate}
\item[\rm (i)] there exists an LR pair over $\mathbb F$
that is associate to $A,B$ and isomorphic to $A',B'$;
\item[\rm (ii)] there exists an LR pair over $\mathbb F$
that is isomorphic to $A,B$ and associate to $A',B'$.
\end{enumerate}
\end{lemma}
\begin{proof} Pick nonzero $\alpha, \beta $ in $\mathbb F$.
By
 Lemma
\ref{lem:twoview}(i),(ii) the LR pair $\alpha A, \beta B$
satisfies condition (i) in the present lemma if and only
if the LR pair $A'/\alpha, B'/\beta$ satisfies condition
(ii) in the present lemma. The result follows. 
\end{proof}

\begin{definition}
\label{def:SIM}
\rm
Let $A,B$ and $A',B'$ denote LR pairs over $\mathbb F$.
These LR pairs will be called {\it similar} whenever
they satisfy the equivalent conditions {\rm (i), (ii)}
in Lemma
\ref{lem:assocIso}.
Similarity is
an equivalence relation.
\end{definition}

\begin{lemma}
Let $A,B$ (resp. $A',B'$) denote an LR pair over $\mathbb F$,
with parameter sequence 
$\lbrace \varphi_i \rbrace_{i=1}^d$ (resp.
$\lbrace \varphi'_i \rbrace_{i=1}^d$).
Then the following are equivalent:
\begin{enumerate}
\item[\rm (i)] the LR pairs
$A,B$ and $A',B'$ are similar;
\item[\rm (ii)] the ratio
$\varphi'_i /\varphi_i$ is independent of $i$ for
$1 \leq i \leq d$.
\end{enumerate}
\end{lemma}
\begin{proof}
${\rm (i)}\Rightarrow {\rm (ii)}$  By
Lemma
\ref{lem:assocIso}
and Definition
\ref{def:SIM}, there exist nonzero $\alpha, \beta$
in $\mathbb F$ such that
$\alpha A, \beta B$ is isomorphic to
$A',B'$.
Recall from Lemma
\ref{lem:alphaBetacom} that
$\alpha A, \beta B$ has parameter sequence
$\lbrace \alpha \beta \varphi_i \rbrace_{i=1}^d$.
Now by Proposition
\ref{prop:LRpairClass},
$\varphi'_i = \alpha \beta \varphi_i$ for
$1 \leq i \leq d$. Therefore $\varphi'_i /\varphi_i$
is independent of $i$ for $1 \leq i \leq d$.
\\
\noindent 
${\rm (ii)}\Rightarrow {\rm (i)}$
Let $\alpha$ denote the common value of $\varphi'_i /\varphi_i$
for $1 \leq i \leq d$.
  By Lemma
\ref{lem:alphaBetacom} and the construction,
 the LR pairs 
$\alpha A,B$ and $A',B'$ have the same parameter sequence
$\lbrace \varphi'_i\rbrace_{i=1}^d$.
Therefore they are isomorphic by Proposition
\ref{prop:LRpairClass}. 
Now 
the LR pairs $A,B$ and $A',B'$ are similar
by Lemma
\ref{lem:assocIso}
and Definition
\ref{def:SIM}.
\end{proof}

\noindent We now describe the bisimilarity relation.
For the rest of this section, assume that $d=2m$ is even.
Until further notice let $A,B$ denote an LR pair on
$V$, with parameter sequence $\lbrace \varphi_i\rbrace_{i=1}^d$
and idempotent sequence $\lbrace E_i \rbrace_{i=0}^d$.
Recall that $\lbrace E_iV\rbrace_{i=0}^d$ is the
$(A,B)$-decomposition of $V$, which is lowered by $A$
and raised by $B$.

\begin{lemma} 
\label{lem:iso2view}
Let $V'$ denote a vector space over $\mathbb F$ with dimension $d+1$,
and let $A',B'$ denote an LR pair on $V'$. Let
$\sigma :V\to V'$ denote an $\mathbb F$-linear bijection.
Then the following are equivalent:
\begin{enumerate}
\item[\rm (i)] $\sigma $ is an isomorphism of LR pairs from
$A,B$ to $A',B'$;
\item[\rm (ii)] all of
\begin{eqnarray}
&& \sigma A_{\rm out} = A'_{\rm out} \sigma,
\qquad \qquad
\sigma A_{\rm in} = A'_{\rm in} \sigma,
\label{eq:Aform}
\\
&&
\sigma B_{\rm out} = B'_{\rm out} \sigma,
\qquad \qquad  
\sigma B_{\rm in} = B'_{\rm in} \sigma.
\label{eq:Bform}
\end{eqnarray}
\end{enumerate}
\end{lemma}
\begin{proof}
${\rm (i)}\Rightarrow {\rm (ii)}$ 
Use Lemma 
\ref{lem:INOutFacts}(i),(ii)
and Lemma
\ref{lem:sigmaJ}.
\\
\noindent
${\rm (ii)}\Rightarrow {\rm (i)}$ 
Add the two equations in (\ref{eq:Aform}) and use
$A=A_{\rm out} +A_{\rm in}$,
$A'=A'_{\rm out} +A'_{\rm in}$ to obtain
$\sigma A = A' \sigma$.
Similarly we obtain
$\sigma B = B' \sigma$.
Now by Definition
\ref{def:isoLRP} the map $\sigma $ is an isomorphism of LR pairs from
$A,B$ to $A',B'$.
\end{proof}

\begin{lemma}
\label{lem:LRPInOut}
Let
$\alpha_{\rm out}, 
\alpha_{\rm in},
\beta_{\rm out},
\beta_{\rm in} 
$
denote nonzero scalars in $\mathbb F$.
 Then the ordered pair
\begin{eqnarray}
\label{eq:LRPNewLRT}
\alpha_{\rm out}A_{\rm out}
+\alpha_{\rm in}A_{\rm in},
\qquad
\beta_{\rm out} B_{\rm out}+
\beta_{\rm in}B_{\rm in}
\end{eqnarray}
is an LR pair on $V$, with idempotent sequence
$\lbrace E_i \rbrace_{i=0}^d$. 
The outer part of $V$ with respect to
{\rm (\ref{eq:LRPNewLRT})}  coincides with the
outer part of $V$ with respect to $A,B$.
The inner part of $V$ with respect to
{\rm (\ref{eq:LRPNewLRT})}  coincides with the
inner part of $V$ with respect to $A,B$.
The projector for the LR pair
{\rm (\ref{eq:LRPNewLRT})}
coincides with the projector for $A,B$.
The LR pair 
{\rm (\ref{eq:LRPNewLRT})}
has parameter
sequence  $\lbrace f_i \varphi_i\rbrace_{i=1}^d$, 
where
\begin{eqnarray*}
&&
f_i = \begin{cases}
\alpha_{\rm out}\beta_{\rm in}
&  {\mbox{\rm if $i$ is even}}; \\
\alpha_{\rm in}\beta_{\rm out}
& {\mbox{\rm if $i$ is odd}}
\end{cases}
\qquad \qquad (1 \leq i \leq d).
\end{eqnarray*}
\end{lemma}
\begin{proof} By construction.
\end{proof}

\begin{lemma}
\label{lem:ApBp}
Let $A',B'$ denote an LR pair on $V$. 
Let
$\alpha_{\rm out}, 
\alpha_{\rm in},
\beta_{\rm out},
\beta_{\rm in} 
$
denote nonzero scalars in $\mathbb F$.
Then the following are equivalent:
\begin{enumerate}
\item[\rm (i)] both
\begin{eqnarray*}
A' = \alpha_{\rm out}A_{\rm out}
+\alpha_{\rm in}A_{\rm in}, \qquad \qquad
B' = 
\beta_{\rm out} B_{\rm out}+
\beta_{\rm in}B_{\rm in};
\end{eqnarray*}
\item[\rm (ii)] all of 
\begin{eqnarray*}
&&
A'_{\rm out} =  
\alpha_{\rm out}A_{\rm out},
\qquad \qquad 
A'_{\rm in} =  
\alpha_{\rm in}A_{\rm in},
\\
&&
B'_{\rm out} =  
\beta_{\rm out}B_{\rm out},
\qquad \qquad 
B'_{\rm in} =  
\beta_{\rm in}B_{\rm in}.
\end{eqnarray*}
\end{enumerate}
\end{lemma}
\begin{proof}
${\rm (i)}\Rightarrow {\rm (ii)}$  Use
Definition
\ref{def:Ainout}
and
 Lemma
\ref{lem:LRPInOut}.
\\
${\rm (ii)}\Rightarrow {\rm (i)}$ 
By 
Definition \ref{def:Ainout} we have
$A' = A'_{\rm out} + A'_{\rm in}$ and
$B' = B'_{\rm out} + B'_{\rm in}$.
\end{proof}

\noindent Referring to Lemmas
\ref{lem:LRPInOut},
\ref{lem:ApBp}, we now consider the case
in which $\alpha_{\rm in}=1$
and $\beta_{\rm in}=1$.

\begin{definition}\rm
Let $A',B'$ denote an LR pair on $V$.
The LR pairs $A,B$ and $A',B'$ will be called 
{\it biassociates} whenever there exist nonzero
$\alpha, \beta$ in $\mathbb F$
such that
\begin{eqnarray*}
A' = \alpha A_{\rm out}
+A_{\rm in},
\qquad \qquad
B' = \beta  B_{\rm out}+
B_{\rm in}.
\end{eqnarray*}
Biassociativity is an equivalence relation.
\end{definition}

\begin{lemma}
\label{lem:twoviewBisim}
Let $V'$ denote a vector space over $\mathbb F$
with dimension $d+1$, and let
$A',B'$ denote an LR pair on $V'$.
Let $\sigma :V\to V'$ denote an $\mathbb F$-linear bijection.
Then for nonzero $\alpha, \beta $ in $\mathbb F$ the following
{\rm (i)--(iii)} are equivalent:
\begin{enumerate}
\item[\rm (i)] $\sigma$ is an isomorphism of LR pairs from
$\alpha A_{\rm out}+A_{\rm in},
\beta B_{\rm out}+B_{\rm in}$ to $A',B'$;
\item[\rm (ii)] $\sigma$ is an isomorphism of LR pairs from
$A,B$ to 
$\alpha^{-1} A'_{\rm out} + A'_{\rm in},
\beta^{-1} B'_{\rm out} + B'_{\rm in}$;
\item[\rm (iii)]  all of
\begin{eqnarray*}
&&\alpha \sigma A_{\rm out}  = A'_{\rm out} \sigma,
\qquad \qquad 
\sigma A_{\rm in}  = A'_{\rm in} \sigma,
\\
&&
\beta \sigma B_{\rm out}  = B'_{\rm out} \sigma,
\qquad \qquad 
\sigma B_{\rm in}  = B'_{\rm in} \sigma.
\end{eqnarray*}
\end{enumerate}
\end{lemma}
\begin{proof} 
By Lemmas
\ref{lem:iso2view},
\ref{lem:ApBp} the assertions (i), (iii) are
equivalent, and the assertions
(ii), (iii) are equivalent.
\end{proof}

\begin{lemma}
\label{lem:LRPBiassocIso}
Let $A,B$ and $A',B'$ denote LR pairs over $\mathbb F$
that have diameter $d$.
Then the following are equivalent:
\begin{enumerate}
\item[\rm (i)] there exists an LR pair  over $\mathbb F$
that is biassociate to $A,B$ and isomorphic to $A',B'$;
\item[\rm (ii)] there exists an LR pair over $\mathbb F$
that is isomorphic to $A,B$ and biassociate to $A',B'$.
\end{enumerate}
\end{lemma}
\begin{proof} Pick nonzero $\alpha, \beta $ in $\mathbb F$.
By Lemma
\ref{lem:twoviewBisim}(i),(ii) the LR pair
$\alpha A_{\rm out} + A_{\rm in}, 
\beta B_{\rm out} + B_{\rm in}$ satisfies condition (i)
in the present lemma if and only if the LR pair
$\alpha^{-1} A'_{\rm out} + A'_{\rm in}, 
\beta^{-1} B'_{\rm out} + B'_{\rm in}$ satisfies condition (ii)
in the present lemma.
The result follows.
\end{proof}

\begin{definition}
\label{def:BIAS}
\rm
Let $A,B$ and $A',B'$ denote LR pairs over $\mathbb F$
that have diameter $d$.
Then $A,B$ and $A',B'$ will be called 
{\it bisimilar} whenever the equivalent conditions {\rm (i), (ii)}
hold in Lemma
\ref{lem:LRPBiassocIso}.
\end{definition}

\begin{lemma}
Let $A,B$ (resp. $A',B'$) denote
an LR pair over $\mathbb F$, with
parameter sequence $\lbrace \varphi_i\rbrace_{i=1}^d$
(resp.
$\lbrace \varphi'_i\rbrace_{i=1}^d$).
Then the following are equivalent:
\begin{enumerate}
\item[\rm (i)] the LR pairs
$A,B$ and $A',B'$ are bisimilar;
\item[\rm (ii)] the ratio $\varphi'_i / \varphi_i$
is independent of $i$ for $i$ even $(1 \leq i \leq d)$,
and 
 independent of $i$ for $i$ odd $(1 \leq i \leq d)$.
\end{enumerate}
\end{lemma}
\begin{proof}
${\rm (i)}\Rightarrow {\rm (ii)}$  
By Lemma
\ref{lem:LRPBiassocIso}
and Definition
\ref{def:BIAS},
there exist nonzero $\alpha, \beta $ in $\mathbb F$
such that
$\alpha A_{\rm out}+A_{\rm in},
\beta B_{\rm out}+B_{\rm in}$ is isomorphic to $A',B'$.
By Lemma
\ref{lem:LRPInOut}
the LR pair
$\alpha A_{\rm out}+A_{\rm in},
\beta B_{\rm out}+B_{\rm in}$ has parameter sequence
$\lbrace f_i \varphi_i\rbrace_{i=1}^d$, where
\begin{eqnarray*}
f_i =
 \begin{cases}
\alpha  &  {\mbox{\rm if $i$ is even}}; \\
\beta & {\mbox{\rm if $i$ is odd}}
\end{cases}
\qquad \qquad (1 \leq i \leq d).
\end{eqnarray*}
By Proposition \ref{prop:LRpairClass}
$\varphi'_i = f_i \varphi_i$ for
$1 \leq i \leq d$.
By these comments
 $\varphi'_i / \varphi_i$
is independent of $i$ for $i$ even $(1 \leq i \leq d)$,
and 
 independent of $i$ for $i$ odd $(1 \leq i \leq d)$.
\\
\noindent
${\rm (ii)}\Rightarrow {\rm (i)}$  
By assumption there exist nonzero $\alpha, \beta$ in
$\mathbb F$ such that
\begin{eqnarray*}
\varphi'_i/\varphi_i =
 \begin{cases}
\alpha  &  {\mbox{\rm if $i$ is even}}; \\
\beta & {\mbox{\rm if $i$ is odd}}
\end{cases}
\qquad \qquad (1 \leq i \leq d).
\end{eqnarray*}
By Lemma
\ref{lem:LRPInOut} and the construction,
the
 LR pairs
$\alpha A_{\rm out} + A_{\rm in},
\beta B_{\rm out} + B_{\rm in} $ and
$A',B'$ have the same
 parameter sequence $\lbrace \varphi'_i\rbrace_{i=1}^d$.
Therefore
they are isomorphic by
Proposition \ref{prop:LRpairClass}.
Now $A,B$ and
$A',B'$ are bisimilar in view of
Lemma
\ref{lem:LRPBiassocIso}
and Definition
\ref{def:BIAS}.
\end{proof}

\section{Constrained sequences}

 In this section we consider a type
of finite sequence, said to be constrained.
We classify the constrained sequences. This classification
will be used
later in the paper.

\medskip

\noindent Throughout this section, $n$ denotes a nonnegative
integer and
 $\lbrace \rho_i \rbrace_{i=0}^n$  denotes a sequence
of scalars taken from $\mathbb F$.

\begin{definition}
\label{def:constrain}
\rm
The sequence $\lbrace \rho_i \rbrace_{i=0}^n$ 
 is said to be {\it constrained} whenever
\begin{enumerate}
\item[\rm (i)] $\rho_i \rho_{n-i}=1$ for $0 \leq i \leq n$;
\item[\rm (ii)] there exist $a,b,c \in \mathbb F$ that
are not all zero and
 $a\rho_{i-1} + b\rho_i + c \rho_{i+1} = 0
$ for $1 \leq i \leq n-1$.
\end{enumerate}
\end{definition}

\noindent Shortly we will classify the constrained sequences.
We will use an inductive  argument  based on the 
 following observation.

\begin{lemma} Assume that $n\geq 2$ and 
the sequence $\lbrace \rho_i \rbrace_{i=0}^n$  is constrained.
Then the sequence $\rho_1, \rho_2, \ldots, \rho_{n-1}$
is constrained.
\end{lemma}

\noindent The sequence $\lbrace \rho_i \rbrace_{i=0}^n$ is called
{\it geometric} whenever $\rho_i \not=0$ for
$0 \leq i \leq n$ and $\rho_i/\rho_{i-1}$ is independent of
$i$ for $1 \leq i \leq n$.
The following are equivalent:
(i) $\lbrace \rho_i \rbrace_{i=0}^n$ is geometric;
(ii) there exist nonzero $r,\xi \in \mathbb F$
such that $\rho_i = \xi r^i$ for $0 \leq i \leq n$.
In this case $\xi = \rho_0$ and
$r= \rho_i /\rho_{i-1}$ for $1 \leq i \leq n$.
\medskip

\noindent 
We now classify the constrained sequences.
The case of $n$ even and $n$ odd will
be treated separately.

\begin{proposition}
\label{prop:nevenPre}
Assume that $n$ is even. Then for
the sequence
 $\lbrace \rho_i \rbrace_{i=0}^n$
 the following {\rm (i)--(iii)} are equivalent:
\begin{enumerate}
\item[\rm (i)] 
$\lbrace \rho_i \rbrace_{i=0}^n$ is constrained;
\item[\rm (ii)]
$\lbrace \rho_i \rbrace_{i=0}^n$ is geometric and $\rho_{n/2}
\in \lbrace 1,-1\rbrace$;
\item[\rm (iii)] there exist
$0 \not=r \in \mathbb F$
and $\varepsilon \in \lbrace 1,-1\rbrace$
 such that
$\rho_i = \varepsilon r^{i-n/2}$ for $0 \leq i \leq n$.
\end{enumerate}
Assume that {\rm (i)--(iii)} hold.
Then $r=\rho_i /\rho_{i-1}$ for $1 \leq i \leq n$,
 and $\varepsilon = \rho_{n/2}$.
\end{proposition}
\begin{proof}
${\rm (i)}\Rightarrow {\rm (iii)}$  Our proof is by induction on
$n$. First assume that $n=0$. Then $\rho^2_0=1$.
Condition 
 (iii) holds with
$\varepsilon = \rho_0$ and arbitrary  $0 \not=r \in \mathbb F$.
Next assume that $n=2$. Then
$\rho_0\rho_2=1$ and $\rho^2_1=1$.
Condition (iii) holds with
$\varepsilon=\rho_1$ and
$r=\rho_1/\rho_0$.
Next assume that $n\geq 4$.
By Definition 
\ref{def:constrain}(ii), there exist $a,b,c\in \mathbb F$
that are not all zero and $a\rho_{i-1}+b\rho_i+c\rho_{i+1}=0$
for $1 \leq i \leq n-1$.
Define $m=n-2$ and $\rho'_i=\rho_{i+1}$ for $0 \leq i \leq m$.
By construction $\rho'_i \rho'_{m-i}=1$ for $0 \leq i \leq m$.
Moreover $a \rho'_{i-1} + b\rho'_i + c \rho'_{i+1} = 0$
for $1 \leq i \leq m-1$. By induction there exist
 $0 \not=r \in \mathbb F$ and
$\varepsilon \in \lbrace 1,-1\rbrace$ 
such that $\rho'_i = \varepsilon r^{i-m/2}$
for $0 \leq i \leq m$.
Now
\begin{eqnarray}
\label{eq:part}
\rho_i = \varepsilon r^{i-n/2} \qquad \qquad (1 \leq i \leq n-1).
\end{eqnarray}
We show that
$\rho_0 = \varepsilon r^{-n/2}$ and
$\rho_n = \varepsilon r^{n/2}$.
Since $\rho_0 \rho_n = 1$, it suffices to show that
$\rho_0 = \varepsilon r^{-n/2}$.
We claim that
\begin{equation}
\label{eq:matcheck}
{\rm det} \left(
\begin{array}{ccc}
\rho_0 & \rho_1 & \rho_{n-2} 
\\
\rho_1& \rho_2 & \rho_{n-1} 
\\
\rho_2 & \rho_3 & \rho_n 
\end{array}
\right)
= \frac{-r^2(\rho_0 -\varepsilon r^{-n/2})^2}{\rho_0}. 
\end{equation}
To verify
(\ref{eq:matcheck}),
evaluate the determinant using
(\ref{eq:part}) 
and $\rho_0 \rho_n = 1$,
and simplify the result.
The claim is proven.
For the matrix
in (\ref{eq:matcheck}), $a({\rm top} \;{\rm row})+
b({\rm middle}\; {\rm row})+c({\rm bottom}\; {\rm row})=0$.
The matrix
is singular, so its determinant is zero.
 Therefore $\rho_0 = \varepsilon r^{-n/2}$ as desired.
\\
${\rm (iii)}\Rightarrow {\rm (i)}$
By construction $\rho_i \rho_{n-i}=1$
for $0 \leq i \leq n$.
Define $a=r$, $b=-1$, $c=0$.
Then $a\rho_{i-1} + b\rho_i + c \rho_{i+1}=0$
for $1 \leq i \leq n-1$.
\\
${\rm (ii)}\Leftrightarrow {\rm (iii)}$
Routine.
\end{proof}

\begin{proposition}
\label{prop:noddPre}
Assume that $n$ is odd.
 Then for the sequence $\lbrace \rho_i \rbrace_{i=0}^n$ 
the following {\rm (i)--(iii)} are equivalent:
\begin{enumerate}
\item[\rm (i)]
  $\lbrace \rho_i \rbrace_{i=0}^n$ is constrained; 
\item[\rm (ii)]
the sequences $\rho_0, \rho_2, \ldots, \rho_{n-1}$
and $\rho^{-1}_n,\rho^{-1}_{n-2},
\ldots, \rho^{-1}_1$ are equal and geometric;
\item[\rm (iii)]
 there exist nonzero 
$s,\xi \in \mathbb F$ such that
\begin{eqnarray}
\label{eq:gammaSol}
\rho_i = \begin{cases}
\xi s^{i/2}  &  {\mbox{\rm if $i$ is even}}; \\
\xi^{-1}s^{(i-n)/2} & {\mbox{\rm if $i$ is odd}}
\end{cases}
\qquad \qquad (0 \leq i \leq n).
\end{eqnarray}
\end{enumerate}
Assume that {\rm (i)--(iii)} hold.
Then $s=\rho_i /\rho_{i-2}$ for $2 \leq i \leq n$ and $\xi = \rho_0$.
\end{proposition}
\begin{proof}
${\rm (i)}\Rightarrow {\rm (iii)}$  Our proof is by induction on
$n$. First assume that $n=1$. Then (iii) holds with
$\xi = \rho_0$ and arbitrary  $0 \not=s \in \mathbb F$.
Next assume that $n=3$. Then (iii) holds with 
$\xi=\rho_0$ and
$s=\rho_2/\rho_0$.
 Next assume that $n\geq 5$.
By Definition
\ref{def:constrain}(ii),
there exist $a,b,c \in \mathbb F$ that
are not all zero and
 $a\rho_{i-1} + b\rho_i + c \rho_{i+1} = 0
$ for $1 \leq i \leq n-1$.
Define $m=n-2$ and $\rho'_i=\rho_{i+1}$ for $0 \leq i \leq m$.
By construction $\rho'_i \rho'_{m-i}=1$ for $0 \leq i \leq m$.
Moreover $a \rho'_{i-1} + b\rho'_i + c \rho'_{i+1} = 0$
for $1 \leq i \leq m-1$. By induction there exist
nonzero $s,x \in \mathbb F$ such that
\begin{eqnarray*}
\rho'_i = \begin{cases}
x s^{i/2}  &  {\mbox{\rm if $i$ is even}}; \\
x^{-1}s^{(i-m)/2} & {\mbox{\rm if $i$ is odd}}
\end{cases}
\qquad \qquad (0 \leq i \leq m).
\end{eqnarray*}
\noindent Define $\xi = x^{-1}s^{-(1+m)/2}$. Then
\begin{eqnarray}
\label{eq:middle}
\rho_i = \begin{cases}
\xi s^{i/2}  &  {\mbox{\rm if $i$ is even}}; \\
\xi^{-1}s^{(i-n)/2} & {\mbox{\rm if $i$ is odd}}
\end{cases}
\qquad \qquad (1 \leq i \leq n-1).
\end{eqnarray}
We show  that
$\rho_0= \xi$ and $\rho_n = \xi^{-1}$. 
Since $\rho_0 \rho_n = 1$, it suffices to show that
$\rho_0 = \xi$.
We claim that
\begin{equation}
\label{eq:mat2}
{\rm det} \left(
\begin{array}{ccc}
\rho_0 & \rho_1 & \rho_2 
\\
\rho_1& \rho_2 & \rho_3 
\\
\rho_2 & \rho_3 & \rho_4 
\end{array}
\right)
+
\xi^2 s^2 {\rm det} \left(
\begin{array}{ccc}
\rho_0 & \rho_1 & \rho_{n-2} 
\\
\rho_1& \rho_2 & \rho_{n-1} 
\\
\rho_2 & \rho_3 & \rho_n 
\end{array}
\right)
= \frac{-(\rho_0 -\xi)^2}{ s^{n-3}\xi^{2}\rho_0}. 
\end{equation}
\noindent To verify
(\ref{eq:mat2}),
evaluate the two determinants
using
(\ref{eq:middle}) and $\rho_0 \rho_n = 1$,
and simplify the result.
The claim is proven.
For each matrix in
(\ref{eq:mat2}),
 $a({\rm top} \;{\rm row})+
b({\rm middle}\; {\rm row})+c({\rm bottom}\; {\rm row})=0$.
Each matrix
is singular, so its determinant is zero.
 Therefore $\rho_0 = \xi$.
\\
${\rm (iii)}\Rightarrow {\rm (i)}$
By construction $\rho_i \rho_{n-i}=1$
for $0 \leq i \leq n$.
Define $a=s$, $b=0$, $c=-1$.
Then $a\rho_{i-1} + b\rho_i + c \rho_{i+1}=0$
for $1 \leq i \leq n-1$.
\\
${\rm (ii)}\Leftrightarrow {\rm (iii)}$
Routine.
\end{proof}

\noindent Assume for the moment that $n$ is odd, and
refer to Proposition
\ref{prop:noddPre}. It could happen that
$\lbrace \rho_i\rbrace_{i=0}^n$ is geometric,
and the
 equivalent conditions
(i)--(iii) hold. 
We now investigate this case.

\begin{lemma}
\label{lem:ConGeo}
Assume that $n$ is odd. Then for the sequence
$\lbrace \rho_i \rbrace_{i=0}^n$ the following
{\rm (i)--(iv)} are equivalent:
\begin{enumerate}
\item[\rm (i)] 
$\lbrace \rho_i \rbrace_{i=0}^n$ is geometric and constrained;
\item[\rm (ii)] 
$\lbrace \rho_i \rbrace_{i=0}^n$ is geometric and 
$\rho_{(n-1)/2}$, $\rho_{(n+1)/2}$ are inverses;
\item[\rm (iii)] there exists $0 \not= t\in \mathbb F$
such that $\rho_i = t^{2i-n}$ for $0 \leq i \leq n$;
\item[\rm (iv)] there exist $s,\xi \in \mathbb F$
that satisfy $\xi^4 s^n=1$ and
Proposition
\ref{prop:noddPre}{\rm (iii)}.
\end{enumerate}
Assume that {\rm (i)--(iv)} hold. Then
 $\xi=\rho_0 = t^{-n}$ and
\begin{eqnarray}
\label{eq:4eq}
s=t^4, 
\qquad 
\rho_{(n-1)/2}=t^{-1},
\qquad 
\rho_{(n+1)/2}=t,
\qquad 
 t^2 = \rho_i/\rho_{i-1}
\quad (1 \leq i \leq n).
\end{eqnarray}
\end{lemma} 
\begin{proof}
${\rm (i)}\Rightarrow {\rm (ii)}$ 
By Definition
\ref{def:constrain}(i).
\\
\noindent 
${\rm (ii)}\Rightarrow {\rm (iii)}$ 
By assumption there exists $0 \not=t \in \mathbb F$
such that
$\rho_{(n-1)/2}=t^{-1}$ and
$\rho_{(n+1)/2}=t$.
Since $\lbrace \rho_i\rbrace_{i=0}^n$ is geometric,
the ratio $\rho_i/\rho_{i-1}$ is independent of
$i$ for $1 \leq i \leq n$. Taking
$i=(n+1)/2$ we find that this ratio is
 $t^2$. By these comments  $\rho_i = t^{2i-n}$ for
$0 \leq i \leq n$.
\\
\noindent 
${\rm (iii)}\Rightarrow {\rm (iv)}$ 
The values
$s=t^4$,
 $\xi =t^{-n}$ 
 meet the requirement.
\\
\noindent 
${\rm (iv)}\Rightarrow {\rm (i)}$ 
Evaluating 
(\ref{eq:gammaSol})
using $\xi^4 s^n=1$ we find that
$\lbrace \rho_i \rbrace_{i=0}^n$ is geometric.
The sequence 
$\lbrace \rho_i \rbrace_{i=0}^n$ is constrained by
Proposition
\ref{prop:noddPre}(i),(iii).
\\
\noindent 
Assume that (i)--(iv) hold. Then the equations
 $\xi=\rho_0 = t^{-n}$ and
(\ref{eq:4eq})
 are readily checked.
\end{proof}

\begin{lemma}
\label{lem:NGd4}
Assume that $\lbrace \rho_i \rbrace_{i=0}^n$ is constrained but
not geometric. Then $n$ is odd and at least 3.
\end{lemma}
\begin{proof} The integer $n$ is odd by Proposition 
\ref{prop:nevenPre}(i),(ii). If $n=1$
then $\rho_0\rho_1=1$, and the sequence
$\rho_0, \rho_1$ is geometric, for a contradiction.
Therefore $n\geq 3$.
\end{proof}

\noindent We have some comments about the scalars
$a,b,c$ from Definition
\ref{def:constrain}(ii).

\begin{definition}
\label{def:lincon}
\rm
Assume that the sequence $\lbrace \rho_i \rbrace_{i=0}^n$
is constrained.
By a {\it linear constraint} for this sequence
we mean a vector $(a,b,c) \in \mathbb F^3$ such
that $a\rho_{i-1}+b\rho_i + c\rho_{i+1}=0$
for $1 \leq i \leq n-1$. 
\end{definition}

\noindent Let $\lambda $ denote an indeterminate,
and let $\mathbb F \lbrack \lambda \rbrack$ denote
the $\mathbb F$-algebra consisting of the polynomials
in $\lambda$ that have all coefficients in $\mathbb F$.
Let $(a,b,c)$ denote a vector in $\mathbb F^3$.
Define  a polynomial $\psi \in \mathbb F\lbrack \lambda \rbrack$
by
\begin{eqnarray*}
\psi = a + b\lambda + c\lambda^2.
\end{eqnarray*}

\begin{proposition} 
\label{prop:nodd}
Assume that $n\geq 2$ and
 $\lbrace \rho_i\rbrace_{i=0}^n$ is constrained.
Then for the above vector $(a,b,c)$ 
the following 
{\rm (i)--(iii)} hold. 
\begin{enumerate}
\item[\rm (i)] 
Assume that $n$ is even.
Then $(a,b,c)$ is a linear constraint for
 $\lbrace \rho_i\rbrace_{i=0}^n$ 
 if and only if 
$\psi(r)=0$, where $r$ is from 
 Proposition 
\ref{prop:nevenPre}.
\item[\rm (ii)] 
Assume that $n$ is odd and
 $\lbrace \rho_i\rbrace_{i=0}^n$ is not geometric.
 Then $(a,b,c)$ is a linear constraint for
 $\lbrace \rho_i\rbrace_{i=0}^n$ 
 if and only if 
$\psi = c(\lambda^2-s)$,
where $s$ is from Proposition
\ref{prop:noddPre}.
\item[\rm (iii)] 
Assume that $n$ is odd and
 $\lbrace \rho_i\rbrace_{i=0}^n$ is geometric.
Then $(a,b,c)$ is a linear constraint for
 $\lbrace \rho_i\rbrace_{i=0}^n$ 
 if and only if 
$\psi(t^2)=0$, where $t$ is from 
Lemma 
\ref{lem:ConGeo}.
\end{enumerate}
\end{proposition}
\begin{proof} (i)
By Proposition
\ref{prop:nevenPre} we have
 $\rho_i = \varepsilon r^{i-n/2}$
for $0 \leq i \leq n$, where
$r = \rho_i /\rho_{i-1}$ for $1 \leq i \leq n$
and $\varepsilon = \rho_{n/2}$.
 For $1 \leq i \leq n-1$,
$a \rho_{i-1} + b\rho_i + c\rho_{i+1} = 0$
if and only if $a+br+cr^2=0$  if and only if
$\psi(r)=0$.
So
$(a,b,c)$ is a linear constraint for
 $\lbrace \rho_i \rbrace_{i=0}^n$
if and only if
$a \rho_{i-1} + b\rho_i + c\rho_{i+1} = 0$
for $1 \leq i \leq n-1$ if and only if
$\psi(r)=0$.
\\
\noindent (ii) The sequence $\lbrace \rho_i \rbrace_{i=0}^n$
is constrained, so by Proposition 
\ref{prop:noddPre} it has the form
(\ref{eq:gammaSol}), where $s=\rho_i/\rho_{i-2}$ for
$2 \leq i \leq n$ and $\xi = \rho_0$.
By assumption
 $\lbrace \rho_i \rbrace_{i=0}^n$ is not geometric,
so $\xi^4 s^n \not=1$ in view of Lemma
\ref{lem:ConGeo}(i),(iv). Define $k = \xi^2 s^{(n+1)/2}$
and note that $k^2=\xi^4 s^{n+1}$.
Therefore $k^2 \not=s$ so
$k \not=k^{-1}s$.
Pick an integer $i$ $(1 \leq i \leq n-1)$.
First assume that $i$ is even.
By
(\ref{eq:gammaSol}),
$a\rho_{i-1}+b\rho_i + c\rho_{i+1}=0$
if and only if $a+bk+ c s=0$.
Next assume that $i$ is odd.
By
(\ref{eq:gammaSol}),
 $a\rho_{i-1}+b\rho_i + c\rho_{i+1}=0$
if and only if $a+bk^{-1}s+ c s=0$.
We now argue that
$(a,b,c)$ is a linear constraint for
$\lbrace \rho_i \rbrace_{i=0}^n$ if and only if
$a\rho_{i-1}+b\rho_i + c\rho_{i+1}=0$ for $1\leq i \leq n-1$
if and only if both 
$a+bk+ c s=0$,
$a+b k^{-1}s+ c s=0$
if and only if $b=0=a+cs$ if and only if $\psi = c(\lambda^2-s)$.
\\
\noindent (iii) Similar to the proof of (i) above.
\end{proof}

\begin{definition}
\label{def:LC}
\rm
Assume that the sequence $\lbrace \rho_i \rbrace_{i=0}^n$ is
constrained. Let LC denote the set of 
 linear constraints for
 $\lbrace \rho_i \rbrace_{i=0}^n$. By
Definition
\ref{def:lincon},
LC is a subspace of the $\mathbb F$-vector
space $\mathbb F^3$.
By Definition
\ref{def:constrain}(ii),
LC is nonzero.
We call LC the {\it linear constraint space} for
 $\lbrace \rho_i \rbrace_{i=0}^n$.
\end{definition}

\begin{proposition}
\label{prop:LCbasis}
Assume that $n\geq 2$ and 
$\lbrace \rho_i \rbrace_{i=0}^n$ is constrained.
Let LC denote the corresponding linear constraint space.
\begin{enumerate}
\item[\rm (i)] Assume that $n$ is even. Then LC has dimension 2.
The vectors $(r,-1,0)$ and $(r^2,0,-1)$ form a basis of LC,
where $r$ is from Proposition
\ref{prop:nevenPre}.
\item[\rm (ii)] Assume that $n$ is odd and
$\lbrace \rho_i \rbrace_{i=0}^n$ is not geometric.
Then LC has dimension 1.
The vector $(s,0,-1)$ forms a basis of LC,
where 
$s$ is from
Proposition
\ref{prop:noddPre}.
\item[\rm (iii)] Assume that $n$ is odd and
$\lbrace \rho_i \rbrace_{i=0}^n$ is geometric.
Then LC has dimension 2.
The vectors $(t^2,-1,0)$ and $(t^4,0,-1)$ form a basis of LC,
where $t$ is from
Lemma 
\ref{lem:ConGeo}.
\end{enumerate}
\end{proposition}
\begin{proof} This is a reformulation of
Proposition
\ref{prop:nodd}.
\end{proof}

\section{Toeplitz matrices and
nilpotent linear transformations}

\noindent We will be discussing 
an upper triangular
matrix of a certain type, said to be Toeplitz.

\begin{definition}
\label{def:top}
\rm
(See \cite[Section~8.12]{dym}.)
Let $\lbrace \alpha_i \rbrace_{i=0}^d$ denote
scalars in 
$\mathbb F$.
Let $T$ denote an upper triangular matrix in 
 ${\rm Mat}_{d+1}(\mathbb F)$. Then $T$ is said to
 {\it Toeplitz, with parameters
$\lbrace \alpha_i \rbrace_{i=0}^d$} whenever
$T$ has $(i,j)$-entry
$\alpha_{j-i}$ for $0 \leq i\leq j\leq d$.
In this case 
\begin{eqnarray*}
T =	\left(
	\begin{array}{ c c cc c c }
	\alpha_0 & \alpha_1   &  \cdot  & \cdot&  \cdot  & \alpha_d  \\
	 & \alpha_0  & \alpha_1  & \cdot& \cdot &  \cdot     \\
	 &   & \alpha_0   & \cdot & \cdot& \cdot  \\
	   &   &  & \cdot  & \cdot & \cdot \\
	     &  & &  & \cdot & \alpha_1 \\
	      {\bf 0}  &&  & &   & \alpha_0  \\
	      \end{array}
	      \right).
\end{eqnarray*}
\end{definition}

\noindent We have some comments.

\begin{note} \rm
The matrix $\tau$ from Definition
\ref{def:tauMat} is Toeplitz, with parameters
\begin{eqnarray*}
\alpha_i = \begin{cases}
1  &  {\mbox{\rm if $i=1$}}; \\
0  & {\mbox{\rm if $i\not=1$}}
\end{cases}
\qquad \qquad (0 \leq i \leq d).
\end{eqnarray*}
\end{note}

\begin{lemma}
\label{lem:ToeChar}
Let 
$T$ denote an upper triangular matrix in
 ${\rm Mat}_{d+1}(\mathbb F)$.  Then $T$ is Toeplitz
 if and only if $T$ commutes with $\tau$. In this case
$T= \sum_{i=0}^d \alpha_i \tau^i$, where $\lbrace \alpha_i \rbrace_{i=0}^d$
are the parameters for $T$.
\end{lemma}
\begin{proof} Use Lemma
\ref{lem:tauPower}.
\end{proof}

\begin{lemma}
\label{lem:ToepTrans}
Let $T$ denote an upper triangular Toeplitz matrix
 in ${\rm Mat}_{d+1}(\mathbb F)$.  Then $T^t={\bf Z}T{\bf Z}$.
\end{lemma}
\begin{proof} Expand ${\bf Z}T{\bf Z}$ by matrix multiplication.
\end{proof}

\noindent
Referring to Definition
\ref{def:top}, assume that $T$ is Toeplitz with
parameters $\lbrace \alpha_i \rbrace_{i=0}^d$. 
Note  that $T$ is invertible if and only if $\alpha_0 \not=0$.
Assume this is the case. Then $T^{-1}$ is upper triangular and
Toeplitz:
\begin{eqnarray*}
T^{-1} =	\left(
	\begin{array}{ c c cc c c }
	\beta_0 & \beta_1   &  \cdot  & \cdot&  \cdot  & \beta_d  \\
	 & \beta_0  & \beta_1  & \cdot& \cdot &  \cdot     \\
	 &   & \beta_0   & \cdot & \cdot& \cdot  \\
	   &   &  & \cdot  & \cdot & \cdot \\
	     &  & &  & \cdot & \beta_1 \\
	      {\bf 0}  &&  & &   & \beta_0  \\
	      \end{array}
	      \right),
\end{eqnarray*}
with parameters $\lbrace \beta_i \rbrace_{i=0}^d$ that are obtained
from 
$\lbrace \alpha_i \rbrace_{i=0}^d$ by
recursively solving
$\alpha_0 \beta_0 = 1$ and
\begin{equation}
\label{eq:recursion}
\alpha_0 \beta_j + \alpha_1 \beta_{j-1} + 
\cdots + \alpha_j \beta_0 = 0 \qquad \qquad (1 \leq j \leq d).
\end{equation}
 We have
\begin{eqnarray*}
\beta_0 &=& \alpha^{-1}_0,
\\
\beta_1 &=& -\alpha_1 \alpha^{-2}_0\;\quad \qquad  \qquad \qquad \mbox{\rm (if $d\geq 1$),}
\\
\beta_2 &=& \frac{ \alpha^2_1-\alpha_0 \alpha_2}{\alpha^3_0}  \;\;\quad \qquad \qquad 
  \mbox{\rm (if $d\geq 2$),}
\\
\beta_3 &=& \frac{2\alpha_0 \alpha_1 \alpha_2 -\alpha^3_1 - \alpha^2_0\alpha_3}{\alpha^4_0} 
  \qquad \mbox{\rm (if $d\geq 3$),}
\\
\beta_4 &=& \frac{\alpha^4_1 +2\alpha^2_0 \alpha_1 \alpha_3 + \alpha^2_0 
\alpha^2_2-3\alpha_0 \alpha^2_1 \alpha_2 - \alpha^3_0 \alpha_4 
}{\alpha^5_0} 
  \qquad \mbox{\rm (if $d\geq 4$).}
\end{eqnarray*}
For $\alpha_0=1$ this becomes
\begin{eqnarray*}
\beta_0 &=& 1,
\\
\beta_1 &=& -\alpha_1, \;\qquad \qquad \quad \qquad  \qquad \qquad \mbox{\rm (if $d\geq 1$),}
\\
\beta_2 &=& \alpha^2_1- \alpha_2 \;\;\qquad \qquad \quad \qquad \qquad 
  \mbox{\rm (if $d\geq 2$),}
\\
\beta_3 &=& 2 \alpha_1 \alpha_2 -\alpha^3_1 - \alpha_3
  \qquad \qquad \qquad \mbox{\rm (if $d\geq 3$),}
\\
\beta_4 &=& \alpha^4_1 +2\alpha_1 \alpha_3 + 
\alpha^2_2-3\alpha^2_1 \alpha_2 - \alpha_4 
  \qquad \mbox{\rm (if $d\geq 4$).}
\end{eqnarray*}

\noindent 
Recall our vector space $V$ over $\mathbb F$
with
dimension $d+1$. Recall the Nil elements in
	      ${\rm End}(V)$ from Definition
\ref{def:Nil}.
We now give a variation on
Lemma
\ref{lem:NilRec}(i),(ii) in terms of vectors.

\begin{lemma} 
\label{lem:reform1}
For 
$A \in 
 {\rm End}(V)$ the following are equivalent:
\begin{enumerate}
\item[\rm (i)] $A$ is Nil;
\item[\rm (ii)] there exists a basis $\lbrace v_i \rbrace_{i=0}^d$
of $V$ such that $Av_0=0$ and $Av_i = v_{i-1}$ for $1 \leq i \leq d$.
\end{enumerate}
\end{lemma}
\begin{proof} By
Lemma
\ref{lem:NilRec}(i),(ii).
\end{proof}

\begin{lemma}
\label{lem:whatDec}
Assume $A \in 
	      {\rm End}(V)$ is Nil. 
For subspaces $\lbrace V_i \rbrace_{i=0}^d$ of $V$ the
following are equivalent:
\begin{enumerate}
\item[\rm (i)] 
$\lbrace V_i \rbrace_{i=0}^d$ is a decomposition of $V$ that
is lowered by $A$;
\item[\rm (ii)] the sum $V=AV+V_d$ is direct, and
$V_i = A^{d-i}V_d$ for $0 \leq i \leq d$.
\end{enumerate}
\end{lemma}
\begin{proof} 
${\rm (i)}\Rightarrow {\rm (ii)}$ 
The sum 
$V=AV+V_d$ 
is direct since $AV=V_0+ \cdots + V_{d-1}$. The
remaining assertion is clear.
\\
${\rm (ii)}\Rightarrow {\rm (i)}$ 
Define 
$U_i = A^{d-i}V$ for $0 \leq i \leq d$. The
sequence $\lbrace U_i\rbrace_{i=0}^d$ is a flag on $V$.
By construction, the sum $U_i = U_{i-1}+V_i$ is
direct for $1 \leq i \leq d$. Therefore
$\lbrace V_i \rbrace_{i=0}^d$ is 
a decomposition of $V$.
By construction $AV_i= V_{i-1}$ for $1 \leq i \leq d$.
Also $AV_0=0$ since $A$ is Nil. By these comments
$\lbrace V_i \rbrace_{i=0}^d$ is lowered by $A$.
\end{proof}

\begin{lemma}
\label{lem:twoView}
Assume $A \in 
	      {\rm End}(V)$ is Nil.
	      For
vectors $\lbrace v_i\rbrace_{i=0}^d$ in $V$ the following
are equivalent:
\begin{enumerate}
\item[\rm (i)] 
$\lbrace v_i\rbrace_{i=0}^d$ is a basis of $V$
such that 
$Av_0=0$ and $Av_i=v_{i-1}$ for $1 \leq i \leq d$;
\item[\rm (ii)] $v_d \not\in  AV$
and $v_i = A^{d-i}v_d$ for $0 \leq i \leq d$.
\end{enumerate}
\end{lemma}
\begin{proof} This is a reformulation of
Lemma
\ref{lem:whatDec}.
\end{proof}

\noindent Assume $A 
\in 
	      {\rm End}(V)$ is Nil.
By Lemma
\ref{lem:twoView},
for $v \in V \backslash AV$ the sequences
$\lbrace A^iv\rbrace_{i=0}^d$
and
$\lbrace A^{d-i}v\rbrace_{i=0}^d$
are bases for $V$. Relative
to these bases
the matrices representing $A$ are, respectively,
\begin{eqnarray*}
	A:\;
	\left(
	\begin{array}{ c c cc c c }
	0 &   &   &&   & \bf 0  \\
	1 & 0  &   &&  &      \\
	 &  1 & 0   & &&  \\
	   &   & \cdot & \cdot  & & \\
	     &  & & \cdot & \cdot & \\
	      {\bf 0}  &&  & &  1 & 0  \\
	      \end{array}
	      \right),
\qquad \qquad 
A:\; 
\left(
\begin{array}{ c c cc c c }
 0 & 1   &   &&   & \bf 0  \\
  & 0  &  1  &&  &      \\ 
   &   &  0   & \cdot &&  \\
     &   &  & \cdot  & \cdot & \\
       &  & &  & \cdot & 1 \\
        {\bf 0}  &&  & &   &  0  \\
	\end{array}
	\right).
	      \end{eqnarray*}

\medskip

\noindent  For  $u,v \in V \backslash AV$, we now compute
the transition matrix from the basis
 $\lbrace A^{d-i}u\rbrace_{i=0}^d$ to the basis
$\lbrace A^{d-i} v\rbrace_{i=0}^d$.
There exist scalars
$\lbrace \alpha_i \rbrace_{i=0}^d$ in $\mathbb F$
such that $\alpha_0 \not=0$
and $v = \sum_{i=0}^d \alpha_i A^i u$.
In this equation, for $0 \leq j \leq d$
apply $A^j$ to each side
and adjust the result to obtain
$A^jv = \sum_{i=j}^d \alpha_{i-j} A^{i} u$.
This yields
\begin{eqnarray}
A^{d-j}v = \sum_{i=0}^j \alpha_{j-i} A^{d-i} u
\qquad \qquad 0 \leq j \leq d.
\label{eq:Toep}
\end{eqnarray}
By
(\ref{eq:Toep}),
the transition matrix from
the basis
 $\lbrace A^{d-i}u\rbrace_{i=0}^d$ to the basis
 $\lbrace A^{d-i}v\rbrace_{i=0}^d$ is
 upper triangular, and 
 Toeplitz with parameters ${\lbrace \alpha_i \rbrace}_{i=0}^d$.
Define $\Phi = \sum_{i=0}^d \alpha_i A^i$. 
By construction $A\Phi=\Phi A$ and $ \Phi u=v$. Therefore
$\Phi  $ sends $A^{d-i}u \mapsto A^{d-i}v$ for
$0 \leq i \leq d$.

\begin{proposition}
\label{prop:Toe}
Let $\lbrace u_i \rbrace_{i=0}^d$ and
$\lbrace v_i \rbrace_{i=0}^d$ denote bases for
$V$. Then the following are equivalent:
\begin{enumerate}
\item[\rm (i)] 
there exists $A \in 
	      {\rm End}(V)$ such that
	      $Au_i = u_{i-1}$ $(1 \leq i \leq d)$,
	      $Au_0 = 0$,
	      $Av_i= v_{i-1}$ $(1 \leq i \leq d)$,
	      $Av_0 = 0$;
\item[\rm (ii)] 
the transition matrix from
 $\lbrace u_i \rbrace_{i=0}^d$ to
$\lbrace v_i \rbrace_{i=0}^d$ is upper
triangular and Toeplitz.
\end{enumerate}
\end{proposition}
\begin{proof} 
${\rm (i)}\Rightarrow {\rm (ii)}$ 
The map $A$ is Nil
by
Lemma \ref{lem:reform1}.
By Lemma
\ref{lem:twoView},
there exist $u,v \in V\backslash AV$ such that
$u_i= A^{d-i}u$ and
$v_i= A^{d-i}v$
for $0 \leq i \leq d$.
Now (ii) follows from the sentence
below
(\ref{eq:Toep}).
\\
\noindent 
${\rm (ii)}\Rightarrow {\rm (i)}$  
Define $A \in 
	      {\rm End}(V)$ such that
	     $Au_0 = 0$ and $Au_i = u_{i-1}$
	     for $1 \leq i \leq d$.
For the transition matrix in question let
 $\lbrace \alpha_i \rbrace_{i=0}^d$
denote the corresponding parameters,
and define $\Phi = \sum_{i=0}^d \alpha_i A^i$.
Then $A\Phi=\Phi A$.
By construction
$\Phi u_i = v_i $ for $0 \leq i \leq d$.
Therefore $Av_0 = 0 $ and $Av_i = v_{i-1}$ for $1 \leq i \leq d$.
\end{proof}

\begin{lemma}
\label{lem:a0a1}
Assume that the two equivalent conditions in Proposition
\ref{prop:Toe} hold,
 and let
 $\lbrace \alpha_i \rbrace_{i=0}^d$
denote the parameters for 
 the Toeplitz
matrix mentioned in the second condition.
Fix $0 \not=r \in \mathbb F$.
\begin{enumerate} 
\item[\rm (i)] 
If we replace $u_i$ by $u'_i = ru_i$
for $0 \leq i \leq d$,
then the equivalent conditions
in
Proposition
\ref{prop:Toe} still hold,
with $A'=A$ and $\alpha'_i = r^{-1}\alpha_i$
for $0 \leq i \leq d$.
\item[\rm (ii)]
If we replace $u_i$ and $v_i$ by $u'_i = r^iu_i$
and $v'_i = r^i v_i$
for $0 \leq i \leq d$,
then the equivalent conditions
in
Proposition
\ref{prop:Toe} still hold,
with $A'=r^{-1}A$ and
$\alpha'_i = r^{i}\alpha_i$ for
$0 \leq i \leq d$.
\end{enumerate}
\end{lemma}
\begin{proof} 
By linear algebra.
\end{proof}

\section{LR triples}
\noindent We now turn our attention to LR triples.
Throughout this section, $V$ denotes
a vector space over $\mathbb F$ with dimension $d+1$.

	      \begin{definition}
	     \label{def:LRT} 
	      \rm
An {\it LR triple on $V$} is a sequence	       
	      $A,B,C$ of elements in ${\rm End}(V)$
	      such that any two of $A,B,C$ form an LR pair on $V$.
This LR triple is said to be {\it over $\mathbb F$}.
We call $V$ the {\it underlying vector space}.
We call $d$ the {\it diameter}.
	      \end{definition}

\begin{definition}
\label{def:isoLRT}
\rm
Let $A,B,C$ denote an LR triple on $V$.
Let $V'$ denote a vector space over $\mathbb F$ with
dimension $d+1$, and let $A',B',C'$ denote an LR triple on
$V'$. By an {\it isomorphism of LR triples from $A,B,C$ to
$A',B',C'$} we mean an $\mathbb F$-linear bijection $\sigma :V \to V'$
such that
\begin{eqnarray*}
\sigma A = A' \sigma,
\qquad \qquad 
\sigma B = B' \sigma,
\qquad \qquad 
\sigma C = C' \sigma.
\end{eqnarray*}
The LR triples
$A,B,C$ and $A',B',C'$ are called {\it isomorphic} whenever
there exists an isomorphism of LR triples from
$A,B,C$ to $A',B',C'$.
\end{definition}

\begin{example}
\label{ex:trivial}
\rm
Assume $d=0$. A sequence of elements $A,B,C$ in 
	      ${\rm End}(V)$ form an LR triple
	      if and only if each of $A,B,C$ is zero.
	      This LR triple will be called {\it trivial}.
	      \end{example}

\noindent We will use the following notational
convention.

\begin{definition}
\label{def:primeConv}
\rm 
Let $A,B,C$ denote an LR triple.
For any object $f$ that we associate with this
LR triple, then $f'$ (resp. $f''$) will denote
the corresponding object for the LR triple
$B,C,A$ (resp. $C,A,B$).
\end{definition}

\begin{definition}
\label{def:LRTpar}
\rm
Let $A,B,C$ denote an LR triple on $V$. By
Definition
\ref{def:LRT}, the pair
$A,B$ (resp. $B,C$) (resp. $C,A$) is an
LR pair on $V$.
Following the notational convention
in Definition 
\ref{def:primeConv},
for these LR pairs the parameter
sequence 
is denoted as follows:

     \bigskip

\centerline{
\begin{tabular}[t]{c|c}
 {\rm LR pair} & {\rm parameter sequence} 
 \\
 \hline
 $A,B$ & $\lbrace \varphi_i \rbrace_{i=1}^d$
   \\
 $B,C$ & $\lbrace \varphi'_i \rbrace_{i=1}^d$
   \\
 $C,A$ & $\lbrace \varphi''_i \rbrace_{i=1}^d$
   \end{tabular}}
     \bigskip

\noindent We call the sequence
\begin{eqnarray}
\label{eq:paLRT}
(\lbrace \varphi_i \rbrace_{i=1}^d;
\lbrace \varphi'_i \rbrace_{i=1}^d;
\lbrace \varphi''_i \rbrace_{i=1}^d)
\end{eqnarray}
the {\it parameter array} of the LR triple $A,B,C$.
\end{definition}

\begin{note}\rm As we will see, not every
LR triple is determined up to isomorphism by its
parameter array.
\end{note}

\begin{lemma}
\label{lem:albega}
Let $A,B,C$ denote an LR triple on $V$, with parameter array
{\rm (\ref{eq:paLRT})}.
Let $\alpha, \beta, \gamma$ denote nonzero scalars in $\mathbb F$.
Then the triple $\alpha A,\beta B,\gamma C$ is an LR triple on
$V$, with parameter array
\begin{eqnarray*}
(\lbrace  \alpha \beta\varphi_i \rbrace_{i=1}^d;
\lbrace \beta\gamma \varphi'_i \rbrace_{i=1}^d;
\lbrace \gamma \alpha  \varphi''_i \rbrace_{i=1}^d).
\end{eqnarray*}
\end{lemma}
\begin{proof} Use Lemma
\ref{lem:alphaBetacom} and
Definition
\ref{def:LRTpar}.
\end{proof}

\begin{lemma}
\label{lem:albegaCor}
Let $A,B,C$ denote a nontrivial LR triple over $\mathbb F$. 
Let 
$\alpha, \beta, \gamma$ denote nonzero scalars in $\mathbb F$.
Then the following are equivalent:
\begin{enumerate}
\item[\rm (i)] the LR triples $A,B,C$ and $\alpha A, \beta B, \gamma C$
have the same parameter array;
\item[\rm (ii)] $\alpha\beta = \beta \gamma = \gamma \alpha = 1$;
\item[\rm (iii)] $\alpha = \beta = \gamma \in \lbrace 1,-1\rbrace$.
\end{enumerate}
\end{lemma}
\begin{proof} Use Lemma
\ref{lem:albega}.
\end{proof}

\begin{lemma}
\label{lem:ABCvar}
Let $A,B,C$ denote an LR triple on $V$,
with parameter array
{\rm (\ref{eq:paLRT})}. Then each permutation of $A,B,C$
is an LR triple on $V$. The corresponding
parameter array is given in the table below:

     \bigskip

\centerline{
\begin{tabular}[t]{c|c}
 {\rm LR triple} & {\rm parameter array}
 \\
 \hline
 \hline
 $A,B,C$ &
$(\lbrace \varphi_i \rbrace_{i=1}^d;
\lbrace \varphi'_i \rbrace_{i=1}^d;
\lbrace \varphi''_i \rbrace_{i=1}^d)$
   \\
 $B,C,A$ &
$(\lbrace \varphi'_i \rbrace_{i=1}^d;
\lbrace \varphi''_i \rbrace_{i=1}^d;
\lbrace \varphi_i \rbrace_{i=1}^d)$
   \\
 $C,A,B$ &
$(\lbrace \varphi''_i \rbrace_{i=1}^d;
\lbrace \varphi_i \rbrace_{i=1}^d;
\lbrace \varphi'_i \rbrace_{i=1}^d)$
   \\
 \hline
 $C,B,A$ &
$(\lbrace \varphi'_{d-i+1} \rbrace_{i=1}^d;
\lbrace \varphi_{d-i+1} \rbrace_{i=1}^d;
\lbrace \varphi''_{d-i+1} \rbrace_{i=1}^d)$
   \\
 $A,C,B$ &
$(\lbrace \varphi''_{d-i+1} \rbrace_{i=1}^d;
\lbrace \varphi'_{d-i+1} \rbrace_{i=1}^d;
\lbrace \varphi_{d-i+1} \rbrace_{i=1}^d)$
   \\
 $B,A,C$ &
$(\lbrace \varphi_{d-i+1} \rbrace_{i=1}^d;
\lbrace \varphi''_{d-i+1} \rbrace_{i=1}^d;
\lbrace \varphi'_{d-i+1} \rbrace_{i=1}^d)$
   \end{tabular}}
     \bigskip

\end{lemma}
\begin{proof} By Lemma
\ref{lem:BAvAB}
and Definition
\ref{def:LRTpar}.
\end{proof}

\begin{definition}
\label{def:ASSOC}
\rm
Let $A,B,C$ and $A',B',C'$ denote LR triples on $V$.
These LR triples will be called {\it associates} whenever
there exist nonzero $\alpha, \beta,\gamma $  in $ \mathbb F$
such that
\begin{eqnarray*}
A'=\alpha A, \qquad
B'= \beta B,
\qquad
C'= \gamma C.
\end{eqnarray*}
Associativity is an equivalence relation.
\end{definition}

\begin{lemma}
\label{lem:assocIsoT}
Let $A,B,C$ and $A',B',C'$ denote LR triples over $\mathbb F$.
Then the following are equivalent:
\begin{enumerate}
\item[\rm (i)] there exists an LR triple over $\mathbb F$
that is associate to $A,B,C$ and isomorphic to $A',B',C'$;
\item[\rm (ii)] there exists an LR triple  over $\mathbb F$
that is isomorphic to $A,B,C$ and associate to $A',B',C'$.
\end{enumerate}
\end{lemma}
\begin{proof} Similar to the proof of
Lemma \ref{lem:assocIso}.
\end{proof}

\begin{definition}
\label{def:simiLar}
\rm
Let $A,B,C$ and $A',B',C'$ denote LR triples over $\mathbb F$.
These LR triples will be called {\it similar} whenever
they satisfy the equivalent conditions {\rm (i), (ii)}
in Lemma
\ref{lem:assocIsoT}.
Similarity is
an equivalence relation.
\end{definition}

\begin{lemma} 
\label{lem:newPA}
Let $A,B,C$ denote an LR triple on $V$,
with parameter array 
{\rm (\ref{eq:paLRT})}.
In each row of the table below, we display
an LR triple on $V^*$ along with its parameter array.

     \bigskip

\centerline{
\begin{tabular}[t]{c|c}
 {\rm LR triple} & {\rm parameter array}
 \\
 \hline
 \hline
 $\tilde A, \tilde B, \tilde C$ &
$(\lbrace \varphi_{d-i+1} \rbrace_{i=1}^d;
\lbrace \varphi'_{d-i+1} \rbrace_{i=1}^d;
\lbrace \varphi''_{d-i+1} \rbrace_{i=1}^d)$
   \\
 $\tilde B, \tilde C, \tilde A$ &
$(\lbrace \varphi'_{d-i+1} \rbrace_{i=1}^d;
\lbrace \varphi''_{d-i+1} \rbrace_{i=1}^d;
\lbrace \varphi_{d-i+1} \rbrace_{i=1}^d)$
   \\
 $\tilde C, \tilde A, \tilde B$ &
$(\lbrace \varphi''_{d-i+1} \rbrace_{i=1}^d;
\lbrace \varphi_{d-i+1} \rbrace_{i=1}^d;
\lbrace \varphi'_{d-i+1} \rbrace_{i=1}^d)$
   \\
 \hline
 $\tilde C, \tilde B, \tilde A$ &
$(\lbrace \varphi'_{i} \rbrace_{i=1}^d;
\lbrace \varphi_{i} \rbrace_{i=1}^d;
\lbrace \varphi''_{i} \rbrace_{i=1}^d)$
   \\
 $\tilde A, \tilde C, \tilde B$ &
$(\lbrace \varphi''_{i} \rbrace_{i=1}^d;
\lbrace \varphi'_{i} \rbrace_{i=1}^d;
\lbrace \varphi_{i} \rbrace_{i=1}^d)$
   \\
 $\tilde B, \tilde A, \tilde C$ &
$(\lbrace \varphi_{i} \rbrace_{i=1}^d;
\lbrace \varphi''_{i} \rbrace_{i=1}^d;
\lbrace \varphi'_{i} \rbrace_{i=1}^d)$
   \end{tabular}}
     \bigskip

\end{lemma}
\begin{proof} 
The given triples
are LR triples by Lemma
\ref{lem:LRDdual}(i)
and Definition
	     \ref{def:LRT}.
To compute their parameter array use
Lemmas
\ref{lem:LRdualPar},
\ref{lem:ABCvar}.
\end{proof}
\begin{definition}
\label{def:prel}
\rm
Let 
 $A,B,C$ denote an LR triple on $V$. By
a {\it relative} of $A,B,C$  we mean an LR triple
from the table in Lemma
\ref{lem:ABCvar}
or
Lemma
\ref{lem:newPA}. A relative of
$A,B,C$ is said to have {\it positive orientation}
(resp. {\it negative orientation}) with respect to
$A,B,C$ 
whenever it is in the top half (resp. bottom half) of the table of
Lemma
\ref{lem:ABCvar}
or the bottom half (resp. top half) of the table in Lemma
\ref{lem:newPA}. We call such a relative a {\it p-relative}
(resp. {\it n-relative}) of $A,B,C$.
Note that an n-relative of $A,B,C$ is the same thing
as a p-relative of $C,B,A$.
\end{definition}

\noindent Let $A,B,C$ denote an LR triple on $V$.
By Lemma
\ref{lem:ABdecIndNil},
each of $A,B,C$ is Nil. By Lemma
\ref{lem:ABdecInd}
the following are mutually opposite flags on $V$:
\begin{equation}
\lbrace A^{d-i}V \rbrace_{i=0}^d, \qquad
\lbrace B^{d-i}V \rbrace_{i=0}^d, \qquad
\lbrace C^{d-i}V \rbrace_{i=0}^d.
\label{eq:LRTMOF}
\end{equation}

\begin{lemma} 
\label{lem:LRTinducedFlag}
Let $A,B,C$ denote an LR triple on $V$. In each row of
the table below, we display a decomposition of $V$
along with its induced flag on $V$:
     
     \bigskip

\centerline{
\begin{tabular}[t]{c|c}
 {\rm decomp. of $V$} & {\rm induced flag on $V$}
 \\
 \hline
 \hline
 $(A,B)$ &
$\lbrace A^{d-i}V \rbrace_{i=0}^d$
\\
 $(B,C)$ &
$\lbrace B^{d-i}V \rbrace_{i=0}^d$
\\
 $(C,A)$ &
$\lbrace C^{d-i}V \rbrace_{i=0}^d$
\\
\hline
 $(B,A)$ &
$\lbrace B^{d-i}V \rbrace_{i=0}^d$
\\
 $(C,B)$ &
$\lbrace C^{d-i}V \rbrace_{i=0}^d$
\\
 $(A,C)$ &
$\lbrace A^{d-i}V \rbrace_{i=0}^d$
\\
   \end{tabular}}
     \bigskip

\end{lemma}
\begin{proof} By Lemma
\ref{lem:ABdecInd}.
\end{proof}

\begin{lemma}
\label{lem:LRTdecDual}
Let $A,B,C$ denote an LR triple on $V$.
In each row of the table below, we display a decomposition
of $V$ along with its dual decomposition of $V^*$.
     
     \bigskip

\centerline{
\begin{tabular}[t]{c|c}
 {\rm decomp. of $V$} & {\rm dual decomp. of $V^*$}
 \\
 \hline
 \hline
 $(A,B)$ &
$(\tilde B, \tilde A)$
\\
 $(B,C)$ &
$(\tilde C, \tilde B)$
\\
 $(C,A)$ &
$(\tilde A, \tilde C)$
\\
\hline
 $(B,A)$ &
$(\tilde A, \tilde B)$
\\
 $(C,B)$ &
$(\tilde B, \tilde C)$
\\
 $(A,C)$ &
$(\tilde C, \tilde A)$
\\
   \end{tabular}}
     \bigskip

\end{lemma}
\begin{proof} 
By Lemma
\ref{lem:LRDdual}.
\end{proof}

\begin{lemma}
\label{lem:LRTflagDual}
Let $A,B,C$ denote an LR triple on $V$.
In each row of the table below, we display a flag
on $V$ along with its dual flag on $V^*$.
     
     \bigskip

\centerline{
\begin{tabular}[t]{c|c}
 {\rm flag on $V$} & {\rm dual flag on $V^*$}
 \\
 \hline
 \hline
 $\lbrace A^{d-i}V\rbrace_{i=0}^d$ &
 $\lbrace \tilde A^{d-i}V^*\rbrace_{i=0}^d$ 
\\
 $\lbrace B^{d-i}V\rbrace_{i=0}^d$ &
 $\lbrace \tilde B^{d-i}V^*\rbrace_{i=0}^d$ 
\\
 $\lbrace C^{d-i}V\rbrace_{i=0}^d$ &
 $\lbrace \tilde C^{d-i}V^*\rbrace_{i=0}^d$ 
\\
   \end{tabular}}
     \bigskip

\end{lemma}
\begin{proof} 
By Lemma
\ref{lem:Nildual}.
\end{proof}

\begin{lemma}
\label{lem:flagaction}
Let $A,B,C$ denote an LR triple on $V$.
In the table below we describe the action of
$A,B,C$ on the flags
{\rm (\ref{eq:LRTMOF})}.
     
     \bigskip

\centerline{
\begin{tabular}[t]{c|c|c}
 {\rm flag on $V$} & {\rm lowered by} & {\rm raised by}
 \\
 \hline
 \hline
 $\lbrace A^{d-i}V\rbrace_{i=0}^d$ & $A$ & $B,C$
\\
 $\lbrace B^{d-i}V\rbrace_{i=0}^d$ &
$B$ & $C,A$
\\
 $\lbrace C^{d-i}V\rbrace_{i=0}^d$ &
$C$ & $A,B$
\\
   \end{tabular}}
     \bigskip

\end{lemma}
\begin{proof} Any two of
$A,B,C$ form an LR pair on $V$. 
Apply
Lemma
\ref{lem:LRFlag} to these pairs.
\end{proof}

\begin{lemma}
\label{lem:LRTaction}
Let $A,B,C$ denote an LR triple on $V$. In each row
of the table below, we display a decomposition
$\lbrace V_i\rbrace_{i=0}^d$ of $V$.
For $0 \leq i \leq d$ we give the action of
$A,B,C$ on $V_i$.
     
     \bigskip

\centerline{
\begin{tabular}[t]{c|ccc}
 {\rm dec.  $\lbrace V_i\rbrace_{i=0}^d$} & 
 {\rm action of $A$ on $V_i$}
 &
 {\rm action of $B$ on $V_i$}
 &
 {\rm action of $C$ on $V_i$}
 \\
 \hline
 \hline
 $(A,B)$ &
 $AV_i = V_{i-1}$ & $BV_i = V_{i+1}$ &
 $CV_i \subseteq V_{i-1}+V_i + V_{i+1}$
\\
 $(B,C)$ &
 $AV_i \subseteq V_{i-1}+V_i + V_{i+1}$ &
 $BV_i = V_{i-1}$ & $CV_i = V_{i+1}$ 
\\
 $(C,A)$ &
 $AV_i = V_{i+1}$ &
$BV_i \subseteq V_{i-1}+V_i + V_{i+1}$ &
 $CV_i = V_{i-1}$ 
\\
\hline
 $(B,A)$ &
 $AV_i = V_{i+1}$ & $BV_i = V_{i-1}$ &
 $CV_i \subseteq V_{i-1}+V_i + V_{i+1}$
\\
 $(C,B)$ &
 $AV_i \subseteq V_{i-1}+V_i + V_{i+1}$ &
 $BV_i = V_{i+1}$ & $CV_i = V_{i-1}$ 
\\
 $(A,C)$ &
 $AV_i = V_{i-1}$ &
$BV_i \subseteq V_{i-1}+V_i + V_{i+1}$ &
 $CV_i = V_{i+1}$ 
\\
   \end{tabular}}
     \bigskip

\end{lemma}
\begin{proof} We verify the first row; the other rows
are similarly verified.
Let $i$ be given.
By construction $AV_i = V_{i-1}$ and
$BV_i = V_{i+1}$.
 We now compute $CV_i$.
By Lemma
\ref{lem:LRTinducedFlag}
the flag $\lbrace A^{d-j}V\rbrace_{j=0}^d$ is induced by
$\lbrace V_j \rbrace_{j=0}^d$.
By Lemma \ref{lem:flagaction}
the flag $\lbrace A^{d-j}V\rbrace_{j=0}^d$ is raised by
 $C$.
By these comments and
Definition
\ref{def:FlagLR},
\begin{equation}
\label{eq:C1}
C V_i \subseteq
C(V_0 + \cdots + V_i) \subseteq 
V_0 +  \cdots +  V_{i+1}.
\end{equation}
Similarly, the flag $\lbrace B^{d-j}V\rbrace_{j=0}^d$ is induced by
$\lbrace V_{d-j} \rbrace_{j=0}^d$ and raised by $C$.
Therefore
\begin{equation}
\label{eq:C2}
C V_i \subseteq 
C (V_i+\cdots + V_d) \subseteq 
V_{i-1} + \cdots + V_d.
\end{equation}
Combining 
(\ref{eq:C1}), 
(\ref{eq:C2}) we obtain
$C V_i \subseteq V_{i-1} + V_{i} + V_{i+1}$.
\end{proof}

\begin{lemma}
\label{lem:ABCGEN}
Let $A,B,C$ denote an LR triple on $V$. Then
the $\mathbb F$-algebra ${\rm End}(V)$ is generated by any two of
$A,B,C$.
\end{lemma}
\begin{proof}
By Corollary
\ref{cor:ABgen}.
\end{proof}

\begin{definition}
\label{def:LRTIdem}
\rm
Let $A,B,C$ denote an LR triple on $V$. Recall
that
the pair $A,B$ (resp. $B,C$) (resp. $C,A$) is an
LR pair on $V$. For these LR pairs the idempotent
sequence from Definition
\ref{def:ABE}
is denoted as follows:

     \bigskip

\centerline{
\begin{tabular}[t]{c|c}
 {\rm LR pair} & {\rm idempotent sequence} 
 \\
 \hline
 $A,B$ & $\lbrace E_i \rbrace_{i=0}^d$
   \\
 $B,C$ & $\lbrace E'_i \rbrace_{i=0}^d$
   \\
 $C,A$ & $\lbrace E''_i \rbrace_{i=0}^d$
   \end{tabular}}
     \bigskip

\noindent 
We call the sequence
\begin{eqnarray}
\label{eq:idseq}
(
\lbrace E_i\rbrace_{i=0}^d;
\lbrace E'_i\rbrace_{i=0}^d;
\lbrace E''_i\rbrace_{i=0}^d)
\end{eqnarray}
the {\it idempotent data} of $A,B,C$.
\end{definition}

\begin{lemma}
\label{lem:isequal}
Let $A,B,C$ denote an LR triple on $V$.
Let $\alpha, \beta, \gamma$ denote nonzero scalars in $\mathbb F$.
Then the idempotent data of $\alpha A,\beta B, \gamma C$ is equal to
the 
idempotent data of $A,B,C$.
\end{lemma}
\begin{proof} By the last assertion of
Lemma
\ref{lem:alphaBeta}.
\end{proof}

\noindent 
Let $A,B,C$ denote an LR triple on $V$.
Our next goal is to compute the idempotent
data for the relatives of $A,B,C$.

\begin{lemma}
\label{lem:ABCEvar}
Let $A,B,C$ denote an LR triple on $V$,
with idempotent data
{\rm (\ref{eq:idseq})}.
In each row of the table below, we display an LR triple on $V$
along with its idempotent data.

     \bigskip

\centerline{
\begin{tabular}[t]{c|c}
 {\rm LR triple} & {\rm idempotent data}
 \\
 \hline
 \hline
 $A,B,C$ &
$(\lbrace E_i \rbrace_{i=0}^d;
\lbrace E'_i \rbrace_{i=0}^d;
\lbrace E''_i \rbrace_{i=0}^d)$
   \\
 $B,C,A$ &
$(\lbrace E'_i \rbrace_{i=0}^d;
\lbrace E''_i \rbrace_{i=0}^d;
\lbrace E_i \rbrace_{i=0}^d)$
   \\
 $C,A,B$ &
$(\lbrace E''_i \rbrace_{i=0}^d;
\lbrace E_i \rbrace_{i=0}^d;
\lbrace E'_i \rbrace_{i=0}^d)$
   \\
 \hline
 $C,B,A$ &
$(\lbrace E'_{d-i} \rbrace_{i=0}^d;
\lbrace E_{d-i} \rbrace_{i=0}^d;
\lbrace E''_{d-i} \rbrace_{i=0}^d)$
   \\
 $A,C,B$ &
$(\lbrace E''_{d-i} \rbrace_{i=0}^d;
\lbrace E'_{d-i} \rbrace_{i=0}^d;
\lbrace E_{d-i} \rbrace_{i=0}^d)$
   \\
 $B,A,C$ &
$(\lbrace E_{d-i} \rbrace_{i=0}^d;
\lbrace E''_{d-i} \rbrace_{i=0}^d;
\lbrace E'_{d-i} \rbrace_{i=0}^d)$
   \end{tabular}}
     \bigskip

\end{lemma}
\begin{proof} Use Lemma
\ref{lem:Ebackward}.
\end{proof}

\begin{lemma}
\label{lem:tildeABCEvar}
Let $A,B,C$ denote an LR triple on $V$,
with idempotent data
{\rm (\ref{eq:idseq})}.
In each row of the table below, we display
an LR triple on $V^*$ along with its idempotent data.

     \bigskip

\centerline{
\begin{tabular}[t]{c|c}
 {\rm LR triple} & {\rm idempotent data}
 \\
 \hline
 \hline
 $\tilde A, \tilde B, \tilde C$ &
$(\lbrace \tilde E_{d-i} \rbrace_{i=0}^d;
\lbrace \tilde E'_{d-i} \rbrace_{i=0}^d;
\lbrace \tilde E''_{d-i} \rbrace_{i=0}^d)$
   \\
 $\tilde B, \tilde C, \tilde A$ &
$(\lbrace \tilde E'_{d-i} \rbrace_{i=0}^d;
\lbrace \tilde E''_{d-i} \rbrace_{i=0}^d;
\lbrace \tilde E_{d-i} \rbrace_{i=0}^d)$
   \\
 $\tilde C, \tilde A, \tilde B$ &
$(\lbrace \tilde E''_{d-i} \rbrace_{i=0}^d;
\lbrace \tilde E_{d-i} \rbrace_{i=0}^d;
\lbrace \tilde E'_{d-i} \rbrace_{i=0}^d)$
   \\
 \hline
 $\tilde C, \tilde B, \tilde A$ &
$(\lbrace \tilde E'_{i} \rbrace_{i=0}^d;
\lbrace \tilde E_{i} \rbrace_{i=0}^d;
\lbrace \tilde E''_{i} \rbrace_{i=0}^d)$
   \\
 $\tilde A, \tilde C, \tilde B$ &
$(\lbrace \tilde E''_{i} \rbrace_{i=0}^d;
\lbrace \tilde E'_{i} \rbrace_{i=0}^d;
\lbrace \tilde E_{i} \rbrace_{i=0}^d)$
   \\
 $\tilde B, \tilde A, \tilde C$ &
$(\lbrace \tilde E_{i} \rbrace_{i=0}^d;
\lbrace \tilde E''_{i} \rbrace_{i=0}^d;
\lbrace \tilde E'_{i} \rbrace_{i=0}^d)$
   \end{tabular}}
     \bigskip

\end{lemma}
\begin{proof} 
By Lemmas
\ref{lem:LRdual},
\ref{lem:ABCEvar}.
\end{proof}

\begin{lemma}
\label{lem:Eform2}
Let $A,B,C$ denote an LR triple on $V$,
with parameter array
{\rm (\ref{eq:paLRT})} and idempotent data
{\rm (\ref{eq:idseq})}.
Then for $0 \leq i \leq d$,
\begin{eqnarray*}
&&
E_i = \frac{A^{d-i}B^d A^i}{\varphi_1 \cdots \varphi_d},
\qquad \qquad 
E'_i = \frac{B^{d-i}C^d B^i}{\varphi'_1 \cdots \varphi'_d},
\qquad \qquad 
E''_i = \frac{C^{d-i}A^d C^i}{\varphi''_1 \cdots \varphi''_d},
\\
&&
E_i = \frac{B^{i}A^d B^{d-i}}{\varphi_1 \cdots \varphi_d},
\qquad \qquad 
E'_i = \frac{C^{i}B^d C^{d-i}}{\varphi'_1 \cdots \varphi'_d},
\qquad \qquad 
E''_i = \frac{A^{i}C^d A^{d-i}}{\varphi''_1 \cdots \varphi''_d}.
\end{eqnarray*}
\end{lemma}
\begin{proof} By Lemma
\ref{lem:Eform}.
\end{proof}

\begin{lemma}
\label{lem:zeroprodABC}
Let $A,B,C$ denote an LR triple on $V$, with
idempotent data
{\rm (\ref{eq:idseq})}.
Then for $0 \leq i<j\leq d$ the following are zero:
\begin{eqnarray*}
&&A^j E_i,
\qquad \qquad \;
E_j A^{d-i},
\qquad \qquad \;
E_i B^j,
\qquad \qquad \;
B^{d-i} E_j,
\\
&&B^j E'_i,
\qquad \qquad \,
E'_j B^{d-i},
\qquad \qquad  \,
E'_i C^j,
\qquad \qquad \,
C^{d-i} E'_j,
\\
&&C^j E''_i,
\qquad \qquad 
E''_j C^{d-i},
\qquad \qquad 
E''_i A^j,
\qquad \qquad 
A^{d-i} E''_j.
\end{eqnarray*}
\end{lemma}
\begin{proof}
By Lemma
\ref{lem:zeroprod}.
\end{proof}


\begin{lemma}
\label{lem:EiEjp}
Let $A,B,C$ denote an LR triple on $V$,
with idempotent data
{\rm (\ref{eq:idseq})}.
Then the following {\rm (i), (ii)} hold for $0 \leq i,j\leq d$.
\begin{enumerate}
\item[\rm (i)] Suppose $i+j<d$. Then 
\begin{eqnarray*}
E_i E'_j = 0, \qquad \qquad 
E'_i E''_j = 0, \qquad \qquad 
E''_i E_j = 0.
\end{eqnarray*}
\item[\rm (ii)] Suppose $i+j>d$. Then 
\begin{eqnarray*}
E'_j E_i = 0, \qquad \qquad 
E''_j E'_i = 0,\qquad \qquad 
E_j E''_i = 0.
\end{eqnarray*}
\end{enumerate}
\end{lemma}
\begin{proof} (i) We show
$E_i E'_j = 0$.
The sequence $\lbrace E'_rV\rbrace_{r=0}^d$ is
the $(B,C)$ decomposition of $V$, which induces the flag
$\lbrace B^{d-r}V\rbrace_{r=0}^d$ on $V$. By this and Lemma
\ref{lem:zeroprodABC},
\begin{eqnarray*}
E_i E'_j V \subseteq E_i (E'_0V+  \cdots + E'_jV) = E_i B^{d-j}V = 0.
\end{eqnarray*}
This shows that $E_iE'_j=0$. The remaining assertions are similarly shown.
\\
\noindent (ii) 
We show
$E'_j E_i = 0$.
The sequence $\lbrace E_{d-r}V\rbrace_{r=0}^d$ is
the $(B,A)$ decomposition of $V$, which induces the flag
$\lbrace B^{d-r}V\rbrace_{r=0}^d$ on $V$. Now by Lemma
\ref{lem:zeroprodABC},
\begin{eqnarray*}
E'_j E_i V \subseteq E'_j (E_iV+ \cdots + E_dV) = E'_j B^iV = 0.
\end{eqnarray*}
This shows that $E'_jE_i=0$. The remaining assertions are similarly shown.
\end{proof}

\begin{lemma} 
\label{lem:tripleProd}
Let $A,B,C$ denote an LR triple on $V$,
with idempotent data
{\rm (\ref{eq:idseq})}. Then for $0 \leq i,j\leq d$,
\begin{eqnarray*}
&&
E_i E'_j E_i = \delta_{i+j,d}E_i,
\qquad \quad
E'_i E''_j E'_i = \delta_{i+j,d}E'_i,
\qquad \quad
E''_i E_j E''_i = \delta_{i+j,d}E''_i,
\\
&&
E_i E''_j E_i = \delta_{i+j,d}E_i,
\qquad \quad
E'_i E_j E'_i = \delta_{i+j,d}E'_i,
\qquad \quad
E''_i E'_j E''_i = \delta_{i+j,d}E''_i.
\end{eqnarray*}
\end{lemma}
\begin{proof}
Consider the product
$E_i E'_j E_i$. 
First assume that $i+j < d$.
Then
$E_i E'_j=0$ by Lemma
\ref{lem:EiEjp}(i), so 
$E_i E'_jE_i=0$.
Next assume that $i+j>d$.
Then $E'_j E_i = 0 $ by
Lemma
\ref{lem:EiEjp}(ii), so
$E_i E'_jE_i=0$.
Next assume that $i+j=d$.
By our results so far,
\begin{eqnarray*}
E_i E'_jE_i =  \sum_{r=0}^d E_i E'_rE_i  = E_i I E_i = E_i.
\end{eqnarray*}
We have verified our assertion for
the product $E_i E'_jE_i$; 
the remaining assertions are similarily verified.
 \end{proof}

\begin{lemma}
\label{lem:doubleTrace}
Let $A,B,C$ denote an LR triple on $V$,
with idempotent data
{\rm (\ref{eq:idseq})}.
Then for $0\leq i,j\leq d$ the products
\begin{eqnarray*}
&&
E_i E'_j, \qquad \qquad
E'_i E''_j, \qquad \qquad
E''_i E_j,
\\
&&
E_i E''_j, \qquad \qquad
E'_i E_j, \qquad \qquad
E''_i E'_j
\end{eqnarray*}
have trace $0$ if $i+j \not=d$ and
trace 1 if $i+j=d$.
\end{lemma}
\begin{proof} For each displayed equation
in Lemma
\ref{lem:tripleProd},
take the trace of each side and simplify
the result using
${\rm tr}(KL)=
{\rm tr}(LK)$.
 \end{proof}

\begin{proposition}
\label{eq:varphiTRACE}
Let $A,B,C$ denote an LR triple on $V$,
with  parameter array
{\rm (\ref{eq:paLRT})} 
and
idempotent data
{\rm (\ref{eq:idseq})}. 
In the table below,
 for each map $F$ in the header row,  we display
the trace of $FE_i$,
$FE'_i$,
$FE''_i$ for $0 \leq i \leq d$.

     \bigskip

\centerline{
\begin{tabular}[t]{c|ccc ccc}
 $F$ & $AB$ & $BA$ & $BC$ & $CB$ & $CA$ & $AC$
 \\
 \hline
 \hline
 ${\rm tr}(FE_i)$ &
  $\varphi_{i+1}$ &  
  $\varphi_{i}$ &  
  $\varphi'_{d-i+1}$ &  
  $\varphi'_{d-i}$ &  
  $\varphi''_{d-i+1}$ &  
  $\varphi''_{d-i}$ 
\\
 ${\rm tr}(FE'_i)$ &
  $\varphi_{d-i+1}$ &  
  $\varphi_{d-i}$  &
  $\varphi'_{i+1}$ &  
  $\varphi'_{i}$ &  
  $\varphi''_{d-i+1}$ &  
  $\varphi''_{d-i}$   
\\
 ${\rm tr}(FE''_i)$ &
  $\varphi_{d-i+1}$ &  
  $\varphi_{d-i}$ &  
  $\varphi'_{d-i+1}$ &  
  $\varphi'_{d-i}$ &
  $\varphi''_{i+1}$ &  
  $\varphi''_{i}$ 
   \end{tabular}}
     \bigskip

\end{proposition}
\begin{proof} We verify the first column of the table.
We have ${\rm tr}(ABE_i) = \varphi_{i+1}$ by
Lemma
\ref{lem:varphiTrace}.
To verify ${\rm tr}(ABE'_i)=\varphi_{d-i+1}$ and
 ${\rm tr}(ABE''_i)=\varphi_{d-i+1}$,
eliminate $AB$ 
using the equation on the left in
(\ref{eq:BiAione}), and evaluate the result using
Lemma
\ref{lem:doubleTrace}.
We have verified the first column of the table;
the remaining columns are similarly verified.
\end{proof}

\noindent In 
Proposition
\ref{eq:varphiTRACE} we used the trace function to
describe the 
parameter array of an LR triple.
We now use the trace function to define some more parameters
for an LR triple.

\begin{definition}
\label{def:traceData}
\rm
Let $A,B,C$ denote an LR triple on $V$,
with idempotent data
{\rm (\ref{eq:idseq})}. For $0 \leq i \leq d$
define
\begin{eqnarray}
a_i = {\rm tr}(CE_i),
\qquad \qquad 
a'_i = {\rm tr}(AE'_i),
\qquad \qquad 
a''_i = {\rm tr}(BE''_i).
\label{eq:aaa}
\end{eqnarray}
We call the sequence
\begin{eqnarray}
\label{eq:tracedata}
(\lbrace a_i \rbrace_{i=0}^d;
\lbrace a'_i \rbrace_{i=0}^d;
\lbrace a''_i \rbrace_{i=0}^d
)
\end{eqnarray}
 the {\it trace data} of $A,B,C$.
\end{definition}

\noindent Our next goal is to describe
the meaning of the trace data from
several points of view.

\begin{lemma}
\label{lem:tracedataI}
Let $A,B,C$ denote an LR triple on $V$,
with trace data
{\rm (\ref{eq:tracedata})}.
Consider a basis for $V$ that induces the
$(A,B)$-decomposition 
(resp. 
$(B,C)$-decomposition)
(resp. 
$(C,A)$-decomposition)
of $V$.
Then for $0 \leq i \leq d$, $a_i$ (resp. $a'_i$)
(resp. $a''_i$)
is the $(i,i)$-entry of the matrix in
 ${\rm Mat}_{d+1}(\F)$ that represents 
$C$ (resp. $A$) (resp. $B$) with respect to this basis.
\end{lemma}
\begin{proof} To obtain the assertion about
$a_i$, 
in the equation on the left in  (\ref{eq:aaa})
represent 
$C$ and $E_i$ by matrices with respect to the given basis.
The other assertions are similarly obtained.
\end{proof}

\begin{lemma}
\label{lem:traceDataII}
Let $A,B,C$ denote an LR triple on $V$,
with idempotent data
{\rm (\ref{eq:idseq})} and
trace data
{\rm (\ref{eq:tracedata})}.
Then for $0 \leq i \leq d$,
\begin{eqnarray*}
E_i C E_i = a_i E_i,
\qquad \qquad 
E'_i A E'_i = a'_i E'_i,
\qquad \qquad 
E''_i B E''_i = a''_i E''_i.
\end{eqnarray*}
\end{lemma}
\begin{proof} We verify the equation on the left.
Since $E_i$ is idempotent and rank 1, there exists
$a \in \mathbb F$ such that $E_iCE_i=a E_i$.
In this equation, take the trace of each side
and use Definition
\ref{def:traceData}
to get $a=a_i$.
\end{proof}

\begin{lemma}
\label{lem:aisum}
Let $A,B,C$ denote an LR triple on $V$,
with trace data 
{\rm (\ref{eq:tracedata})}. Then
\begin{eqnarray*}
 0 = \sum_{i=0}^d a_i,
\qquad \quad
 0 = \sum_{i=0}^d a'_i,
\qquad \quad
 0 = \sum_{i=0}^d a''_i.
\end{eqnarray*}
\end{lemma}
\begin{proof}
The sum $\sum_{i=0}^d a_i$ is the trace of $C$, 
which is zero since $C$ is nilpotent.
The remaining assertions are similarly shown.
\end{proof}

\begin{lemma}
\label{lem:traceAdj}
Let $A,B,C$ denote an LR triple on $V$,
with trace data
{\rm (\ref{eq:tracedata})}.
Let $\alpha, \beta, \gamma$ denote nonzero
scalars in $\mathbb F$. Then the 
LR triple $\alpha A, \beta B, \gamma C$
has trace data
\begin{eqnarray*}
(\lbrace \gamma a_i \rbrace_{i=0}^d;
\lbrace \alpha a'_i \rbrace_{i=0}^d;
\lbrace \beta a''_i \rbrace_{i=0}^d).
\end{eqnarray*}
\end{lemma}
\begin{proof}
Use Lemma
\ref{lem:isequal}
and Definition
\ref{def:traceData}.
\end{proof}

\noindent Let $A,B,C$ denote an LR triple on $V$. Our next goal
is to compute the trace data for the relatives of $A,B,C$.

\begin{lemma}
\label{lem:tracedataAlt}
Let $A,B,C$ denote an LR triple on $V$,
with trace data
{\rm   (\ref{eq:tracedata})}.
In each row of the table below, we display an
LR triple on $V$ along with its trace data.

     \bigskip

\centerline{
\begin{tabular}[t]{c|c}
 {\rm LR triple} & {\rm trace data}
 \\
 \hline
 \hline
 $A,B,C$ &
$(\lbrace a_i \rbrace_{i=0}^d;
\lbrace a'_i \rbrace_{i=0}^d;
\lbrace a''_i \rbrace_{i=0}^d)$
   \\
 $B,C,A$ &
$(\lbrace a'_i \rbrace_{i=0}^d;
\lbrace a''_i \rbrace_{i=0}^d;
\lbrace a_i \rbrace_{i=0}^d)$
   \\
 $C,A,B$ &
$(\lbrace a''_i \rbrace_{i=0}^d;
\lbrace a_i \rbrace_{i=0}^d;
\lbrace a'_i \rbrace_{i=0}^d)$
   \\
 \hline
 $C,B,A$ &
$(\lbrace a'_{d-i} \rbrace_{i=0}^d;
\lbrace a_{d-i} \rbrace_{i=0}^d;
\lbrace a''_{d-i} \rbrace_{i=0}^d)$
   \\
 $A,C,B$ &
$(\lbrace a''_{d-i} \rbrace_{i=0}^d;
\lbrace a'_{d-i} \rbrace_{i=0}^d;
\lbrace a_{d-i} \rbrace_{i=0}^d)$
   \\
 $B,A,C$ &
$(\lbrace a_{d-i} \rbrace_{i=0}^d;
\lbrace a''_{d-i} \rbrace_{i=0}^d;
\lbrace a'_{d-i} \rbrace_{i=0}^d)$
   \end{tabular}}
     \bigskip

\end{lemma}
\begin{proof} 
By Lemma
\ref{lem:ABCEvar}
and Definition
\ref{def:traceData}.
\end{proof}

\begin{lemma} 
\label{lem:tracedataDual}
Let $A,B,C$ denote an LR triple on $V$,
with trace data
{\rm   (\ref{eq:tracedata})}.
In each row of the table below, we display
an LR triple on $V^*$ along with its trace data.

     \bigskip

\centerline{
\begin{tabular}[t]{c|c}
 {\rm LR triple} & {\rm trace data}
 \\
 \hline
 \hline
 $\tilde A, \tilde B, \tilde C$ &
$(\lbrace a_{d-i} \rbrace_{i=0}^d;
\lbrace a'_{d-i} \rbrace_{i=0}^d;
\lbrace a''_{d-i} \rbrace_{i=0}^d)$
   \\
 $\tilde B, \tilde C, \tilde A$ &
$(\lbrace a'_{d-i} \rbrace_{i=0}^d;
\lbrace a''_{d-i} \rbrace_{i=0}^d;
\lbrace a_{d-i} \rbrace_{i=0}^d)$
   \\
 $\tilde C, \tilde A, \tilde B$ &
$(\lbrace a''_{d-i} \rbrace_{i=0}^d;
\lbrace a_{d-i} \rbrace_{i=0}^d;
\lbrace a'_{d-i} \rbrace_{i=0}^d)$
   \\
 \hline
 $\tilde C, \tilde B, \tilde A$ &
$(\lbrace a'_{i} \rbrace_{i=0}^d;
\lbrace a_{i} \rbrace_{i=0}^d;
\lbrace a''_{i} \rbrace_{i=0}^d)$
   \\
 $\tilde A, \tilde C, \tilde B$ &
$(\lbrace a''_{i} \rbrace_{i=0}^d;
\lbrace a'_{i} \rbrace_{i=0}^d;
\lbrace a_{i} \rbrace_{i=0}^d)$
   \\
 $\tilde B, \tilde A, \tilde C$ &
$(\lbrace a_{i} \rbrace_{i=0}^d;
\lbrace a''_{i} \rbrace_{i=0}^d;
\lbrace a'_{i} \rbrace_{i=0}^d)$
   \end{tabular}}
     \bigskip

\end{lemma}
\begin{proof} 
An element in ${\rm End}(V)$ has the same trace as its adjoint.
The result follows from this along with
Lemma
\ref{lem:tildeABCEvar}
and Definition
\ref{def:traceData}.
\end{proof}

\noindent  Let
$A,B,C$ denote an LR triple on $V$, with parameter array
(\ref{eq:paLRT}),
idempotent data
{\rm (\ref{eq:idseq})},
and trace data
(\ref{eq:tracedata}).
Associated with $A,B,C$ are 12 types of bases for $V$:

\begin{eqnarray}
&&
(A,B),  \quad\qquad \mbox{\rm inverted}\;(A,B), \quad \qquad
(B,A),  \quad \qquad \mbox{\rm inverted}\; (B,A),
\label{eq:typeAB}
\\
&&
(B,C),  \qquad \quad \mbox{\rm inverted}\; (B,C), \quad \qquad
(C,B),  \qquad\quad \mbox{\rm inverted}\; (C,B),
\label{eq:typeBC}
\\
&&
(C,A),  \qquad \quad \,\mbox{\rm inverted}\; (C,A), \quad \qquad
(A,C),  \qquad \quad \,\mbox{\rm inverted}\; (A,C).
\label{eq:typeCA}
\end{eqnarray}
\noindent We now consider the actions of $A,B,C$ on these
bases. We will use the following notation.

\begin{definition}
\label{def:natural}
\rm
For the above LR triple $A,B,C$ consider the 12 types of
bases for $V$ from
{\rm (\ref{eq:typeAB})--(\ref{eq:typeCA})}.
For each type $\natural$ in the list and $F \in 
{\rm End}(V)$ let $F^\natural$ denote the matrix
in ${\rm Mat}_{d+1}(\F)$ that represents
$F$
with respect to a basis for $V$ of type $\natural$.
Note that the map $\natural : 
{\rm End}(V) \to 
{\rm Mat}_{d+1}(\F)$, $F \mapsto F^\natural$
is an $\mathbb F$-algebra isomorphism.
\end{definition}

\begin{proposition}
\label{prop:matrixRep}
For the above LR triple $A,B,C$  
consider the 12 types of bases for $V$
from
{\rm (\ref{eq:typeAB})--(\ref{eq:typeCA})}.
For each type $\natural $ in the list,
the entries of
$A^\natural$, $B^\natural$, $C^\natural$ are given in the table below.
All entries not shown are zero.

     \bigskip

\centerline{
\begin{tabular}[t]{c|ccc|ccc|ccc}
 $\natural$ &
 $A^\natural_{i,i-1}$ & $A^\natural_{i,i}$ & $A^\natural_{i-1,i}$
 & 
 $B^\natural_{i,i-1}$ & $B^\natural_{i,i}$ & $B^\natural_{i-1,i}$
 &
 $C^\natural_{i,i-1}$ & $C^\natural_{i,i}$ & $C^\natural_{i-1,i}$
 \\
 \hline
 \hline
 $(A,B)$ &
$0$ & $0$ & $1$ 
&
$\varphi_i$ & $0$ & $0$ 
&
$\varphi''_{d-i+1}$ & $a_i$ & $\frac{\varphi'_{d-i+1}}{\varphi_i}$ 
\\
{\rm inv. $(A,B)$} &
$1$ & $0$ & $0$ 
&
$0$ & $0$ & $\varphi_{d-i+1}$ 
&
$\frac{\varphi'_{i}}{\varphi_{d-i+1}}$ & $a_{d-i}$ & $\varphi''_{i}$ 
\\
$(B,A)$ &
$\varphi_{d-i+1}$ & $0$ & $0$ 
&
$0$ & $0$ & $1$ 
&
$\varphi'_{i}$ & $a_{d-i}$ & $\frac{\varphi''_{i}}{\varphi_{d-i+1}}$ 
\\
{\rm inv. $(B,A)$} &
$0$ & $0$ & $\varphi_i $ 
&
$1$ & $0$ & $0$ 
&
$\frac{\varphi''_{d-i+1}}{\varphi_i}$ & $a_i$ & 
$\varphi'_{d-i+1}$ 
\\
\hline
 $(B,C)$ &
$\varphi_{d-i+1}$ & $a'_{i}$ & $\frac{\varphi''_{d-i+1}}{\varphi'_i}$ 
&
$0$ & $0$ & $1$ 
&
$\varphi'_i$ & $0$ & $0$ 
\\
{\rm inv. $(B,C)$} &
$\frac{\varphi''_{i}}{\varphi'_{d-i+1}}$ & $a'_{d-i}$ & $\varphi_{i}$ 
&
$1$ & $0$ & $0$ 
&
$0$ & $0$ & $\varphi'_{d-i+1}$ 
\\
$(C,B)$ &
$\varphi''_{i}$ & $a'_{d-i}$ & $\frac{\varphi_{i}}{\varphi'_{d-i+1}}$ 
&
$\varphi'_{d-i+1}$ & $0$ & $0$ 
&
$0$ & $0$ & $1$ 
\\
{\rm inv. $(C,B)$} &
$\frac{\varphi_{d-i+1}}{\varphi'_i}$ & $a'_{i}$ & 
$\varphi''_{d-i+1}$ 
&
$0$ & $0$ & $\varphi'_i $ 
&
$1$ & $0$ & $0$ 
\\
\hline
 $(C,A)$ &
$\varphi''_i$ & $0$ & $0$ 
&
$\varphi'_{d-i+1}$ & $a''_{i}$ & $\frac{\varphi_{d-i+1}}{\varphi''_i}$ 
&
$0$ & $0$ & $1$ 
\\
{\rm inv. $(C,A)$} &
$0$ & $0$ & $\varphi''_{d-i+1}$ 
&
$\frac{\varphi_{i}}{\varphi''_{d-i+1}}$ & $a''_{d-i}$ & $\varphi'_{i}$ 
&
$1$ & $0$ & $0$ 
\\
$(A,C)$ &
$0$ & $0$ & $1$ 
&
$\varphi_{i}$ & $a''_{d-i}$ & $\frac{\varphi'_{i}}{\varphi''_{d-i+1}}$ 
&
$\varphi''_{d-i+1}$ & $0$ & $0$ 
\\
{\rm inv. $(A,C)$} &
$1$ & $0$ & $0$ 
&
$\frac{\varphi'_{d-i+1}}{\varphi''_i}$ & $a''_{i}$ & 
$\varphi_{d-i+1}$ 
&
$0$ & $0$ & $\varphi''_i $ 
\\
\end{tabular}}
     \bigskip

\end{proposition}
\begin{proof}
 We verify the first row of the table.
Consider a basis for $V$ of type $\natural=(A,B)$.
The entries of $A^\natural$ and
$B^\natural$ are given in
Lemma \ref{lem:ABmatrix}. 
We now compute the entries of
$C^\natural$. This matrix is tridiagonal
by Lemma
\ref{lem:LRTaction}.
The diagonal entries of $C^\natural$ are given in
 Lemma
\ref{lem:tracedataI}.
As we compute additional entries of $C^\natural$,
we will use the fact that
for $0 \leq j \leq d$ the matrix $E^\natural_j$ has
$(j,j)$-entry 1 and all other entries $0$.
For $1 \leq i \leq d$ we now compute the
$(i,i-1)$-entry of $C^\natural$.
We evaluate ${\rm tr}(CAE_i)$ in two ways.
On one hand, by Proposition
\ref{eq:varphiTRACE} this trace is
equal to 
$\varphi''_{d-i+1}$.
On the other hand, by linear algebra this trace is equal to
 ${\rm tr}(C^\natural A^\natural E^\natural_i)$, which is equal to
 the $(i,i)$-entry of $C^\natural A^\natural $ by the form of
 $E^\natural_i$. By the form of $A^\natural $,
 the $(i,i)$-entry of $C^\natural A^\natural $ is equal to
$C^\natural_{i,i-1}$.
By these comments 
$C^\natural_{i,i-1}=\varphi''_{d-i+1}$.
Next we compute
the $(i-1,i)$-entry of $C^\natural$.
We evaluate ${\rm tr}(BCE_i)$ in two ways.
On one hand, by Proposition
\ref{eq:varphiTRACE} this trace is
equal to 
$\varphi'_{d-i+1}$.
On the other hand, 
by linear algebra this trace is equal to
 ${\rm tr}(B^\natural C^\natural E^\natural_i)$, which is equal to
 the $(i,i)$-entry of $B^\natural C^\natural $ by the form of
 $E^\natural_i$. By the form of $B^\natural$,
 the $(i,i)$-entry of $B^\natural C^\natural $ is equal to
$\varphi_i C^\natural_{i-1,i}$.
By these comments 
$C^\natural_{i-1,i}=\varphi'_{d-i+1}/\varphi_i$.
We have verified the first row of the table, and
the remaining rows are similarly verified.
\end{proof}

\begin{proposition}
\label{prop:IsoParTrace}
An LR triple is uniquely determined up to isomorphism
by its parameter array and trace data.
\end{proposition}
\begin{proof}
In 
Proposition \ref{prop:matrixRep},
the matrix entries are determined by the
parameter array and trace data.
\end{proof}

\noindent Let 
$A,B,C$ denote an LR triple on $V$.
Recall the 12 types of bases for $V$
from
{\rm (\ref{eq:typeAB})--(\ref{eq:typeCA})}.
We now consider how
 these bases are related. As we proceed,
keep in mind that any permutation of
$A,B,C$ is an LR triple on $V$.

\begin{lemma}
\label{lem:xyz}
Let $A,B,C$ denote an LR triple on $V$.
Let $\lbrace u_i \rbrace_{i=0}^d$ denote an
 $(A,C)$-basis of
$V$, and let $\lbrace v_i\rbrace_{i=0}^d$
denote an $(A,B)$-basis of $V$.
Then the transition matrix from
$\lbrace u_i \rbrace_{i=0}^d$ to
$\lbrace v_i \rbrace_{i=0}^d$ is upper triangular
and Toeplitz.
\end{lemma}
\begin{proof} 
Let $S \in 
{\rm Mat}_{d+1}(\F)$ denote the transition matrix
in question.
By Definition
\ref{def:ABbasis}  and the construction,
 $Au_i = u_{i-1}$ 
$(1 \leq i \leq d)$, $Au_0=0$, 
 $Av_i = v_{i-1}$
$(1 \leq i \leq d)$, $Av_0=0$.
Now by Proposition
\ref{prop:Toe}, $S$ is upper triangular and Toeplitz.
\end{proof}

\begin{definition}\rm
\label{def:compat}
Two bases of $V$ will be called 
{\it compatible}
whenever the transition matrix from
one basis to the other is upper triangular and Toeplitz,
with all diagonal entries 1.
\end{definition}

\begin{lemma}
\label{lem:exCompat}
Let $A,B,C$ denote an LR triple on $V$. 
Given an $(A,C)$-basis of $V$, there exists
a compatible $(A,B)$-basis of $V$.
\end{lemma}
\begin{proof} Let $\lbrace u_i\rbrace_{i=0}^d$ denote the
$(A,C)$-basis in question.
Let 
$\lbrace v_i\rbrace_{i=0}^d$ denote an 
$(A,B)$-basis of $V$.
Let $S \in 
{\rm Mat}_{d+1}(\F)$ denote the transition matrix
from 
$\lbrace u_i\rbrace_{i=0}^d$ to
$\lbrace v_i\rbrace_{i=0}^d$.
By construction $S$ is invertible.
By Lemma
\ref{lem:xyz},
$S$ is upper triangular and Toeplitz.
Let $\lbrace \alpha_i \rbrace_{i=0}^d$ denote
the corresponding parameters, and note that $\alpha_0\not=0$.
Define $v'_i = v_i /\alpha_0$ for $0 \leq i \leq d$.
Then $\lbrace v'_i\rbrace_{i=0}^d$ is an $(A,B)$-basis
of $V$. The transition matrix from
$\lbrace u_i\rbrace_{i=0}^d$ to 
$\lbrace v'_i\rbrace_{i=0}^d$ is $S/\alpha_0$.
This matrix is upper triangular and Toeplitz, with
all diagonal entries 1.
Now by Definition
\ref{def:compat}
the basis $\lbrace v'_i\rbrace_{i=0}^d$ 
is compatible with 
 $\lbrace u_i\rbrace_{i=0}^d$.
\end{proof}

\begin{definition}\rm
\label{def:TTT}
Let $A,B,C$ denote an LR triple on $V$.
We define matrices $T, T', T''$ in
${\rm Mat}_{d+1}(\F)$ as follows:
\begin{enumerate}
\item[\rm (i)] 
 $T$
is the transition matrix
from a $(C,B)$-basis of $V$ to a compatible
$(C,A)$-basis of $V$;
\item[\rm (ii)] 
 $T'$
is the transition matrix
from an $(A,C)$-basis of $V$ to a compatible
$(A,B)$-basis of $V$;
\item[\rm (iii)] 
 $T''$
is the transition matrix
from a $(B,A)$-basis of $V$ to a compatible
$(B,C)$-basis of $V$.
\end{enumerate}
\end{definition}

\begin{definition}
\label{def:TTT1}
\rm
Let $A,B,C$ denote an LR triple on $V$.
By 
Definition
\ref{def:compat}
the associated matrix
$T$ (resp. $T'$)  (resp. $T''$) is upper triangular and
Toeplitz; let 
$\lbrace \alpha_i \rbrace_{i=0}^d$ 
(resp. $\lbrace \alpha'_i \rbrace_{i=0}^d$) 
(resp. $\lbrace \alpha''_i \rbrace_{i=0}^d$) 
denote the
corresponding parameters. Let
$\lbrace \beta_i \rbrace_{i=0}^d$,
$\lbrace \beta'_i \rbrace_{i=0}^d$,
$\lbrace \beta''_i \rbrace_{i=0}^d$
denote the
parameters for 
$T^{-1}$,
$(T')^{-1}$,
 $(T'')^{-1}$ respectively.
We call the 6-tuple
\begin{eqnarray}
(
\lbrace \alpha_i \rbrace_{i=0}^d,
\lbrace \beta_i \rbrace_{i=0}^d;
\lbrace \alpha'_i \rbrace_{i=0}^d,
\lbrace \beta'_i \rbrace_{i=0}^d;
\lbrace \alpha''_i \rbrace_{i=0}^d,
\lbrace \beta''_i \rbrace_{i=0}^d 
)
\label{eq:ToeplitzData}
\end{eqnarray}
the {\it Toeplitz data} for
$A,B,C$. 
For notational convenience
define each of the following to be zero: 
\begin{eqnarray*}
\alpha_{d+1},
\qquad \alpha'_{d+1},
\qquad \alpha''_{d+1},
\qquad 
\beta_{d+1},
\qquad  \beta'_{d+1},
\qquad \beta''_{d+1}.
\end{eqnarray*}
\end{definition}

\begin{lemma} 
\label{lem:alpha0}
Referring to Definition
\ref{def:TTT1},
\begin{eqnarray}
&&\alpha_0 = 1, \qquad \qquad 
\alpha'_0=1, \qquad \qquad 
\alpha''_0=1,
\label{eq:list1}
\\
&&\beta_0 = 1, \qquad \qquad 
\beta'_0=1, \qquad \qquad 
\beta''_0=1.
\label{eq:list2}
\end{eqnarray}
Moreover
\begin{eqnarray}
\label{eq:list3}
\beta_1 = - \alpha_1, 
\qquad \qquad
\beta'_1 = - \alpha'_1, 
\qquad \qquad
\beta''_1 = - \alpha''_1.
\end{eqnarray}
\end{lemma}
\begin{proof}
Concerning
(\ref{eq:list1}),
(\ref{eq:list2})
the matrices $T,T',T''$ and their inverses are
all transition matrices between
a pair of compatible bases.
So their diagonal entries are all 1 by Definition
\ref{def:compat}.
Line (\ref{eq:list3})
comes from above
Lemma \ref{lem:reform1}.
\end{proof}

\begin{lemma} Referring to Definitions
\ref{def:tauMat},
\ref{def:TTT},
\begin{eqnarray*}
&&T = \sum_{i=0}^d \alpha_i \tau^i \qquad \qquad
T' = \sum_{i=0}^d \alpha'_i \tau^i \qquad \qquad
T'' = \sum_{i=0}^d \alpha''_i \tau^i,
\\
&&
T^{-1} = \sum_{i=0}^d \beta_i \tau^i,
\qquad \quad
(T')^{-1} = \sum_{i=0}^d \beta'_i \tau^i,
\qquad \quad
(T'')^{-1} = \sum_{i=0}^d \beta''_i \tau^i.
\end{eqnarray*}
Moreover $T,T', T'', \tau$ mutually commute.
\end{lemma}
\begin{proof} By Lemma
\ref{lem:ToeChar}.
\end{proof}

\begin{lemma}
\label{lem:ToeplitzAdjust}
Let $A,B,C$ denote an LR triple on $V$, with
Toeplitz data
{\rm (\ref{eq:ToeplitzData})}. Let $\alpha, \beta, \gamma$
denote nonzero scalars in $\mathbb F$. Then the
LR triple $\alpha A, \beta B, \gamma C$ has Toeplitz data
\begin{eqnarray}
(
\lbrace \gamma^{-i}\alpha_i \rbrace_{i=0}^d,
\lbrace \gamma^{-i}\beta_i \rbrace_{i=0}^d;
\lbrace \alpha^{-i} \alpha'_i \rbrace_{i=0}^d,
\lbrace \alpha^{-i} \beta'_i \rbrace_{i=0}^d;
\lbrace \beta^{-i} \alpha''_i \rbrace_{i=0}^d,
\lbrace \beta^{-i} \beta''_i \rbrace_{i=0}^d 
).
\label{eq:NewToep}
\end{eqnarray}
\end{lemma}
\begin{proof} Use Lemma
\ref{lem:a0a1}(ii).
\end{proof}

\noindent Let $A,B,C$ denote an LR triple on $V$. Our next goal
is to compute the Toeplitz data for the relatives of $A,B,C$.

\begin{lemma}
\label{lem:ToeplitzData}
Let $A,B,C$ denote an LR triple on $V$, with
Toeplitz data
{\rm (\ref{eq:ToeplitzData})}.
In each row of the table below, we display
an LR triple on $V$ along  with its Toeplitz data.

     \bigskip

\centerline{
\begin{tabular}[t]{c|c}
 {\rm LR triple} & {\rm Toeplitz data}
 \\
 \hline
 \hline
 $A,B,C$ &
$(
\lbrace \alpha_i \rbrace_{i=0}^d,
\lbrace \beta_i \rbrace_{i=0}^d;
\lbrace \alpha'_i \rbrace_{i=0}^d,
\lbrace \beta'_i \rbrace_{i=0}^d;
\lbrace \alpha''_i \rbrace_{i=0}^d,
\lbrace \beta''_i \rbrace_{i=0}^d 
)$
   \\
 $B,C,A$ &
$(
\lbrace \alpha'_i \rbrace_{i=0}^d,
\lbrace \beta'_i \rbrace_{i=0}^d;
\lbrace \alpha''_i \rbrace_{i=0}^d,
\lbrace \beta''_i \rbrace_{i=0}^d;
\lbrace \alpha_i \rbrace_{i=0}^d,
\lbrace \beta_i \rbrace_{i=0}^d 
)$
   \\
 $C,A,B$ &
$(
\lbrace \alpha''_i \rbrace_{i=0}^d,
\lbrace \beta''_i \rbrace_{i=0}^d;
\lbrace \alpha_i \rbrace_{i=0}^d,
\lbrace \beta_i \rbrace_{i=0}^d;
\lbrace \alpha'_i \rbrace_{i=0}^d,
\lbrace \beta'_i \rbrace_{i=0}^d 
)$
   \\
 \hline
 $C,B,A$ &
$(
\lbrace \beta'_i \rbrace_{i=0}^d, 
\lbrace \alpha'_i \rbrace_{i=0}^d;
\lbrace \beta_i \rbrace_{i=0}^d,
\lbrace \alpha_i \rbrace_{i=0}^d;
\lbrace \beta''_i \rbrace_{i=0}^d,
\lbrace \alpha''_i \rbrace_{i=0}^d
)$
   \\
 $A,C,B$ &
$(
\lbrace \beta''_i \rbrace_{i=0}^d,
\lbrace \alpha''_i \rbrace_{i=0}^d;
\lbrace \beta'_i \rbrace_{i=0}^d,
\lbrace \alpha'_i \rbrace_{i=0}^d;
\lbrace \beta_i \rbrace_{i=0}^d,
\lbrace \alpha_i \rbrace_{i=0}^d
)$
   \\
 $B,A,C$ &
$(
\lbrace \beta_i \rbrace_{i=0}^d,
\lbrace \alpha_i \rbrace_{i=0}^d;
\lbrace \beta''_i \rbrace_{i=0}^d,
\lbrace \alpha''_i \rbrace_{i=0}^d;
\lbrace \beta'_i \rbrace_{i=0}^d,
\lbrace \alpha'_i \rbrace_{i=0}^d
)$
\\
   \end{tabular}}
\bigskip

\end{lemma}
\begin{proof} By Definition
\ref{def:TTT}
and the construction.
\end{proof}


\begin{lemma}
\label{lem:compatFlip}
Let $\lbrace u_i \rbrace_{i=0}^d$ 
 (resp. $\lbrace v_i \rbrace_{i=0}^d$) denote a
 basis of $V$, and let
$\lbrace u'_i \rbrace_{i=0}^d$ 
 (resp. $\lbrace v'_i \rbrace_{i=0}^d$) denote the dual
 basis of $V^*$. Then the following are equivalent:
\begin{enumerate}
\item[\rm (i)]  
 $\lbrace u_i \rbrace_{i=0}^d$ and 
 $\lbrace v_i \rbrace_{i=0}^d$ are compatible;
\item[\rm (ii)]  
 $\lbrace u'_{d-i} \rbrace_{i=0}^d$ and 
 $\lbrace v'_{d-i} \rbrace_{i=0}^d$ are compatible.
 \end{enumerate}
Moreover, suppose {\rm (i), (ii)} hold. Then
the transition matrix from
 $\lbrace u_i \rbrace_{i=0}^d$ to 
 $\lbrace v_i \rbrace_{i=0}^d$ is the inverse of
 the transition matrix from
 $\lbrace u'_{d-i} \rbrace_{i=0}^d$ to 
 $\lbrace v'_{d-i} \rbrace_{i=0}^d$.
 \end{lemma}
\begin{proof}
Let $S \in
{\rm Mat}_{d+1}(\F)$ denote the transition matrix
from 
 $\lbrace u_i \rbrace_{i=0}^d$ to
 $\lbrace v_i \rbrace_{i=0}^d$.
Then $S^t$ is the transition matrix
from 
 $\lbrace v'_i \rbrace_{i=0}^d$ to
 $\lbrace u'_i \rbrace_{i=0}^d$.
Then $(S^t)^{-1}$ is the transition matrix
from 
 $\lbrace u'_i \rbrace_{i=0}^d$ to
 $\lbrace v'_i \rbrace_{i=0}^d$.
Then  ${\bf Z} (S^t)^{-1}{\bf Z}$ is the transition matrix
from 
 $\lbrace u'_{d-i} \rbrace_{i=0}^d$ to
 $\lbrace v'_{d-i} \rbrace_{i=0}^d$.
Note that $S$ is upper triangular and Toeplitz with all diagonal entries 1,
if and only if $S^{-1}$ is 
upper triangular and Toeplitz with all diagonal entries 1.
By this and
 Lemma
\ref{lem:ToepTrans},
we see 
that $S$ is upper triangular and Toeplitz with all diagonal entries 1,
if and only if 
 ${\bf Z} (S^t)^{-1}{\bf Z}$ 
is upper triangular and Toeplitz with all diagonal entries 1,
and in this case 
 ${\bf Z} (S^t)^{-1}{\bf Z}=S^{-1}$. 
The result follows.
\end{proof}

\begin{lemma}
\label{lem:ToeplitzDataD}
Let $A,B,C$ denote an LR triple on $V$, 
with Toeplitz data 
{\rm (\ref{eq:ToeplitzData})}.
In each row of the table below, we display
an LR triple on $V^*$ along with its Toeplitz data.

     \bigskip

\centerline{
\begin{tabular}[t]{c|c}
 {\rm LR triple} & {\rm Toeplitz data}
 \\
 \hline
 \hline
 $\tilde A, \tilde B, \tilde C$ &
$(
\lbrace \beta_i \rbrace_{i=0}^d,
\lbrace \alpha_i \rbrace_{i=0}^d;
\lbrace \beta'_i \rbrace_{i=0}^d,
\lbrace \alpha'_i \rbrace_{i=0}^d;
\lbrace \beta''_i \rbrace_{i=0}^d,
\lbrace \alpha''_i \rbrace_{i=0}^d
)$
   \\
 $\tilde B, \tilde C, \tilde A$ &
$(
\lbrace \beta'_i \rbrace_{i=0}^d,
\lbrace \alpha'_i \rbrace_{i=0}^d;
\lbrace \beta''_i \rbrace_{i=0}^d,
\lbrace \alpha''_i \rbrace_{i=0}^d;
\lbrace \beta_i \rbrace_{i=0}^d,
\lbrace \alpha_i \rbrace_{i=0}^d
)$
   \\
 $\tilde C, \tilde A, \tilde B$ &
$(
\lbrace \beta''_i \rbrace_{i=0}^d,
\lbrace \alpha''_i \rbrace_{i=0}^d;
\lbrace \beta_i \rbrace_{i=0}^d,
\lbrace \alpha_i \rbrace_{i=0}^d;
\lbrace \beta'_i \rbrace_{i=0}^d,
\lbrace \alpha'_i \rbrace_{i=0}^d
)$
   \\
 \hline
 $\tilde C, \tilde B, \tilde A$ &
$(
\lbrace \alpha'_i \rbrace_{i=0}^d,
\lbrace \beta'_i \rbrace_{i=0}^d;
\lbrace \alpha_i \rbrace_{i=0}^d,
\lbrace \beta_i \rbrace_{i=0}^d;
\lbrace \alpha''_i \rbrace_{i=0}^d,
\lbrace \beta''_i \rbrace_{i=0}^d
)$
   \\
 $\tilde A, \tilde C, \tilde B$ &
$(
\lbrace \alpha''_i \rbrace_{i=0}^d,
\lbrace \beta''_i \rbrace_{i=0}^d;
\lbrace \alpha'_i \rbrace_{i=0}^d,
\lbrace \beta'_i \rbrace_{i=0}^d;
\lbrace \alpha_i \rbrace_{i=0}^d,
\lbrace \beta_i \rbrace_{i=0}^d
)$
   \\
 $\tilde B, \tilde A, \tilde C$ &
$(
\lbrace \alpha_i \rbrace_{i=0}^d,
\lbrace \beta_i \rbrace_{i=0}^d;
\lbrace \alpha''_i \rbrace_{i=0}^d,
\lbrace \beta''_i \rbrace_{i=0}^d;
\lbrace \alpha'_i \rbrace_{i=0}^d,
\lbrace \beta'_i \rbrace_{i=0}^d
)$
\\
   \end{tabular}}
\bigskip

\end{lemma}
\begin{proof} Use 
Lemma \ref{lem:LRBdual},
Definition
\ref{def:TTT}, and
Lemma
\ref{lem:compatFlip}.
\end{proof}

\noindent 
Until further notice fix an LR triple
$A,B,C$ on $V$, with parameter array
(\ref{eq:paLRT}),
idempotent data
{\rm (\ref{eq:idseq})},
trace data
(\ref{eq:tracedata}), and Toeplitz data 
(\ref{eq:ToeplitzData}).
Recall the 12 types of bases for $V$
from
{\rm (\ref{eq:typeAB})--(\ref{eq:typeCA})}.
As we consider how
 these bases are related, 
it is convenient to work with specific bases of each type.
Fix nonzero vectors
\begin{eqnarray}
&&
 \eta \in A^dV, \qquad \qquad  \;\;
 \eta' \in B^dV, \qquad \qquad \;\;
 \eta'' \in C^dV,
\label{eq:eta}
\\
&&
\tilde \eta \in \tilde A^dV^*, \qquad \qquad
 \tilde \eta' \in \tilde B^dV^*, \qquad \qquad
 \tilde \eta'' \in \tilde C^dV^*.
\label{eq:etaDual}
\end{eqnarray}

\noindent By construction,
\begin{eqnarray}
&&
A \eta = 0, \qquad \qquad
B \eta' = 0, \qquad \qquad
C \eta'' = 0,
\label{eq:ABCzero}
\\
&&
\tilde A \tilde \eta = 0, \qquad \qquad
\tilde B \tilde \eta' = 0, \qquad \qquad
\tilde C \tilde \eta'' = 0.
\label{eq:TABCzero}
\end{eqnarray}

\noindent We mention a result for later use.
\begin{lemma}
\label{lem:lateruse}
The following scalars are nonzero:
\begin{eqnarray}
&&
\label{eq:IP1}
(\eta, \tilde \eta'),
\qquad \qquad
(\eta', \tilde \eta''),
\qquad \qquad
(\eta'', \tilde \eta),
\\
&&
\label{eq:IP2}
(\eta, \tilde \eta''),
\qquad \qquad
(\eta', \tilde \eta),
\qquad \qquad
(\eta'', \tilde \eta').
\end{eqnarray}
For $d\geq 1$ the following scalars are zero:
\begin{eqnarray}
(\eta, \tilde \eta),
\qquad \qquad
(\eta', \tilde \eta'),
\qquad \qquad
(\eta'', \tilde \eta'').
\label{eq:IP3}
\end{eqnarray}
\end{lemma}
\begin{proof} We show
that $(\eta, \tilde \eta')\not=0$. By assumption
$0 \not= \eta \in A^dV$ and
$0 \not= \tilde \eta' \in \tilde B^dV^*$.
The flags 
$\lbrace B^{d-i}V\rbrace_{i=0}^d$
and
$\lbrace \tilde B^{d-i}V^*\rbrace_{i=0}^d$
are dual by
Lemma \ref{lem:LRTflagDual};
therefore
$BV$ is the orthogonal complement of
$ \tilde B^dV^*$. 
The flags
$\lbrace A^{d-i}V\rbrace_{i=0}^d$
and 
$\lbrace B^{d-i}V\rbrace_{i=0}^d$ are opposite;
therefore 
$A^dV \cap BV=0$. By these comments
$(\eta, \tilde \eta')\not=0$.
The other five inner products in 
(\ref{eq:IP1}),
(\ref{eq:IP2})
are similarly shown to be nonzero.
Next assume that $d\geq 1$. We show that
$(\eta, \tilde \eta)=0$.
By construction
$\eta \in A^dV$ and
$\tilde \eta \in \tilde A^dV^*$.
The flags
$\lbrace A^{d-i}V\rbrace_{i=0}^d$ and
$\lbrace \tilde A^{d-i}V^*\rbrace_{i=0}^d$ are
dual by
Lemma \ref{lem:LRTflagDual};
 therefore
$AV$ is the orthogonal complement of 
$\tilde A^{d}V^*$.
The subspace $AV$ contains $A^dV$ since $d\geq 1$;
therefore 
$A^dV$  is orthogonal to $\tilde A^d V^*$.
By these comments $(\eta, \tilde \eta)=0$.
The other two inner products in 
(\ref{eq:IP3})
are similarly shown to be zero.
\end{proof}

\noindent We now display some  bases for $V$
of the types
(\ref{eq:typeAB})--(\ref{eq:typeCA}).

\begin{lemma}
\label{lem:Allbases}
In each row of the  tables below, for
$0 \leq i \leq d$ we display a vector $v_i \in V$. The
vectors $\lbrace v_i \rbrace_{i=0}^d$ form a basis for
$V$; we give the type and the induced decomposition of $V$.

     \bigskip

\centerline{
\begin{tabular}[t]{c|c|c}
 $v_i$ & {\rm type of basis} & {\rm induced dec. of $V$} 
 \\
 \hline
 \hline
$B^i \eta$ & {\rm inverted $(B,A)$} & $(A,B)$
\\
$B^{d-i} \eta$ & {\rm $(B,A)$} & $(B,A)$
\\
$(\varphi_1 \cdots \varphi_i)^{-1}B^i \eta$ & $(A,B)$ & $(A,B)$
\\
$(\varphi_1 \cdots \varphi_{d-i})^{-1}B^{d-i} \eta$ & {\rm inverted $(A,B)$} & $(B,A)$
\\
\hline
$C^i \eta$ & {\rm inverted $(C,A)$} & $(A,C)$
\\
$C^{d-i} \eta$ & {\rm $(C,A)$} & $(C,A)$
\\
$(\varphi''_d \cdots \varphi''_{d-i+1})^{-1}C^i \eta$ & $(A,C)$ & $(A,C)$
\\
$(\varphi''_d \cdots \varphi''_{i+1})^{-1}C^{d-i} \eta$ & {\rm inverted $(A,C)$} & $(C,A)$
\\
     \end{tabular}}

\bigskip

\centerline{
\begin{tabular}[t]{c|c|c}
 $v_i$ & {\rm type of basis} & {\rm induced dec. of $V$} 
 \\
 \hline
 \hline
$C^i \eta'$ & {\rm inverted $(C,B)$} & $(B,C)$
\\
$C^{d-i} \eta'$ & {\rm $(C,B)$} & $(C,B)$
\\
$(\varphi'_1 \cdots \varphi'_i)^{-1}C^i \eta'$ & $(B,C)$ & $(B,C)$
\\
$(\varphi'_1 \cdots \varphi'_{d-i})^{-1}C^{d-i} \eta'$ & {\rm inverted $(B,C)$}
& $(C,B)$
\\
\hline
$A^i \eta'$ & {\rm inverted $(A,B)$} & $(B,A)$
\\
$A^{d-i} \eta'$ & {\rm $(A,B)$} & $(A,B)$
\\
$(\varphi_d \cdots \varphi_{d-i+1})^{-1}A^i \eta'$ & $(B,A)$ & $(B,A)$
\\
$(\varphi_d \cdots \varphi_{i+1})^{-1}A^{d-i} \eta'$ & {\rm inverted $(B,A)$}
& $(A,B)$
\\
     \end{tabular}}

 \bigskip

\centerline{
\begin{tabular}[t]{c|c|c}
 $v_i$ & {\rm type of basis} & {\rm induced dec. of $V$} 
 \\
\hline
\hline
$A^i \eta''$ & {\rm inverted $(A,C)$} & $(C,A)$
\\
$A^{d-i} \eta''$ & {\rm $(A,C)$} & $(A,C)$
\\
$(\varphi''_1 \cdots \varphi''_i)^{-1}A^i \eta''$ & $(C,A)$ & $(C,A)$
\\
$(\varphi''_1 \cdots \varphi''_{d-i})^{-1}A^{d-i} \eta''$ &
{\rm inverted $(C,A)$}
& $(A,C)$
\\
\hline
$B^i \eta''$ & {\rm inverted $(B,C)$} & $(C,B)$
\\
$B^{d-i} \eta''$ & {\rm $(B,C)$} & $(B,C)$
\\
$(\varphi'_d \cdots \varphi'_{d-i+1})^{-1}B^i \eta''$ & $(C,B)$ & $(C,B)$
\\
$(\varphi'_d \cdots \varphi'_{i+1})^{-1}B^{d-i} \eta''$ & {\rm inverted $(C,B)$}
& $(B,C)$
\\
     \end{tabular}}
     \bigskip

\end{lemma}
\begin{proof}
Use Lemmas
\ref{lem:5Char},
\ref{lem:5CharInv},
\ref{lem:5CharBA},
\ref{lem:5CharBAinv}.
\end{proof}

\begin{lemma}
\label{lem:moreTTTp}
 In each of {\rm (i)--(iii)} below
we describe two bases from the tables
in
Lemma \ref{lem:Allbases}.  These two
bases are compatible.
\begin{enumerate}
\item[\rm (i)]
the $(A,B)$-basis in the first table, and the
$(A,C)$-basis in the first table;
\item[\rm (ii)] 
the $(B,C)$-basis in the second table, and the
$(B,A)$-basis in the second table;
\item[\rm (iii)]
the $(C,A)$-basis in the third table, and the
$(C,B)$-basis in the third table.
\end{enumerate}
\end{lemma}
\begin{proof} (i)
Let $\lbrace u_i \rbrace_{i=0}^d$ denote
the $(A,C)$-basis in the first table, and let
$\lbrace v_i \rbrace_{i=0}^d$ denote
 the
$(A,B)$-basis in the first table.
Let $S \in
{\rm Mat}_{d+1}(\F)$ denote the transition matrix
from
 $\lbrace u_i \rbrace_{i=0}^d$ to
$\lbrace v_i \rbrace_{i=0}^d$.
By Lemma
\ref{lem:xyz},
$S$ is upper triangular and Toeplitz;
let $\lbrace s_i \rbrace_{i=0}^d$ denote the corresponding
parameters.
Note that $u_0$ and $v_0$ are both equal to
$\eta$; therefore $s_0=1$.
 Consequently the diagonal entries
of $S$ are all 1.
The bases $\lbrace u_i \rbrace_{i=0}^d$ and
$\lbrace v_i \rbrace_{i=0}^d$ are compatible by
Definition
\ref{def:compat}.
\\
\noindent (ii), (iii) Similar to the proof of (i) above.
\end{proof}

\begin{lemma}
\label{lem:moreTTT}
Referring to Lemma
\ref{lem:Allbases},
\begin{enumerate}
\item[\rm (i)] $T$ is the transition matrix from
the $(C,B)$-basis in the third table, to the
$(C,A)$-basis in the third table;
\item[\rm (ii)] $T'$ is the transition matrix from
the $(A,C)$-basis in the first table, to the
$(A,B)$-basis in the first table;
\item[\rm (iii)] $T''$ is the transition matrix from
the $(B,A)$-basis in the second table, to the
$(B,C)$-basis in the second table.
\end{enumerate}
\end{lemma}
\begin{proof} By Definition 
\ref{def:TTT}
and 
Lemma
\ref{lem:moreTTTp}.
\end{proof}

\begin{lemma} 
\label{lem:TransI}
For $0 \leq j \leq d$,
\begin{eqnarray*}
&&
B^j \eta = \sum_{i=0}^j \alpha'_{j-i} \frac{\varphi_1 \cdots \varphi_j}
{\varphi''_d \cdots \varphi''_{d-i+1}}
\,C^i \eta,
\qquad \qquad
\;\;\;\;C^j \eta = \sum_{i=0}^j \beta'_{j-i} \frac{\varphi''_d \cdots \varphi''_{d-j+1}}
{\varphi_1 \cdots \varphi_{i}}
\,B^i \eta,
\\
&&
C^j \eta' = \sum_{i=0}^j \alpha''_{j-i} \frac{\varphi'_1 \cdots \varphi'_j}
{\varphi_d \cdots \varphi_{d-i+1}}
\,A^i \eta',
\qquad \qquad \;\;
A^j \eta' = \sum_{i=0}^j \beta''_{j-i} \frac{\varphi_d \cdots \varphi_{d-j+1}}
{\varphi'_1 \cdots \varphi'_{i}}
\,C^i \eta',
\\
&&
A^j \eta'' = \sum_{i=0}^j \alpha_{j-i} \frac{\varphi''_1 \cdots \varphi''_j}
{\varphi'_d \cdots \varphi'_{d-i+1}}
\,B^i \eta'',
\qquad \qquad
B^j \eta'' =
\sum_{i=0}^j \beta_{j-i} \frac{\varphi'_d \cdots \varphi'_{d-j+1}}
{\varphi''_1 \cdots \varphi''_{i}}
\,A^i \eta''.
\end{eqnarray*}
\end{lemma}
\begin{proof}
Use Lemmas
\ref{lem:Allbases},
\ref{lem:moreTTT}.
\end{proof}

\begin{lemma}
\label{lem:alphaSum}
We have
\begin{eqnarray}
&&
\label{eq:one}
\eta = \frac{(\eta,\tilde \eta'')}{(\eta',\tilde \eta'')}
\sum_{i=0}^d \alpha_i C^i \eta',
\qquad \qquad
\eta = \frac{(\eta,\tilde \eta')}{(\eta'',\tilde \eta')}
\sum_{i=0}^d \beta''_i B^i \eta'',
\\
&&
\label{eq:two}
\eta' = \frac{(\eta',\tilde \eta)}{(\eta'',\tilde \eta)}
\sum_{i=0}^d \alpha'_i A^i \eta'',
\qquad \qquad
\eta' = \frac{(\eta',\tilde \eta'')}{(\eta,\tilde \eta'')}
\sum_{i=0}^d \beta_i C^i \eta,
\\
&&
\label{eq:three}
\eta'' = \frac{(\eta'',\tilde \eta')}{(\eta,\tilde \eta')}
\sum_{i=0}^d \alpha''_i B^i \eta,
\qquad \qquad
\eta'' = \frac{(\eta'',\tilde \eta)}{(\eta',\tilde \eta)}
\sum_{i=0}^d \beta'_i A^i \eta'.
\end{eqnarray}
\end{lemma}
\begin{proof}
We verify the equation on the left in
(\ref{eq:one}). By Lemma
\ref{lem:Allbases}, 
$\lbrace C^{d-i}\eta'\rbrace_{i=0}^d$ is a $(C,B)$-basis
of $V$. Let $\lbrace v_i \rbrace_{i=0}^d$ denote a compatible
$(C,A)$-basis of $V$. By Definition
\ref{def:TTT}(i) $T$ is the transition matrix from
$\lbrace C^{d-i}\eta'\rbrace_{i=0}^d$ to
$\lbrace v_i \rbrace_{i=0}^d$.
By the discussion in
Definition
\ref{def:TTT1},
$T$ is upper triangular and Toeplitz with
parameters $\lbrace \alpha_i \rbrace_{i=0}^d$.
By these comments 
$v_d = \sum_{i=0}^d \alpha_i C^i\eta'$.
By construction $v_d$ is a basis of $A^dV$. Since
$\eta$ is also a basis of $A^dV$, there exists
$0 \not=\zeta \in \mathbb F$ such that
$\eta =  \zeta v_d$. Therefore
\begin{equation}
\label{eq:vdsum}
\eta = \zeta \sum_{i=0}^d \alpha_i C^i\eta'.
\end{equation}
We now compute $\zeta$.
In the equation
(\ref{eq:vdsum}),
take the inner
product of each side with $\tilde \eta''$.
By Lemma 
\ref{lem:LRTdecDual} (row 2 of the table) we have
$(C^i \eta', \tilde \eta'') = 0 $ for $1 \leq i \leq d$.
By this and $ \alpha_0=1$ we obtain
$(\eta,\tilde \eta'') =  \zeta (\eta',\tilde \eta'')$. Therefore
\begin{eqnarray}
\label{lem:zetaVal}
\zeta = \frac{
(\eta,\tilde \eta'')}
{(\eta',\tilde \eta'')}.
\end{eqnarray}
Combining
(\ref{eq:vdsum}),
(\ref{lem:zetaVal})
we obtain the equation on the left in
(\ref{eq:one}). The remaining equations in
(\ref{eq:one})--(\ref{eq:three})
 are similarly verified.
\end{proof}

\noindent In the next two lemmas we give some results
about $\alpha_1$. Similar results hold for $\alpha'_1, \alpha''_1$.

\begin{lemma} 
\label{lem:etapp} 
Assume  $d\geq 1$.
Then the vector $\eta''$ is
an eigenvector for 
$A/\varphi''_1- 
B/\varphi'_d$
with eigenvalue $\alpha_1$.
\end{lemma}
\begin{proof}
In Lemma \ref{lem:TransI}, set $j=1$
in either equation from the third row.
\end{proof}

\begin{lemma} 
\label{lem:matrixA}
Assume $d\geq 1$. Then the column vector
 $(\alpha''_0,\alpha''_1,\ldots, \alpha''_d)^t$ is
 an eigenvector for the matrix
\begin{equation}
\label{eq:matrixEig}
\left(
\begin{array}{ c c cc c c }
 0 & \varphi_1/\varphi''_1   &   &&   & \bf 0  \\
-1/\varphi'_d & 0  & \varphi_2/\varphi''_1  &&  &      \\ 
   & -1/\varphi'_d  &  0   & \cdot &&  \\
     &   & \cdot & \cdot  & \cdot & \\
       &  & & \cdot & \cdot & \varphi_d/\varphi''_1 \\
        {\bf 0}  &&  & & -1/\varphi'_d  &  0  \\
	\end{array}
	\right).
\end{equation}
The corresponding eigenvalue is $\alpha_1$.
\end{lemma}
\begin{proof}
In Lemma
\ref{lem:etapp}, represent everything
with respect to
$\lbrace B^i \eta\rbrace_{i=0}^d$,
which is an inverted $(B,A)$-basis of $V$.
In this calculation use
Proposition
\ref{prop:matrixRep}
and the equation on the left in
(\ref{eq:three}).
\end{proof}

\begin{lemma} The column vector
 $(\alpha''_0,\alpha''_1,\ldots, \alpha''_d)^t$ is
 an eigenvector for the matrix
\begin{equation}
\label{eq:matrixEig2}
\left(
\begin{array}{ c c cc c c }
 a_0 & \varphi'_d   &   &&   & \bf 0  \\
\varphi''_d/\varphi_1 & a_1  & \varphi'_{d-1}  &&  &      \\ 
   & \varphi''_{d-1}/\varphi_2  &  a_2   & \cdot &&  \\
     &   & \cdot & \cdot  & \cdot & \\
       &  & & \cdot & \cdot & \varphi'_1 \\
        {\bf 0}  &&  & & \varphi''_1/\varphi_d   &  a_d  \\
	\end{array}
	\right).
\end{equation}
The corresponding eigenvalue is $0$.
\end{lemma}
\begin{proof}
In the equation $C\eta''=0$, represent everything
with respect to
$\lbrace B^i \eta\rbrace_{i=0}^d$,
which is an inverted $(B,A)$-basis of $V$.
In this calculation use
Proposition
\ref{prop:matrixRep}
and the equation on the left in
(\ref{eq:three}).
\end{proof}

\begin{lemma}
\label{lem:prodFormula}
For $0 \leq i \leq d$,
\begin{eqnarray*}
&&
\varphi_1 \varphi_2 \cdots \varphi_i \alpha''_i \beta'_d = \beta'_{d-i},
\qquad \quad
\varphi'_1 \varphi'_2 \cdots \varphi'_i \alpha_i \beta''_d = \beta''_{d-i},
\qquad \quad
\varphi''_1 \varphi''_2 \cdots \varphi''_i \alpha'_i \beta_d = \beta_{d-i},
\\
&&
\varphi_1 \varphi_2 \cdots \varphi_i \alpha'_i \beta''_d = \beta''_{d-i},
\qquad \quad
\varphi'_1 \varphi'_2 \cdots \varphi'_i \alpha''_i \beta_d = \beta_{d-i},
\qquad \quad
\varphi''_1 \varphi''_2 \cdots \varphi''_i \alpha_i \beta'_d = \beta'_{d-i}.
\end{eqnarray*}
\end{lemma}
\begin{proof} We verify the first equation.
By Lemma
\ref{lem:TransI} with $j=d$,
\begin{eqnarray}
\label{eq:jd}
C^d \eta = \sum_{i=0}^d \beta'_{d-i}
\frac{\varphi''_1 \cdots \varphi''_{d}}
{\varphi_1 \cdots \varphi_{i}}
\,B^i \eta.
\end{eqnarray}
The vectors $C^d \eta $ and $\eta''$ are both bases for $C^d V$,
so
there exists $0 \not= \vartheta \in \mathbb F$ such that
$ C^d \eta =  \vartheta \eta''$. Use this to compare
(\ref{eq:jd}) with the equation on left in
(\ref{eq:three}).  We find that for $0 \leq i \leq d$,
\begin{eqnarray}
\label{eq:geni}
 \beta'_{d-i} \frac{\varphi''_1 \cdots \varphi''_d}{\varphi_1 \cdots 
\varphi_i} = \frac{(\eta'',\tilde \eta')}{
(\eta,\tilde \eta')} \vartheta \alpha''_i.
\end{eqnarray}
Setting $i=0$ in 
(\ref{eq:geni}),
\begin{eqnarray}
\label{eq:i0}
\beta'_d \varphi''_1 \cdots \varphi''_d
= \frac{(\eta'',\tilde \eta')}{
(\eta,\tilde \eta')} \vartheta.
\end{eqnarray}
Eliminating $\vartheta$ in
(\ref{eq:geni}) using 
(\ref{eq:i0}), we obtain the first equation 
in the lemma statement. Apply this equation to the p-relatives
of $A,B,C$ to get the
remaining
equations in the lemma statement.
\end{proof}

\begin{lemma} 
\label{lem:NZ}
\label{lem:identity}
The following scalars are nonzero:
\begin{eqnarray*}
\alpha_d, \qquad \alpha'_d, \qquad \alpha''_d,
\qquad \beta_d, \qquad \beta'_d, \qquad \beta''_d.
\end{eqnarray*}
Moreover
\begin{eqnarray*}
&&
\varphi_1 \varphi_2 \cdots \varphi_d = \frac{1}{
\alpha'_d \beta''_d} = 
\frac{1}{
\alpha''_d \beta'_d}, 
\qquad \qquad 
\varphi'_1 \varphi'_2 \cdots \varphi'_d = \frac{1}{
\alpha''_d \beta_d} = 
 \frac{1}{
\alpha_d \beta''_d}, 
\\
&&\varphi''_1 \varphi''_2 \cdots \varphi''_d = \frac{1}{
\alpha_d \beta'_d}
=
 \frac{1}{
\alpha'_d \beta_d}.
\end{eqnarray*}
\end{lemma}
\begin{proof} Set $i=d$ in Lemma
\ref{lem:prodFormula} and use Lemma
\ref{lem:alpha0}.
\end{proof}

\begin{lemma}
\label{lem:alphaInv}
For $0 \leq i \leq d$,
\begin{eqnarray}
\label{eq:alphaInv}
&&
\frac{\alpha_i}{\varphi_1 \varphi_2 \cdots \varphi_i}
= 
\frac{\alpha'_i}{\varphi'_1 \varphi'_2 \cdots \varphi'_i}
=
\frac{\alpha''_i}{\varphi''_1 \varphi''_2 \cdots \varphi''_i}
\end{eqnarray}
and also
\begin{eqnarray}
\frac{\beta_i}{\varphi_d \varphi_{d-1} \cdots \varphi_{d-i+1}}
= 
\frac{\beta'_i}{\varphi'_d \varphi'_{d-1} \cdots \varphi'_{d-i+1}}
=
\frac{\beta''_i}{\varphi''_d \varphi''_{d-1} \cdots \varphi''_{d-i+1}}.
\label{eq:betaInv}
\end{eqnarray}
\end{lemma}
\begin{proof} To obtain
(\ref{eq:alphaInv}), 
use Lemma
\ref{lem:prodFormula}
and the first assertion of Lemma
\ref{lem:NZ}. To obtain
(\ref{eq:betaInv}), apply
(\ref{eq:alphaInv}) to the LR triple
$\tilde A, \tilde B, \tilde C$ and use
Lemmas
\ref{lem:newPA},
\ref{lem:ToeplitzDataD}.
\end{proof}

\begin{lemma} 
\label{lem:COB}
For $0 \leq j \leq d$,
\begin{eqnarray*}
&&B^j \eta = \frac{1}{\beta''_d} \,
\frac{(\eta,\tilde \eta'')}{(\eta',\tilde \eta'')}
\,
\frac{A^{d-j}\eta'}{\varphi_d \cdots \varphi_{j+1}},
\qquad \quad
C^j \eta = \frac{1}{\alpha_d} \,
\frac{(\eta,\tilde \eta')}{(\eta'',\tilde \eta')}
\,
\frac{A^{d-j}\eta''}{\varphi''_1 \cdots \varphi''_{d-j}},
\\
&&C^j \eta' = \frac{1}{\beta_d} \,
\frac{(\eta',\tilde \eta)}{(\eta'',\tilde \eta)}
\,
\frac{B^{d-j}\eta''}{\varphi'_d \cdots \varphi'_{j+1}},
\qquad \quad
A^j \eta' = \frac{1}{\alpha'_d} \,
\frac{(\eta',\tilde \eta'')}{(\eta,\tilde \eta'')}
\,
\frac{B^{d-j}\eta}{\varphi_1 \cdots \varphi_{d-j}},
\\
&&A^j \eta'' = \frac{1}{\beta'_d} \,
\frac{(\eta'',\tilde \eta')}{(\eta,\tilde \eta')}
\,
\frac{C^{d-j}\eta}{\varphi''_d \cdots \varphi''_{j+1}},
\qquad \quad 
B^j \eta'' = \frac{1}{\alpha''_d} \,
\frac{(\eta'',\tilde \eta)}{(\eta',\tilde \eta)}
\,
\frac{C^{d-j}\eta'}{\varphi'_1 \cdots \varphi'_{d-j}}.
\end{eqnarray*}
\end{lemma}
\begin{proof}
We verify the first equation.
In 
Lemma
\ref{lem:Allbases}, 
 compare the $(A,B)$-basis of $V$ from
the first table, with the
 $(A,B)$-basis of $V$ from
the second table. By Lemma
\ref{lem:BAbasisU} there exists $0 \not=\zeta \in \mathbb F$
such that
\begin{eqnarray}
\frac{B^j \eta } {\varphi_1 \cdots \varphi_j} = \zeta 
A^{d-j} \eta' \qquad \qquad (0 \leq j \leq d).
\label{eq:compareA}
\end{eqnarray}
We now find $\zeta$. Setting $j=0$ in
(\ref{eq:compareA}) we find
$\eta = \zeta A^d \eta'$.
Use this 
to compare
the equation in Lemma
\ref{lem:TransI}
(row 2, column 2, $j=d$),
with the equation on the left in
 (\ref{eq:one}).
In this comparison consider the summands
 for $i=0$ to obtain
\begin{eqnarray}
\label{eq:compareA5}
\zeta = \frac{1}{\beta''_d}\,
\frac{1}{\varphi_1 \cdots \varphi_d}\,
\frac{
(\eta, \tilde \eta'')
}{
(\eta', \tilde \eta'')
}.
\end{eqnarray}
Evaluating 
(\ref{eq:compareA}) using
(\ref{eq:compareA5})  we get the first displayed equation 
in the lemma statement.
Applying our results so far to the LR triples
in Lemma
\ref{lem:ABCvar},
we obtain the remaining equations in the
 lemma statement.
\end{proof}

\noindent We emphasize a special case of Lemma
\ref{lem:COB}.

\begin{lemma}
\label{lem:COBSC}
We have
\begin{eqnarray*}
&&B^d \eta = 
\frac{(\eta,\tilde \eta'')}{(\eta',\tilde \eta'')}
\,
\frac{\eta'}{\beta''_d},
\qquad \quad
C^d \eta =
\frac{(\eta,\tilde \eta')}{(\eta'',\tilde \eta')}
\,
\frac{\eta''}{\alpha_d},
\\
&&C^d \eta' = 
\frac{(\eta',\tilde \eta)}{(\eta'',\tilde \eta)}
\,
\frac{\eta''}{\beta_d },
\qquad \quad
A^d \eta' =
\frac{(\eta',\tilde \eta'')}{(\eta,\tilde \eta'')}
\,
\frac{\eta}{\alpha'_d },
\\
&&A^d \eta'' =
\frac{(\eta'',\tilde \eta')}{(\eta,\tilde \eta')}
\,
\frac{\eta}{\beta'_d },
\qquad \quad 
B^d \eta'' =
\frac{(\eta'',\tilde \eta)}{(\eta',\tilde \eta)}
\,
\frac{\eta'}{\alpha''_d}.
\end{eqnarray*}
\end{lemma}
\begin{proof} Set $j=d$ in
Lemma
\ref{lem:COB}.
\end{proof}


\begin{proposition}
\label{prop:sixBil}
Each of 
$\alpha_d/\beta_d$,
$\alpha'_d/\beta'_d$,
$\alpha''_d/\beta''_d$ 
is equal to
\begin{eqnarray}
\frac{
(\eta,\tilde \eta')
(\eta',\tilde \eta'')
(\eta'',\tilde \eta)
}
{
(\eta,\tilde \eta'')
(\eta',\tilde \eta)
(\eta'',\tilde \eta')
}.
\label{eq:TripleProd}
\end{eqnarray}
\end{proposition}
\begin{proof}
The scalars 
$\alpha_d/\beta_d$,
$\alpha'_d/\beta'_d$,
$\alpha''_d/\beta''_d$ 
are equal by
Lemma
\ref{lem:NZ}.
To see that $\alpha_d /\beta_d$ is equal to
(\ref{eq:TripleProd}), compare the two
equations in 
(\ref{eq:one}) using
Lemma
\ref{lem:COB} (row 2, column 1). 
The result follows after a brief computation.
\end{proof}

\begin{note}\rm
By Proposition
\ref{prop:sixBil} the 
scalars in (\ref{eq:IP1}), 
(\ref{eq:IP2}) 
are determined by the Toeplitz data
(\ref{eq:ToeplitzData}) 
and the sequence
\begin{eqnarray}
\label{eq:5list}
(\eta, \tilde \eta'),
\qquad 
(\eta', \tilde \eta''),
\qquad 
(\eta'', \tilde \eta),
\qquad 
(\eta, \tilde \eta''),
\qquad 
(\eta', \tilde \eta).
\end{eqnarray}
The scalars
(\ref{eq:5list}) are ``free'' in the following sense.
Given a sequence $\Psi$ of five nonzero
scalars in $\mathbb F$, there exist nonzero vectors
$\eta, \eta',\eta''$ and
$\tilde \eta, \tilde \eta', \tilde \eta''$ 
as in
(\ref{eq:eta}),
(\ref{eq:etaDual})
such that the sequence 
(\ref{eq:5list}) is equal to $\Psi$.
\end{note}

\noindent We display some transition matrices for later use.

\begin{lemma} 
\label{lem:TI}
Referring to Lemma
\ref{lem:Allbases}, the following {\rm (i)--(iii)} hold.
\begin{enumerate}
\item[\rm (i)] The transition matrix from the
inverted $(B,A)$-basis in the first table to
the inverted $(B,A)$-basis in the second table is
\begin{eqnarray*}
 \frac{(\eta', \tilde \eta'')}{(\eta, \tilde \eta'')} \,\beta''_d I.
 \end{eqnarray*}
\item[\rm (ii)] The transition matrix from the
inverted $(C,B)$-basis in the second table to
the inverted $(C,B)$-basis in the third table is
\begin{eqnarray*}
 \frac{(\eta'', \tilde \eta)}{(\eta', \tilde \eta)} \,\beta_d I.
 \end{eqnarray*}
\item[\rm (iii)] The transition matrix from the
inverted $(A,C)$-basis in the third table to
the inverted $(A,C)$-basis in the first table is
\begin{eqnarray*}
 \frac{(\eta, \tilde \eta')}{(\eta'', \tilde \eta')} \,\beta'_d I.
 \end{eqnarray*}
\end{enumerate}
\end{lemma}
\begin{proof}
Use Lemma \ref{lem:COB}.
\end{proof}

\noindent The following definition is motivated by 
Definition \ref{def:Dmat}.

\begin{definition}
\label{def:DDD}
\rm Let $D$ (resp. $D'$) (resp. $D''$)
denote the diagonal matrix in 
${\rm Mat}_{d+1}(\F)$ with
$(i,i)$-entry
$\varphi_1 \varphi_2 \cdots \varphi_i$
(resp.
$\varphi'_1 \varphi'_2 \cdots \varphi'_i$)
(resp.
$\varphi''_1 \varphi''_2 \cdots \varphi''_i$)
for $0 \leq i \leq d$.
\end{definition}

\noindent 
The following result is reminiscent of Lemma
\ref{lem:Dmeaning}.

\begin{lemma}
\label{lem:DDD}
Referring to Lemma 
\ref{lem:Allbases},
\begin{enumerate}
\item[\rm (i)] $D$ is the transition matrix from the $(A,B)$-basis
in the first table to the inverted $(B,A)$-basis
in the first table;
\item[\rm (ii)] $D'$ is the transition matrix from the $(B,C)$-basis
in the second table to the inverted $(C,B)$-basis
in the second table;
\item[\rm (iii)] $D''$ is the transition matrix from the $(C,A)$-basis
in the third table to the inverted $(A,C)$-basis
in the third table.
\end{enumerate}
\end{lemma}
\begin{proof} By
Lemma
\ref{lem:Allbases} and
Definition
\ref{def:DDD}.
\end{proof}

\begin{definition}
\label{def:theta}
\rm
Let $\theta$ denote the scalar
(\ref{eq:TripleProd}).
By Proposition
\ref{prop:sixBil},
\begin{equation}
\label{eq:thetaEx}
\theta=
\frac{\alpha_d}{\beta_d}
=
\frac{\alpha'_d}{\beta'_d}
=
\frac{\alpha''_d}{\beta''_d}.
\end{equation}
Note that 
$\theta \not=0$.
\end{definition}

\begin{proposition}
\label{prop:12cycle}
We have
\begin{equation}
\label{eq:9cycle}
T' D {\bf Z} T''D'{\bf Z} T D'' {\bf Z} =\frac{I}{\theta \beta_d \beta'_d \beta''_d},
\end{equation}
where $\theta$ is from Definition
\ref{def:theta}.
\end{proposition}
\begin{proof}
Consider the following 12 bases from
Lemma
\ref{lem:Allbases}.
In each row of the table below, for
$0 \leq i \leq d$ we display a vector $v_i \in V$. The
vectors $\lbrace v_i \rbrace_{i=0}^d$ form a basis for
$V$; we give the type and the induced decomposition of $V$.

     \bigskip

\centerline{
\begin{tabular}[t]{c|c|c}
 $v_i$ & {\rm type of basis} & {\rm induced dec. of $V$} 
 \\
 \hline
 \hline
$(\varphi_1 \cdots \varphi_i)^{-1}B^i \eta$
&
$(A,B)$
&
$(A,B)$
\\
$B^i \eta$
&
{\rm inverted $(B,A)$}
&
$(A,B)$
\\
$(\varphi_d \cdots \varphi_{i+1})^{-1}A^{d-i} \eta'$
&
{\rm inverted $(B,A)$}
&
$(A,B)$
\\
$(\varphi_d \cdots \varphi_{d-i+1})^{-1}A^{i} \eta'$
&
$(B,A)$
&
$(B,A)$
\\
\hline
$(\varphi'_1 \cdots \varphi'_{i})^{-1}C^{i} \eta'$
&
$(B,C)$
&
$(B,C)$
\\
$C^{i} \eta'$
&
{\rm inverted $(C,B)$}
&
$(B,C)$
\\
$(\varphi'_d \cdots \varphi'_{i+1})^{-1}B^{d-i} \eta''$
&
{\rm inverted $(C,B)$}
&
$(B,C)$
\\
$(\varphi'_d \cdots \varphi'_{d-i+1})^{-1}B^{i} \eta''$
&
$(C,B)$
&
$(C,B)$
\\
\hline
$(\varphi''_1 \cdots \varphi''_{i})^{-1}A^{i} \eta''$
&
$(C,A)$
&
$(C,A)$
\\
$A^{i} \eta''$
&
{\rm inverted $(A,C)$}
&
$(C,A)$
\\
$(\varphi''_d \cdots \varphi''_{i+1})^{-1}C^{d-i} \eta$
&
{\rm inverted $(A,C)$}
&
$(C,A)$
\\
$(\varphi''_d \cdots \varphi''_{d-i+1})^{-1}C^{i} \eta$
&
$(A,C)$
&
$(A,C)$
   \\
     \end{tabular}}
     \bigskip

\noindent  We cycle through the bases in the above table,
starting with the basis in the bottom row, jumping
to the basis in the top row, and then going down through the
rows until we return to the basis in the bottom row.
For each basis in the sequence, consider the transition matrix
to the next basis in the sequence.
This gives a sequence of transition matrices.
Compute the product of these transition matrices in the given order.
This product is evaluated  in two ways.
On one hand, the product is equal to the identity matrix.
On the other hand, each factor in the product is computed 
using
Lemmas
\ref{lem:moreTTT},
\ref{lem:TI},
\ref{lem:DDD}
and the definition of $\bf Z$ in Section 2.
Evaluate the resulting equation using
 Proposition
\ref{prop:sixBil}.
The result follows.
\end{proof}

\noindent In Section 34 
we use the equation
(\ref{eq:9cycle}) to characterize
 the LR triples.
\medskip

\noindent We mention some other results involving
the scalar $\theta$ from Definition
\ref{def:theta}.

\begin{lemma}
\label{lem:ABCtrace}
We have
\begin{eqnarray}
{\rm tr}(A^dB^dC^d) =
\frac{\theta}{\alpha_d \alpha'_d \alpha''_d},
\qquad \qquad
{\rm tr}(C^d B^d A^d) =
\frac{1}{\theta \beta_d \beta'_d \beta''_d}.
\label{eq:abctrace}
\end{eqnarray}
\end{lemma}
\begin{proof}
We verify the equation on the left in
(\ref{eq:abctrace}). Let $\lbrace u_i\rbrace_{i=0}^d$
denote an $(A,C)$-basis of $V$, such that
 $u_0 = \eta$. Let $S$
denote the matrix
in
${\rm Mat}_{d+1}(\mathbb F)$ that represents
$A^dB^dC^d$ with respect to 
 $\lbrace u_i\rbrace_{i=0}^d$.
The map $C$ raises the
$(A,C)$-decomposition of $V$. Therefore
$C^d u_i = 0$ for $1 \leq i \leq d$.
By Lemma
\ref{lem:COBSC} and Definition
\ref{def:theta},
\begin{eqnarray*}
A^d B^d C^d \eta = \frac{\theta \eta}{\alpha_d \alpha'_d \alpha''_d}.
\end{eqnarray*} 
By these comments
$S$ has $(0,0)$-entry
$\theta/(\alpha_d \alpha'_d \alpha''_d)$ and all other
entries zero. We have verified the
 equation on the left in
(\ref{eq:abctrace}). The other equation is similarily
verified.
\end{proof} 

\begin{lemma}
\label{lem:thetaTrace}
 For $0 \leq i \leq d$, the trace of
$E'_{d-i}E_i E''_{d-i} E'_i E_{d-i} E''_i$ is
\begin{eqnarray*}
\theta \,\frac{\varphi_d \cdots \varphi_{d-i+1}}{\varphi_1 \cdots \varphi_i}
\, \frac{\varphi'_d \cdots \varphi'_{d-i+1}}{\varphi'_1 \cdots \varphi'_i}
\, \frac{\varphi''_d \cdots \varphi''_{d-i+1}}{\varphi''_1 \cdots \varphi''_i},
\end{eqnarray*}
and the trace of
$E_{d-i}E'_i E''_{d-i} E_i E'_{d-i} E''_i$ is
\begin{eqnarray*}
\frac{1}{\theta}
\, \frac{\varphi_1 \cdots \varphi_{i}}{\varphi_d \cdots \varphi_{d-i+1}}
\, \frac{\varphi'_1 \cdots \varphi'_{i}}{\varphi'_d \cdots \varphi'_{d-i+1}}
\, \frac{\varphi''_1 \cdots \varphi''_{i}}{\varphi''_d \cdots 
\varphi''_{d-i+1}}.
\end{eqnarray*}
\end{lemma}
\begin{proof}
We verify the first assertion.
In the product
$E'_{d-i}E_i E''_{d-i} E'_i E_{d-i} E''_i$,  evaluate each
factor using
Lemma \ref{lem:Eform2}, and
simplify
the result using
Lemma
\ref{lem:ABCtrace} along with the 
meaning of the parameter array.
The first assertion follows  after a brief computation.
The second assertion is similarly verified.
\end{proof}

\begin{corollary}
\label{cor:thetaint}
The trace of $E'_d E_0 E''_d E'_0 E_d E''_0$ is $\theta$.
The trace of $E_d E'_0 E''_d E_0 E'_d E''_0$ is $\theta^{-1}$.
\end{corollary}
\begin{proof} Set $i=0$ in Lemma
\ref{lem:thetaTrace}.
\end{proof}

\section{
How the parameter array, trace data, and 
Toeplitz data are related, I}

\noindent  Throughout this section
and the next,
let $V$ denote a vector space over $\mathbb F$ with
dimension $d+1$.
Fix an LR triple $A,B,C$ on
$V$. We consider how its
 parameter array
(\ref{eq:paLRT}),
trace data
(\ref{eq:tracedata}), and Toeplitz data
(\ref{eq:ToeplitzData}) are related.

\medskip
\noindent Recall
Definition
\ref{def:natural}.
Let $\lbrace u_i \rbrace_{i=0}^d $ denote a basis for
$V$ of type $\natural =(A,C)$.
 Let
$C^\natural \in
{\rm Mat}_{d+1}(\mathbb F)$ represent 
$C$ with respect to
$\lbrace u_i \rbrace_{i=0}^d $.
The entries of $C^\natural $ are given in Proposition
\ref{prop:matrixRep},
row $(A,C)$ of the table.
Let $\lbrace v_i \rbrace_{i=0}^d $ denote a compatible 
basis for $V$ of 
type $\sharp = (A,B)$.
Let
$C^\sharp \in
{\rm Mat}_{d+1}(\mathbb F)$ represent 
$C$ with respect to
$\lbrace v_i \rbrace_{i=0}^d $. 
The entries of $C^\sharp$ are given in
Proposition
\ref{prop:matrixRep},
row $(A,B)$ of the table.
By Definition
\ref{def:TTT}(ii),
$T'$ is the transition matrix from
$\lbrace u_i \rbrace_{i=0}^d$ to
$\lbrace v_i \rbrace_{i=0}^d$. By linear algebra,
\begin{equation}
\label{eq:NT=TS}
C^\natural T' = T' C^\sharp.
\end{equation}
Consequently
\begin{equation}
\label{eq:TiNT=S}
(T')^{-1}C^\natural T' = C^\sharp.
\end{equation}

\begin{proposition}
\label{lem:ai}
For $0 \leq i \leq d$,
\begin{eqnarray}
&&
a_{d-i}=\alpha'_0 \beta'_1 \varphi''_i +
\alpha'_1 \beta'_0 \varphi''_{i+1}
=
\alpha''_0 \beta''_1 \varphi'_i +
\alpha''_1 \beta''_0 \varphi'_{i+1},
\label{eq:aip}
\\
&&
a'_{d-i}= \alpha''_0 \beta''_1 \varphi_i +
\alpha''_1 \beta''_0 \varphi_{i+1}
=
\alpha_0 \beta_1 \varphi''_i +
\alpha_1 \beta_0 \varphi''_{i+1},
\label{eq:aipp}
\\
&&
a''_{d-i}= \alpha_0 \beta_1 \varphi'_i +
\alpha_1 \beta_0 \varphi'_{i+1}
=
\alpha'_0 \beta'_1 \varphi_i +
\alpha'_1 \beta'_0 \varphi_{i+1}.
\label{eq:ai}
\end{eqnarray}
\end{proposition}
\begin{proof} 
We verify the equation on the left in
(\ref{eq:aip}).
In the equation
(\ref{eq:TiNT=S}),
compute the $(d-i,d-i)$-entry of each side,
and evaluate
the result using
 Proposition
\ref{prop:matrixRep} and Definition
\ref{def:TTT}.
This yields the equation on the left in
(\ref{eq:aip}). To finish the proof,
apply this equation to the relatives of
$A,B,C$.
\end{proof}

\noindent We mention some variations on
Proposition
\ref{lem:ai}.

\begin{corollary}
\label{cor:aiVar}
For $0 \leq i \leq d$,
\begin{eqnarray*}
&& 
a_0 + a_{1} + \cdots + a_{d-i} = 
\beta'_1 \varphi''_i
= 
\beta''_1 \varphi'_i,
\\
&& 
a'_0 + a'_{1} + \cdots + a'_{d-i} = 
\beta''_1 \varphi_i
= 
\beta_1 \varphi''_i,
\\
&& 
a''_0 + a''_{1} + \cdots + a''_{d-i} = 
\beta_1 \varphi'_i
= 
\beta'_1 \varphi_i.
\end{eqnarray*}
\end{corollary}
\begin{proof}
To verify each equation,
evaluate the sum on the left
using
Proposition
\ref{lem:ai},
and simplify the result using 
Lemma
\ref{lem:alpha0}.
\end{proof}

\begin{corollary}
\label{cor:aiVar2}
For $0 \leq i \leq d$,
\begin{eqnarray*}
&& 
a_d + a_{d-1} + \cdots + a_{d-i} = 
\alpha'_1 \varphi''_{i+1}
= 
\alpha''_1 \varphi'_{i+1},
\\
&& 
a'_d + a'_{d-1} + \cdots + a'_{d-i} = 
\alpha''_1 \varphi_{i+1}
= 
\alpha_1 \varphi''_{i+1},
\\
&& 
a''_d + a''_{d-1} + \cdots + a''_{d-i} = 
\alpha_1 \varphi'_{i+1}
= 
\alpha'_1 \varphi_{i+1}.
\end{eqnarray*}
\end{corollary}
\begin{proof}
To verify each equation,
evaluate the sum on the left
using
Proposition
\ref{lem:ai},
and simplify the result using 
Lemma
\ref{lem:alpha0}.
\end{proof}

\begin{corollary}
\label{lem:a0ad}
We have
\begin{eqnarray*}
&&
a_0 = \beta'_1 \varphi''_d = \beta''_1 \varphi'_d,
\qquad \qquad 
a'_0 = \beta''_1 \varphi_d = \beta_1 \varphi''_d,
\qquad \qquad 
a''_0 = \beta_1 \varphi'_d = \beta'_1 \varphi_d,
\\
&&
a_d = \alpha'_1 \varphi''_1 = \alpha''_1 \varphi'_1,
\qquad \qquad 
a'_d = \alpha''_1 \varphi_1 = \alpha_1 \varphi''_1,
\qquad \qquad 
a''_d = \alpha_1 \varphi'_1 = \alpha'_1 \varphi_1.
\end{eqnarray*}
\end{corollary}
\begin{proof} Set $i=d$ in Corollary
\ref{cor:aiVar}, and
$i=0$ in Corollary
\ref{cor:aiVar2}. 
\end{proof}

\begin{corollary}
\label{cor:alphaOneInv}
For $1 \leq i \leq d$,
\begin{eqnarray}
\label{eq:alphaOneInv}
\frac{\alpha_1}{\varphi_i} = 
\frac{\alpha'_1}{\varphi'_i} = 
\frac{\alpha''_1}{\varphi''_i},
\qquad \qquad \qquad 
\frac{\beta_1}{\varphi_i} = 
\frac{\beta'_1}{\varphi'_i} = 
\frac{\beta''_1}{\varphi''_i}.
\end{eqnarray}
\end{corollary}
\begin{proof} Use Corollaries
\ref{cor:aiVar},
\ref{cor:aiVar2}.
\end{proof}

\begin{proposition}
\label{prop:goodrec}
For $1 \leq i \leq d$,
\begin{eqnarray*}
&&\frac{\varphi'_i}{\varphi''_{d-i+1}} = 
\alpha'_0 \beta'_2 \varphi_{i-1}+
\alpha'_1 \beta'_1 \varphi_{i}+
\alpha'_2 \beta'_0 \varphi_{i+1},
\\
&&\frac{\varphi''_i}{\varphi_{d-i+1}} = 
\alpha''_0 \beta''_2 \varphi'_{i-1}+
\alpha''_1 \beta''_1 \varphi'_{i}+
\alpha''_2 \beta''_0 \varphi'_{i+1},
\\
&&\frac{\varphi_i}{\varphi'_{d-i+1}} = 
\alpha_0 \beta_2 \varphi''_{i-1}+
\alpha_1 \beta_1 \varphi''_{i}+
\alpha_2 \beta_0 \varphi''_{i+1}
\end{eqnarray*}
and also
\begin{eqnarray*}
&&\frac{\varphi''_i}{\varphi'_{d-i+1}} = 
\alpha''_0 \beta''_2 \varphi_{i-1}+
\alpha''_1 \beta''_1 \varphi_{i}+
\alpha''_2 \beta''_0 \varphi_{i+1},
\\
&&\frac{\varphi_i}{\varphi''_{d-i+1}} = 
\alpha_0 \beta_2 \varphi'_{i-1}+
\alpha_1 \beta_1 \varphi'_{i}+
\alpha_2 \beta_0 \varphi'_{i+1},
\\
&&\frac{\varphi'_i}{\varphi_{d-i+1}} = 
\alpha'_0 \beta'_2 \varphi''_{i-1}+
\alpha'_1 \beta'_1 \varphi''_{i}+
\alpha'_2 \beta'_0 \varphi''_{i+1}.
\end{eqnarray*}
\end{proposition}
\begin{proof} 
We verify the last equation in the proposition statement.
In the equation
(\ref{eq:TiNT=S}),
compute the $(d-i,d-i+1)$-entry of each side,
and evaluate
the result using
 Proposition
\ref{prop:matrixRep}
and Definition \ref{def:TTT}.
This yields the last equation in the proposition statement.
To finish the proof,
apply this equation to the relatives of
$A,B,C$.
\end{proof}

\begin{proposition}
\label{prop:longrec}
For $3 \leq r \leq d+1$ and
$0 \leq i \leq d-r+1$,
\begin{eqnarray*}
&&
0 =
\alpha'_0 \beta'_r \varphi_i + 
\alpha'_1 \beta'_{r-1} \varphi_{i+1} + 
\cdots 
+
\alpha'_r \beta'_0 \varphi_{i+r},
\\
&&
0 =
\alpha''_0 \beta''_r \varphi'_i + 
\alpha''_1 \beta''_{r-1} \varphi'_{i+1} + 
\cdots 
+
\alpha''_r \beta''_0 \varphi'_{i+r},
\\
&&
0 =
\alpha_0 \beta_r \varphi''_i + 
\alpha_1 \beta_{r-1} \varphi''_{i+1} + 
\cdots 
+
\alpha_r \beta_0 \varphi''_{i+r}
\end{eqnarray*}
and also
\begin{eqnarray*}
&&
0 =
\alpha''_0 \beta''_r \varphi_i + 
\alpha''_1 \beta''_{r-1} \varphi_{i+1} + 
\cdots 
+
\alpha''_r \beta''_0 \varphi_{i+r},
\\
&&
0 =
\alpha_0 \beta_r \varphi'_i + 
\alpha_1 \beta_{r-1} \varphi'_{i+1} + 
\cdots 
+
\alpha_r \beta_0 \varphi'_{i+r},
\\
&&
0 =
\alpha'_0 \beta'_r \varphi''_i + 
\alpha'_1 \beta'_{r-1} \varphi''_{i+1} + 
\cdots 
+
\alpha'_r \beta'_0 \varphi''_{i+r}.
\end{eqnarray*}
\end{proposition}
\begin{proof}
We verify the last equation in the proposition statement.
In the equation
(\ref{eq:TiNT=S}),
compute the $(d-i-r+1,d-i)$-entry of each side,
and evaluate
the result using
 Proposition
\ref{prop:matrixRep}
and Definition \ref{def:TTT}.
This yields the last equation in the proposition statement.
To finish the proof,
apply this equation to the relatives of
$A,B,C$.
\end{proof}

\section{
How the parameter array, trace data, and 
Toeplitz data are related, II}

We continue to discuss our LR triple $A,B,C$ on
$V$, with
 parameter array
(\ref{eq:paLRT}),
trace data
(\ref{eq:tracedata}), and Toeplitz data
(\ref{eq:ToeplitzData}).
In the previous section we found a 
relationship among
 these scalars,
using the equation
(\ref{eq:TiNT=S}). In the present section
we describe this relationship 
from the point of view of
(\ref{eq:NT=TS}).

\begin{proposition} 
\label{prop:abI}
For $1\leq i\leq d$ and $0 \leq j \leq d-i$,
\begin{eqnarray*}
&&\alpha'_{i-1} \frac{\varphi'_{j+1}}{\varphi''_{d-j}} + \alpha'_i a''_{d-j} +
\alpha'_{i+1} \varphi_j = 
\alpha'_{i+1} \varphi_{i+j+1},
\\
&&\alpha''_{i-1} \frac{\varphi''_{j+1}}{\varphi_{d-j}} + \alpha''_i a_{d-j} +
\alpha''_{i+1} \varphi'_j = 
\alpha''_{i+1} \varphi'_{i+j+1},
\\
&&\alpha_{i-1} \frac{\varphi_{j+1}}{\varphi'_{d-j}} + \alpha_i a'_{d-j} +
\alpha_{i+1} \varphi''_j = 
\alpha_{i+1} \varphi''_{i+j+1}
\end{eqnarray*}
and also
\begin{eqnarray*}
&&\alpha''_{i-1} \frac{\varphi''_{j+1}}{\varphi'_{d-j}} + \alpha''_i a'_{d-j} +
\alpha''_{i+1} \varphi_j = 
\alpha''_{i+1} \varphi_{i+j+1},
\\
&&\alpha_{i-1} \frac{\varphi_{j+1}}{\varphi''_{d-j}} + \alpha_i a''_{d-j} +
\alpha_{i+1} \varphi'_j = 
\alpha_{i+1} \varphi'_{i+j+1},
\\
&&\alpha'_{i-1} \frac{\varphi'_{j+1}}{\varphi_{d-j}} + \alpha'_i a_{d-j} +
\alpha'_{i+1} \varphi''_j = 
\alpha'_{i+1} \varphi''_{i+j+1}.
\end{eqnarray*}
\end{proposition}
\begin{proof}
We verify the last equation in the proposition statement.
In the equation
(\ref{eq:NT=TS}),
compute the $(d-i-j,d-j)$-entry of each side,
and evaluate
the result using
 Proposition
\ref{prop:matrixRep} and Definition
\ref{def:TTT}.
This yields the last equation of the proposition
statement. To finish the proof,
apply this equation to the p-relatives of
$A,B,C$.
\end{proof}

\noindent We point out some special cases of Proposition
\ref{prop:abI}.

\begin{corollary}
\label{prop:abII}
For $1\leq i\leq d-1$,
\begin{eqnarray*}
 &&
 \alpha'_{i-1}\frac{\varphi'_{d-i+1}}{\varphi''_i} + \alpha'_i a''_{i} +
 \alpha'_{i+1} \varphi_{d-i} = 0,
\\
 &&
 \alpha''_{i-1}\frac{\varphi''_{d-i+1}}{\varphi_i} + \alpha''_i a_{i} +
 \alpha''_{i+1} \varphi'_{d-i} = 0,
\\
 &&
 \alpha_{i-1}\frac{\varphi_{d-i+1}}{\varphi'_i} + \alpha_i a'_{i} +
 \alpha_{i+1} \varphi''_{d-i} = 0
\end{eqnarray*}
and also
\begin{eqnarray*}
 &&
 \alpha''_{i-1}\frac{\varphi''_{d-i+1}}{\varphi'_i} + \alpha''_i a'_{i} +
 \alpha''_{i+1} \varphi_{d-i} = 0,
\\
 &&
 \alpha_{i-1}\frac{\varphi_{d-i+1}}{\varphi''_i} + \alpha_i a''_{i} +
 \alpha_{i+1} \varphi'_{d-i} = 0,
\\
 &&
 \alpha'_{i-1}\frac{\varphi'_{d-i+1}}{\varphi_i} + \alpha'_i a_{i} +
 \alpha'_{i+1} \varphi''_{d-i} = 0.
\end{eqnarray*}
\end{corollary}
\begin{proof}
In Proposition
\ref{prop:abI}, assume $i\leq d-1$ and
 $j=d-i$.
\end{proof}

\begin{corollary}
\label{prop:abIII}
For $1\leq i\leq d-1$,
\begin{eqnarray*}
&&
\alpha'_{i-1}\frac{\varphi'_1}{\varphi''_d} + \alpha'_i a''_d
= \alpha'_{i+1} \varphi_{i+1},
\qquad \quad 
\alpha''_{i-1}\frac{\varphi''_1}{\varphi'_d} + \alpha''_i a'_d
= \alpha''_{i+1} \varphi_{i+1},
\\
&&
\alpha''_{i-1}\frac{\varphi''_1}{\varphi_d} + \alpha''_i a_d
= \alpha''_{i+1} \varphi'_{i+1},
\qquad \quad 
\alpha_{i-1}\frac{\varphi_1}{\varphi''_d} + \alpha_i a''_d
= \alpha_{i+1} \varphi'_{i+1},
\\
&&
\alpha_{i-1}\frac{\varphi_1}{\varphi'_d} + \alpha_i a'_d
= \alpha_{i+1} \varphi''_{i+1},
\qquad \quad
\alpha'_{i-1}\frac{\varphi'_1}{\varphi_d} + \alpha'_i a_d
= \alpha'_{i+1} \varphi''_{i+1}.
\end{eqnarray*}
\end{corollary}
\begin{proof}
In 
 Proposition
\ref{prop:abI}, assume $i \leq d-1$ and
$j=0$.
\end{proof}

\begin{corollary}
\label{prop:abIV}
For $d\geq 1$,
\begin{eqnarray*}
&& \alpha'_{d-1} \frac{\varphi'_1}{\varphi''_d}+ \alpha'_d a''_d = 0,
\quad \qquad 
 \alpha''_{d-1} \frac{\varphi''_1}{\varphi_d}+ \alpha''_d a_d =  0,
 \quad \qquad
 \alpha_{d-1} \frac{\varphi_1}{\varphi'_d}+ \alpha_d a'_d =  0,
\\
&&
 \alpha''_{d-1} \frac{\varphi''_1}{\varphi'_d}+ \alpha''_d a'_d = 0,
\qquad \quad
 \alpha_{d-1} \frac{\varphi_1}{\varphi''_d}+ \alpha_d a''_d = 0,
\qquad \quad
\alpha'_{d-1} \frac{\varphi'_1}{\varphi_d}+ \alpha'_d a_d= 0.
\end{eqnarray*}
\end{corollary}
\begin{proof}
In Proposition
\ref{prop:abI}, assume
$i=d$ and $j=0$.
\end{proof}

\begin{proposition}
\label{prop:baI}
For $1 \leq i \leq d$ and $0 \leq j\leq d-i$,

\begin{eqnarray*}
&&\beta'_{i-1} \frac{\varphi'_{d-j}}{\varphi''_{j+1}} + \beta'_i a''_{j} +
\beta'_{i+1} \varphi_{d-j+1} = 
\beta'_{i+1} \varphi_{d-i-j},
\\
&&\beta''_{i-1} \frac{\varphi''_{d-j}}{\varphi_{j+1}} + \beta''_i a_{j} +
\beta''_{i+1} \varphi'_{d-j+1} = 
\beta''_{i+1} \varphi'_{d-i-j},
\\
&&\beta_{i-1} \frac{\varphi_{d-j}}{\varphi'_{j+1}} + \beta_i a'_{j} +
\beta_{i+1} \varphi''_{d-j+1} = 
\beta_{i+1} \varphi''_{d-i-j}
\end{eqnarray*}
and also
\begin{eqnarray*}
&&\beta''_{i-1} \frac{\varphi''_{d-j}}{\varphi'_{j+1}} + \beta''_i a'_{j} +
\beta''_{i+1} \varphi_{d-j+1} = 
\beta''_{i+1} \varphi_{d-i-j},
\\
&&\beta_{i-1} \frac{\varphi_{d-j}}{\varphi''_{j+1}} + \beta_i a''_{j} +
\beta_{i+1} \varphi'_{d-j+1} = 
\beta_{i+1} \varphi'_{d-i-j},
\\
&&\beta'_{i-1} \frac{\varphi'_{d-j}}{\varphi_{j+1}} + \beta'_i a_{j} +
\beta'_{i+1} \varphi''_{d-j+1} = 
\beta'_{i+1} \varphi''_{d-i-j}.
\end{eqnarray*}
\end{proposition}
\begin{proof} Apply Proposition
\ref{prop:abI}
to the LR triple $\tilde A, \tilde B, \tilde C$.
\end{proof}

\begin{corollary}
For $1\leq i\leq d-1$,
\begin{eqnarray*}
 &&
 \beta'_{i-1}\frac{\varphi'_{i}}{\varphi''_{d-i+1}} + \beta'_i a''_{d-i} +
 \beta'_{i+1} \varphi_{i+1} = 0,
\\
 &&
 \beta''_{i-1}\frac{\varphi''_{i}}{\varphi_{d-i+1}} + \beta''_i a_{d-i} +
 \beta''_{i+1} \varphi'_{i+1} = 0,
\\
 &&
 \beta_{i-1}\frac{\varphi_{i}}{\varphi'_{d-i+1}} + \beta_i a'_{d-i} +
 \beta_{i+1} \varphi''_{i+1} = 0
\end{eqnarray*}
and also
\begin{eqnarray*}
 &&
 \beta''_{i-1}\frac{\varphi''_{i}}{\varphi'_{d-i+1}} + \beta''_i a'_{d-i} +
 \beta''_{i+1} \varphi_{i+1} = 0,
\\
 &&
 \beta_{i-1}\frac{\varphi_{i}}{\varphi''_{d-i+1}} + \beta_i a''_{d-i} +
 \beta_{i+1} \varphi'_{i+1} = 0,
\\ 
 &&
 \beta'_{i-1}\frac{\varphi'_{i}}{\varphi_{d-i+1}} + \beta'_i a_{d-i} +
 \beta'_{i+1} \varphi''_{i+1} = 0.
\end{eqnarray*}
\end{corollary}
\begin{proof} 
In Proposition
\ref{prop:baI}, assume
$i \leq d-1$ and
$j=d-i$.
\end{proof}

\begin{corollary}
 For $1\leq i\leq d-1$,
\begin{eqnarray*}
&&
\beta'_{i-1}\frac{\varphi'_d}{\varphi''_1} + \beta'_i a''_0
= \beta'_{i+1} \varphi_{d-i},
\qquad \quad 
\beta''_{i-1}\frac{\varphi''_d}{\varphi'_1} + \beta''_i a'_0
= \beta''_{i+1} \varphi_{d-i},
\\
&&
\beta''_{i-1}\frac{\varphi''_d}{\varphi_1} + \beta''_i a_0
= \beta''_{i+1} \varphi'_{d-i},
\qquad \quad 
\beta_{i-1}\frac{\varphi_d}{\varphi''_1} + \beta_i a''_0
= \beta_{i+1} \varphi'_{d-i},
\\
&&
\beta_{i-1}\frac{\varphi_d}{\varphi'_1} + \beta_i a'_0
= \beta_{i+1} \varphi''_{d-i},
\qquad \quad
\beta'_{i-1}\frac{\varphi'_d}{\varphi_1} + \beta'_i a_0
= \beta'_{i+1} \varphi''_{d-i}.
\end{eqnarray*}
\end{corollary}
\begin{proof} 
In Proposition
\ref{prop:baI}, assume
$i \leq d-1$ and
$j=0$.
\end{proof}

\begin{corollary}
For $d\geq 1$,
\begin{eqnarray*}
&& \beta'_{d-1} \frac{\varphi'_d}{\varphi''_1}+ \beta'_d a''_0 = 0,
\quad \qquad 
 \beta''_{d-1} \frac{\varphi''_d}{\varphi_1}+ \beta''_d a_0 =  0,
 \quad \qquad
 \beta_{d-1} \frac{\varphi_d}{\varphi'_1}+ \beta_d a'_0 =  0,
\\
&&
 \beta''_{d-1} \frac{\varphi''_d}{\varphi'_1}+ \beta''_d a'_0 = 0,
\qquad \quad
 \beta_{d-1} \frac{\varphi_d}{\varphi''_1}+ \beta_d a''_0 = 0,
\qquad \quad
\beta'_{d-1} \frac{\varphi'_d}{\varphi_1}+ \beta'_d a_0 = 0.
\end{eqnarray*}
\end{corollary}
\begin{proof}
In Proposition
\ref{prop:baI}, assume
$i =d$ and
$j=0$.
\end{proof}

\noindent We have displayed many equations
relating the
 parameter array
(\ref{eq:paLRT}),
trace data
(\ref{eq:tracedata}), and Toeplitz data
(\ref{eq:ToeplitzData}).
From these equations
it is apparent that we can improve on Proposition 
\ref{prop:IsoParTrace}.
We now give some results in this direction. To avoid
trivialities we assume $d\geq 1$.

\begin{proposition}
\label{prop:extra}
Assume $d\geq 1$. Then the LR triple $A,B,C$ is
uniquely determined up to isomorphism
by its parameter
array along with any one of the following 12 scalars:
\begin{eqnarray}
\label{eq:12values}
a_0, a'_0, a''_0;
\quad \qquad 
a_d, a'_d, a''_d;
\quad \qquad 
\alpha_1, \alpha'_1, \alpha''_1;
\quad \qquad 
\beta_1, \beta'_1, \beta''_1.
\end{eqnarray}
\end{proposition}
\begin{proof} Use
Proposition
\ref{prop:IsoParTrace} along with
(\ref{eq:list3}),
Proposition
\ref{lem:ai},
and Corollary
\ref{lem:a0ad}.
\end{proof}

\noindent In our discussion going forward,
among the scalars 
(\ref{eq:12values})
we will put the emphasis on
$\alpha_1$. We call $\alpha_1$ the {\it first Toeplitz number}
of the LR triple $A,B,C$.

\begin{lemma}
\label{lem:SpittingField}
Assume $d\geq 1$.
For the LR triple $A,B,C$ let $\mathbb K$ denote a
subfield of $\mathbb F$ that contains
the scalars
{\rm (\ref{eq:paLRT})} and the first Toeplitz number
$\alpha_1$.
Then there exists an LR triple over 
$\mathbb K$ that has parameter array
{\rm (\ref{eq:paLRT})} and first Toeplitz number
$\alpha_1$.
\end{lemma}
\begin{proof} Represent $A,B,C$ by matrices, using
 the first row in the table of
Proposition
\ref{prop:matrixRep}. For the resulting three matrices
each entry
is in $\mathbb K$. So each matrix represents
a $\mathbb K$-linear transformation of a vector space 
over $\mathbb K$. The resulting
three $\mathbb K$-linear transformations form
an LR triple over $\mathbb K$
 that has parameter array
(\ref{eq:paLRT}) and first Toeplitz number
$\alpha_1$.
\end{proof}

\begin{lemma}
\label{lem:whatisIso}
For the LR triple $A,B,C$ the following are equivalent:
\begin{enumerate}
\item[\rm (i)] $\varphi_i = \varphi'_i = \varphi''_i
$ for $1 \leq i \leq d$;
\item[\rm (ii)] the p-relatives of $A,B,C$ are mutually
isomorphic;
\item[\rm (iii)] the n-relatives of $A,B,C$ are mutually
isomorphic.
\end{enumerate}
\noindent Assume that {\rm (i)--(iii)} hold. Then for
$0 \leq i \leq d$,
\begin{eqnarray}
\label{eq:aaaCom}
a_i = a'_i = a''_i, 
\qquad \qquad
\alpha_i = \alpha'_i = \alpha''_i, 
\qquad \qquad
\beta_i = \beta'_i = \beta''_i.
\end{eqnarray}
\end{lemma}
\begin{proof} Assume $d\geq 1$; otherwise (i)--(iii)
and
(\ref{eq:aaaCom}) 
all hold.
\\
\noindent 
${\rm (i)}\Rightarrow {\rm (ii)}$  
We have $\alpha_1 = \alpha'_1=\alpha''_1$ by
Corollary
\ref{cor:alphaOneInv}.
The result follows by
Proposition
\ref{prop:extra}, along with
Lemmas
\ref{lem:ABCvar},
\ref{lem:newPA}
and
Definition
\ref{def:prel}.
\\
\noindent 
${\rm (ii)}\Rightarrow {\rm (i)}$ By 
Lemmas
\ref{lem:ABCvar},
\ref{lem:newPA}
and
Definition
\ref{def:prel}.
\\
\noindent 
${\rm (i)}\Leftrightarrow {\rm (iii)}$ Similar
to the proof of
${\rm (i)}\Leftrightarrow {\rm (ii)}$ above.
\\
Assume that (i)--(iii) hold. Then
(\ref{eq:aaaCom}) holds by
Lemmas
\ref{lem:tracedataAlt},
\ref{lem:tracedataDual},
\ref{lem:ToeplitzData},
\ref{lem:ToeplitzDataD}.
\end{proof}

\noindent
We now compute the Toeplitz data
(\ref{eq:ToeplitzData}) 
in terms of the  
 parameter array
(\ref{eq:paLRT}) and any scalar
from
(\ref{eq:12values}).
 We will focus on
 $\lbrace \alpha_i\rbrace_{i=0}^d$  and
 $\lbrace \beta_i\rbrace_{i=0}^d$.

\begin{proposition} 
\label{prop:AlphaRecursion}
For $d\geq 1$ the following {\rm (i), (ii)} hold.
\begin{enumerate}
\item[\rm (i)]
The sequence $\lbrace \alpha_i\rbrace_{i=0}^d$ 
is computed as follows: 
$\alpha_0=1$ and $\alpha_1$ is from
Corollary
\ref{lem:a0ad}. Moreover
\begin{eqnarray}
&&\alpha_{i+1} = \frac{
\alpha_1 \alpha_{i}\varphi''_1 + \alpha_{i-1}
\varphi_1 (\varphi'_d)^{-1}}{\varphi''_{i+1}}
\qquad \qquad 
(1 \leq i \leq d-1).
\label{eq:alphaRec}
\end{eqnarray}
\item[\rm (ii)]
The sequence $\lbrace \beta_i\rbrace_{i=0}^d$ is
computed as follows:
$\beta_0=1$ and $\beta_1$ is from
Corollary 
\ref{lem:a0ad}. Moreover
\begin{eqnarray}
&&\beta_{i+1} = \frac{
\beta_1 \beta_{i}\varphi''_d + \beta_{i-1}
\varphi_d (\varphi'_1)^{-1}}{\varphi''_{d-i}}
\qquad \qquad (1 \leq i \leq d-1).
\label{eq:betaRec}
\end{eqnarray}
\end{enumerate}
\end{proposition}
\begin{proof} (i) We verify
(\ref{eq:alphaRec}). Consider the displayed equation
in Corollary
\ref{prop:abIII} (row 3, column 1). In this equation
solve for $\alpha_{i+1}$, 
and eliminate $a'_d$ using the equation
$a'_d = \alpha_1 \varphi''_1$ from Corollary
\ref{lem:a0ad}.
\\
\noindent (ii) Similar to the proof of (i) above.
\end{proof}

\noindent We now give some more ways to compute
 $\lbrace \alpha_i\rbrace_{i=0}^d$  and
 $\lbrace \beta_i\rbrace_{i=0}^d$.

\begin{proposition} 
\label{prop:AlphaRecursion2}
For $d\geq 1$ the following {\rm (i), (ii)} hold.
\begin{enumerate}
\item[\rm (i)]
The sequence $\lbrace \alpha_i\rbrace_{i=0}^d$ 
is computed as follows: 
$\alpha_0=1$ and $\alpha_1$ is from
Corollary
\ref{lem:a0ad}. Moreover
\begin{eqnarray}
&&\alpha_{i+1} = \frac{
\alpha_1 \alpha_{i}(\varphi''_{d-i}-\varphi''_{d-i+1})
- \alpha_{i-1}
\varphi_{d-i+1}(\varphi'_{i})^{-1}}{\varphi''_{d-i}}
 \qquad 
(1 \leq i \leq d-1).
\label{eq:alphaRec2}
\end{eqnarray}
\item[\rm (ii)]
The sequence $\lbrace \beta_i\rbrace_{i=0}^d$ is
computed as follows:
$\beta_0=1$ and $\beta_1$ is from
Corollary 
\ref{lem:a0ad}. Moreover
\begin{eqnarray}
&&\beta_{i+1} = \frac{
\beta_1 \beta_{i}(\varphi''_{i+1}-\varphi''_{i})
- \beta_{i-1}
\varphi_{i}(\varphi'_{d-i+1})^{-1}}{\varphi''_{i+1}}
 \qquad 
(1 \leq i \leq d-1).
\label{eq:betaRec2}
\end{eqnarray}
\end{enumerate}
\end{proposition}
\begin{proof} (i) We verify
(\ref{eq:alphaRec2}). Consider the displayed equation
in Corollary
\ref{prop:abII} (row 3). In this equation
solve for $\alpha_{i+1}$, 
and eliminate $a'_i$ using the equation
$a'_i = \alpha_1 (\varphi''_{d-i+1}-\varphi''_{d-i})$
from Proposition
\ref{lem:ai}.
\\
\noindent (ii) Similar to the proof of (i) above.
\end{proof}

\begin{proposition} 
\label{prop:AlphaRecursion3}
For $d\geq 1$ the following {\rm (i)--(iv)} hold:
\begin{enumerate}
\item[\rm (i)] for $1 \leq i \leq d-1$,
\begin{eqnarray*}
&&\alpha_1 \alpha_{i} 
\biggl(
1-\frac{\varphi''_1}{\varphi''_{i+1}}
- 
\frac{\varphi''_{d-i+1}}{\varphi''_{d-i}}
\biggr)
=
\alpha_{i-1}
\biggl(
\frac{\varphi_1}{\varphi'_d \varphi''_{i+1}}
+ 
\frac{\varphi_{d-i+1}}{\varphi'_i\varphi''_{d-i}}
\biggr);
\end{eqnarray*}
\item[\rm (ii)]
$\alpha_1 \alpha_d \varphi''_1 = - \alpha_{d-1} \varphi_1 /\varphi'_d$;
\item[\rm (iii)]
for $1 \leq i \leq d-1$,
\begin{eqnarray*}
&&\beta_1 \beta_{i} 
\biggl(
1-\frac{\varphi''_d}{\varphi''_{d-i}}
- 
\frac{\varphi''_{i}}{\varphi''_{i+1}}
\biggr)
=
\beta_{i-1}
\biggl(
\frac{\varphi_d}{\varphi'_1 \varphi''_{d-i}}
+ 
\frac{\varphi_{i}}{\varphi'_{d-i+1}\varphi''_{i+1}}
\biggr);
\end{eqnarray*}
\item[\rm (iv)]
$\beta_1 \beta_d \varphi''_d = 
- \beta_{d-1} \varphi_d /\varphi'_1$.
\end{enumerate}
\end{proposition}
\begin{proof} (i) 
Subtract 
(\ref{eq:alphaRec2}) 
from
(\ref{eq:alphaRec}) 
and simplify the result.
\\
\noindent 
(ii) In the displayed equation of Corollary 
\ref{prop:abIV} (row 1, column 3) eliminate
$a'_d$ using the equation
$a'_d = \alpha_1 \varphi''_1$ from
Corollary
\ref{lem:a0ad}.
\\
\noindent (iii), (iv) Similar to the proof of (i), (ii) above.
\end{proof}

\begin{note}
\label{note:alphaOne}
\rm Referring to Proposition
\ref{prop:extra}, if we replace the LR triple
$A,B,C$ by the LR triple $-A,-B,-C$
then the parameter array is unchanged, and
each scalar in 
(\ref{eq:12values})
is replaced by its opposite.
So in general, the LR triple $A,B,C$ is not determined
up to isomorphism by its parameter array.
\end{note}

\noindent Referring to Proposition
\ref{prop:extra} and in light of
Note \ref{note:alphaOne},
 we now consider the extent to which
$\alpha^2_1$ 
is determined by the parameter array
(\ref{eq:paLRT}). 

\begin{lemma}
\label{lem:alphaOneOK}
For $d\geq 1$
the scalar $\alpha^2_1$
is related to the parameter array 
{\rm (\ref{eq:paLRT})} 
in the following way.
\begin{enumerate}
\item[\rm (i)]
Assume $d=1$. Then
\begin{eqnarray}
\label{eq:alphaD1OK}
\alpha^2_1 = - \frac{\varphi_1}{\varphi'_1 \varphi''_1}.
\end{eqnarray}
\item[\rm (ii)]
Assume $d=2$. Then $\alpha_1=0$ or
\begin{eqnarray}
\label{eq:alphaD2OK}
\alpha^2_1 = 
- \frac{\varphi_1+\varphi_2}{\varphi'_2 \varphi''_1}.
\end{eqnarray}
\item[\rm (iii)]
Assume $d\geq 2$. Then
\begin{eqnarray}
\label{eq:DG2OK}
\alpha^2_1\biggl(
1-
\frac{\varphi''_1}{\varphi''_{2}} 
-
\frac{\varphi''_d}{\varphi''_{d-1}}
\biggr)
=
\frac{\varphi_d}{\varphi''_{d-1}}
\frac{1}{\varphi'_1}
+
\frac{\varphi_1}{\varphi''_{2}}
\frac{1}{\varphi'_d}.
\end{eqnarray}
Moreover 
for $1 \leq i \leq d$,
\begin{eqnarray}
\alpha^2_1\biggl(
\frac{\varphi''_d}{\varphi''_{d-1}} \varphi''_{i-1}
- \varphi''_i
+
\frac{\varphi''_1}{\varphi''_{2}} \varphi''_{i+1}
\biggr)
=
\frac{\varphi_i}{\varphi'_{d-i+1}}
-
\frac{\varphi_d}{\varphi''_{d-1}}
\frac{\varphi''_{i-1}}{\varphi'_1}
-
\frac{\varphi_1}{\varphi''_{2}}
\frac{\varphi''_{i+1}}{\varphi'_d}.
\label{eq:dLarge}
\end{eqnarray}
\end{enumerate}
\end{lemma}
\begin{proof} (i), (ii) Compute the eigenvalues of
the matrix
(\ref{eq:matrixEig}).
\\
\noindent (iii)
Using Proposition
\ref{prop:AlphaRecursion},
solve for
$\alpha_2$, $\beta_2$ in terms of
$\alpha_1$ and the parameter array
(\ref{eq:paLRT}).
To obtain (\ref{eq:DG2OK}), use
the above solutions and
$\beta_2 = \alpha^2_1 - \alpha_2$.
To obtain
(\ref{eq:dLarge}), use
the above solutions and
the third
displayed equation in Proposition
\ref{prop:goodrec}.
\end{proof}

\noindent Referring to Lemma
\ref{lem:alphaOneOK}(iii),
it sometimes happens that in each equation
(\ref{eq:DG2OK}),
(\ref{eq:dLarge}) 
the coefficient of $\alpha^2_1$ is zero.
We illustrate with two examples.

\begin{definition}\rm
\label{def:WEYL}
The LR triple $A,B,C$ is said to have {\it Weyl type}
whenever the LR pairs $A,B$ and $B,C$ and $C,A$
all have Weyl type, in the sense of
Definition
\ref{def:Weyl}.
In this case, $p=d+1$ is prime and
${\rm Char}(\mathbb F)=p$. Moreover
\begin{eqnarray}
&&AB-BA=I, \qquad \qquad 
BC-CB=I, \qquad\qquad 
CA-AC=I,
\label{eq:tripleWeyl}
\\
&&
\label{eq:eqeq}
 \qquad \qquad \qquad 
\varphi_i = \varphi'_i= \varphi''_i = i
\qquad \qquad (1 \leq i \leq d).
\end{eqnarray}
\end{definition}

\begin{lemma}
Assume that the LR triple $A,B,C$ has Weyl type.
Then each p-relative of $A,B,C$ has Weyl type.
\end{lemma}
\begin{proof} By
Lemma
\ref{lem:whatisIso}
and (\ref{eq:eqeq}).
\end{proof}

\begin{lemma}
\label{lem:except2WEYL}
Assume that $d\geq 2$ and
$A,B,C$ has Weyl type.
Then in
each of {\rm (\ref{eq:DG2OK})},
{\rm (\ref{eq:dLarge})}  the coefficient of 
$\alpha^2_1$ is zero. Moreover the right-hand side
is zero.
\end{lemma}
\begin{proof} This is readily checked using
(\ref{eq:eqeq}).
\end{proof}

\noindent 
Assume that $A,B,C$ has Weyl type.
Then equations
(\ref{eq:DG2OK}),
(\ref{eq:dLarge}) give no information about
$\alpha_1$. 
To compute 
$\alpha_1$ we use the following result.

\begin{lemma}
\label{lem:ABCsum}
Assume that $A,B,C$ has Weyl type. Then
\begin{eqnarray}
A + B + C = \alpha_1 I.
\label{eq:ABCsum}
\end{eqnarray}
\end{lemma}
\begin{proof}
Represent $A,B,C$ by matrices, using
for example the first row in the table
of Proposition
\ref{prop:matrixRep}. By 
Proposition \ref{lem:ai} we have
$a_i = \alpha_1 $
for $0  \leq i \leq d$.
\end{proof}

\begin{lemma}
\label{lem:WEYLalphaZero}
Assume that    $A,B,C$ has Weyl type.
Then $\alpha_1=1$ if $d=1$,
and 
$\alpha_1=0$ if $d\geq 2$.
\end{lemma}
\begin{proof} Recall from Definition
\ref{def:WEYL}
 that $p=d+1$ is
prime and
${\rm Char}(\mathbb F)=p$. 
First assume that $d=1$. Then
by Lemma
\ref{lem:alphaOneOK}(i) 
and since 
${\rm Char}(\mathbb F)=2$,
we obtain
$\alpha^2_1=1$. Again using
${\rm Char}(\mathbb F)=2$ we find
$\alpha_1=1 $.
Next assume that $d\geq 2$.
By Lemma \ref{lem:ABCsum}, 
$C-\alpha_1 I = -A-B$.
On one hand, the pair $B,C$ is an LR pair
on $V$, so $C$ is Nil by
Lemma \ref{lem:ABdecIndNil}.
On the other hand, by Lemma
\ref{lem:curious} 
the pair $B,-A-B$ is an LR pair on $V$,
so $-A-B$ is Nil by Lemma
\ref{lem:ABdecIndNil}.
By these comments, 
both $C$ and $C-\alpha_1 I$  are Nil.
Considering their eigenvalues we obtain
$\alpha_1=0$.
\end{proof}

\begin{lemma}
\label{lem:alphabetaWeyl}
Assume that $d\geq 2$ and $A,B,C$ has Weyl type. Then
\begin{eqnarray*}
\alpha_{2i}= \frac{(-1)^i}{ 2^{i}i!},
\qquad \qquad 
\beta_{2i}= \frac{1}{ 2^{i}i!}
\qquad \qquad 
(0 \leq i \leq d/2).
\end{eqnarray*}
Also $\alpha_{2i+1}=0$ and $\beta_{2i+1}= 0$ for $0 \leq i<d/2$.
\end{lemma}
\begin{proof}
Use
Proposition \ref{prop:AlphaRecursion}
and Lemma
\ref{lem:WEYLalphaZero}.
\end{proof}

\begin{proposition}
\label{prop:Weyclass}
The following are equivalent:
\begin{enumerate}
\item[\rm (i)] 
$p=d+1$ is prime and
${\rm Char}(\mathbb F)=p$;
\item[\rm (ii)] 
 there exists
an LR triple $A,B,C$ over $\mathbb F$ that
has diameter $d$ and Weyl type.
\end{enumerate}
Assume that {\rm (i)}, {\rm (ii)} hold.
Then $A,B,C$ is unique up to isomorphism.
\end{proposition}
\begin{proof} 
${\rm (i)}\Rightarrow {\rm (ii)}$ 
By Lemma
\ref{ex:Weylback},
there exists an LR pair $A,B$ over $\mathbb F$
that has diameter $d$ and
 Weyl type. Define $C=I-A-B$ if $d=1$, and
$C=-A-B$ if $d\geq 2$.
For $d=1$ one routinely verifies
that $A,B,C$ is an LR triple of Weyl type.
Assume that $d\geq 2$. By
Lemma
\ref{lem:curious}
and Definition
\ref{def:WEYL}
the sequence $A,B,C$ is an LR triple of Weyl type.
\\
\noindent ${\rm (ii)}\Rightarrow {\rm (i)}$ 
By Definition
\ref{def:WEYL}.
\\
\noindent Assume that (i), (ii) hold.
The LR triple $A,B,C$ is unique  up to isomorphism
by Proposition
\ref{prop:extra},
line
(\ref{eq:eqeq}), and
Lemma
\ref{lem:WEYLalphaZero}.
\end{proof}

\noindent We continue to discuss our LR triple $A,B,C$ on
$V$, with parameter array
(\ref{eq:paLRT}),
trace data
(\ref{eq:tracedata}), and Toeplitz data
(\ref{eq:ToeplitzData}).

\begin{definition}\rm
\label{def:qExceptional}
Pick a nonzero $q \in \mathbb F$ such that
$q^2\not=1$.
The LR triple $A,B,C$ is said to have {\it $q$-Weyl type}
whenever the LR pairs $A,B$ and $B,C$ and $C,A$
all have $q$-Weyl type, in the sense of
Definition
\ref{def:qWeyl}.
In this case $d,q$ or $d,-q$ is standard. Moreover 
\begin{eqnarray}
\label{eq:WWW}
&&\frac{qAB-q^{-1}BA}{q-q^{-1}}=I,
\qquad \quad
\frac{qBC-q^{-1}CB}{q-q^{-1}}=I,
\qquad \quad
\frac{qCA-q^{-1}AC}{q-q^{-1}}=I,
\\
&& \qquad \qquad \qquad 
\varphi_i = \varphi'_i= \varphi''_i = 1-q^{-2i} 
\qquad \qquad (1 \leq i \leq d).
\label{eq:vvv}
\end{eqnarray}
\end{definition}

\begin{lemma}
With reference to Definition
\ref{def:qExceptional}, assume that
$A,B,C$ has $q$-Weyl type.
Then $A,B,C$ has $(-q)$-Weyl type.
\end{lemma}
\begin{proof}
Use
Note \ref{note:minusq}.
\end{proof}

\begin{lemma}
\label{lem:qiWeyl}
With reference to Definition
\ref{def:qExceptional}, assume that
$A,B,C$ has $q$-Weyl type.
Then each p-relative of $A,B,C$ has $q$-Weyl type.
Moreover each n-relative of $A,B,C$ has $(q^{-1})$-Weyl type.
\end{lemma}
\begin{proof} By
Note 
\ref{note:minusq}
and
Definition
\ref{def:prel}.
\end{proof}

\begin{lemma}
\label{lem:except2}
With reference to Definition
\ref{def:qExceptional}, assume that
 $d\geq 2$ and
$A,B,C$ has $q$-Weyl type.
Then in
each of {\rm (\ref{eq:DG2OK})},
{\rm (\ref{eq:dLarge})}  the coefficient of 
$\alpha^2_1$ is zero. Moreover the right-hand side
is zero.
\end{lemma}
\begin{proof} This is readily checked.
\end{proof}

\noindent 
With reference to Definition
\ref{def:qExceptional}, assume that
 $d\geq 2$ and
$A,B,C$ has $q$-Weyl type.
Then
(\ref{eq:DG2OK}),
(\ref{eq:dLarge}) give no information about
$\alpha_1$. 
To get some information about
$\alpha_1$ we turn to
 Lemma
\ref{lem:etapp}. 

\begin{lemma}
With reference to Definition
\ref{def:qExceptional}, assume that
$A,B,C$ has $q$-Weyl type.
Then 
\begin{eqnarray}
A/\varphi_1-B/\varphi_d = \frac{qA+q^{-1}B}{q-q^{-1}}.
\label{eq:ABcheck}
\end{eqnarray}
\end{lemma}
\begin{proof}
Use
(\ref{eq:vvv}).
\end{proof}

\noindent 
Recall Assumption \ref{def:sqRoot}.

\begin{lemma}
\label{lem:alphaOneList}
With reference to Assumption
\ref{def:sqRoot} and Definition
\ref{def:qExceptional}, 
assume that $A,B,C$ has $q$-Weyl type.
Then for  
the element {\rm (\ref{eq:ABcheck})}
the roots of the characteristic polynomial
are
\begin{eqnarray}
\frac{q^{j+1/2}+q^{-j-1/2}}{q-q^{-1}}
\qquad \qquad (0 \leq j \leq d).
\label{eq:thetaCalc2}
\end{eqnarray}
Moreover $\alpha_1$ is contained in the list
{\rm (\ref{eq:thetaCalc2})}.
\end{lemma}
\begin{proof} 
The first assertion follows from
 Lemma
\ref{lem:qAqiB}.
The second assertion follows from
Lemma
\ref{lem:etapp}
and
(\ref{eq:vvv}).
\end{proof}

\noindent We now consider which values 
of 
(\ref{eq:thetaCalc2}) could equal $\alpha_1$.

\begin{lemma}
\label{lem:sixeqs}
With reference to
Definition
\ref{def:qExceptional}, 
assume that $A,B,C$ has $q$-Weyl type. Then 
$\alpha_1 (q-q^{-1})I$
is equal to each of
\begin{eqnarray*}
&&
q A + q^{-1}B + qC - qABC, \qquad \qquad q^{-1}A+qB+q^{-1}C- q^{-1}CBA,
\\
&&
q B + q^{-1}C + qA - qBCA, \qquad \qquad q^{-1}B+qC+q^{-1}A- q^{-1}ACB,
\\
&&
q C + q^{-1}A + q B - qCAB, \qquad \qquad q^{-1}C+qA+q^{-1}B- q^{-1}BAC.
\end{eqnarray*}
\end{lemma}
\begin{proof}
Represent $A,B,C$ by matrices, using for
example the first row in the table of 
Proposition
\ref{prop:matrixRep}. By Proposition
\ref{lem:ai}
we have $a_i = \alpha_1 q^{2i+1}(q-q^{-1})$ for $0 \leq i \leq d$.
\end{proof}

\begin{proposition} 
\label{prop:AllqWeyl}
Assume that $\mathbb F$ is
algebraically closed.
With reference to Assumption
\ref{def:sqRoot},
pick an integer $j$ $(0 \leq j \leq d)$ and
define
\begin{eqnarray}
\alpha_1= \frac{q^{j+1/2} + q^{-j-1/2}}{q-q^{-1}}.
\label{eq:alphafix}
\end{eqnarray}
Then there exists an LR triple over $\mathbb F$ that has
diameter $d$ and $q$-Weyl type, with first Toeplitz number
$\alpha_1$.
This LR triple is uniquely determined up to isomorphism
by $d,q,j$.
For this LR triple,
\begin{eqnarray}
\label{eq:alphaiBasic}
&&
\alpha_i = \frac{(-1)^i q^{-i/2}}{(q^{-1};q^{-1})_i}\;
 {}_3\phi_2 \Biggl(
\genfrac{}{}{0pt}{}
 {q^{i}, \;-q^{-j-1},\;-q^{j}}
 {0, \;\;-q^{-1}}
 \;\Bigg\vert \; q^{-1},\;q^{-1}\Biggr)
\qquad \qquad (0 \leq i \leq d),
\\
&&
\beta_i = \frac{(-1)^i q^{i/2}}{(q;q)_i}\;
 {}_3\phi_2 \Biggl(
\genfrac{}{}{0pt}{}
 {q^{-i}, \;-q^{j+1},\;-q^{-j}}
 {0, \;\;-q}
 \;\Bigg\vert \; q,\;q\Biggr)
\label{eq:betaiBasic}
\qquad \qquad (0 \leq i \leq d).
 \end{eqnarray}
\end{proposition}
\begin{proof}
By Lemma \ref{ex:exceptional}
there exists an LR pair $A,B$ on $V$
that has $q$-Weyl type.
Its parameter sequence $\lbrace \varphi_i \rbrace_{i=1}^d$
satisfies $\varphi_i = 1-q^{-2i}$ for $1 \leq i \leq d$.
With respect to an $(A,B)$-basis of $V$ the matrices representing
$A,B$ are given as shown in the first row of the table in
Proposition
\ref{prop:matrixRep}. 
Define $C\in {\rm End}(V)$ such that
with respect to the $(A,B)$-basis,
the matrix representing $C$ is given as shown in the 
first row of the table, using
$a_i = \alpha_1 q^{2i+1}(q-q^{-1})$
for $0 \leq i \leq d$ and $\varphi'_i = \varphi''_i = \varphi_i$
for $1 \leq i \leq d$.
We show that $A,B,C$ is an LR triple on $V$ that has $q$-Weyl type.
To do this, it suffices to show that $B,C$ and $C,A$ are
LR pairs on $V$ that have
$q$-Weyl type.
We now show that 
$B,C$ is an LR pair on $V$ that has $q$-Weyl type.
To this end we apply Lemma
\ref{lem:qWeylABextend}
to the pair $B,C$.
From the matrix representions 
we see that $B,C$ satisfy the middle equation in
(\ref{eq:WWW}). The map $B$ is not invertible, since $B$
is Nil by Lemma
\ref{lem:ABdecIndNil}. We show that $C$ is not invertible.
From the matrix representations we obtain
\begin{eqnarray}
\label{eq:tocheck}
\alpha_1 (q-q^{-1})I = qA+q^{-1}B + qC -qABC.
\end{eqnarray}
Rearranging 
(\ref{eq:tocheck}),
\begin{eqnarray}
\alpha_1 (q-q^{-1})I - qA - q^{-1}B = q(1-AB)C.
\label{eq:BAC}
\end{eqnarray}
By assumption
(\ref{eq:alphafix})
along with
Definition
\ref{def:thetaI}
and Lemma
\ref{lem:qAqiB},  
$\alpha_1(q-q^{-1})$ is an
eigenvalue for $qA+q^{-1}B$.
So in the equation
(\ref{eq:BAC}) the expression on the left is not
invertible.
Therefore $(1-AB)C$ 
is not invertible.
Note that $I-AB$ is diagonalizable with eigenvalues
$\lbrace 1-\varphi_{i+1}\rbrace_{i=0}^d$.
Moreover $1-\varphi_{i+1} = q^{-2i-2}\not=0$ for
$0 \leq i \leq d$.
Therefore $I-AB$ is invertible. By these comments $C$
is not invertible.
Applying Lemma
\ref{lem:qWeylABextend}
to the pair $B,C$ we find that $B,C$ is an LR pair
on $V$ that has $q$-Weyl type.
One similarly shows that
$C,A$  is an LR pair on $V$ that has $q$-Weyl type.
Now by Definition
\ref{def:qExceptional} the triple
$A,B,C$ is an LR triple on $V$
that has $q$-Weyl type.
Comparing
(\ref{eq:tocheck}) with the first expression
in the display of Lemma
\ref{lem:sixeqs},
 we see that $A,B,C$ has first Toeplitz
number 
$\alpha_1$.
We have displayed an LR triple over $\mathbb F$
that has diameter $d$ and $q$-Weyl type, with
first Toeplitz number $\alpha_1$.
This LR triple is unique up to isomorphism
by Proposition
\ref{prop:extra} and line
(\ref{eq:vvv}).
To obtain  (\ref{eq:alphaiBasic}) use
the eigenvector assertion in Lemma
\ref{lem:bMatrixPre} along with 
Lemma
\ref{lem:matrixA}.
To obtain  (\ref{eq:betaiBasic}),
apply
 (\ref{eq:alphaiBasic}) to any n-relative of $A,B,C$
and use 
Lemma
\ref{lem:qiWeyl}.
\end{proof}

\section{Bipartite LR triples}

\noindent 
Throughout this section the following notation is in effect.
Let $V$ denote a vector space over $\mathbb F$ with
dimension $d+1$.
Let $A,B,C$ denote
an LR triple on
$V$, with
 parameter array
(\ref{eq:paLRT}),
idempotent data
(\ref{eq:idseq}),
trace data
(\ref{eq:tracedata}), and Toeplitz data
(\ref{eq:ToeplitzData}).
We describe a condition on $A,B,C$ called bipartite.

\begin{definition}
\label{def:LRTbip}
\rm 
The LR triple $A,B,C$  is called
{\it bipartite} whenever
each of $a_i,
a'_i,
a''_i$ is zero for 
$0 \leq i \leq d$.
\end{definition}

\begin{lemma}
\label{lem:TrivBip}
If $A,B,C$ is trivial then it is
 bipartite. 
\end{lemma}
\begin{proof} 
Set $d=0$ in Lemma
\ref{lem:aisum} to 
see that each of
$a_0, a'_0, a''_0$ is zero.
\end{proof}

\begin{lemma} 
\label{lem:bipRel}
Assume
that $A,B,C$ is bipartite (resp. nonbipartite). Then
each relative of 
$A,B,C$ is bipartite (resp. nonbipartite). 
\end{lemma}
\begin{proof} Use Lemmas
\ref{lem:tracedataAlt},
\ref{lem:tracedataDual}.
\end{proof}

\begin{lemma}
\label{lem:BPabc}
Assume
that $A,B,C$ is bipartite (resp. nonbipartite). Let
$\alpha, \beta, \gamma$ denote nonzero scalars in
$\mathbb F$. Then the LR triple
$\alpha A, \beta B, \gamma C$ is bipartite (resp. nonbipartite).
\end{lemma}
\begin{proof}
Use Lemma
\ref{lem:traceAdj}.
\end{proof}

\begin{lemma}
\label{lem:NotB}
Assume that $A,B,C$ is nonbipartite. Then $d\geq 1$. Moreover
each of
\begin{eqnarray}
\label{eq:NotBip}
\alpha_1, \qquad
\alpha'_1, \qquad
\alpha''_1, \qquad
\beta_1, \qquad
\beta'_1, \qquad
\beta''_1
\end{eqnarray}
is nonzero.
\end{lemma}
\begin{proof}
We have $d\geq 1 $  by
Lemma
\ref{lem:TrivBip}.
We show $\alpha_1 \not= 0 $.
Suppose $\alpha_1 = 0$.
By Corollary
\ref{cor:alphaOneInv} we obtain
$\alpha'_1= 0$
and 
$\alpha''_1= 0$.
Observe that $\beta_1 = -\alpha_1= 0$.
Similarly 
$\beta'_1= 0$
and 
$\beta''_1= 0$. Now
by Proposition
\ref{lem:ai}
each of $a_i, a'_i, a''_i$ is zero
for $0 \leq i \leq d$.
Now $A,B,C$ is bipartite, for a contradiction.
We have shown that $\alpha_1 \not=0$; by
Lemma
\ref{lem:bipRel}
the other
scalars in
(\ref{eq:NotBip}) are nonzero.
\end{proof}

\begin{lemma}
\label{lem:case}
Assume that $A,B,C$ is bipartite. Then $d$ is even.
Moreover
for  $0 \leq i \leq d$, 
each of
\begin{eqnarray}
\label{eq:BipZero}
\alpha_i, \qquad
\alpha'_i, \qquad
\alpha''_i, \qquad
\beta_i, \qquad
\beta'_i, \qquad
\beta''_i
\end{eqnarray}
is zero if $i$ is odd and nonzero if $i$ is even.
\end{lemma}
\begin{proof}
We claim that $\alpha_i$
is zero if $i$ is odd and nonzero if $i$ is even.
We prove the claim by induction on $i$.
The claim holds for $i=0$, since $\alpha_0=1$.
The claim holds for $i=1$, since
 $\alpha_1=0$ by
Corollary
\ref{lem:a0ad}
and our assumption that $A,B,C$ is bipartite.
The claim holds for $2 \leq i \leq d$
by Proposition
\ref{prop:AlphaRecursion}(i)
and induction. The claim is proven.
By Lemma \ref{lem:bipRel}
the other scalars in
(\ref{eq:BipZero}) are zero if $i$ is odd and
nonzero if $i$ is even.
The diameter $d$ must be even by the
first assertion of Lemma
\ref{lem:NZ}.
\end{proof}


\noindent As we continue to discuss LR triples, we will often treat
the bipartite and nonbipartite cases separately.
For the next few results, we consider the nonbipartite case.

\begin{lemma}
\label{lem:alphaiAlpha1}
Assume that $A,B,C$ is nonbipartite. Then
for $0 \leq i \leq d$,
\begin{eqnarray}
\label{ex:alphaiAlpha1}
\frac{\alpha_i}{\alpha^i_1}
=
\frac{\alpha'_i}{(\alpha'_1)^i}
=
\frac{\alpha''_i}{(\alpha''_1)^i},
\qquad \qquad \qquad
\frac{\beta_i}{\beta^i_1}
=
\frac{\beta'_i}{(\beta'_1)^i}
=
\frac{\beta''_i}{(\beta''_1)^i}.
\end{eqnarray}
\end{lemma}
\begin{proof} By Lemma
\ref{lem:alphaInv}
and Corollary
\ref{cor:alphaOneInv}.
\end{proof}

\begin{lemma}
Assume that $A,B,C$ is nonbipartite. Let $\alpha, \beta, \gamma$
denote nonzero scalars in $\mathbb F$. Then the following are
equivalent:
\begin{enumerate}
\item[\rm (i)] 
the LR triples $A,B,C$ and $\alpha A,\beta B, \gamma C$ are isomorphic;
\item[\rm (ii)] 
$\alpha = \beta = \gamma =1$.
\end{enumerate}
\end{lemma}
\begin{proof} Use
Lemma \ref{lem:ToeplitzAdjust}.
\end{proof}

\noindent We turn our attention to bipartite
LR triples.

\begin{lemma}
\label{lem:BiPbasic}
Assume that $A,B,C$ is bipartite and nontrivial. Then
\begin{enumerate}
\item[\rm (i)] 
 $\alpha_1, \alpha'_1, \alpha''_1 $ and
 $\beta_1, \beta'_1, \beta''_1 $ are all zero;
\item[\rm (ii)] 
 $\alpha_2, \alpha'_2, \alpha''_2 $ and
 $\beta_2, \beta'_2, \beta''_2 $ are all nonzero; 
\item[\rm (iii)] We have 
\begin{eqnarray}
\beta_2 = - \alpha_2, \qquad\qquad
\beta'_2 = - \alpha'_2, \qquad\qquad
\beta''_2 = - \alpha''_2.
\label{eq:Al2Be2}
\end{eqnarray}
\end{enumerate}
\end{lemma}
\begin{proof}
(i), (ii)
By Lemma
\ref{lem:case}. 
\\
\noindent (iii)
From above
Lemma \ref{lem:reform1}.
\end{proof}

\begin{lemma}
\label{lem:UNIQUE}
A bipartite LR triple is uniquely determined up to isomorphism
by its parameter array.
\end{lemma}
\begin{proof}
By Proposition
\ref{prop:IsoParTrace}
and Definition
\ref{def:LRTbip}.
\end{proof}

\begin{lemma}
Assume that $A,B,C$ is bipartite. Let $\alpha, \beta, \gamma$
denote nonzero scalars in $\mathbb F$. Then the following are
equivalent:
\begin{enumerate}
\item[\rm (i)] 
the LR triples $A,B,C$ and $\alpha A,\beta B, \gamma C$ are isomorphic;
\item[\rm (ii)] 
$\alpha = \beta = \gamma \in \lbrace 1,-1\rbrace$.
\end{enumerate}
\end{lemma}
\begin{proof} Use
Lemmas
\ref{lem:albegaCor},
\ref{lem:UNIQUE}.
\end{proof}

\begin{lemma}
\label{lem:V0V1}
Assume that $A,B,C$ is bipartite,
so that  $d=2 m$ is even.
\begin{enumerate}
\item[\rm (i)]
The following subspaces are equal:
\begin{equation}
\label{eq:3verOut}
\sum_{j=0}^m E_{2j} V,
\qquad \quad 
 \sum_{j=0}^m E'_{2j} V,
\qquad \quad 
\sum_{j=0}^m E''_{2j} V.
\end{equation}
\item[\rm (ii)]
The following subspaces are equal:
\begin{equation}
\label{eq:3verIn}
\sum_{j=0}^{m-1} E_{2j+1} V,
\qquad \quad 
\sum_{j=0}^{m-1} E'_{2j+1} V,
\qquad \quad 
\sum_{j=0}^{m-1} E''_{2j+1} V.
\end{equation}
\item[\rm (iii)] Let 
 $V_{\rm out}$ and $V_{\rm in}$ denote
the common values of
{\rm (\ref{eq:3verOut})} and
{\rm (\ref{eq:3verIn})}, respectively.
Then 
\begin{equation}
\label{eq:V0V1}
 V= V_{\rm out}+V_{\rm in} \qquad \quad  {\mbox{\rm (direct sum).}} 
\end{equation}
\item[\rm (iv)] We have
\begin{equation}
 {\rm dim}(V_{\rm out}) = m+1,
\qquad \qquad 
 {\rm dim}(V_{\rm in}) = m.
\label{eq:V0V1dim}
\end{equation}
\end{enumerate}
\end{lemma}
\begin{proof} (i) 
Denote the sequence in
(\ref{eq:3verOut})
by $U$, $U'$, $U''$.
We show
$U' = U$.
The sequence
$\lbrace E_iV\rbrace_{i=0}^d$ 
is the
$(A,B)$-decomposition of $V$. Therefore
$\lbrace E_{d-i}V\rbrace_{i=0}^d$ 
is the
$(B,A)$-decomposition of $V$. 
The sequence
$\lbrace E'_{i}V\rbrace_{i=0}^d$ 
is the
$(B,C)$-decomposition of $V$. 
Let $\lbrace u_i\rbrace_{i=0}^d$ denote a $(B,A)$-basis for
$V$.
Let $\lbrace v_i\rbrace_{i=0}^d$ denote a compatible $(B,C)$-basis for
$V$. Thus for $0 \leq i \leq d$, $u_i$ (resp. $v_i$) is
a basis for $E_{d-i}V$ (resp. $E'_iV$).
Consequently $\lbrace u_{2j}\rbrace_{j=0}^m$ and
$\lbrace v_{2j}\rbrace_{j=0}^m$ are bases for $U$ and $U'$, respectively.
The matrix $T''$ from Definition
\ref{def:TTT}(iii) is the
transition matrix from
$\lbrace u_i\rbrace_{i=0}^d$ to
$\lbrace v_i\rbrace_{i=0}^d$.
By construction
$T''$ is upper triangular with
$(i,r)$-entry
$\alpha''_{r-i}$ for $0 \leq i\leq r\leq d$.
By Lemma
\ref{lem:case} the scalars $\alpha''_1, \alpha''_3,\ldots, \alpha''_{d-1}$
are zero. So the $(i,r)$-entry
of $T''$ is zero if $r-i$ is odd $(0 \leq i \leq r\leq d)$.
By these comments $v_{2j} \in U$  for $0 \leq j \leq m$.
Therefore $U' \subseteq U$. In this inclusion each side has
dimension $m+1$, so $U'=U$.
 One similarly shows that
$U''=U'$.
\\
\noindent (ii) Similar to the proof of (i) above.
\\
\noindent (iii), (iv) The sequence
$\lbrace E_iV\rbrace_{i=0}^d$  is a decomposition of $V$.
\end{proof}

\begin{definition}\rm
Referring to Lemma
\ref{lem:V0V1}
and following  Definition
\ref{def:OUTERINNER},
we call $V_{\rm out}$ (resp. $V_{\rm in}$)
the {\it outer part} (resp. {\it inner part})
of $V$ with respect to $A,B,C$.
\end{definition}

\begin{lemma}
\label{lem:trivialT}
Assume that $A,B,C$ is bipartite.
Then $V_{\rm out}\not=0$ and
 $V_{\rm in}\not=V$.
Moreover the following are equivalent:
{\rm (i)} 
$A,B$ is trivial;
{\rm (ii)} 
 $V_{\rm out}=V$; {\rm (iii)}
 $V_{\rm in}=0$.
\end{lemma}
\begin{proof} 
By 
Lemma
\ref{lem:LRPtrivialT}.
\end{proof}

\begin{lemma}
\label{lem:bipABCact}
Assume that $A,B,C$ is bipartite.
Then 
\begin{eqnarray*}
&& AV_{\rm out} = V_{\rm in},
 \qquad \qquad 
 BV_{\rm out} = V_{\rm in},
 \qquad \qquad 
 CV_{\rm out} = V_{\rm in},
\\
&&
 AV_{\rm in} \subseteq V_{\rm out},
 \qquad \qquad 
 BV_{\rm in} \subseteq V_{\rm out},
 \qquad \qquad 
 CV_{\rm in} \subseteq V_{\rm out}.
\end{eqnarray*}
Moreover
\begin{eqnarray*}
&& A^2V_{\rm out} \subseteq V_{\rm out},
 \qquad \qquad 
 B^2V_{\rm out} \subseteq V_{\rm out},
 \qquad \qquad 
 C^2V_{\rm out} \subseteq V_{\rm out},
\\
&&
 A^2V_{\rm in} \subseteq V_{\rm in},
 \qquad \qquad 
 B^2V_{\rm in} \subseteq V_{\rm in},
 \qquad \qquad 
 C^2V_{\rm in} \subseteq V_{\rm in}.
\end{eqnarray*}
\end{lemma}
\begin{proof} 
Use Lemma
\ref{lem:ABaction}.
\end{proof}

\begin{definition}\rm
\label{def:BipNotation}
For notational convenience define
\begin{eqnarray*}
t_i = \frac{\varphi'_{d-i+1} \varphi''_{d-i+1}}{\varphi_i},
\qquad \quad
t'_i = \frac{\varphi''_{d-i+1} \varphi_{d-i+1}}{\varphi'_i},
\qquad \quad
t''_i = \frac{\varphi_{d-i+1} \varphi'_{d-i+1}}{\varphi''_i}
\end{eqnarray*}
for $1 \leq i \leq d$, and
\begin{eqnarray*}
t_0=0, \qquad  t'_0 =0, \qquad  t''_0=0,
\qquad 
t_{d+1}=0, \qquad  t'_{d+1} = 0, \qquad  t''_{d+1}=0.
\end{eqnarray*}
\end{definition}

\noindent The following two lemmas are obtained by routine computation.

\begin{lemma}
\label{lem:A2B2C2Out}
Assume that $A,B,C$ is bipartite. Then 
the action of
 $A^2, B^2, C^2$ on
$V_{\rm out}$ is an LR triple with diameter $m=d/2$.
For this LR triple,
\begin{enumerate}
\item[\rm (i)]
the parameter array is
\begin{eqnarray*}
(
\lbrace \varphi_{2j-1}\varphi_{2j}\rbrace_{j=1}^m;
\lbrace \varphi'_{2j-1}\varphi'_{2j}\rbrace_{j=1}^m;
\lbrace \varphi''_{2j-1}\varphi''_{2j}\rbrace_{j=1}^m);
\end{eqnarray*}
\item[\rm (ii)]
the idempotent data is
\begin{eqnarray*}
(
\lbrace E_{2j}\rbrace_{j=0}^m;
\lbrace E'_{2j}\rbrace_{j=0}^m;
\lbrace E''_{2j}\rbrace_{j=0}^m
);
\end{eqnarray*}
\item[\rm (iii)]
the trace data is (using the notation of Definition
\ref{def:BipNotation})
\begin{eqnarray*}
(
\lbrace t_{2j} + t_{2j+1} \rbrace_{j=0}^m;
\lbrace t'_{2j} + t'_{2j+1} \rbrace_{j=0}^m;
\lbrace t''_{2j} + t''_{2j+1} \rbrace_{j=0}^m);
\end{eqnarray*}
\item[\rm (iv)]
the Toeplitz data is
\begin{eqnarray*}
(
\lbrace \alpha_{2j}\rbrace_{j=0}^m,
\lbrace \beta_{2j}\rbrace_{j=0}^m;
\lbrace \alpha'_{2j}\rbrace_{j=0}^m,
\lbrace \beta'_{2j}\rbrace_{j=0}^m;
\lbrace \alpha''_{2j}\rbrace_{j=0}^m,
\lbrace \beta''_{2j}\rbrace_{j=0}^m
).
\end{eqnarray*}
\end{enumerate}
\end{lemma}

\begin{lemma}
\label{lem:A2B2C2In}
Assume that $A,B,C$ is bipartite
and nontrivial.
Then the action of
 $A^2, B^2, C^2$ on
$V_{\rm in}$ is an LR triple with diameter $m-1$, where $m=d/2$.
For this LR triple,
\begin{enumerate}
\item[\rm (i)]
the parameter array is
\begin{eqnarray*}
(
\lbrace \varphi_{2j}\varphi_{2j+1}\rbrace_{j=1}^{m-1};
\lbrace \varphi'_{2j}\varphi'_{2j+1}\rbrace_{j=1}^{m-1};
\lbrace \varphi''_{2j}\varphi''_{2j+1}\rbrace_{j=1}^{m-1});
\end{eqnarray*}
\item[\rm (ii)]
the idempotent data is
\begin{eqnarray*}
(
\lbrace E_{2j+1}\rbrace_{j=0}^{m-1};
\lbrace E'_{2j+1}\rbrace_{j=0}^{m-1};
\lbrace E''_{2j+1}\rbrace_{j=0}^{m-1}
);
\end{eqnarray*}
\item[\rm (iii)]
the trace data is (using the notation of Definition
\ref{def:BipNotation})
\begin{eqnarray*}
(
\lbrace t_{2j+1} + t_{2j+2} \rbrace_{j=0}^{m-1};
\lbrace t'_{2j+1} + t'_{2j+2} \rbrace_{j=0}^{m-1};
\lbrace t''_{2j+1} + t''_{2j+2} \rbrace_{j=0}^{m-1});
\end{eqnarray*}
\item[\rm (iv)]
the Toeplitz data is
\begin{eqnarray*}
(
\lbrace \alpha_{2j}\rbrace_{j=0}^{m-1},
\lbrace \beta_{2j}\rbrace_{j=0}^{m-1};
\lbrace \alpha'_{2j}\rbrace_{j=0}^{m-1},
\lbrace \beta'_{2j}\rbrace_{j=0}^{m-1};
\lbrace \alpha''_{2j}\rbrace_{j=0}^{m-1},
\lbrace \beta''_{2j}\rbrace_{j=0}^{m-1}
).
\end{eqnarray*}
\end{enumerate}
\end{lemma}

\begin{lemma}
\label{lem:NormHalf}
Assume that $A,B,C$ is bipartite, and
consider the action of $A^2,B^2,C^2$ on 
$V_{\rm out}$.
\begin{enumerate}
\item[\rm (i)]
Assume that $A,B,C$ is trivial. Then
so is the action of $A^2,B^2,C^2$ on 
$V_{\rm out}$.
\item[\rm (ii)]
Assume that $A,B,C$ is nontrivial. Then
the action of $A^2,B^2,C^2$ on 
$V_{\rm out}$ is nonbipartite.
\end{enumerate}
\end{lemma}
\begin{proof} (i)
By Example
\ref{ex:trivial} and Lemma
\ref{lem:trivialT}.
\\
\noindent (ii)
Evaluate the Toeplitz data in
Lemma
\ref{lem:A2B2C2Out}(iv)
using Lemmas
\ref{lem:NotB},
\ref{lem:case}.
\end{proof}

\begin{lemma} 
\label{lem:NormHalfIn}
Assume that $A,B,C$ is bipartite and
nontrivial.
Consider the action of $A^2,B^2,C^2$ on 
$V_{\rm in}$.
\begin{enumerate}
\item[\rm (i)] 
Assume that $d=2$.
Then the action of $A^2,B^2,C^2$ on 
$V_{\rm in}$ is trivial.
\item[\rm (ii)]
Assume that $d\geq 4$.
Then the action of $A^2,B^2,C^2$ on 
$V_{\rm in}$ is nonbipartite.
\end{enumerate}
\end{lemma}
\begin{proof} 
(i)
By Example
\ref{ex:trivial} and Lemma
\ref{lem:trivialT}.
\\
\noindent (ii)
Evaluate the Toeplitz data in
Lemma
\ref{lem:A2B2C2In}(iv)
using Lemmas
\ref{lem:NotB},
\ref{lem:case}.
\end{proof}

\begin{lemma}
\label{lem:Balpha2Inv}
Assume that $A,B,C$ is bipartite and nontrivial.
Then for $2 \leq i \leq d$,
\begin{eqnarray}
\label{eq:Balpha2Inv}
\frac{\alpha_2}{\varphi_{i-1}\varphi_i} = 
\frac{\alpha'_2}{\varphi'_{i-1}\varphi'_i} = 
\frac{\alpha''_2}{\varphi''_{i-1}\varphi''_i},
 \qquad \qquad 
\frac{\beta_2}{\varphi_{i-1}\varphi_i} = 
\frac{\beta'_2}{\varphi'_{i-1}\varphi'_i} = 
\frac{\beta''_2}{\varphi''_{i-1}\varphi''_i}.
\end{eqnarray}
\end{lemma}
\begin{proof}
Apply Corollary
\ref{cor:alphaOneInv}
to the LR triples in
Lemmas \ref{lem:A2B2C2Out},
\ref{lem:A2B2C2In}.
\end{proof}

\begin{lemma}
\label{lem:Bippipd}
Assume that $A,B,C$ is bipartite and nontrivial.
Then the following {\rm (i), (ii)} hold for $1 \leq i,j\leq d$.
\begin{enumerate}
\item[\rm (i)] Assume that $i,j$ have opposite parity. Then
\begin{eqnarray}
\label{eq:Bippipd}
\frac{\alpha_2}{\varphi_{i}\varphi_{j}} = 
\frac{\alpha'_2}{\varphi'_{i}\varphi'_{j}} = 
\frac{\alpha''_2}{\varphi''_{i}\varphi''_{j}},
 \quad \qquad 
\frac{\beta_2}{\varphi_{i}\varphi_{j}} = 
\frac{\beta'_2}{\varphi'_{i}\varphi'_{j}} = 
\frac{\beta''_2}{\varphi''_{i}\varphi''_{j}}.
\end{eqnarray}
\item[\rm (ii)] Assume that $i,j$ have the same parity.
Then
\begin{eqnarray}
\label{eq:Bippipd2}
\frac{\varphi_{i}}{\varphi_{j}} = 
\frac{\varphi'_{i}}{\varphi'_{j}} = 
\frac{\varphi''_{i}}{\varphi''_{j}}.
\end{eqnarray}
\end{enumerate}
\end{lemma}
\begin{proof}
We have a preliminary remark.
For $2\leq k \leq d$ define
$x_k = \alpha_2(\varphi_{k-1}\varphi_k)^{-1}$,
and note that $x_k=x'_k=x''_k$ by
 Lemma
\ref{lem:Balpha2Inv}.
\\
\noindent 
(i)
We may assume without loss that
$i<j$.
Observe that
\begin{eqnarray*}
\frac{\alpha_2}{\varphi_{i}\varphi_{j}} = 
 \frac{x_{i+1}x_{i+3} \cdots x_{j}}{x_{i+2}x_{i+4}\cdots x_{j-1}}.
\end{eqnarray*}
By this and the preliminary remark,
we obtain the equations on the left in
(\ref{eq:Bippipd}).
The equations on the right in
(\ref{eq:Bippipd})
are similarly obtained.
\\
\noindent (ii)
We may assume without loss that
$i<j$. Observe that
\begin{eqnarray*}
\frac{\varphi_{i}}{\varphi_{j}} = 
 \frac{x_{i+2}x_{i+4} \cdots x_{j}}{x_{i+1}x_{i+3}\cdots x_{j-1}}.
\end{eqnarray*}
By this and the preliminary remark, we obtain the equations
(\ref{eq:Bippipd2}).
\end{proof}

\begin{lemma} 
\label{lem:BipAlphaInvar}
Assume that $A,B,C$ is bipartite and nontrivial.
Then for $0 \leq j \leq d/2$,
\begin{eqnarray}
\label{ex:BipAlphaInvar}
\frac{\alpha_{2j}}{\alpha^j_2}
= 
\frac{\alpha'_{2j}}{(\alpha'_2)^j}
= 
\frac{\alpha''_{2j}}{(\alpha''_2)^j},
\qquad \qquad \qquad
\frac{\beta_{2j}}{\beta^j_2}
= 
\frac{\beta'_{2j}}{(\beta'_2)^j}
= 
\frac{\beta''_{2j}}{(\beta''_2)^j}.
\end{eqnarray}
\end{lemma}
\begin{proof}
Apply Lemma
\ref{lem:alphaiAlpha1}
to the LR triple in
Lemma \ref{lem:A2B2C2Out}. This LR triple is
nonbipartite by Lemma
\ref{lem:NormHalf}(ii).
\end{proof}

\begin{definition}\rm
\label{def:projector}
Assume that $A,B,C$ is bipartite.
An ordered pair of elements chosen from $A,B,C$ form an
 LR pair;
consider the corresponding projector map from
Definition
\ref{def:J}. By Lemma 
\ref{lem:V0V1}
this projector is independent of the choice; denote
this common projector by $J$. We call $J$ the
{\it projector} for $A,B,C$.
\end{definition}

\noindent In Section 9 we discussed in detail the projector
map for LR pairs. We now adapt a few points to LR triples.

\begin{lemma}
\label{lem:ABCJtriv}
Assume that $A,B,C$ is bipartite.
Then its projector map $J$ is nonzero.
If $A,B,C$ is trivial then
$J=I$. If $A,B,C$ is nontrivial
then $J, I$ are linearly independent over $\mathbb F$.
\end{lemma}
\begin{proof} 
By Lemma
\ref{lem:Jtriv} and Definition
\ref{def:projector}.
\end{proof}

\begin{lemma}
\label{lem:EEEJ}
Assume that $A,B,C$ is bipartite.
Then its projector map $J$ satisfies
\begin{eqnarray*}
J =
\sum_{j=0}^{d/2} E_{2j} = 
\sum_{j=0}^{d/2} E'_{2j} = 
\sum_{j=0}^{d/2} E''_{2j}.
\end{eqnarray*}
Moreover $J^2=J$. Also, $J$ commutes with
each of $E_i,E'_i,E''_i$ for $0 \leq i \leq d$.
\end{lemma}
\begin{proof} 
By Lemma
\ref{lem:Jfacts} and Definition
\ref{def:projector}.
\end{proof}

\begin{lemma}
\label{lem:ABCJ}
Assume that $A,B,C$ is bipartite.
Then its projector map $J$ satisfies
\begin{eqnarray*}
A = A J + J A, \qquad \qquad
B = B J + J B, \qquad \qquad
C = C J + J C.
\end{eqnarray*}
\end{lemma}
\begin{proof} By Lemma
\ref{lem:INOutFacts}(iii)
and Definition \ref{def:projector}.
\end{proof}

\begin{lemma}
\label{lem:JMat}
Assume that $A,B,C$ is bipartite.
With respect to any of the 12 bases
{\rm (\ref{eq:typeAB})--(\ref{eq:typeCA})}, the matrix
representing $J$ is 
${\rm diag}(1,0,1,0,\ldots,0,1)$.
\end{lemma}
\begin{proof} By construction and linear algebra.
\end{proof}

\noindent The following definition is 
motivated by Definition
\ref{def:Ainout}.
\begin{definition}
\label{def:ABCoutIn}
\rm
Assume that $A,B,C$ is bipartite.
Define  
\begin{eqnarray}
A_{\rm out}, 
\qquad 
A_{\rm in},
\qquad
B_{\rm out}, 
\qquad 
B_{\rm in},
\qquad 
C_{\rm out},
\qquad 
C_{\rm in}
\label{eq:6list}
\end{eqnarray}
 in ${\rm End}(V)$ as follows.
The map $A_{\rm out}$ acts on
 $V_{\rm out}$ as $A$, and on
 $V_{\rm in}$ as zero.
The map $A_{\rm in}$ acts on
 $V_{\rm in}$ as $A$, and on
 $V_{\rm out}$ as zero.
The other maps in
(\ref{eq:6list}) are similarly defined.
By construction
\begin{equation*}
A = A_{\rm out}+ 
 A_{\rm in}, 
 \qquad \qquad
B = B_{\rm out}+ 
 B_{\rm in}, 
 \qquad \qquad
C = C_{\rm out}+ 
 C_{\rm in}.
\end{equation*}
\end{definition}

\begin{lemma}
\label{lem:ABCINOUTJ}
Assume that $A,B,C$ is bipartite. Then 
\begin{eqnarray*}
&&
A_{\rm out} = AJ=(I-J)A,
\qquad \qquad  
A_{\rm in} = JA = A(I-J),
\\
&&
B_{\rm out} = BJ=(I-J)B,
\qquad \qquad  
B_{\rm in} = JB = B(I-J),
\\
&&
C_{\rm out} = CJ=(I-J)C,
\qquad \qquad  
C_{\rm in} = JC = C(I-J).
\end{eqnarray*}
\end{lemma}
\begin{proof} By Lemma
\ref{lem:INOutFacts}(i),(ii)
and Definition \ref{def:projector}.
\end{proof}

\begin{lemma}
\label{lem:InOut}
Assume that $A,B,C$ is bipartite.
Let
\begin{eqnarray}
\label{eq:AlphaInOut}
\alpha_{\rm out}, \quad 
\alpha_{\rm in}, \quad 
\beta_{\rm out}, \quad 
\beta_{\rm in}, \quad 
\gamma_{\rm out}, \quad 
\gamma_{\rm in}
\end{eqnarray}
denote nonzero scalars in $\mathbb F$. Then the sequence
\begin{eqnarray}
\label{eq:NewLRT}
\alpha_{\rm out}A_{\rm out}
+\alpha_{\rm in}A_{\rm in},
\qquad
\beta_{\rm out} B_{\rm out}+
\beta_{\rm in}B_{\rm in},
\qquad
\gamma_{\rm out}C_{\rm out}+
\gamma_{\rm in} C_{\rm in}
\end{eqnarray}
is a bipartite LR triple on $V$.
\end{lemma}
\begin{proof} By construction.
\end{proof}

\noindent Our next goal is to obtain the
parameter array, idempotent data, and Toeplitz data for
the LR triple in
(\ref{eq:NewLRT}).
The following definition is for notational convenience.

\begin{definition}
\label{def:NewLRTnot}
\rm
Adopt the assumptions and notation of 
Lemma
\ref{lem:InOut}.
For $1 \leq i \leq d$
define
\begin{eqnarray*}
&&
f_i
= \alpha_{\rm out}\beta_{\rm in},
\qquad
f'_i = 
\beta_{\rm out}\gamma_{\rm in},
\qquad
f''_i = 
\gamma_{\rm out}\alpha_{\rm in}
\qquad \qquad {\mbox {\rm (if $i$ is even)}},
\\
&&
f_i =
\alpha_{\rm in}\beta_{\rm out},
\qquad
f'_i = 
\beta_{\rm in}\gamma_{\rm out},
\qquad
f''_i =
\gamma_{\rm in}\alpha_{\rm out}
\qquad \qquad {\mbox {\rm (if $i$ is odd)}}.
\end{eqnarray*}
\noindent Also define
\begin{eqnarray*}
&&
g_i = 
(\alpha_{\rm out}\alpha_{\rm in})^{-i/2},
\qquad
g'_i = 
(\beta_{\rm out}\beta_{\rm in})^{-i/2},
\qquad 
g''_i
= (\gamma_{\rm out}\gamma_{\rm in})^{-i/2}
\qquad \quad {\mbox {\rm (if $i$ is even)}},
\\
&&
g_i =
0,
\qquad
g'_i = 0,
\qquad
g''_i = 0
\qquad \qquad {\mbox {\rm (if $i$ is odd)}}.
\end{eqnarray*}
\end{definition}

\begin{lemma}
\label{lem:PAnewLRT}
Referring to Lemma
\ref{lem:InOut}, let the nonzero scalars
{\rm (\ref{eq:AlphaInOut})}
be given, 
and consider the LR triple
in {\rm (\ref{eq:NewLRT})}.
For this LR triple,
\begin{enumerate}
\item[\rm (i)]
the parameter array is (using the notation of Definition
\ref{def:NewLRTnot})
\begin{eqnarray*}
(
\lbrace 
\varphi_i f_i\rbrace_{i=1}^d;
\lbrace \varphi'_i f'_i\rbrace_{i=1}^d;
\lbrace \varphi''_i f''_i\rbrace_{i=1}^d
);
\end{eqnarray*}
\item[\rm (ii)]
the idempotent data is 
equal to the
idempotent data for $A,B,C$;
\item[\rm (iii)]
the Toeplitz data is (using the notation of Definition
\ref{def:NewLRTnot})
\begin{eqnarray*}
(
\lbrace \alpha_i g''_i\rbrace_{i=0}^d,
\lbrace \beta_i g''_i\rbrace_{i=0}^d;
\lbrace \alpha'_i g_i\rbrace_{i=0}^d,
\lbrace \beta'_i g_i\rbrace_{i=0}^d;
\lbrace \alpha''_i g'_i\rbrace_{i=0}^d,
\lbrace \beta''_i g'_i \rbrace_{i=0}^d
).
\end{eqnarray*}
\end{enumerate}
\end{lemma}
\begin{proof} (i) Use Lemma
\ref{lem:LRPInOut}.
\\
\noindent 
(ii) Similar to the proof of Lemma
 \ref{lem:isequal}.
 \\
\noindent (iii) 
Similar to the proof of Lemma \ref{lem:ToeplitzAdjust}.
\end{proof}

\begin{definition}
\label{def:BIASSOC}
\rm
Assume that $A,B,C$  is bipartite. Let 
$A',B',C'$ denote a bipartite LR triple on $V$.
Then $A,B,C$ and $A',B',C'$ will be called 
{\it biassociate} whenever there exist nonzero scalars
$\alpha, \beta, \gamma$ in $\mathbb F$
such that
\begin{eqnarray*}
A' = \alpha A_{\rm out}
+A_{\rm in},
\qquad
B' = \beta  B_{\rm out}+
B_{\rm in},
\qquad
C' = \gamma C_{\rm out}+
C_{\rm in}.
\end{eqnarray*}
Biassociativity is an equivalence relation.
\end{definition}

\begin{lemma}
\label{lem:BiassocIso}
Assume that $A,B,C$ is bipartite.
Let $A',B',C'$ denote a bipartite LR triple  over $\mathbb F$.
Then the following are equivalent:
\begin{enumerate}
\item[\rm (i)] there exists a bipartite LR triple  over $\mathbb F$
that is biassociate to $A,B,C$ and isomorphic to $A',B',C'$;
\item[\rm (ii)] there exists a bipartite LR triple over $\mathbb F$
that is isomorphic to $A,B,C$ and biassociate to $A',B',C'$.
\end{enumerate}
\end{lemma}
\begin{proof} Similar to the proof of Lemma
\ref{lem:assocIso}.
\end{proof}

\begin{definition}
\label{def:biSim}
\rm
Assume that $A,B,C$  is bipartite. Let 
$A',B',C'$ denote a bipartite LR triple over $\mathbb F$.
Then $A,B,C$ and $A',B',C'$ will be called 
{\it bisimilar} whenever the equivalent conditions {\rm (i), (ii)}
hold in Lemma
\ref{lem:BiassocIso}.
Bisimilarity is an equivalence relation.
\end{definition}

\section{Equitable LR triples}

\noindent Throughout this section the following notation is in effect.
Let $V$ denote a vector space over $\mathbb F$ with
dimension $d+1$.
Let $A,B,C$ denote
an LR triple on $V$, with
 parameter array
(\ref{eq:paLRT}),
idempotent data
(\ref{eq:idseq}),
trace data
(\ref{eq:tracedata}), and Toeplitz data
(\ref{eq:ToeplitzData}). We describe a condition on
$A,B,C$ called equitable.

\begin{definition}
\label{def:equitNorm}
\rm 
The LR triple $A,B,C$ is called {\it equitable}
whenever 
$\alpha_i = \alpha'_i = \alpha''_i$ for $0 \leq i \leq d$.
\end{definition}

\begin{lemma}
\label{lem:TrivEquit}
If 
$A,B,C$ is trivial, then it is equitable.
\end{lemma}
\begin{proof} Recall that $\alpha_0=\alpha'_0 = \alpha''_0=1$.
\end{proof}

\begin{lemma}
\label{lem:equitBasicBeta}
Assume that $A,B,C$ is equitable. Then
$\beta_i= \beta'_i = \beta''_i$ for $0 \leq i \leq d$.
\end{lemma}
\begin{proof}
Refer to Definitions
\ref{def:TTT},
\ref{def:TTT1}. The matrices $T,T',T''$
coincide, so their inverses coincide.
The result follows.
\end{proof}

\begin{lemma}
If  $A,B,C$ is equitable, then so are its relatives.
\end{lemma}
\begin{proof} By
Lemmas
\ref{lem:ToeplitzData},
\ref{lem:ToeplitzDataD},
\ref{lem:equitBasicBeta}
and
Definition
\ref{def:equitNorm}.
\end{proof}

\noindent As we investigate the equitable property,
we will treat the bipartite and nonbipartite cases
separately. 
We begin with the nonbipartite case.

\begin{lemma}
\label{lem:equitNBmin}
Assume that $A,B,C$ is nonbipartite.
Then $A,B,C$ is equitable if and only if
$\alpha_1 = \alpha'_1 = \alpha''_1$.
\end{lemma}
\begin{proof}
By Lemma
\ref{lem:alphaiAlpha1}
and Definition
\ref{def:equitNorm}.
\end{proof}

\begin{lemma}
\label{lem:equitBasic}
Assume that $A,B,C$ is nonbipartite and equitable.
Then the following hold:
\begin{enumerate}
\item[\rm (i)] 
$\varphi_i = \varphi'_i = \varphi''_i$ for $1 \leq i \leq d$;
\item[\rm (ii)] 
$a_i = a'_i = a''_i =\alpha_1(\varphi_{d-i+1}-\varphi_{d-i})$ 
for $0 \leq i \leq d$.
\end{enumerate}
\end{lemma}
\begin{proof}
(i) By Corollary
\ref{cor:alphaOneInv}
and Lemma
\ref{lem:NotB}.
\\
\noindent (ii) Use
(\ref{eq:list3})  and
Proposition
\ref{lem:ai}.
\end{proof}

\noindent The following definition is for later use.

\begin{definition}
\label{def:Rhoi}
\rm
Assume that $A,B,C$ is nonbipartite and equitable.
Define
\begin{eqnarray}
\rho_i = \frac{\varphi_{i+1}}{\varphi_{d-i}}
\qquad \qquad (0 \leq i \leq d-1).
\label{eq:mainrec2}
\end{eqnarray}
Note by Lemma
\ref{lem:equitBasic}(i) that
$\rho_i = \rho'_i = \rho''_i$ for $0 \leq i\leq d-1$.
\end{definition}

\begin{lemma}
\label{def:RhoiCom}
Assume that $A,B,C$ is nonbipartite and equitable.
Then $\rho_i \rho_{d-i-1} = 1$ for $0 \leq i \leq d-1$.
\end{lemma}
\begin{proof} By Definition
\ref{def:Rhoi}.
\end{proof}

\begin{lemma}
\label{lem:NBipEquitAdj}
Assume that $A,B,C$ is nonbipartite.
 Let $\alpha,\beta,\gamma$
denote nonzero scalars in $\mathbb F$. Then the following
are equivalent:
\begin{enumerate}
\item[\rm (i)] the LR triple $\alpha A,\beta B,\gamma C$ is equitable;
\item[\rm (ii)] $\alpha/\alpha'_1 = \beta/\alpha''_1 = \gamma/\alpha_1 $.
\end{enumerate}
\end{lemma}
\begin{proof}
By Lemmas
\ref{lem:ToeplitzAdjust},
\ref{lem:equitNBmin}.
\end{proof}

\begin{lemma} 
\label{lem:EqAdjust}
Assume that $A,B,C$ is nonbipartite.
Then there
exists an equitable LR triple on $V$ that is
associate
to $A,B,C$.
\end{lemma}
\begin{proof}
By
Definition
\ref{def:ASSOC}
and
Lemma
\ref{lem:NBipEquitAdj}.
\end{proof}

\noindent We turn our attention to bipartite LR triples.

\begin{lemma}
\label{lem:BPEquitMin}
Assume that $A,B,C$ is bipartite and nontrivial. 
Then $A,B,C$ is equitable if and only if 
$\alpha_2 = \alpha'_2 = \alpha''_2$.
\end{lemma}
\begin{proof}
By Lemmas
\ref{lem:case},
\ref{lem:BipAlphaInvar}
and
Definition
\ref{def:equitNorm}.
\end{proof}

\begin{lemma}
\label{lem:basic2}
Assume that $A,B,C$ is bipartite, nontrivial, and equitable.
Then 
$\varphi_{i-1}\varphi_i =
\varphi'_{i-1} \varphi'_i =
\varphi''_{i-1} \varphi''_i$ for $2 \leq i \leq d$.
\end{lemma}
\begin{proof} 
By Lemma \ref{lem:BiPbasic}(ii)
 and Lemma
\ref{lem:Balpha2Inv}.
\end{proof}

\begin{lemma}
\label{lem:basic2A}
Assume that $A,B,C$ is bipartite  and equitable.
Then the following {\rm (i), (ii)}  hold for $0 \leq i \leq d$.
\begin{enumerate}
\item[\rm (i)] For $i$ even,
\begin{eqnarray*}
&&
\varphi_1 \varphi_2 \cdots \varphi_i = 
\varphi'_1 \varphi'_2 \cdots \varphi'_i = 
\varphi''_1 \varphi''_2 \cdots \varphi''_i,
\\
&&
\varphi_d \varphi_{d-1} \cdots \varphi_{d-i+1} = 
\varphi'_d \varphi'_{d-1} \cdots \varphi'_{d-i+1} = 
\varphi''_d \varphi''_{d-1} \cdots \varphi''_{d-i+1}.
\end{eqnarray*}
\item[\rm (ii)] For $i$  odd,
\begin{eqnarray*}
&&
\varphi_2 \varphi_3 \cdots \varphi_i = 
\varphi'_2 \varphi'_3 \cdots \varphi'_i = 
\varphi''_2 \varphi''_3 \cdots \varphi''_i,
\\
&&
\varphi_{d-1} \varphi_{d-2} \cdots \varphi_{d-i+1} = 
\varphi'_{d-1} \varphi'_{d-2} \cdots \varphi'_{d-i+1} = 
\varphi''_{d-1} \varphi''_{d-2} \cdots \varphi''_{d-i+1}.
\end{eqnarray*}
\end{enumerate}
\end{lemma}
\begin{proof} By Lemma
\ref{lem:basic2} and since $d$ is even.
\end{proof}

\begin{lemma}
\label{lem:basic2new}
Assume that $A,B,C$ is bipartite  and equitable.
Then for
 $0 \leq i \leq d-1$,
\begin{eqnarray*}
\frac{\varphi'_{i+1}}{\varphi''_{d-i}}
=
\frac{\varphi''_{i+1}}{\varphi'_{d-i}},
\qquad \qquad 
\frac{\varphi''_{i+1}}{\varphi_{d-i}}
=
\frac{\varphi_{i+1}}{\varphi''_{d-i}},
\qquad \qquad 
\frac{\varphi_{i+1}}{\varphi'_{d-i}}
=
\frac{\varphi'_{i+1}}{\varphi_{d-i}}.
\end{eqnarray*}
\end{lemma}
\begin{proof}
Assume that $A,B,C$ is nontrivial; otherwise there
is nothing to prove. Now use Lemma
\ref{lem:Bippipd}(i) with $j=d-i+1$. The integers
$i,j$ have opposite parity since $d$ is even.
\end{proof}

\begin{definition}
\label{def:RHOD}
Assume that $A,B,C$ is bipartite and equitable.
Then for $0 \leq i \leq d-1$ define
\begin{eqnarray*}
&&
\rho_i =
\frac{\varphi'_{i+1}}{\varphi''_{d-i}} 
= 
\frac{\varphi''_{i+1}}{\varphi'_{d-i}},
\\
&&
\rho'_i=
\frac{\varphi''_{i+1}}{\varphi_{d-i}} 
= 
\frac{\varphi_{i+1}}{\varphi''_{d-i}},
\\
&&
\rho''_i = 
\frac{\varphi_{i+1}}{\varphi'_{d-i}} 
= 
\frac{\varphi'_{i+1}}{\varphi_{d-i}}.
\end{eqnarray*}
We emphasize that for $d\geq 1$,
\begin{eqnarray}
\rho_0 = \frac{\varphi'_1}{\varphi''_d} =
\frac{\varphi''_1}{\varphi'_d},
\qquad \qquad
\rho'_0 = \frac{\varphi''_1}{\varphi_d} =
\frac{\varphi_1}{\varphi''_d},
\qquad \qquad
\rho''_0 = \frac{\varphi_1}{\varphi'_d} =
\frac{\varphi'_1}{\varphi_d}.
\label{def:rho0def}
\end{eqnarray}
\end{definition}

\begin{lemma}
\label{def:BipRhoiCom}
Assume that $A,B,C$ is bipartite and equitable.
Then for $0 \leq i \leq d-1$,
\begin{eqnarray}
\rho_i \rho_{d-i-1}=1, \qquad 
\rho'_i \rho'_{d-i-1}=1, \qquad 
\rho''_i \rho''_{d-i-1}=1.
\label{eq:ThreeRho}
\end{eqnarray}
\end{lemma}
\begin{proof} By Definition
\ref{def:RHOD}.
\end{proof}

\begin{lemma}
\label{lem:RhoRelBip}
Assume that $A,B,C$ is bipartite, nontrivial, and equitable.
Then the following {\rm (i)--(iii)} hold:
\begin{enumerate}
\item[\rm (i)]
for $1 \leq i \leq d$,
\begin{eqnarray*}
&&
\frac{\varphi_i}{\rho_0} = 
\frac{\varphi'_i}{\rho'_0} = 
\frac{\varphi''_i}{\rho''_0} 
 \qquad \qquad \qquad {\mbox{\rm if $i$ is even}},
\\
&&
\varphi_i\rho_0 = 
\varphi'_i \rho'_0 = 
\varphi''_i\rho''_0 
 \qquad \qquad {\mbox{\rm if $i$ is odd}};
\end{eqnarray*} 
\item[\rm (ii)]
for $0 \leq i \leq d$,
\begin{eqnarray*}
&&
\rho_0 \rho_1 \cdots \rho_{i-1} = 
\rho'_0 \rho'_1 \cdots \rho'_{i-1} = 
\rho''_0 \rho''_1 \cdots \rho''_{i-1}
 \qquad \qquad {\mbox{\rm if $i$ is even}},
\\
&&
\rho_1 \rho_2 \cdots \rho_{i-1} = 
\rho'_1 \rho'_2 \cdots \rho'_{i-1} = 
\rho''_1 \rho''_2 \cdots \rho''_{i-1}
 \qquad \qquad {\mbox{\rm if $i$ is odd}};
\end{eqnarray*}
\item[\rm (iii)] for $0 \leq i \leq d-1$,
\begin{eqnarray*}
&&
\frac{\rho_i}{\rho_0} = 
\frac{\rho'_i}{\rho'_0} = 
\frac{\rho''_i}{\rho''_0}
 \qquad \qquad \qquad {\mbox{\rm if $i$ is even}},
\\
&&
\rho_i\rho_0 = 
\rho'_i\rho'_0 = 
\rho''_i\rho''_0
 \qquad \qquad {\mbox{\rm if $i$ is odd}}.
\end{eqnarray*}
\end{enumerate}
\end{lemma}
\begin{proof} Use Lemma
\ref{lem:Bippipd}
and Definition
\ref{def:RHOD}.
\end{proof}

\begin{lemma}
\label{lem:A2B2C2Equit}
Assume that $A,B,C$ is bipartite and equitable.
Then:
\begin{enumerate}
\item[\rm (i)] 
the action of $A^2,B^2,C^2$ on $V_{\rm out}$ is
equitable;
\item[\rm (ii)] for $A,B,C$ nontrivial 
the action of $A^2,B^2,C^2$ on $V_{\rm in}$ is equitable.
\end{enumerate}
\end{lemma}
\begin{proof} (i) By Lemma
\ref{lem:A2B2C2Out}(iv) and Definition
\ref{def:equitNorm}.
\\
\noindent (ii) By Lemma
\ref{lem:A2B2C2In}(iv) and Definition
\ref{def:equitNorm}.
\end{proof}

\begin{lemma}
\label{lem:BipEquitAdj}
Referring to Lemma
\ref{lem:InOut}, assume that $A,B,C$ is nontrivial,
and let the nonzero scalars 
{\rm (\ref{eq:AlphaInOut})}
be given.
Then the LR triple
{\rm (\ref{eq:NewLRT})} 
is equitable
 if and only if
\begin{eqnarray} 
 \alpha_{\rm out}\alpha_{\rm in}/\alpha'_2= 
 \beta_{\rm out}\beta_{\rm in}/\alpha''_2= 
 \gamma_{\rm out}\gamma_{\rm in} /\alpha_2.
\label{eq:neededForEquit}
\end{eqnarray}
\end{lemma}
\begin{proof}
By Lemma
\ref{lem:PAnewLRT}(iii)
and Definition
\ref{def:equitNorm}.
\end{proof}

\begin{lemma} 
\label{lem:EqAdjustBip}
Assume that $A,B,C$ is bipartite and nontrivial.
Then there
exists an equitable LR triple on $V$ that is
 biassociate to 
$A,B,C$.
\end{lemma}
\begin{proof}
By
Definition
\ref{def:BIASSOC}
and
Lemma
\ref{lem:BipEquitAdj}.
\end{proof}

\noindent We have a comment about general LR triples, bipartite 
or not.

\begin{lemma}
\label{lem:combine}
Assume that $A,B,C$ is equitable. Then
\begin{eqnarray}
\label{eq:combine}
\varphi_1 \varphi_2 \cdots \varphi_d =
\varphi'_1 \varphi'_2 \cdots \varphi'_d =
\varphi''_1 \varphi''_2 \cdots \varphi''_d.
\end{eqnarray}
\end{lemma}
\begin{proof} For $A,B,C$ nonbipartite,
the result follows from Lemma
\ref{lem:equitBasic}(i).
For $A,B,C$ bipartite, the result
follows from Lemma
\ref{lem:basic2A}(i) 
and since $d$ is even.
\end{proof}


\section{Normalized LR triples}

\noindent 
Throughout this section the following notation is in effect.
Let $V$ denote a vector space over $\mathbb F$ with
dimension $d+1$.
Let $A,B,C$ denote
an LR triple on
$V$, with
 parameter array
(\ref{eq:paLRT}),
idempotent data
(\ref{eq:idseq}),
trace data
(\ref{eq:tracedata}), and Toeplitz data
(\ref{eq:ToeplitzData}).
We describe a condition on $A,B,C$
called normalized. The condition is defined a bit
differently in the trivial, nonbipartite, and bipartite nontrivial cases.
We first dispense with the trivial case.

\begin{definition}
\rm 
\label{def:BNormTriv}
Assume that $A,B,C$ is trivial. Then we declare $A,B,C$ to be 
{\it normalized}.
\end{definition}

\begin{definition}
\label{def:NBNorm}
\rm
Assume that $A,B,C$ is nonbipartite.
Then $A,B,C$ is called {\it normalized} whenever
\begin{eqnarray*}
\alpha_1 = 1,
\qquad \quad 
\alpha'_1 = 1,
\qquad \quad 
\alpha''_1 = 1.
\end{eqnarray*}
\end{definition}

\begin{lemma}
Assume that $A,B,C$ is nonbipartite and normalized.
Then $A,B,C$ is equitable.
\end{lemma}
\begin{proof}
By
Lemma \ref{lem:equitNBmin}
and since $\alpha_1 = \alpha'_1 = \alpha''_1$.
\end{proof}

\begin{lemma}
Assume that $A,B,C$ is nonbipartite and normalized. Then so are 
its p-relatives.
\end{lemma}
\begin{proof}
By Definition
\ref{def:prel}
and Lemmas
\ref{lem:ToeplitzData},
\ref{lem:ToeplitzDataD}.
\end{proof}

\noindent 
A nonbipartite LR triple 
can be normalized as follows.

\begin{lemma} 
\label{lem:NBnormalize}
Assume that $A,B,C$ is nonbipartite.
Let $\alpha, \beta, \gamma$ denote nonzero scalars in 
$\mathbb F$. Then the LR triple
 $\alpha A, \beta B, \gamma C$ 
  is normalized
 if and only if
\begin{eqnarray*}
\alpha = \alpha'_1, \qquad \qquad
\beta = \alpha''_1, \qquad \qquad
\gamma = \alpha_1.
\end{eqnarray*}
\end{lemma}
\begin{proof}
Use Lemmas \ref{lem:ToeplitzAdjust},
\ref{lem:BPabc}
and
Definition
\ref{def:NBNorm}.
\end{proof}

\begin{corollary}
\label{prop:NormalNBunique}
Assume that $A,B,C$ is nonbipartite. Then
there exists a unique sequence $\alpha, \beta,\gamma$ of
nonzero scalars in $\mathbb F$ such that
$\alpha A, \beta B, \gamma C$ is normalized.
\end{corollary}
\begin{proof} By
Lemma
\ref{lem:NBnormalize}.
\end{proof}

\begin{corollary}
\label{prop:normalNBunique}
Assume that $A,B,C$ is nonbipartite. Then
$A,B,C$ is associate to a unique normalized nonbipartite LR triple over
$\mathbb F$.
\end{corollary}
\begin{proof} By
Definition
\ref{def:ASSOC},
Lemma
\ref{lem:BPabc},
and
Corollary
\ref{prop:NormalNBunique}.
\end{proof}

\begin{lemma}
\label{lem:Nbeta}
Assume that $A,B,C$ is nonbipartite and normalized.
Then
\begin{eqnarray*}
\beta_1 = -1,
\qquad \quad 
\beta'_1 = -1,
\qquad \quad 
\beta''_1 = -1.
\end{eqnarray*}
\end{lemma}
\begin{proof} By
(\ref{eq:list3}) and Definition
\ref{def:NBNorm}.
\end{proof}

\begin{lemma}
Assume that $A,B,C$ is nonbipartite and normalized. Then
so is the LR triple $-C,-B,-A$.
\end{lemma}
\begin{proof}
By Lemma
\ref{lem:ToeplitzData} (row 4 of the table)
along with 
Lemmas
\ref{lem:NBnormalize},
\ref{lem:Nbeta}.
\end{proof}

\begin{lemma} 
\label{lem:NBnormIso}
Assume that $A,B,C$ is nonbipartite and normalized.
Then $A,B,C$ is uniquely determined
up to isomorphism by its parameter array.
\end{lemma}
\begin{proof}
By Proposition
\ref{prop:IsoParTrace} the LR triple 
$A,B,C$ is uniquely
determined up to isomorphism by its
parameter array and trace data.
The trace data
 is determined by the
parameter array using
Lemma
\ref{lem:equitBasic}(ii) and $\alpha_1=1$.
The result follows.
\end{proof}

\noindent We turn our attention to  bipartite nontrivial LR triples.

\begin{definition}
\label{def:BNorm}
\rm
Assume  that
$A,B,C$ is bipartite 
and nontrivial.
Then $A,B,C$ is called {\it normalized} whenever
\begin{eqnarray*}
\alpha_2 = 1,
\qquad \quad 
\alpha'_2 = 1,
\qquad \quad 
\alpha''_2 = 1.
\end{eqnarray*}
\end{definition}

\begin{lemma}
Assume that $A,B,C$ is bipartite, nontrivial, and
normalized. Then $A,B,C$ is equitable.
\end{lemma}
\begin{proof}
By Lemma
\ref{lem:BPEquitMin}
and since 
$\alpha_2 = \alpha'_2 = \alpha''_2$.
\end{proof}

\begin{lemma}
Assume that $A,B,C$ is bipartite, nontrivial, and normalized. Then so are 
its p-relatives.
\end{lemma}
\begin{proof}
By Definition
\ref{def:prel}
and
Lemmas
\ref{lem:ToeplitzData},
\ref{lem:ToeplitzDataD}.
\end{proof}


\noindent A bipartite nontrivial LR triple can be
normalized as follows.

\begin{lemma} 
\label{lem:howtoNorm}
Referring to Lemma
\ref{lem:InOut}, assume that $A,B,C$ is nontrivial,
and let the nonzero scalars 
{\rm (\ref{eq:AlphaInOut})}
be given.
Then the LR triple
{\rm (\ref{eq:NewLRT})} 
is normalized  
 if and only if
\begin{eqnarray*}
\alpha_{\rm out} \alpha_{\rm in} = \alpha'_2, \qquad \qquad
\beta_{\rm out} \beta_{\rm in} = \alpha''_2, \qquad \qquad
\gamma_{\rm out} \gamma_{\rm in} = \alpha_2.
\end{eqnarray*}
\end{lemma}
\begin{proof} Use Lemma
\ref{lem:PAnewLRT}(iii) and Definition
\ref{def:BNorm}.
\end{proof}

\begin{corollary}
\label{lem:NormUnique}
Assume that $A,B,C$ is bipartite and nontrivial. 
Then there exists a unique sequence $\alpha, \beta, \gamma$
of nonzero scalars in $\mathbb F$ such that
\begin{eqnarray*}
\alpha A_{\rm out} + A_{\rm in},
\qquad \quad
\beta B_{\rm out} + B_{\rm in},
\qquad \quad
\gamma C_{\rm out} + C_{\rm in}
\end{eqnarray*}
is normalized.
\end{corollary}
\begin{proof}
In Lemma
\ref{lem:howtoNorm} set 
$
\alpha_{\rm in} =1$,
$
\beta_{\rm in} =1$,
$\gamma_{\rm in} =1$ to see that
$\alpha= \alpha'_2$,
$\beta= \alpha''_2$,
$\gamma= \alpha_2$ is the unique solution.
\end{proof}

\begin{corollary}
\label{lem:biassocUnique}
Assume that $A,B,C$ is bipartite and nontrivial. Then $A,B,C$
is biassociate to a unique biparitite normalized LR triple over
$\mathbb F$.
\end{corollary}
\begin{proof} 
By
Lemma
\ref{lem:InOut},
Definition
\ref{def:BIASSOC},
and
Corollary \ref{lem:NormUnique}.
\end{proof}

\begin{lemma}
\label{lem:BPBeta2}
Assume that $A,B,C$ is bipartite, nontrivial, and
normalized. Then 
\begin{eqnarray*}
\beta_2 = -1,
\qquad \quad 
\beta'_2 = -1,
\qquad \quad 
\beta''_2 = -1.
\end{eqnarray*}
\end{lemma}
\begin{proof}
By 
(\ref{eq:Al2Be2}) and
Definition
\ref{def:BNorm}.
\end{proof}

\begin{lemma} Assume that $A,B,C$ is bipartite, nontrivial, and
normalized. Then so is the LR triple
\begin{eqnarray*}
C_{\rm out}- C_{\rm in}, \quad \qquad
B_{\rm out}- B_{\rm in}, \quad \qquad
A_{\rm out}- A_{\rm in}.
\end{eqnarray*}
\end{lemma}
\begin{proof} By 
Lemma
\ref{lem:ToeplitzData} (row 4 of the table)
and Lemmas
\ref{lem:howtoNorm},
\ref{lem:BPBeta2}.
\end{proof}

\begin{lemma}
\label{lem:NormHalf2}
Assume that $A,B,C$ is bipartite, nontrivial, and normalized.
Then:
\begin{enumerate}
\item[\rm (i)]
the action of $A^2,B^2,C^2$ on 
$V_{\rm out}$ is normalized;
\item[\rm (ii)]
the action of $A^2,B^2,C^2$ on 
$V_{\rm in}$ is normalized.
\end{enumerate}
\end{lemma}
\begin{proof}
(i)
Evaluate the Toeplitz data in
Lemma
\ref{lem:A2B2C2Out}(iv)
using 
Lemma
\ref{lem:NormHalf}(ii) 
and
Definition
\ref{def:NBNorm}.
\\
\noindent (ii) For $d=2$,
the action of $A^2,B^2,C^2$ on 
$V_{\rm in}$ is trivial and hence normalized.
For $d\geq 4$,
evaluate the Toeplitz data in
Lemma
\ref{lem:A2B2C2In}(iv)
using 
Lemma
\ref{lem:NormHalfIn}(ii)
and
Definition
\ref{def:NBNorm}.
\end{proof}

\section{The idempotent centralizers for an LR triple}

\noindent 
Throughout this section the following notation is in effect.
Let $V$ denote a vector space over $\mathbb F$ with
dimension $d+1$.
Let $A,B,C$ denote
an LR triple on $V$, with 
 parameter array
(\ref{eq:paLRT}),
idempotent data
(\ref{eq:idseq}),
trace data
(\ref{eq:tracedata}), and Toeplitz data
(\ref{eq:ToeplitzData}).
We discuss a type of element in 
 ${\rm End}(V)$ called an idempotent centralizer.

\begin{definition}
\label{def:IC}
\rm
By an
{\it idempotent centralizer} for $A,B,C$
we mean an element in ${\rm End}(V)$
that commutes with each of $E_i, E'_i, E''_i$
for $0 \leq i \leq d$.
\end{definition}

\begin{lemma}
\label{lem:ICmeaning}
For 
$X \in {\rm End}(V)$ the following are equivalent:
\begin{enumerate}
\item[\rm (i)] $X$ is an idempotent centralizer
for $A,B,C$;
\item[\rm (ii)] for $0 \leq i \leq d$,
\begin{eqnarray*}
X E_iV \subseteq E_iV,\qquad \qquad
X E'_iV \subseteq E'_iV,\qquad \qquad
X E''_iV \subseteq E''_iV.
\end{eqnarray*}
\end{enumerate}
\end{lemma}
\begin{proof} By Definition
\ref{def:IC}
and linear algebra.
\end{proof}

\begin{example}
\label{ex:trivIC}
\rm The identity $I \in  
{\rm End}(V)$ is an idempotent centralizer for $A,B,C$.
\end{example}

\begin{definition}
\label{def:ICspace}
\rm 
Let $\mathcal I$ denote the set of idempotent centralizers
for $A,B,C$. Note that
$\mathcal I$ is a subalgebra of the $\mathbb F$-algebra
${\rm End}(V)$. We call 
$\mathcal I$ the {\it idempotent centralizer algebra}
for $A,B,C$.
\end{definition}

\noindent Referring to Definition
\ref{def:ICspace},
our next goal is to display a basis
for the $\mathbb F$-vector space
$\mathcal I$. 
Recall the projector $J$ from Definition
\ref{def:projector}.

\begin{proposition}
\label{prop:centBasis}
The following {\rm (i)--(iii)} hold.
\begin{enumerate}
\item[\rm (i)] Assume that $A,B,C$ is trivial. Then
$I$ is a basis for $\mathcal I$.
\item[\rm (ii)] Assume that $A,B,C$ is nonbipartite.
Then 
$I$ is a basis for $\mathcal I$.
\item[\rm (iii)] Assume that $A,B,C$ is bipartite
and nontrivial. Then $I,J$ is a basis for 
$\mathcal I$.
\end{enumerate}
\end{proposition}
\begin{proof}
(i) Routine.
\\
\noindent (ii), (iii) 
Assume that $A,B,C$ is nontrivial.
Let the set $S$ consist of $I$ (if $A,B,C$ is nonbipartite)
and $I,J$ (if $A,B,C$ is bipartite).
 We show that $S$ is a basis for $\mathcal I$.
By 
Lemma
\ref{lem:ABCJtriv}
and Lemma \ref{lem:EEEJ},
$S$ is a linearly independent
subset of $\mathcal I$.
We show that $S$ spans $\mathcal I$.
 Let $\lbrace u_i\rbrace_{i=0}^d$ denote
an $(A,C)$-basis of $V$, and let 
 $\lbrace v_i\rbrace_{i=0}^d$ denote
a compatible $(A,B)$-basis of $V$.
The transition matrix from
 $\lbrace u_i\rbrace_{i=0}^d$ to
 $\lbrace v_i\rbrace_{i=0}^d$ is the matrix $T'$ from
Definition
\ref{def:TTT}. The matrix $T'$ is upper triangular
and Toeplitz, with parameters $\lbrace \alpha'_i\rbrace_{i=0}^d$.
Let $X \in \mathcal I$.
By Lemma
\ref{lem:ICmeaning}
there exist scalars $\lbrace r_i\rbrace_{i=0}^d$ in
$\mathbb F$ such that $Xu_i = r_i u_{i}$ for
$0 \leq i \leq d$.
Also by Lemma
\ref{lem:ICmeaning},
there exist scalars $\lbrace s_i\rbrace_{i=0}^d$ in
$\mathbb F$ such that $Xv_i = s_i v_{i}$ for
$0 \leq i \leq d$.
Let the matrix $M \in 
{\rm Mat}_{d+1}(\mathbb F)$ represent
$X$ with respect to $\lbrace u_i \rbrace_{i=0}^d$.
Then $M$ is diagonal with $(i,i)$-entry 
$r_i$ for $0 \leq i \leq d$.
Let the matrix $N \in 
{\rm Mat}_{d+1}(\mathbb F)$ represent
$X$ with respect to $\lbrace v_i \rbrace_{i=0}^d$.
Then $N$ is diagonal with $(i,i)$-entry $s_i$ for 
$0 \leq i \leq d$.
By linear algebra $M T' = T' N$.
In this equation, for $0 \leq i \leq d$
 compare the $(i,i)$-entry of
each side, to obtain $r_i=s_i$.
Until further notice assume that $A,B,C$ is nonbipartite.
Then $\alpha'_1 \not=0$. In the equation
 $M T' = T' N$, for $1 \leq i \leq d$
  compare the $(i-1,i)$-entry of each side, to
obtain $r_{i-1} = r_{i}$. So
$r_i=r_0$ for $0 \leq i \leq d$.
Consequently $X-r_0I$ vanishes
on $u_i$ for $0 \leq i \leq d$. Therefore $X=r_0I$.
Next assume that $A,B,C$ is bipartite.
Then $\alpha'_1 =0$ and $\alpha'_2 \not=0$.
In the equation
 $M T' = T' N$, for $2 \leq i \leq d$
 compare the $(i-2,i)$-entry of each side, to
obtain $r_{i-2} = r_{i}$.
For $0 \leq i \leq d$ we have $r_i = r_0$ (if $i$ is even)
and $r_i = r_1$ (if $i$ is odd).
Consequently $X-r_0 J - r_1 (I-J)$ vanishes
on $u_i$ for $0 \leq i \leq d$. Therefore
 $X=r_0 J + r_1 (I-J)$.
We have shown that the set $S$ spans $\mathcal I$. The result follows.
\end{proof}

\noindent We have some comments about
Proposition
\ref{prop:centBasis}(iii).

\begin{lemma}
\label{def:bipIC}
Assume that $A,B,C$ is bipartite and nontrivial.
Let $X$ denote an idempotent centralizer for $A,B,C$. Then
\begin{enumerate}
\item[\rm (i)] $XV_{\rm out} \subseteq V_{\rm out}$ 
and
$XV_{\rm in} \subseteq V_{\rm in}$; 
\item[\rm (ii)]
$XJ=JX$.
\end{enumerate}
\end{lemma}
\begin{proof} (i) By Lemmas
\ref{lem:V0V1},
\ref{lem:ICmeaning}.
\\
\noindent (ii) The map $J$ acts on $V_{\rm out}$ as the identity, and
on $V_{\rm in}$ as zero. The result follows from this
and (i) above.
\end{proof}

\begin{definition}
\label{def:ABCJ}
\rm
Assume that $A,B,C$ is bipartite and nontrivial.
Let $X$ denote an idempotent centralizer for $A,B,C$.
Then $X$ is called {\it outer} (resp. $\it inner$)
whenever $X$ is zero on $V_{\rm in}$ (resp. $V_{\rm out}$).
Let ${\mathcal I}_{\rm out}$
(resp. ${\mathcal I}_{\rm in}$) denote the set of
outer (resp. inner) idempotent centralizers for $A,B,C$.
Note that 
${\mathcal I}_{\rm out}$
and ${\mathcal I}_{\rm in}$ are ideals in the algebra
$\mathcal I$.
\end{definition}

\begin{proposition}
\label{prop:ICbasis}
Assume that $A,B,C$ is bipartite and nontrivial.
Then the following {\rm (i)--(iii)} hold:
\begin{enumerate}
\item[\rm (i)] the sum
$\mathcal I = {\mathcal I}_{\rm out} + {\mathcal I}_{\rm in}$
is direct;
\item[\rm (ii)] $J$ is a basis for ${\mathcal I}_{\rm out}$;
\item[\rm (iii)]
$I-J$ is a basis for ${\mathcal I}_{\rm in}$.
\end{enumerate}
\end{proposition}
\begin{proof}
By Definitions
\ref{def:J},
\ref{def:ABCJ}
we find that $J \in {\mathcal I}_{\rm out}$
and 
 $I-J \in {\mathcal I}_{\rm in}$.
Also ${\mathcal I}_{\rm out} \cap {\mathcal I}_{\rm in} = 0$
by Definition
\ref{def:ABCJ} and since the sum $V= V_{\rm out}+V_{\rm in}$
is direct.
The result follows in view of
Proposition
\ref{prop:centBasis}(iii).
\end{proof}

\section{The double lowering spaces for an LR triple}

\noindent
Throughout this section the following notation is in effect.
Let $V$ denote a vector space over $\mathbb F$ with
dimension $d+1$.
Let $A,B,C$ denote
an  LR triple on $V$, with
 parameter array
(\ref{eq:paLRT}),
idempotent data
(\ref{eq:idseq}),
trace data
(\ref{eq:tracedata}), and Toeplitz data
(\ref{eq:ToeplitzData}).
We discuss some subspaces of 
${\rm End}(V)$ called the double lowering spaces.

\begin{definition}
\rm 
 Let $\lbrace V_i\rbrace_{i=0}^d$ denote a
decomposition of $V$.
For $X \in 
{\rm End}(V)$, we say that $X$ {\it weakly lowers
  $\lbrace V_i\rbrace_{i=0}^d$}
whenever 
$XV_i \subseteq V_{i-1}$ for  $1 \leq i \leq d$
and 
$X V_0 = 0$.
\end{definition}

\begin{definition}
\label{def:DL}
\rm
Let $\overline A$ denote the set of elements in
${\rm End}(V)$ that weakly lower both
the $(A,B)$-decomposition of $V$
and the
 $(A,C)$-decomposition of $V$.
The sets $\overline B$, $\overline C$ are similarly defined.
Note that $\overline A$, $\overline B$, $\overline C$ are subspaces of the
$\mathbb F$-vector space
${\rm End}(V)$.
We call
$\overline A, \overline B, \overline C$ the {\it double lowering spaces}
for the LR triple $A,B,C$.
\end{definition}

\noindent We now describe the $\mathbb F$-vector space
 $\overline A$; similar results
hold for  $\overline B$ and  $\overline C$.

\begin{theorem}
\label{thm:DL}
The following {\rm (i)--(iii)} hold.
\begin{enumerate}
\item[\rm (i)] Assume that $A,B,C$ is trivial. Then $\overline A=0$.
\item[\rm (ii)]
 Assume that $A,B,C$ is nonbipartite.
Then $A$ is a basis for $\overline A$.  Moreover 
 $\overline A$
has dimension 1.
\item[\rm (ii)]
 Assume that $A,B,C$ is bipartite and nontrivial.
Then $A_{\rm out}, A_{\rm in}$ form a basis for $\overline A$. 
 Moreover $\overline A$
has dimension 2.
\end{enumerate}
\end{theorem}
\begin{proof} (i) $\overline A$ is zero on $E_0V$, and
$E_0V=V$.
\\
\noindent (ii), (iii)
Assume that $A,B,C$ is nontrivial.
Let the set $S$ consist of $A$ (if $A,B,C$ is nonbipartite)
and $A_{\rm out}, A_{\rm in}$ (if $A,B,C$ is bipartite).
 We show that $S$ is a basis for $\overline A$.
By 
Lemma
\ref{lem:AiAoLI}
and the construction,
$S$ is a linearly independent
subset of $\overline A$. We show that $S$ spans $\overline A$.
 Let $\lbrace u_i\rbrace_{i=0}^d$ denote
an $(A,C)$-basis of $V$, and let 
 $\lbrace v_i\rbrace_{i=0}^d$ denote
a compatible $(A,B)$-basis of $V$.
The transition matrix from
 $\lbrace u_i\rbrace_{i=0}^d$ to
 $\lbrace v_i\rbrace_{i=0}^d$ is the matrix $T'$ from
Definition
\ref{def:TTT}. The matrix $T'$ is upper triangular
and Toeplitz, with parameters $\lbrace \alpha'_i\rbrace_{i=0}^d$.
Let $X \in \overline A$.
The map $X$ weakly lowers the $(A,C)$-decomposition of $V$,
so there exist scalars $\lbrace r_i\rbrace_{i=1}^d$ in
$\mathbb F$ such that $Xu_i = r_i u_{i-1}$ for
$1 \leq i \leq d$ and $Xu_0=0$.
The map $X$ weakly lowers the $(A,B)$-decomposition of $V$,
so there exist scalars $\lbrace s_i\rbrace_{i=1}^d$ in
$\mathbb F$ such that $Xv_i = s_i v_{i-1}$ for
$1 \leq i \leq d$ and $Xv_0=0$.
Let the matrix $M \in 
{\rm Mat}_{d+1}(\mathbb F)$ represent
$X$ with respect to $\lbrace u_i \rbrace_{i=0}^d$.
Then $M$ has $(i-1,i)$-entry $r_i$ for $1 \leq i \leq d$,
and all other entries 0.
Let the matrix $N \in 
{\rm Mat}_{d+1}(\mathbb F)$ represent
$X$ with respect to $\lbrace v_i \rbrace_{i=0}^d$.
Then $N$ has $(i-1,i)$-entry $s_i$ for $1 \leq i \leq d$,
and all other entries 0.
By linear algebra $M T' = T' N$.
In this equation, for $1 \leq i \leq d$
 compare the $(i-1,i)$-entry of
each side, to obtain $r_i=s_i$.
Until further notice assume that $A,B,C$ is nonbipartite.
Then $\alpha'_1 \not=0$. In the equation
 $M T' = T' N$, for $1 \leq i \leq d-1$
  compare the $(i-1,i+1)$-entry of each side, to
obtain $r_i = r_{i+1}$. So
$r_i=r_1$ for $1 \leq i \leq d$.
By construction $A u_i = u_{i-1}$ for $1 \leq i \leq d$
and $Au_0=0$. By these comments $X-r_1A$ vanishes
on $u_i$ for $0 \leq i \leq d$. Therefore $X=r_1A$.
Next assume that $A,B,C$ is bipartite.
Then $\alpha'_1 =0$ and $\alpha'_2 \not=0$.
In the equation
 $M T' = T' N$, for $2 \leq i \leq d-1$
 compare the $(i-2,i+1)$-entry of each side, to
obtain $r_{i-1} = r_{i+1}$.
For $1 \leq i \leq d$ we have $r_i = r_2$ (if $i$ is even)
and $r_i = r_1$ (if $i$ is odd).
For $0 \leq i \leq d$ define $\varepsilon_i$ to be $0$ (if $i$ is even)
and $1$ (if $i$ is odd). By construction
 $A_{\rm out} u_i = (1-\varepsilon_i) u_{i-1}$
for $1 \leq i \leq d$ and 
 $A_{\rm out} u_0 = 0$.
Also by construction
 $A_{\rm in} u_i = \varepsilon_i u_{i-1}$
for $1 \leq i \leq d$ and 
 $A_{\rm in} u_0 = 0$.
By the above comments $X-r_1 A_{\rm in} - r_2 A_{\rm out}$ vanishes
on $u_i$ for $0 \leq i \leq d$. Therefore
 $X=r_1 A_{\rm in} + r_2 A_{\rm out}$.
We have shown that the set $S$ spans $\overline A$. The result follows.
\end{proof}

\section{The unipotent maps for an LR triple}

\noindent
Throughout this section the following notation is in effect.
Let $V$ denote a vector space over $\mathbb F$ with
dimension $d+1$.
Let $A,B,C$ denote
an  LR triple on $V$, with
 parameter array
(\ref{eq:paLRT}),
idempotent data
(\ref{eq:idseq}),
trace data
(\ref{eq:tracedata}), and Toeplitz data
(\ref{eq:ToeplitzData}).
Using $A,B,C$ we define three elements
$\mathbb A$,
$\mathbb B$,
$\mathbb C$ in
${\rm End}(V)$ called the unipotent maps.
This name is motivated by Lemma
\ref{lem:unip}
below.

\begin{definition}
\label{def:Del}
\rm
Define
\begin{eqnarray*}
{\mathbb A} = \sum_{i=0}^d E_{d-i} E''_{i},
\qquad \qquad 
{\mathbb B} = \sum_{i=0}^d E'_{d-i} E_{i},
\qquad \qquad 
{\mathbb C} = \sum_{i=0}^d E''_{d-i} E'_i.
\end{eqnarray*}
We call 
$\mathbb A$, 
$\mathbb B$, 
$\mathbb C$  the {\it unipotent maps} for
$A,B,C$. 
\end{definition}

\begin{lemma} 
\label{lem:UniTriv}
Assume that $A,B,C$ is trivial. Then
$\mathbb A=
\mathbb B=
\mathbb C = I$.
\end{lemma}
\begin{proof} For $d=0$ we have 
$E_0=E'_0 = E''_0=I$.
\end{proof}

\begin{lemma}
\label{lem:DeltaInv}
The maps $\mathbb A, \mathbb B, \mathbb C$ are invertible.
Their inverses are
\begin{eqnarray*}
\mathbb {A}^{-1} = \sum_{i=0}^d E''_{d-i} E_{i},
\qquad \qquad 
\mathbb B^{-1} = \sum_{i=0}^d E_{d-i} E'_{i},
\qquad \qquad 
\mathbb C^{-1} = \sum_{i=0}^d E'_{d-i} E''_i.
\end{eqnarray*}
\end{lemma}
\begin{proof}
Concerning $\mathbb A$ and using Lemma
\ref{lem:tripleProd},
\begin{eqnarray*}
\mathbb A \sum_{j=0}^d E''_{d-j} E_j= \sum_{j=0}^d E_jE''_{d-j}E_j
= \sum_{j=0}^d E_j = I.
\end{eqnarray*}
\end{proof}

\begin{lemma}
\label{lem:delA}
For $0 \leq i \leq d$,
\begin{eqnarray*}
\mathbb A E''_i = E_{d-i} \mathbb A,
\qquad \qquad
\mathbb B E_i = E'_{d-i} \mathbb B,
\qquad \qquad
\mathbb C E'_i = E''_{d-i} \mathbb C.
\end{eqnarray*} 
\end{lemma}
\begin{proof} These equations are verified
by evaluating each side using Definition
\ref{def:Del}.
\end{proof}

\begin{lemma}
\label{lem:moveDec}
For $0 \leq i \leq d$,
\begin{eqnarray*}
\mathbb A E''_iV = E_{d-i}V,
\qquad \qquad
\mathbb B E_iV = E'_{d-i}V,
\qquad \qquad
\mathbb C E'_iV = E''_{d-i}V.
\end{eqnarray*} 
\end{lemma}
\begin{proof}
Use Lemmas
\ref{lem:DeltaInv},
\ref{lem:delA}.
\end{proof}

\begin{lemma} \label{lem:AAdecSend}
 The following {\rm (i)--(iii)} hold:
\begin{enumerate}
\item[\rm (i)]  $\mathbb A$ sends the
$(A,C)$-decomposition of $V$ to the $(A,B)$-decomposition of $V$;
\item[\rm (ii)] $\mathbb B$ sends the
$(B,A)$-decomposition of $V$ to the $(B,C)$-decomposition of $V$;
\item[\rm (iii)] $\mathbb C$ sends the
$(C,B)$-decomposition of $V$ to the $(C,A)$-decomposition of $V$.
\end{enumerate}
\end{lemma}
\begin{proof} This is a reformulation of Lemma
\ref{lem:moveDec}.
\end{proof}

\noindent We now consider how the maps $\mathbb A,\mathbb B, \mathbb C$
act on the three flags 
(\ref{eq:LRTMOF}). 

\begin{lemma} The following {\rm (i)--(iii)} hold:
\begin{enumerate}
\item[\rm (i)]  $\mathbb A$ fixes 
$\lbrace A^{d-i}V\rbrace_{i=0}^d$ and sends
$\lbrace C^{d-i}V\rbrace_{i=0}^d$ to
$\lbrace B^{d-i}V\rbrace_{i=0}^d$;
\item[\rm (ii)]  $\mathbb B$ fixes 
$\lbrace B^{d-i}V\rbrace_{i=0}^d$ and sends
$\lbrace A^{d-i}V\rbrace_{i=0}^d$ to
$\lbrace C^{d-i}V\rbrace_{i=0}^d$;
\item[\rm (iii)]  $\mathbb C$ fixes 
$\lbrace C^{d-i}V\rbrace_{i=0}^d$ and sends
$\lbrace B^{d-i}V\rbrace_{i=0}^d$ to
$\lbrace A^{d-i}V\rbrace_{i=0}^d$.
\end{enumerate}
\end{lemma}
\begin{proof} 
By Lemmas
\ref{lem:LRTinducedFlag},
 \ref{lem:AAdecSend}.
\end{proof}

\noindent We now consider how the maps $\mathbb A,\mathbb B, \mathbb C$
act on some bases for $V$.

\begin{lemma}
\label{lem:bbAaction}
 The following {\rm (i)--(iii)} hold:
\begin{enumerate}
\item[\rm (i)] $\mathbb A$ sends each $(A,C)$-basis of $V$
to a compatible $(A,B)$-basis of $V$;
\item[\rm (ii)] $\mathbb B$ sends each $(B,A)$-basis of $V$
to a compatible $(B,C)$-basis of $V$;
\item[\rm (iii)] $\mathbb C$ sends each $(C,B)$-basis of $V$
to a compatible $(C,A)$-basis of $V$.
\end{enumerate}
\end{lemma}
\begin{proof}(i) Let $\lbrace u_i\rbrace_{i=0}^d$
denote an $(A,C)$-basis of $V$, and let
$\lbrace v_i \rbrace_{i=0}^d$ denote a compatible
$(A,B)$-basis of $V$. We have $Au_i=u_{i-1}$ for
$1 \leq i \leq d$ and $Au_0=0$. Also $v_i=B^iv_0/(\varphi_1\cdots \varphi_i)
$
for $0 \leq i \leq d$. Moreover $u_0=v_0$. 
For $0 \leq i \leq d$,
\begin{eqnarray*}
\mathbb A u_i = E_iE''_{d-i}u_i =
 E_i u_i = \frac{A^{d-i}B^dA^i}{\varphi_1 \cdots \varphi_d}
u_i = \frac{A^{d-i}B^du_0}{\varphi_1\cdots \varphi_d}
=
\frac{B^iu_0}{\varphi_1\cdots \varphi_i}=v_i.
\end{eqnarray*}
\\
\noindent (ii), (iii) Similar to the proof of (i) above.
\end{proof}

\begin{lemma} 
\label{lem:TandA}
The following {\rm (i)--(iii)} hold:
\begin{enumerate}
\item[\rm (i)] the matrix $T'$ represents $\mathbb A$ with
respect to each $(A,C)$-basis of $V$;
\item[\rm (ii)] the matrix $T''$ represents $\mathbb B$ with
respect to each $(B,A)$-basis of $V$;
\item[\rm (iii)] the matrix $T$ represents $\mathbb C$ with
respect to each $(C,B)$-basis of $V$.
\end{enumerate}
\end{lemma}
\begin{proof} By Definition
\ref{def:TTT}
and Lemma
\ref{lem:bbAaction}.
\end{proof}

\noindent 
Recall the vectors $\eta,\eta',\eta''$
 and $\tilde \eta, \tilde \eta', \tilde \eta''$
from (\ref{eq:eta}),
(\ref{eq:etaDual}).

\begin{lemma}
\label{lem:AiBiCi}
 For $0 \leq i \leq d$,
\begin{eqnarray*}
&&
\mathbb A C^i \eta = 
\frac{\varphi''_d \cdots \varphi''_{d-i+1}}{\varphi_1 \cdots \varphi_i}\,B^i \eta,
\qquad \qquad 
\mathbb A^{-1} B^i \eta = \frac{\varphi_1 \cdots \varphi_{i}}{\varphi''_d \cdots \varphi''_{d-i+1}}\,C^i \eta,
\\
&&
\mathbb B A^i \eta' = 
\frac{\varphi_d \cdots \varphi_{d-i+1}}{\varphi'_1 \cdots \varphi'_i}\,C^i \eta',
\qquad \qquad 
\mathbb B^{-1} C^i \eta' = \frac{\varphi'_1 \cdots \varphi'_{i}}{\varphi_d \cdots \varphi_{d-i+1}}\,A^i \eta',
\\
&&
\mathbb C B^i \eta'' = 
\frac{\varphi'_d \cdots \varphi'_{d-i+1}}{\varphi''_1 \cdots \varphi''_i}\,A^i \eta'',
\qquad \qquad 
\mathbb C^{-1} A^i \eta'' = \frac{\varphi''_1 \cdots \varphi''_{i}}{\varphi'_d \cdots \varphi'_{d-i+1}}\,B^i \eta''.
\end{eqnarray*}
\end{lemma}
\begin{proof}
By Lemmas
\ref{lem:moreTTTp},
\ref{lem:bbAaction}.
\end{proof}

\begin{lemma}
\label{lem:AAi}
For $0 \leq i \leq d$,
\begin{eqnarray*}
&&
\mathbb A A^i \eta''
 =
 \frac{(\eta'',\tilde \eta)}{(\eta',\tilde \eta)} A^i \eta',
\qquad \qquad 
\mathbb A^{-1} A^i \eta'
 =
 \frac{(\eta',\tilde \eta)}{(\eta'',\tilde \eta)} A^i \eta'',
\\
&&
\mathbb B B^i \eta
 =
 \frac{(\eta,\tilde \eta')}{(\eta'',\tilde \eta')} B^i \eta'',
\qquad \qquad 
\mathbb B^{-1} B^i \eta''
 =
 \frac{(\eta'',\tilde \eta')}{(\eta,\tilde \eta')} B^i \eta,
\\
&&
\mathbb C C^i \eta'
 =
 \frac{(\eta',\tilde \eta'')}{(\eta,\tilde \eta'')} C^i \eta,
\qquad \qquad 
\mathbb C^{-1} C^i \eta
 =
 \frac{(\eta,\tilde \eta'')}{(\eta',\tilde \eta'')} C^i \eta'.
\end{eqnarray*}
\end{lemma}
\begin{proof} 
Evaluate the displayed equations
in 
Lemma
\ref{lem:AiBiCi}
using
Lemma
\ref{lem:COB}, and simplify the results using
Lemma
\ref{lem:NZ}
and Proposition
\ref{prop:sixBil}.
\end{proof}

\begin{lemma}
\label{lem:iIszero}
The following {\rm (i)--(iii)} hold:
\begin{enumerate}
\item[\rm (i)] $\mathbb A$ fixes $\eta$ and sends
$\eta''$ to $(\eta'',\tilde \eta)/(\eta',\tilde \eta)\eta'$;
\item[\rm (ii)] $\mathbb B$ fixes $\eta'$ and sends
$\eta$ to $(\eta,\tilde \eta')/(\eta'',\tilde \eta')\eta''$;
\item[\rm (iii)] $\mathbb C$ fixes $\eta''$ and sends
$\eta'$ to $(\eta',\tilde \eta'')/(\eta,\tilde \eta'')\eta$.
\end{enumerate}
\end{lemma}
\begin{proof} Set $i=0$ in Lemmas
\ref{lem:AiBiCi},
\ref{lem:AAi}.
\end{proof}

\begin{proposition}
\label{prop:Apoly}
We have
\begin{eqnarray}
\label{eq:rotator1}
\mathbb A = \sum_{i=0}^d \alpha'_i A^i,
\qquad \qquad 
\mathbb B = \sum_{i=0}^d \alpha''_i B^i,
\qquad \qquad 
\mathbb C = \sum_{i=0}^d \alpha_i C^i.
\end{eqnarray}
Moreover
\begin{eqnarray}
\label{eq:rotator2}
\mathbb A^{-1} = \sum_{i=0}^d \beta'_i A^i,
\qquad \qquad 
\mathbb B^{-1} = \sum_{i=0}^d \beta''_i B^i,
\qquad \qquad 
\mathbb C^{-1} = \sum_{i=0}^d \beta_i C^i.
\end{eqnarray}
\end{proposition}
\begin{proof} We verify
$\mathbb A = \sum_{i=0}^d \alpha'_i A^i$.
Let $\lbrace u_i\rbrace_{i=0}^d$ denote an $(A,C)$-basis
of $V$.
Recall the matrix $\tau$ from Definition
\ref{def:tauMat}.
By Proposition
\ref{prop:matrixRep},
$\tau$ represents $A$
with respect to
$\lbrace u_i\rbrace_{i=0}^d$.
By Lemma \ref{lem:ToeChar}, $T' = \sum_{i=0}^d \alpha'_i \tau^i$.
So $T'$ represents 
 $\sum_{i=0}^d \alpha'_i A^i$ with respect to
$\lbrace u_i\rbrace_{i=0}^d$.
By Lemma
\ref{lem:TandA}(i), $T'$ represents $\mathbb A$ with respect to
$\lbrace u_i\rbrace_{i=0}^d$.
Therefore $\mathbb A = \sum_{i=0}^d \alpha'_i A^i$.
The remaining assertions of the lemma
 are similarly verified.
\end{proof}

\noindent We emphasize one aspect of Proposition
\ref{prop:Apoly}.

\begin{corollary}
\label{cor:AAcom}
The element $A$ (resp. $B$) (resp. $C$) commutes with
$\mathbb A$ (resp. $\mathbb B$) (resp. $\mathbb C$).
\end{corollary}

\noindent An element $X \in {\rm End}(V)$ is
called {\it unipotent} whenever $X-I$ is nilpotent.

\begin{lemma}
\label{lem:unip}
Each of $\mathbb A$,
 $\mathbb B$,
$\mathbb C$ is unipotent.
\end{lemma}
\begin{proof} The element
$\mathbb A-I$ is nilpotent, since
it is a linear combination of $\lbrace A^i\rbrace_{i=1}^d$
and $A$ is nilpotent. Therefore $\mathbb A$
is unipotent. The maps
$\mathbb B$, $\mathbb C$ are similarly shown
to be unipotent.
\end{proof}

\begin{definition}
\label{def:unidata}
\rm Call the sequence
$\mathbb A, \mathbb B, \mathbb C$ the
{\it unipotent data} for $A,B,C$.
\end{definition}

\begin{lemma}
\label{lem:uniNew}
Let $\alpha, \beta, \gamma$ denote nonzero scalars in $\mathbb F$.
Then the LR triples $A,B,C$ and 
$\alpha A, \beta B, \gamma C$ have the same unipotent data.
\end{lemma}
\begin{proof}
By Lemma
\ref{lem:isequal}
and Definition
\ref{def:Del}.
\end{proof}

\begin{lemma} 
\label{lem:newUniData}
In the table below, we display
some LR triples on $V$ along with their unipotent data.

     \bigskip

\centerline{
\begin{tabular}[t]{c|c}
 {\rm LR triple} & {\rm unipotent data}
 \\
 \hline
 \hline
 $ A,  B, C$ &
$\mathbb A, \mathbb B, \mathbb C$
   \\
 $ B, C, A$ &
$\mathbb B, \mathbb C, \mathbb A$
   \\
 $C, A, B$ &
$\mathbb C, \mathbb A, \mathbb B$
   \\
 \hline
 $ C, B, A$ &
$\mathbb C^{-1}, \mathbb B^{-1}, \mathbb A^{-1}$
   \\
 $ A, C, B$ &
$\mathbb A^{-1}, \mathbb C^{-1}, \mathbb B^{-1}$
   \\
 $ B, A,  C$ &
$\mathbb B^{-1}, \mathbb A^{-1}, \mathbb C^{-1}$
   \end{tabular}}
     \bigskip

\end{lemma}
\begin{proof}
Use Definition
\ref{def:Del}
and 
Lemmas
\ref{lem:ABCEvar},
\ref{lem:DeltaInv}.
\end{proof}

\begin{lemma}
\label{lem:newUniDataDual}
In the table below, we display
some LR triples on $V^*$ along with their unipotent data.

     \bigskip

\centerline{
\begin{tabular}[t]{c|c}
 {\rm LR triple} & {\rm unipotent data}
 \\
 \hline
 \hline
 $ \tilde A,  \tilde B,\tilde C$ &
 $\tilde {\mathbb A}^{-1}, \tilde{ \mathbb B}^{-1}, \tilde {\mathbb C}^{-1}$
   \\
 $ \tilde B, \tilde C, \tilde A$ &
$\tilde {\mathbb B}^{-1}, \tilde {\mathbb C}^{-1}, \tilde {\mathbb A}^{-1}$
   \\
 $\tilde C, \tilde A, \tilde B$ &
$\tilde {\mathbb C}^{-1}, \tilde {\mathbb A}^{-1}, \tilde {\mathbb B}^{-1}$
   \\
 \hline
 $ \tilde C, \tilde B, \tilde A$ &
$\tilde {\mathbb C}, \tilde {\mathbb B}, \tilde {\mathbb A}$
   \\
 $ \tilde A, \tilde C, \tilde B$ &
$\tilde{ \mathbb A}, \tilde {\mathbb C}, \tilde {\mathbb B}$
   \\
 $ \tilde B, \tilde A,  \tilde C$ &
$\tilde {\mathbb B}, \tilde {\mathbb A}, \tilde {\mathbb C}$
\end{tabular}}
     \bigskip

\end{lemma}
\begin{proof} 
Use Definition
\ref{def:Del}
and 
 Lemmas
\ref{lem:tildeABCEvar},
\ref{lem:DeltaInv}. Keep in mind that the
adjoint map is an antiisomorphism.
\end{proof}

\begin{lemma} Assume that $A,B,C$ is bipartite.
Then the projector $J$ commutes with each of $\mathbb A$,
$\mathbb B$,
$\mathbb C$.
\end{lemma}
\begin{proof} 
By Definition
\ref{def:Del},
and since
$J$ commutes with each of
$E_i, E'_i, E''_i$ for
$0 \leq i \leq d$.
\end{proof}

\begin{lemma} Assume that $A,B,C$ is bipartite.
Then
\begin{eqnarray*}
&&
\mathbb A V_{\rm out} = V_{\rm out},
\qquad \qquad
\mathbb B V_{\rm out} = V_{\rm out},
\qquad \qquad
\mathbb C V_{\rm out} = V_{\rm out},
\\
&&
\mathbb A V_{\rm in} = V_{\rm in},
\qquad \qquad \;\;
\mathbb B V_{\rm in} = V_{\rm in},
\qquad \qquad 
\quad 
\mathbb C V_{\rm in} = V_{\rm in}.
\end{eqnarray*}
\end{lemma}
\begin{proof} By
Lemma
\ref{lem:V0V1},
Definition
\ref{def:Del}, and since $d$ is even,
we find that $V_{\rm out}$ and
$V_{\rm in}$ are invariant under each of
$\mathbb A, \mathbb B, \mathbb C$.
By Lemma
\ref{lem:DeltaInv} the maps
$\mathbb A, \mathbb B, \mathbb C$
are invertible.
The result follows.
\end{proof}

\noindent The next two lemmas follow from the
construction.

\begin{lemma}
\label{lem:unipA2}
Assume that $A,B,C$ is bipartite, so that
$A^2,B^2,C^2$ act on $V_{\rm out}$  as
an LR triple. The unipotent data for this triple
is given by the actions of
$\mathbb A, \mathbb B, \mathbb C$  on
$V_{\rm out}$.
\end{lemma}

\begin{lemma}
\label{lem:unipA2In}
Assume that $A,B,C$ is bipartite and nontrivial, so that
$A^2,B^2,C^2$ act on $V_{\rm in}$  as
an LR triple. The unipotent data for this triple
is given by the actions of
$\mathbb A, \mathbb B, \mathbb C$  on
$V_{\rm in}$.
\end{lemma}

\begin{lemma}
Assume that $A,B,C$ is bipartite.
Let
\begin{eqnarray*}
\alpha_{\rm out}, \quad 
\alpha_{\rm in}, \quad 
\beta_{\rm out}, \quad 
\beta_{\rm in}, \quad 
\gamma_{\rm out}, \quad 
\gamma_{\rm in}
\end{eqnarray*}
denote nonzero scalars in $\mathbb F$, so that the sequence
\begin{eqnarray*}
\alpha_{\rm out}A_{\rm out}
+\alpha_{\rm in}A_{\rm in},
\qquad
\beta_{\rm out} B_{\rm out}+
\beta_{\rm in}B_{\rm in},
\qquad
\gamma_{\rm out}C_{\rm out}+
\gamma_{\rm in} C_{\rm in}
\end{eqnarray*}
is a bipartite LR triple on $V$. This LR triple
has the same unipotent data as $A,B,C$.
\end{lemma}
\begin{proof}
By Lemma
\ref{lem:PAnewLRT}(ii) and
Definition
\ref{def:Del}.
\end{proof}

\section{The rotators for an LR triple} 

\noindent 
Throughout this section the following notation is in effect.
Let $V$ denote a vector space over $\mathbb F$ with
dimension $d+1$.
Let $A,B,C$ denote
an LR triple on $V$, with
 parameter array
(\ref{eq:paLRT}),
idempotent data
(\ref{eq:idseq}),
trace data
(\ref{eq:tracedata}), and Toeplitz data
(\ref{eq:ToeplitzData}).
We discuss a type of element in 
${\rm End}(V)$ called a rotator.

\begin{definition}
\label{def:ABCROT}
\rm
By
a {\it rotator} for $A,B,C$ we mean
an element
$R \in {\rm End}(V)$ such that
for $0 \leq i \leq d$,
\begin{eqnarray}
\label{eq:ROT}
E_i R  =  R E'_i, \qquad \qquad
E'_i R = R E''_i, \qquad \qquad
E''_i R  = R E_i.
\end{eqnarray}
\end{definition}

\begin{lemma}
\label{lem:ROTmeaning}
For $R \in {\rm End}(V)$ the following are equivalent:
\begin{enumerate}
\item[\rm (i)] 
$R$ is a rotator for $A,B,C$;
\item[\rm (ii)] 
for $0 \leq i \leq d$,
\begin{eqnarray*}
&&
R E'_iV \subseteq E_iV,
\;\qquad \qquad
R E''_iV \subseteq  E'_iV,
\;\qquad \qquad
R E_iV \subseteq  E''_iV.
\end{eqnarray*}
\end{enumerate}
\end{lemma}
\begin{proof}
By Definition
\ref{def:ABCROT}
and linear algebra.
\end{proof}

\begin{lemma}
\label{ex:trivRot}
\rm Assume that $A,B,C$ is trivial. Then
each element of ${\rm End}(V)$ is a rotator
for $A,B,C$.
\end{lemma}
\begin{proof} For $d=0$ we have
$E_0=E'_0=E''_0=I$.
\end{proof}

\begin{lemma}
Let $R$ denote a rotator for $A,B,C$. Then
\begin{eqnarray}
\mathbb A R = R \mathbb B, \qquad \qquad
\mathbb B R = R \mathbb C, \qquad \qquad
\mathbb C R = R \mathbb A.
\label{eq:RotM2}
\end{eqnarray}
\end{lemma}
\begin{proof} Use Definitions
\ref{def:Del},
\ref{def:ABCROT}.
\end{proof}

\begin{definition}
\label{def:ABCROTspace} 
\rm
Let $\mathcal R$ denote the  set of rotators for
$A,B,C$. Note that $\mathcal R$ is a subspace of
the $\mathbb F$-vector space ${\rm End}(V)$.
We call $\mathcal R$ the {\it rotator space} for
$A,B,C$.
\end{definition}


\begin{definition}
\label{lem:Rbasis}
\rm Assume that $A,B,C$ is trivial.
Then the identity $I$ of ${\rm End}(V)=\mathcal R$
is a basis for $\mathcal R$.
We call $I$ the {\it standard rotator} for $A,B,C$.
\end{definition}

\noindent Assume for the moment that $A,B,C$ is nontrivial.
We are going to show that $\mathcal R$ has dimension
1 (if $A,B,C$ is nonbipartite) and 2
(if $A,B,C$ is bipartite). In each case,
we will display an explicit basis for
$\mathcal R$.
We now obtain some results that will be used to
construct these bases.

\begin{lemma}
\label{lem:3rot1} The following {\rm (i)--(iii)} hold.
\begin{enumerate}
\item[\rm (i)] $\mathbb B^{-1} C \mathbb B$ is zero on $E_0V$. Moreover for
$1 \leq i \leq d$ and on $E_iV$,
\begin{eqnarray*}
\mathbb B^{-1} C \mathbb B = \frac{\varphi'_{d-i+1}}{\varphi_{i}} A.
\end{eqnarray*}
\item[\rm (ii)] $\mathbb C^{-1} A \mathbb C$ is zero on $E'_0V$. Moreover for
$1 \leq i \leq d$ and on $E'_iV$,
\begin{eqnarray*}
\mathbb C^{-1} A \mathbb C = \frac{\varphi''_{d-i+1}}{\varphi'_{i}} B.
\end{eqnarray*}
\item[\rm (iii)] $\mathbb A^{-1} B \mathbb A$ is zero on $E''_0V$. Moreover for
$1 \leq i \leq d$ and on $E''_iV$,
\begin{eqnarray*}
\mathbb A^{-1} B \mathbb A = \frac{\varphi_{d-i+1}}{\varphi''_{i}} C.
\end{eqnarray*}
\end{enumerate}
\end{lemma}
\begin{proof}
(i) The vector $A^{d-i}\eta'$ is a basis for
$E_iV$. Apply $\mathbb B^{-1}C\mathbb B$ to this vector and
evaluate the result using 
Lemma
\ref{lem:AiBiCi} (middle row).
\\
\noindent (ii), (iii) Similar to the proof of (i) above.
\end{proof}

\begin{lemma}
\label{lem:3rot2} The following {\rm (i)--(iii)} hold.
\begin{enumerate}
\item[\rm (i)] $\mathbb A C \mathbb A^{-1}$ 
is zero on $E_dV$. Moreover for
$0 \leq i \leq d-1$ and on $E_iV$,
\begin{eqnarray*}
\mathbb A C \mathbb A^{-1} = \frac{\varphi''_{d-i}}{\varphi_{i+1}} B.
\end{eqnarray*}
\item[\rm (ii)] $\mathbb B A \mathbb B^{-1}$ is zero on $E'_dV$. Moreover for
$0 \leq i \leq d-1$ and on $E'_iV$,
\begin{eqnarray*}
\mathbb B A \mathbb B^{-1} = \frac{\varphi_{d-i}}{\varphi'_{i+1}} C.
\end{eqnarray*}
\item[\rm (iii)] $\mathbb C B \mathbb C^{-1}$ 
is zero on $E''_dV$. Moreover for
$0 \leq i \leq d-1$ and on $E''_iV$,
\begin{eqnarray*}
\mathbb C B \mathbb C^{-1} = \frac{\varphi'_{d-i}}{\varphi''_{i+1}} A.
\end{eqnarray*}
\end{enumerate}
\end{lemma}
\begin{proof}
(i) The vector $B^{i}\eta$ is a basis for
$E_iV$. Apply $\mathbb A C\mathbb A^{-1}$ to this vector and
evaluate the result using 
Lemma
\ref{lem:AiBiCi} (top row).
\\
\noindent (ii), (iii) Similar to the proof of (i) above.
\end{proof}

\begin{lemma}
\label{lem:newLR}
The following {\rm (i)--(iii)} hold:
\begin{enumerate}
\item[\rm (i)] the $(A,B)$-decomposition of $V$
is lowered by $\mathbb B^{-1}C\mathbb B$ and raised by
$\mathbb A C \mathbb A^{-1}$;
\item[\rm (ii)] the $(B,C)$-decomposition of $V$
is lowered by $\mathbb C^{-1}A\mathbb C$ and raised by
$\mathbb B A \mathbb B^{-1}$;
\item[\rm (iii)] the $(C,A)$-decomposition of $V$
is lowered by $\mathbb A^{-1}B\mathbb A$ and raised by
$\mathbb C B \mathbb C^{-1}$.
\end{enumerate}
\end{lemma}
\begin{proof} (i) The sequence $\lbrace E_iV\rbrace_{i=0}^d$
is 
the $(A,B)$-decomposition of
$V$. This decomposition is lowered by $A$ and raised by $B$. The result follows
in view of Lemmas
\ref{lem:3rot1}(i),
\ref{lem:3rot2}(i).
\end{proof}

\begin{lemma}
\label{lem:3and3}
 We have
\begin{eqnarray*}
&&
\mathbb B^{-1} C \mathbb B \Biggl(
\sum_{i=0}^d \frac{\varphi_1  \cdots \varphi_i}
{\varphi'_d  \cdots \varphi'_{d-i+1}} E_i
\Biggr) =
\Biggl(
\sum_{i=0}^d \frac{\varphi_1 \cdots \varphi_i}
{\varphi'_d  \cdots \varphi'_{d-i+1}} E_i
\Biggr) A,
\\
&&
\mathbb C^{-1} A \mathbb C \Biggl(
\sum_{i=0}^d \frac{\varphi'_1  \cdots \varphi'_i}
{\varphi''_d  \cdots \varphi''_{d-i+1}} E'_i
\Biggr) =
\Biggl(
\sum_{i=0}^d \frac{\varphi'_1 \cdots \varphi'_i}
{\varphi''_d  \cdots \varphi''_{d-i+1}} E'_i
\Biggr) B,
\\
&&
\mathbb A^{-1} B \mathbb A \Biggl(
\sum_{i=0}^d \frac{\varphi''_1  \cdots \varphi''_i}
{\varphi_d  \cdots \varphi_{d-i+1}} E''_i
\Biggr) =
\Biggl(
\sum_{i=0}^d \frac{\varphi''_1  \cdots \varphi''_i}
{\varphi_d  \cdots \varphi_{d-i+1}} E''_i
\Biggr) C
\end{eqnarray*}
and also
\begin{eqnarray*}
&&
\Biggl(\sum_{i=0}^d \frac{\varphi_1  \cdots \varphi_i}
{\varphi''_d \cdots \varphi''_{d-i+1}} E_i\Biggr) \mathbb A C \mathbb A^{-1} = B 
\Biggl(\sum_{i=0}^d \frac{\varphi_1 \cdots \varphi_i}
{\varphi''_d 
\cdots \varphi''_{d-i+1}} E_i\Biggr),
\\
&&
\Biggl(\sum_{i=0}^d \frac{\varphi'_1 \cdots \varphi'_i}
{\varphi_d\cdots \varphi_{d-i+1}} E'_i\Biggr) \mathbb B A \mathbb B^{-1} = C
\Biggl(\sum_{i=0}^d \frac{\varphi'_1 \cdots \varphi'_i}
{\varphi_d \cdots \varphi_{d-i+1}} E'_i\Biggr),
\\
&&
\Biggl(\sum_{i=0}^d \frac{\varphi''_1  \cdots \varphi''_i}
{\varphi'_d \cdots \varphi'_{d-i+1}} E''_i\Biggr) 
\mathbb C B \mathbb C^{-1} = A 
\Biggl(\sum_{i=0}^d \frac{\varphi''_1 \cdots \varphi''_i}
{\varphi'_d \cdots \varphi'_{d-i+1}} E''_i\Biggr).
\end{eqnarray*}
\end{lemma}
\begin{proof} 
To verify the first (resp. fourth) displayed equation in the lemma statement,
for $0 \leq i \leq d$ apply each side to $E_iV$,
and evaluate the result using
Lemma \ref{lem:3rot1}(i) (resp. Lemma \ref{lem:3rot2}(i)).
 The remaining equations are similarly
verified.
\end{proof}

\noindent For the next few results, it is convenient
to assume that $A,B,C$ is equitable. Shortly we will
return to the general case.

\begin{proposition}
\label{prop:EquitWOW}
Assume that $A,B,C$ is equitable. Then
\begin{eqnarray}
&&
\mathbb C 
\Biggl(\sum_{i=0}^d \frac{\varphi'_1 \cdots \varphi'_i}
{\varphi''_d \cdots \varphi''_{d-i+1}} E'_i\Biggr) \mathbb B 
=
\mathbb A 
\Biggl(\sum_{i=0}^d \frac{\varphi''_1 \cdots \varphi''_i}
{\varphi'_d \cdots \varphi'_{d-i+1}} E''_i\Biggr) \mathbb C, 
\label{eq:Om2}
\\
&&
\mathbb A 
\Biggl(\sum_{i=0}^d \frac{\varphi''_1 \cdots \varphi''_i}
{\varphi_d \cdots \varphi_{d-i+1}} E''_i\Biggr) \mathbb C 
=
\mathbb B 
\Biggl(\sum_{i=0}^d \frac{\varphi_1 \cdots \varphi_i}
{\varphi''_d \cdots \varphi''_{d-i+1}} E_i\Biggr) \mathbb A,
\label{eq:Om3}
\\
&&
\mathbb B 
\Biggl(\sum_{i=0}^d \frac{\varphi_1 \cdots \varphi_i}
{\varphi'_d \cdots \varphi'_{d-i+1}} E_i\Biggr) \mathbb A 
=
\mathbb C
\Biggl(\sum_{i=0}^d \frac{\varphi'_1 \cdots \varphi'_i}
{\varphi_d \cdots \varphi_{d-i+1}} E'_i\Biggr) \mathbb B.
\label{eq:Om1}
\end{eqnarray}
\end{proposition}
\begin{proof} We prove
(\ref{eq:Om2}).
Define
\begin{eqnarray}
\label{eq:XX}
X = \mathbb C 
\Biggl(\sum_{i=0}^d \frac{\varphi'_1 \cdots \varphi'_i}
{\varphi''_d \cdots \varphi''_{d-i+1}} E'_i\Biggr).
\end{eqnarray}
We claim that
\begin{eqnarray}
\label{eq:XY}
X 
=
\Biggl(\sum_{i=0}^d \frac{\varphi''_1 \cdots \varphi''_i}
{\varphi'_d \cdots \varphi'_{d-i+1}} E''_i\Biggr) 
\mathbb C.
\end{eqnarray}
To verify
(\ref{eq:XY}), 
evaluate the right-hand side of
(\ref{eq:XX}) using Lemma
\ref{lem:delA}, and simplify the result using
Lemma \ref{lem:combine}.
The claim is proven.
By
(\ref{eq:XY}) 
 and the last diplayed equation in 
Lemma
\ref{lem:3and3}, $XB=AX$. So $XB^i=A^iX$ for $0 \leq i \leq d$.
By  Definition
\ref{def:equitNorm}
and
line (\ref{eq:rotator1}),
$\mathbb A = \sum_{i=0}^d \alpha_i A^i$ and
$\mathbb B = \sum_{i=0}^d \alpha_i B^i$.
By these comments
$X \mathbb B = \mathbb A X$.
In this equation evaluate the
$X$ on the left and  right using
(\ref{eq:XX}) and
(\ref{eq:XY}),
 respectively.
This yields 
(\ref{eq:Om2}).
The equations
(\ref{eq:Om3}),
(\ref{eq:Om1}) are similarly obtained.
\end{proof}

\begin{definition}
\label{def:Om123}
\rm
Assume that $A,B,C$ is equitable.
Let $\Omega, \Omega', \Omega''$ denote the common
values of 
{\rm (\ref{eq:Om2})},
{\rm (\ref{eq:Om3})},
{\rm (\ref{eq:Om1})} respectively.
\end{definition}

\begin{lemma}
\label{lem:OmegaTriv}
Assume that $A,B,C$ is trivial. Then
$\Omega = \Omega' = \Omega''=I$.
\end{lemma}
\begin{proof}
By Lemma \ref{lem:UniTriv},
Definition
\ref{def:Om123}, and since
$E_0 = E'_0 = E''_0=I$.
\end{proof}

\begin{lemma} 
\label{lem:OmegaCom2}
Assume that $A,B,C$ is equitable. Then for $0 \leq i \leq d$,
\begin{eqnarray*}
E_i  \Omega  = \Omega E'_i, \qquad \qquad
E'_i \Omega'  = \Omega' E''_i, \qquad \qquad
E''_i \Omega''  = \Omega'' E_i.
\end{eqnarray*}
\end{lemma}
\begin{proof}
To verify
$E_i  \Omega  = \Omega E'_i$,
eliminate $\Omega$ using the formula on
the right in
(\ref{eq:Om2}), and evaluate the
result using
 Lemma
\ref{lem:delA}.
The remaining equations are similary verified.
\end{proof}

\begin{lemma} 
\label{lem:OmegaCom}
Assume that $A,B,C$ is equitable. Then 
\begin{eqnarray}
\label{eq:AOmB}
A \Omega  = \Omega B, \qquad \qquad
B \Omega'  = \Omega' C, \qquad \qquad
C \Omega''  = \Omega'' A.
\end{eqnarray}
\end{lemma}
\begin{proof}
To verify $A \Omega = \Omega B$,
eliminate $\Omega$ using the formula
on the right in 
(\ref{eq:Om2}), and evaluate the result
using 
Corollary
\ref{cor:AAcom}
and the last displayed equation in
Lemma
\ref{lem:3and3}. The remaining equations in
(\ref{eq:AOmB})
are similarly
verified.
\end{proof}


\begin{lemma}
\label{lem:OmInv}
 Assume that $A,B,C$ is equitable.
Then the elements $\Omega, \Omega', \Omega''$ are invertible.
Moreover 
\begin{eqnarray*}
\Omega^{-1} = 
\mathbb B^{-1} 
\Biggl(\sum_{i=0}^d \frac{\varphi''_d \cdots \varphi''_{d-i+1}}
{\varphi'_1 \cdots \varphi'_i} E'_i\Biggr) \mathbb C^{-1} 
=
\mathbb C^{-1} 
\Biggl(\sum_{i=0}^d \frac{\varphi'_d\cdots \varphi'_{d-i+1}}
{\varphi''_1 \cdots \varphi''_i} E''_i\Biggr) \mathbb A^{-1}, 
&&
\\
(\Omega')^{-1}=
\mathbb C^{-1} 
\Biggl(\sum_{i=0}^d \frac{\varphi_d \cdots \varphi_{d-i+1}}
{\varphi''_1 \cdots \varphi''_i} E''_i\Biggr) \mathbb A^{-1} 
=
\mathbb A^{-1} 
\Biggl(\sum_{i=0}^d \frac{\varphi''_d\cdots \varphi''_{d-i+1}}
{\varphi_1 \cdots \varphi_i} E_i\Biggr) \mathbb B^{-1}, 
&&
\\
(\Omega'')^{-1}=
\mathbb A^{-1} 
\Biggl(\sum_{i=0}^d \frac{\varphi'_d \cdots \varphi'_{d-i+1}}
{\varphi_1 \cdots \varphi_i} E_i\Biggr) \mathbb B^{-1} 
=
\mathbb B^{-1} 
\Biggl(\sum_{i=0}^d \frac{\varphi_d\cdots \varphi_{d-i+1}}
{\varphi'_1 \cdots \varphi'_i} E'_i\Biggr) \mathbb C^{-1}. 
&&
\end{eqnarray*}
\end{lemma}
\begin{proof}
Use Proposition
\ref{prop:EquitWOW}
and Definition
\ref{def:Om123}.
\end{proof}


\begin{lemma}
\label{prop:NBequit}
Assume that $A,B,C$ is equitable and nonbipartite.
Then $\Omega = \Omega'=\Omega''$, and this 
common value is equal to
\begin{eqnarray*}
&&\mathbb B 
\Biggl(\sum_{i=0}^d \frac{\varphi_1 \cdots \varphi_i}
{\varphi_d \cdots \varphi_{d-i+1}} E_i\Biggr) \mathbb A
= 
\mathbb C 
\Biggl(\sum_{i=0}^d \frac{\varphi_1 \cdots \varphi_i}
{\varphi_d \cdots \varphi_{d-i+1}} E'_i\Biggr) \mathbb B 
\nonumber
\\
&& \qquad \qquad \qquad 
=
\mathbb A 
\Biggl(\sum_{i=0}^d \frac{\varphi_1 \cdots \varphi_i}
{\varphi_d \cdots \varphi_{d-i+1}} E''_i\Biggr) \mathbb C.
\end{eqnarray*}
\end{lemma}
\begin{proof}
By Lemma
\ref{lem:equitBasic}(i)
along with
(\ref{eq:Om2})--(\ref{eq:Om1})
and Definition 
\ref{def:Om123}.
\end{proof}
\noindent For the past few results we assumed  
that $A,B,C$ is equitable. We now drop the equitable assumption
and return to the general case.

\begin{theorem}
\label{cor:everythingGen}
\label{cor:everything}
Assume that $A,B,C$ is nonbipartite. Then the following
{\rm (i)--(v)} hold.
\begin{enumerate}
\item[\rm (i)] We have
\begin{eqnarray*}
&&\mathbb B 
\Biggl(\sum_{i=0}^d \frac{\varphi_1 \cdots \varphi_i}
{\varphi_d \cdots \varphi_{d-i+1}} E_i\Biggr) \mathbb A
= 
\mathbb C 
\Biggl(\sum_{i=0}^d \frac{\varphi_1 \cdots \varphi_i}
{\varphi_d \cdots \varphi_{d-i+1}} E'_i\Biggr) \mathbb B 
\\
&& \qquad \qquad \qquad 
=
\mathbb A 
\Biggl(\sum_{i=0}^d \frac{\varphi_1 \cdots \varphi_i}
{\varphi_d \cdots \varphi_{d-i+1}} E''_i\Biggr) \mathbb C.
\end{eqnarray*}
Denote this common value by $\Omega$.
\item[\rm (ii)]
 $\Omega$
is invertible, and $\Omega^{-1}$ is equal to
\begin{eqnarray*}
&&
\mathbb A^{-1} 
\Biggl(\sum_{i=0}^d \frac{\varphi_d \cdots \varphi_{d-i+1}}
{\varphi_1 \cdots \varphi_{i}} E_i\Biggr) \mathbb B^{-1} 
= 
\mathbb B^{-1} 
\Biggl(\sum_{i=0}^d \frac{\varphi_d \cdots \varphi_{d-i+1}}
{\varphi_1 \cdots \varphi_{i}} E'_i\Biggr) \mathbb C^{-1} 
\nonumber
\\
&& \qquad \qquad \qquad 
=
\mathbb C^{-1} 
\Biggl(\sum_{i=0}^d \frac{\varphi_d \cdots \varphi_{d-i+1}}
{\varphi_1 \cdots \varphi_{i}} E''_i\Biggr) \mathbb A^{-1}.
\end{eqnarray*}
\item[\rm (iii)]
For
 $0 \leq i \leq d$,
\begin{eqnarray*}
E_i \Omega = \Omega E'_i,
\qquad \qquad 
E'_i \Omega = \Omega E''_i,
\qquad \qquad 
E''_i \Omega = \Omega E_i.
\end{eqnarray*}
\item[\rm (iv)]
We have 
\begin{eqnarray*}
\alpha'_1 A \Omega = \alpha''_1 \Omega B, \qquad \qquad
\alpha''_1 B \Omega = \alpha_1 \Omega C, \qquad \qquad
\alpha_1 C \Omega = \alpha'_1 \Omega A.
\end{eqnarray*}
\item[\rm (v)]
For $A,B,C$  equitable,
\begin{eqnarray*}
A \Omega =  \Omega B, \qquad \qquad
 B \Omega = \Omega C, \qquad \qquad
C \Omega =  \Omega A.
\end{eqnarray*}
\end{enumerate}
\end{theorem}
\begin{proof} Apply
Lemmas
\ref{lem:OmegaCom2}--\ref{prop:NBequit}
to the equitable LR triple $\alpha'_1 A, \alpha''_1 B, \alpha_1 C$
and use Lemmas
\ref{lem:albega},
\ref{lem:isequal},
\ref{lem:uniNew}.
\end{proof}


\begin{proposition}
\label{lem:OmEi}
 Assume that $A,B,C$ is nonbipartite.
Then $\Omega$ is a rotator for $A,B,C$.
\end{proposition}
\begin{proof}
By 
Definition
\ref{def:ABCROT} and
Theorem
\ref{cor:everythingGen}(iii).
\end{proof}

\begin{lemma}
\label{lem:OmegaSendEV}
Assume that $A,B,C$ is nonbipartite.
Then for $0 \leq i \leq d$, 
\begin{eqnarray*}
\Omega E'_iV= E_iV,
\qquad \qquad
\Omega E''_iV= E'_iV,
\qquad \qquad
\Omega E_iV= E''_iV.
\end{eqnarray*}
\end{lemma}
\begin{proof} 
By Proposition \ref{lem:OmEi} the map
$\Omega$ is a rotator for $A,B,C$.
The result follows by
Lemma
\ref{lem:ROTmeaning},
and since $\Omega$ is invertible by Theorem
\ref{cor:everythingGen}(ii).
\end{proof}
\noindent Recall the rotator space $\mathcal R$ from Definition
\ref{def:ABCROTspace}.

\begin{proposition} 
\label{prop:OmegaBasis}
Assume that $A,B,C$ is nonbipartite.
Then $\Omega$ is a basis for the $\mathbb F$-vector space
$\mathcal R$.
\end{proposition}
\begin{proof}
We have $\Omega \in \mathcal R$
by 
Proposition
\ref{lem:OmEi}. 
The map
$\Omega$ is invertible by
Theorem
\ref{cor:everythingGen}(ii),
 so of course $\Omega \not=0$.
We show
that $\Omega $ spans $\mathcal R$.
Let $R \in \mathcal R$. 
By assumption $R$ and $\Omega$ are rotators for
$A,B,C$. So by Definition
\ref{def:ABCROT},
$\Omega^{-1}R$ 
commutes with each of $E_i$, $E'_i$, $E''_i$ for
$0 \leq i \leq d$. Now by Definition
\ref{def:IC},
$\Omega^{-1}R$ is an idempotent
centralizer for $A,B,C$.
By Definition
\ref{def:ICspace}
and
Proposition \ref{prop:centBasis}(ii),
there exists $\zeta \in \mathbb F$ such that
$\Omega^{-1}R = \zeta I$.
Therefore 
$R = \zeta \Omega$. We have shown that
$\Omega $ spans $\mathcal R$. The result follows.
\end{proof}

\begin{definition}
\label{def:Rotator}
\rm
Assume that $A,B,C$ is nonbipartite.
By the {\it standard rotator} for $A,B,C$ we mean the
map $\Omega$ from 
Theorem 
\ref{cor:everythingGen}
and Proposition
\ref{prop:OmegaBasis}.
\end{definition}

\begin{lemma} 
\label{lem:OmegaSendsEta}
Assume that $A,B,C$ is nonbipartite.
Then $\Omega$ sends
\begin{eqnarray*}
\eta \to \frac{(\eta,\tilde \eta')}{(\eta'',\tilde \eta')}\eta'',
\qquad \qquad
\eta' \to \frac{(\eta',\tilde \eta'')}{(\eta,\tilde \eta'')}\eta,
\qquad \qquad
\eta'' \to \frac{(\eta'',\tilde \eta)}{(\eta',\tilde \eta)}\eta'.
\end{eqnarray*}
\end{lemma}
\begin{proof} Use
Lemma
\ref{lem:iIszero}
and
the formulae for $\Omega$ given in
Theorem 
\ref{cor:everythingGen}(i).
\end{proof}

\begin{proposition} 
\label{thm:OmegaCubed}
Assume that $A,B,C$ is 
nonbipartite. Then $\Omega^3= \theta I$, where
$\theta$ is from Definition
\ref{def:theta}.
\end{proposition}
\begin{proof}
By Theorem
\ref{cor:everythingGen}(iii),
$\Omega^3$ commutes with 
each of $E_i, E'_i, E''_i$ for
$0 \leq i \leq d$.
By Definition
\ref{def:IC},
$\Omega^{3}$ is an idempotent
centralizer for $A,B,C$.
By Definition
\ref{def:ICspace}
and
Proposition \ref{prop:centBasis}(ii),
there exists $\zeta \in \mathbb F$ such that
$\Omega^{3}= \zeta I$.
Considering the action of $\Omega^3$ on $\eta$ and using
Lemma
\ref{lem:OmegaSendsEta},
we obtain $\zeta = \theta$.
\end{proof}

\begin{lemma}
\label{lem:SRadj}
Assume that $A,B,C$ is nonbipartite.
Let $\alpha, \beta, \gamma$ denote nonzero
scalars in $\mathbb F$.
Then the
LR triples
$A,B,C$ and
 $\alpha A, \beta B, \gamma C$
 have the same standard rotator.
\end{lemma}
\begin{proof}
Use Lemmas
\ref{lem:albega},
\ref{lem:isequal},
\ref{lem:uniNew}
and any formula for $\Omega$ given in Theorem
\ref{cor:everythingGen}(i).
\end{proof}

\begin{lemma}
\label{lem:sameROT}
 Assume that $A,B,C$ is nonbipartite.
\begin{enumerate}
\item[\rm (i)] The following LR triples have the same standard rotator:
\begin{eqnarray*}
A,B,C \qquad \qquad 
B,C,A\qquad \qquad 
C,A,B.
\end{eqnarray*}
\item[\rm (ii)] The following LR triples have the same standard rotator:
\begin{eqnarray*}
C,B,A \qquad \qquad 
A,C,B \qquad \qquad 
B,A,C.
\end{eqnarray*}
\item[\rm (iii)] The standard 
rotators in {\rm (i)}, {\rm (ii)} above are inverses.
\end{enumerate}
\end{lemma}
\begin{proof}
By Lemmas
\ref{lem:EqAdjust},
\ref{lem:SRadj}
we may assume without loss that $A,B,C$ is equitable.
Now use Lemmas
\ref{lem:ABCvar},
\ref{lem:ABCEvar},
\ref{lem:equitBasic}(i),
\ref{lem:newUniData}
and the formulae for $\Omega, \Omega^{-1}$ given in
Theorem
\ref{cor:everythingGen}(i),(ii).
\end{proof}

\noindent Recall from Lemma
\ref{lem:newPA}
the LR triple $\tilde A,\tilde B,\tilde C$
on the dual space $V^*$.

\begin{lemma} Assume that $A,B,C$ is nonbipartite.
Then the following are inverse:
\begin{enumerate}
\item[\rm (i)] the adjoint of the standard rotator for $A,B,C$;
\item[\rm (ii)] the standard rotator for $\tilde A, \tilde B, \tilde C$.
\end{enumerate}
\end{lemma}
\begin{proof}
Use Lemmas
\ref{lem:newPA},
\ref{lem:tildeABCEvar},
\ref{lem:newUniDataDual}
 and any formula for $\Omega$ given in 
Theorem
\ref{cor:everythingGen}(i).
\end{proof}

\begin{proposition}
\label{cor:RinvSend}
Assume that $A,B,C$ is nonbipartite.
Let $R$ denote a nonzero rotator for $A,B,C$. Then
the following {\rm (i)--(iv)} hold.
\begin{enumerate}
\item[\rm (i)] 
 $R$ is invertible.
\item[\rm (ii)]  We have
\begin{eqnarray*}
\alpha'_1 A R = \alpha''_1 R B, \qquad \qquad
\alpha''_1 B R = \alpha_1 RC, \qquad \qquad
\alpha_1 C R= \alpha'_1 R A.
\end{eqnarray*}
\item[\rm (iii)] We have
\begin{eqnarray*}
\overline A R = R \overline B , \qquad \qquad
 \overline B R = R\overline C, \qquad \qquad
 \overline C R = R\overline A.
\end{eqnarray*}
\item[\rm (iv)]
For $A,B,C$  equitable,
\begin{eqnarray*}
A R =  R B, \qquad \qquad
 B R = R C, \qquad \qquad
C R =  R A.
\end{eqnarray*}
\end{enumerate}
\end{proposition}
\begin{proof}
By Proposition
\ref{prop:OmegaBasis} there exists $0 \not=\zeta \in \mathbb F$
such that $R = \zeta \Omega$.
The results follow in view of 
Theorems
\ref{thm:DL}(ii),
\ref{cor:everythingGen}.
\end{proof}

\noindent 
We now turn our attention to the case in which
$A,B,C$ is bipartite and nontrivial.
\medskip

\begin{lemma}
\label{lem:RoutInPre}
Assume that $A,B,C$ is bipartite and nontrivial.
Let $R$ denote a rotator for $A,B,C$. Then
the following {\rm (i)--(iii)} hold:
\begin{enumerate}
\item[\rm (i)] $R V_{\rm out} \subseteq V_{\rm out}$ and
$R V_{\rm in} \subseteq V_{\rm in}$;
\item[\rm (ii)] $R J=JR$;
\item[\rm (iii)] $R J$ is a rotator for $A,B,C$.
\end{enumerate}
\end{lemma}
\begin{proof}(i)
By
Lemmas \ref{lem:V0V1},
\ref{lem:ROTmeaning}.
\\
\noindent (ii) The map $J$ acts on 
$V_{\rm out}$ as the identity, and on $V_{\rm in}$ as zero.
The result follows from this and (i) above.
\\
\noindent (iii) By Definition
\ref{def:ABCROT}, and since $J$ is an idempotent
centralizer for $A,B,C$ by
Proposition
\ref{prop:centBasis}(iii).
\end{proof}

\begin{definition}
\label{def:RoutRin}
\rm Assume that $A,B,C$ is bipartite and nontrivial.
Let $R$ denote a rotator for $A,B,C$.
Then $R$ is called {\it outer}
(resp. {\it inner}) whenever
$R$ is zero on $V_{\rm in}$
(resp. $V_{\rm out}$).
Let ${\mathcal R}_{\rm out}$
(resp. ${\mathcal R}_{\rm in}$) denote the
set of outer (resp. inner) rotators
for $A,B,C$. Note that 
${\mathcal R}_{\rm out}$
and ${\mathcal R}_{\rm in}$ are subspaces
of the $\mathbb F$-vector space $\mathcal R$.
\end{definition}

\begin{definition}
\label{def:OmegaOutIn}
\rm
Assume that $A,B,C$ is bipartite and nontrivial.
Define elements $\Omega_{\rm out}$,
$\Omega_{\rm in}$ in ${\rm End}(V)$ as follows.
Recall by Lemmas
\ref{lem:A2B2C2Out},
\ref{lem:NormHalf}(ii)
that $A^2,B^2,C^2$ acts on
$V_{\rm out}$ as a nonbipartite LR triple.
The map 
$\Omega_{\rm out}$ acts on $V_{\rm out}$
as the standard rotator for this LR triple.
The map
$\Omega_{\rm out}$ 
 acts on 
$V_{\rm in}$ as zero.
The map $\Omega_{\rm in}$ acts
on $V_{\rm out}$ as zero.
Recall 
by Lemmas
\ref{lem:A2B2C2In},
\ref{lem:NormHalfIn}(i),(ii)
that $A^2,B^2,C^2$ acts on
$V_{\rm in}$ as an LR triple
that is nonbipartite or trivial.
The map $\Omega_{\rm in}$ acts on $V_{\rm in}$
as the standard rotator for this LR triple.
\end{definition}

\begin{lemma} 
\label{lem:OmegaBasic}
With reference to Definition \ref{def:OmegaOutIn},
\begin{eqnarray*}
\Omega_{\rm out} V_{\rm out} = V_{\rm out},
\quad \qquad
\Omega_{\rm out} V_{\rm in} = 0,
\quad \qquad 
\Omega_{\rm in} V_{\rm out} = 0,
\quad \qquad
\Omega_{\rm in} V_{\rm in} = V_{\rm in}.
\end{eqnarray*}
\end{lemma}
\begin{proof} By
Definition
\ref{def:OmegaOutIn} and the construction.
\end{proof}

\begin{proposition}
\label{prop:bipRoutIN}
Assume that $A,B,C$ is bipartite and nontrivial.
Then the following {\rm (i)--(iii)} hold:
\begin{enumerate}
\item[\rm (i)] the sum
${\mathcal R}  = {\mathcal R}_{\rm out} + 
{\mathcal R}_{\rm in}$ is direct;
\item[\rm (ii)] $\Omega_{\rm out}$ is a basis
for
${\mathcal R}_{\rm out}$;
\item[\rm (iii)] $\Omega_{\rm in}$ is a basis
for
${\mathcal R}_{\rm in}$.
\end{enumerate}
\end{proposition}
\begin{proof}
By Definitions
\ref{def:RoutRin},
\ref{def:OmegaOutIn} we find
$\Omega_{\rm out} \in \mathcal R_{\rm out}$
and 
$\Omega_{\rm in} \in \mathcal R_{\rm in}$.
We mentioned in
Definition \ref{def:OmegaOutIn}
that $A^2,B^2,C^2$ acts on $V_{\rm out}$
as a nonbipartite LR triple.
Denote the corresponding rotator subspace 
and standard rotator by 
$\mathcal R^{\rm out}$ and
$\Omega^{\rm out}$, respectively.
By construction
$\mathcal R^{\rm out}$ is a subspace of
${\rm End}(V_{\rm out})$. By Proposition
\ref{prop:OmegaBasis},
$\Omega^{\rm out}$ is a basis
for
$\mathcal R^{\rm out}$.
For $R \in \mathcal R$ the restriction 
$R \vert_{V_{\rm out}}$ 
is contained in 
${\mathcal R}^{\rm out}$, and the map
$\mathcal R \to
{\mathcal R}^{\rm out}$, $R \mapsto 
R \vert_{V_{\rm out}}$ 
is $\mathbb F$-linear. 
This map has kernel $\mathcal R_{\rm in}$. 
This map
sends $\Omega_{\rm out} \mapsto \Omega^{\rm out}$
and is therefore surjective.
By these comments
 $\Omega_{\rm out}$ forms a basis for a complement of
$\mathcal R_{\rm in}$ in $\mathcal R$.
Similarly
 $\Omega_{\rm in}$ forms a basis for a complement of
$\mathcal R_{\rm out}$ in $\mathcal R$.
Note that ${\mathcal R}_{\rm out} \cap  
{\mathcal R}_{\rm in} = 0$ by
Definition 
\ref{def:RoutRin} and since the sum
$V  = {V}_{\rm out} + 
{V}_{\rm in}$ is direct.
The result follows.
\end{proof}

\begin{definition}
\label{def:StandardOmega}
\rm
With reference to Definition \ref{def:OmegaOutIn}
and Proposition \ref{prop:bipRoutIN},
we call 
$\Omega_{\rm out}$ (resp.  $\Omega_{\rm in}$) the
{\it standard outer rotator}
(resp. 
{\it standard inner rotator})
for $A,B,C$.
\end{definition}

\noindent We now describe
$\Omega_{\rm out}$ and $\Omega_{\rm in}$
in more detail.

\begin{theorem}
\label{lem:OmegaOutD}
Assume that $A,B,C$ is bipartite and nontrivial.
Then the following {\rm (i)--(v)} hold.
\begin{enumerate}
\item[\rm (i)] We have
\begin{eqnarray*}
&&\Omega_{\rm out} = 
\mathbb B 
\Biggl(\sum_{j=0}^{d/2} \frac{\varphi_1 \varphi_2\cdots \varphi_{2j}}
{\varphi_d \varphi_{d-1}\cdots \varphi_{d-2j+1}} E_{2j}\Biggr) \mathbb A
= 
\mathbb C 
\Biggl(\sum_{j=0}^{d/2} \frac{\varphi_1 \varphi_2 \cdots \varphi_{2j}}
{\varphi_d \varphi_{d-1}\cdots \varphi_{d-2j+1}} E'_{2j}\Biggr) \mathbb B 
\\
&& \qquad \qquad \qquad 
=
\mathbb A 
\Biggl(\sum_{j=0}^{d/2} \frac{\varphi_1 \varphi_2 \cdots \varphi_{2j}}
{\varphi_d \varphi_{d-1}\cdots \varphi_{d-2j+1}} E''_{2j}\Biggr) \mathbb C.
\end{eqnarray*}
\item[\rm (ii)]
For $0 \leq i \leq d$,
\begin{eqnarray}
\label{eq:OutCom}
E_i \Omega_{\rm out} = \Omega_{\rm out} E'_i,
\qquad \qquad 
E'_i \Omega_{\rm out} = \Omega_{\rm out} E''_i,
\qquad \qquad 
E''_i \Omega_{\rm out} = \Omega_{\rm out} E_i.
\end{eqnarray}
\item[\rm (iii)]
Referring to 
{\rm (\ref{eq:OutCom})}, if $i$ is odd then
for each equation 
both sides are zero.
\item[\rm (iv)] We have
\begin{eqnarray*}
\alpha'_2 A^2 \Omega_{\rm out} = \alpha''_2 \Omega_{\rm out} B^2,
\qquad 
\alpha''_2 B^2 \Omega_{\rm out} = \alpha_2 \Omega_{\rm out} C^2,
\qquad 
\alpha_2 C^2 \Omega_{\rm out} = \alpha'_2 \Omega_{\rm out} A^2.
\end{eqnarray*}
\item[\rm (v)] For $A,B,C$ equitable,
\begin{eqnarray*}
 A^2 \Omega_{\rm out} =  \Omega_{\rm out} B^2,
\qquad \qquad
 B^2 \Omega_{\rm out} = \Omega_{\rm out} C^2,
\qquad \qquad
 C^2 \Omega_{\rm out} = \Omega_{\rm out} A^2.
\end{eqnarray*}
\end{enumerate}
\end{theorem}
\begin{proof}
Apply Theorem
\ref{cor:everything}
to the LR triple
in Lemma
\ref{lem:A2B2C2Out}, and evaluate
the result using
Lemmas
\ref{lem:unipA2},
\ref{lem:OmegaBasic}.
\end{proof}

\begin{theorem}
\label{lem:OmegaOutInD}
Assume that $A,B,C$ is bipartite and nontrivial.
Then the following {\rm (i)--(v)} hold.
\begin{enumerate}
\item[\rm (i)]
We have
\begin{eqnarray*}
&&\Omega_{\rm in} = 
\mathbb B 
\Biggl(\sum_{j=0}^{d/2-1} \frac{\varphi_2 \varphi_3\cdots \varphi_{2j+1}}
{\varphi_{d-1} \varphi_{d-2}\cdots \varphi_{d-2j}} E_{2j+1}\Biggr) \mathbb A
= 
\mathbb C 
\Biggl(\sum_{j=0}^{d/2-1} \frac{\varphi_2 \varphi_3 \cdots \varphi_{2j+1}}
{\varphi_{d-1} \varphi_{d-2}\cdots \varphi_{d-2j}} E'_{2j+1}\Biggr) \mathbb B 
\nonumber
\\
&& \qquad \qquad \qquad 
=
\mathbb A 
\Biggl(\sum_{j=0}^{d/2-1} \frac{\varphi_2 \varphi_3\cdots \varphi_{2j+1}}
{\varphi_{d-1} \varphi_{d-2} \cdots \varphi_{d-2j}} E''_{2j+1}\Biggr)
\mathbb C.
\end{eqnarray*}
\item[\rm (ii)] 
For $0 \leq i \leq d$,
\begin{eqnarray}
\label{eq:InCom}
E_i \Omega_{\rm in} = \Omega_{\rm in} E'_i,
\qquad \qquad 
E'_i \Omega_{\rm in} = \Omega_{\rm in} E''_i,
\qquad \qquad 
E''_i \Omega_{\rm in} = \Omega_{\rm in} E_i.
\end{eqnarray}
\item[\rm (iii)]
Referring to
{\rm (\ref{eq:InCom})}, if $i$ is even then 
for each equation both sides are zero.
\item[\rm (iv)] We have
\begin{eqnarray*}
\alpha'_2 A^2 \Omega_{\rm in} = \alpha''_2 \Omega_{\rm in} B^2,
\qquad
\alpha''_2 B^2 \Omega_{\rm in} = \alpha_2 \Omega_{\rm in} C^2,
 \qquad
\alpha_2 C^2 \Omega_{\rm in} = \alpha'_2 \Omega_{\rm in} A^2.
\end{eqnarray*}
\item[\rm (v)] For $A,B,C$ equitable,
\begin{eqnarray*}
 A^2 \Omega_{\rm in} =  \Omega_{\rm in} B^2,
\qquad \qquad
 B^2 \Omega_{\rm in} =  \Omega_{\rm in} C^2,
\qquad \qquad
 C^2 \Omega_{\rm in} = \Omega_{\rm in} A^2.
\end{eqnarray*}

\end{enumerate}
\end{theorem}
\begin{proof}
Apply Theorem
\ref{cor:everything}
to the LR triple
in Lemma
\ref{lem:A2B2C2In}, and evaluate
the result using
Lemmas
\ref{lem:unipA2In},
\ref{lem:OmegaBasic}.
\end{proof}

\noindent Recall the maps $\Omega, \Omega', \Omega''$
from Definition
\ref{def:Om123}.

\begin{proposition}
\label{lem:OmegaOmega}
Assume that $A,B,C$ is equitable, bipartite, and
nontrivial. Then
\begin{eqnarray*}
\Omega = \Omega_{\rm out} + \rho_0 \Omega_{\rm in},
\qquad \qquad 
\Omega' = \Omega_{\rm out} + \rho'_0 \Omega_{\rm in},
\qquad \qquad
\Omega'' = \Omega_{\rm out} + \rho''_0 \Omega_{\rm in}.
\end{eqnarray*}
\end{proposition}
\begin{proof}
Compare the formulae for $\Omega$, $\Omega'$,
$\Omega''$ 
given in 
Proposition
\ref{prop:EquitWOW},
with the
formulae for
$\Omega_{\rm out}$,
$\Omega_{\rm in}$ given in
Theorems 
\ref{lem:OmegaOutD}(i),
\ref{lem:OmegaOutInD}(i).
 The result follows in
view of
Lemma
\ref{lem:basic2A} and
(\ref{def:rho0def}).
\end{proof}

\begin{lemma} 
\label{lem:OmegaMC}
Assume that $A,B,C$ is equitable, bipartite, and
nontrivial. Then
$\Omega, \Omega', \Omega''$ are rotators
for $A,B,C$.
\end{lemma}
\begin{proof}
By 
Propositions
\ref{prop:bipRoutIN},
\ref{lem:OmegaOmega}.
\end{proof}

\begin{proposition}
\label{prop:rhoShow}
Assume that $A,B,C$ is bipartite and
nontrivial. Then
\begin{eqnarray*}
&&
\varphi'_d 
A \Omega_{\rm out} = \varphi''_1 \Omega_{\rm in} B,
\qquad \qquad
\varphi''_d 
B \Omega_{\rm out} = \varphi_1 \Omega_{\rm in} C,
\qquad \qquad
\varphi_d 
C \Omega_{\rm out} = \varphi'_1 \Omega_{\rm in} A,
\\
&&
\varphi'_1
A \Omega_{\rm in} =  \varphi''_d\Omega_{\rm out} B,
\qquad \qquad
\varphi''_1
 B \Omega_{\rm in} =  
  \varphi_d
 \Omega_{\rm out} C,
\qquad \qquad
\varphi_1
C \Omega_{\rm in} = 
  \varphi'_d
\Omega_{\rm out} A.
\end{eqnarray*}
Moreover for $A,B,C$ equitable,
\begin{eqnarray*}
&&
A \Omega_{\rm out} = \rho_0 \Omega_{\rm in} B,
\qquad \qquad
B \Omega_{\rm out} = \rho'_0 \Omega_{\rm in} C,
\qquad \qquad
C \Omega_{\rm out} = \rho''_0 \Omega_{\rm in} A,
\\
&&
\rho_0
A \Omega_{\rm in} =  \Omega_{\rm out} B,
\qquad \qquad
\rho'_0
 B \Omega_{\rm in} =  
 \Omega_{\rm out} C,
\qquad \qquad
\rho''_0
C \Omega_{\rm in} = 
\Omega_{\rm out} A.
\end{eqnarray*}
\end{proposition}
\begin{proof}
First assume that $A,B,C$ is equitable.
To obtain the result under this assumption,
evaluate 
(\ref{eq:AOmB})
using
Proposition
\ref{lem:OmegaOmega},
Lemmas
\ref{lem:bipABCact},
\ref{lem:OmegaBasic}, and line
(\ref{def:rho0def}).
We have verified the result under the assumption that
$A,B,C$ is equitable. To remove the assumption,
apply the result so far to the LR triple 
(\ref{eq:NewLRT}) in
 Lemma
\ref{lem:InOut},
made equitable by chosing
the parameters 
(\ref{eq:AlphaInOut})
to
satisfy
(\ref{eq:neededForEquit}).
\end{proof}

\begin{proposition}
\label{prop:rotatorRaction}
Assume that $A,B,C$ is  bipartite and
nontrivial.
Let $R$ denote a rotator for $A,B,C$ and write
$R=r\Omega_{\rm out} + s \Omega_{\rm in}$ 
with $r,s \in \mathbb F$.
Then
\begin{eqnarray*}
&&
s \varphi'_d 
 A_{\rm out} R = r \varphi''_1 R B_{\rm out},
\qquad 
s \varphi''_d 
 B_{\rm out} R = r \varphi_1 R C_{\rm out},
\qquad 
s \varphi_d 
 C_{\rm out} R = r \varphi'_1 R A_{\rm out},
\\
&&
r\varphi'_1
A_{\rm in} R = s \varphi''_d R B_{\rm in},
\quad \qquad 
r\varphi''_1
B_{\rm in} R = s \varphi_d R C_{\rm in},
\quad \qquad 
r\varphi_1
 C_{\rm in} R = s \varphi'_d R A_{\rm in}.
\end{eqnarray*}
Moreover for $A,B,C$ equitable,
\begin{eqnarray*}
&&
s 
 A_{\rm out} R = r \rho_0 R B_{\rm out},
\qquad 
s 
 B_{\rm out} R = r \rho'_0 R C_{\rm out},
\qquad 
s 
 C_{\rm out} R = r \rho''_0 R A_{\rm out},
\\
&&
r\rho_0
A_{\rm in} R = s  R B_{\rm in},
\quad \qquad 
r\rho'_0
B_{\rm in} R = s R C_{\rm in},
\quad \qquad 
r\rho''_0
 C_{\rm in} R = s R A_{\rm in}.
\end{eqnarray*}
\end{proposition}
\begin{proof}
To verify these equations,
eliminate $R$ using
$R= r \Omega_{\rm out}+s \Omega_{\rm in}$ 
and evaluate the result using
Definition
\ref{def:ABCoutIn}
together with
Proposition
\ref{prop:rhoShow}.
\end{proof}

\begin{lemma}
\label{lem:OmegaOutSend}
Assume that $A,B,C$ is bipartite and nontrivial. Then 
 $\Omega_{\rm out}$ sends
\begin{eqnarray*}
\eta \to \frac{(\eta,\tilde \eta')}{(\eta'',\tilde \eta')}\eta'',
\qquad \qquad
\eta' \to \frac{(\eta',\tilde \eta'')}{(\eta,\tilde \eta'')}\eta,
\qquad \qquad
\eta'' \to \frac{(\eta'',\tilde \eta)}{(\eta',\tilde \eta)}\eta'.
\end{eqnarray*}
\end{lemma}
\begin{proof} 
Similar to the proof of Lemma \ref{lem:OmegaSendsEta}.
\end{proof}

\noindent Recall the scalar $\theta$ from Definition
\ref{def:theta}.
\begin{proposition}
\label{lem:OmegaOutcubed}
Assume that $A,B,C$ is  bipartite and nontrivial. Then 
the following {\rm (i), (ii)} hold.
\begin{enumerate}
\item[\rm (i)]
$\Omega^3_{\rm out} = \theta I$
on 
$V_{\rm out}$.
\item[\rm (ii)] 
$\Omega^3_{\rm in} = \rho^{-1} \theta I$
 on $V_{\rm in}$, 
where 
$\rho= 
\varphi_1 \varphi'_1 \varphi''_1/
(\varphi_d \varphi'_d \varphi''_d)$.
\end{enumerate}
\end{proposition}
\begin{proof}
(i) 
Similar to the proof of Proposition \ref{thm:OmegaCubed}.
\\
\noindent (ii) 
By Proposition
\ref{prop:rhoShow} we obtain
$A \Omega^3_{\rm out} = \rho
\Omega^3_{\rm in} A
$. For this equation apply each side to $V_{\rm out}$
and use the fact that $AV_{\rm out} = V_{\rm in}$.
The result follows in view of (i) above.
\end{proof}

\begin{lemma}
\label{lem:whenRinv}
Assume that $A,B,C$ is bipartite and nontrivial.
Let $R$ denote a rotator for $A,B,C$ and write
$R = r \Omega_{\rm out} + s \Omega_{\rm in}$ with
$r,s\in \mathbb F$. Then $R$ is invertible if and only if
$r,s$ are nonzero.
\end{lemma}
\begin{proof} By Lemma
\ref{lem:OmegaBasic}.
\end{proof}

\begin{proposition}
Assume that $A,B,C$ is  bipartite and nontrivial.
Let $R$ denote an invertible rotator for $A,B,C$. Then
\begin{eqnarray*}
\overline A R =   R\overline B,
\qquad \qquad 
\overline B R =   R\overline C,
\qquad \qquad 
\overline C R =   R \overline A.
\end{eqnarray*}
\end{proposition}
\begin{proof}
Use 
Theorem \ref{thm:DL}(iii) and the comment
above that theorem, along with
Proposition
\ref{prop:rotatorRaction}
and 
Lemma \ref{lem:whenRinv}.
\end{proof}

\section{The reflectors for an LR triple}

\noindent 
Throughout this section the following notation is in effect.
Let $V$ denote a vector space over $\mathbb F$ with
dimension $d+1$.
Let $A,B,C$ denote
an LR triple on $V$, with
 parameter array
(\ref{eq:paLRT}),
idempotent data
(\ref{eq:idseq}),
trace data
(\ref{eq:tracedata}), and Toeplitz data
(\ref{eq:ToeplitzData}).
Recall the reflector antiautomorphism
concept discussed in 
Proposition \ref{prop:antiaut}
and Definition
\ref{def:REFdag}. There are three reflectors associated with $A,B,C$; 
the $(A,B)$-reflector, the
$(B,C)$-reflector, and the
$(C,A)$-reflector. We now consider
how these reflectors behave. 
In order to keep things simple, throughout this section we assume
that $A,B,C$ is equitable.

\begin{proposition}
Assume that $A,B,C$ is equitable and nonbipartite, with
standard rotator $\Omega$. Then
the following {\rm (i)--(iii)} hold.
\begin{enumerate}
\item[\rm (i)] The $(A,B)$-reflector
swaps $A,B$ and fixes $C$. 
It swaps $\mathbb A, \mathbb B$ and fixes $\mathbb C$.
 It fixes $\Omega$. For $0 \leq i \leq d$
it fixes $E_i$ and swaps $E'_i,E''_{i}$.
\item[\rm (ii)] The $(B,C)$-reflector
swaps $B,C$ and fixes $A$. 
It swaps $\mathbb B, \mathbb C$ and fixes $\mathbb A$.
 It fixes $\Omega$. For $0 \leq i \leq d$
it fixes $E'_i$ and swaps $E''_i,E_{i}$.
\item[\rm (iii)] The $(C,A)$-reflector
swaps $C,A$ and fixes $B$. 
It swaps $\mathbb C, \mathbb A$ and fixes $\mathbb B$.
 It fixes $\Omega$. For $0 \leq i \leq d$
it fixes $E''_i$ and swaps $E_i,E'_{i}$.
\end{enumerate}
\end{proposition}
\begin{proof} (i) Denote
the $(A,B)$-reflector by $\dagger$.
The map
$\dagger$ swaps $A,B$ by
Proposition \ref{prop:antiaut}
and Definition
\ref{def:REFdag}.
We have
 $\mathbb A = \sum_{i=0}^d \alpha_i A^i$ and
 $\mathbb B = \sum_{i=0}^d \alpha_i B^i$ by
Proposition \ref{prop:Apoly},
so 
$\dagger$ swaps $\mathbb A,\mathbb B$.
To see that $\dagger$
fixes $\Omega$, use the
first formula for $\Omega$ given in
Theorem
\ref{cor:everything}(i),
along with  Definition
\ref{def:REF} and
Lemma
\ref{lem:psiFix}.
From 
Theorem \ref{cor:everythingGen}(v)
we obtain
$C=\Omega A \Omega^{-1}$ and $C=\Omega^{-1}B\Omega$.
By these and 
since $\dagger$  fixes $\Omega$,
we see that $\dagger$  fixes $C$.
Consequently $\dagger$ fixes $\mathbb C = \sum_{i=0}^d 
\alpha_i C^i$.
For $0 \leq i\leq d$ the map
$\dagger$ fixes $E_i$ by
Lemma 
\ref{lem:daggerE}.
Also 
by Lemma
\ref{lem:Eform2}
and Lemma \ref{lem:equitBasic}(i)
 the map $\dagger$  swaps $E'_i, E''_{i}$.
\\
\noindent (ii), (iii) Use (i) above and
Lemma
\ref{lem:sameROT}(i).
\end{proof}

\begin{proposition}
\label{eq:EBRef}
Assume that $A,B,C$ is equitable, bipartite, and nontrivial. Then 
 the following {\rm (i)--(iii)} hold.
\begin{enumerate}
\item[\rm (i)] The $(A,B)$-reflector sends
\begin{eqnarray*}
&&
A_{\rm out}\rightarrow B_{\rm in},
\qquad \qquad
B_{\rm out}\rightarrow A_{\rm in},
\qquad \qquad 
C_{\rm out}\rightarrow (\rho''_0/\rho'_0) C_{\rm in},
\\
&&
A_{\rm in}\rightarrow B_{\rm out},
\qquad \qquad
B_{\rm in}\rightarrow A_{\rm out},
\qquad \qquad 
C_{\rm in}\rightarrow (\rho'_0/\rho''_0) C_{\rm out}.
\end{eqnarray*}
It swaps $\mathbb A, \mathbb B$ and fixes $\mathbb C$.
It fixes $J$ and everything in $\mathcal R$.
For $0 \leq i \leq d$ it fixes $E_i$ and swaps
$E'_i, E''_i$.
\item[\rm (ii)] The $(B,C)$-reflector sends
\begin{eqnarray*}
&&
B_{\rm out}\rightarrow C_{\rm in},
\qquad \qquad
C_{\rm out}\rightarrow B_{\rm in},
\qquad \qquad 
A_{\rm out}\rightarrow (\rho_0/\rho''_0) A_{\rm in},
\\
&&
B_{\rm in}\rightarrow C_{\rm out},
\qquad \qquad
C_{\rm in}\rightarrow B_{\rm out},
\qquad \qquad 
A_{\rm in}\rightarrow (\rho''_0/\rho_0) A_{\rm out}.
\end{eqnarray*}
It swaps $\mathbb B, \mathbb C$ and fixes $\mathbb A$.
It fixes $J$ and everything in $\mathcal R$.
For $0 \leq i \leq d$ it fixes $E'_i$ and swaps
$E''_i, E_i$.
\item[\rm (iii)] The $(C,A)$-reflector sends
\begin{eqnarray*}
&&
C_{\rm out}\rightarrow A_{\rm in},
\qquad \qquad
A_{\rm out}\rightarrow C_{\rm in},
\qquad \qquad 
B_{\rm out}\rightarrow (\rho'_0/\rho_0) B_{\rm in},
\\
&&
C_{\rm in}\rightarrow A_{\rm out},
\qquad \qquad
A_{\rm in}\rightarrow C_{\rm out},
\qquad \qquad 
B_{\rm in}\rightarrow (\rho_0/\rho'_0) B_{\rm out}.
\end{eqnarray*}
It swaps $\mathbb C, \mathbb A$ and fixes $\mathbb B$.
It fixes  $J$ and everything in 
$\mathcal R$.
For $0 \leq i \leq d$ it fixes $E''_i$ and swaps
$E_i, E'_i$.
\end{enumerate}
\end{proposition}
\begin{proof} (i) Denote the $(A,B)$-reflector by $\dagger$.
The map $\dagger$ swaps $A,B$ by Proposition 
\ref{prop:antiaut} and Definition
\ref{def:REFdag}. 
For $0 \leq i \leq d$ the map $\dagger$ fixes
$E_i$ by
Lemma
\ref{lem:daggerE}.
By this and Lemma
\ref{lem:Jfacts}(i), the
 map $\dagger$ fixes $J$.
By this and Lemma
\ref{lem:INOutFacts}(i),(ii) 
the map $\dagger$ sends
$A_{\rm out} \leftrightarrow B_{\rm in}$ and
$A_{\rm in} \leftrightarrow B_{\rm out}$.
We have
$\mathbb A = \sum_{i=0}^d \alpha_i A^i$
and 
$\mathbb B = \sum_{i=0}^d \alpha_i B^i$
by
Proposition \ref{prop:Apoly},
so $\dagger $ swaps $\mathbb A, \mathbb B$.
We show that $\dagger$ fixes everything in
$\mathcal R$. By Proposition
\ref{prop:bipRoutIN}
it suffices to show that
$\dagger$ fixes $\Omega_{\rm out}$ and
 $\Omega_{\rm in}$.
To see that $\dagger$
fixes 
$\Omega_{\rm out}$
(resp. 
$\Omega_{\rm in}$),
 use the
first formula for $\Omega_{\rm out}$ (resp.
$\Omega_{\rm in}$)
given in
Theorem
\ref{lem:OmegaOutD}(i)
(resp. Theorem
\ref{lem:OmegaOutInD}(i)),
along with  Definition
\ref{def:PoutPin} and
Lemma
\ref{lem:psiFixOutIn}.
For $0 \leq i \leq d$ we show
that $\dagger$ swaps $E'_i, E''_i$.
Pick an invertible  $R \in \mathcal R$.
By Definition
\ref{def:ABCROT},
$E_iR = RE'_i$ and
$E''_iR = RE_i$.
In either equation,
apply $\dagger$ to each side and compare the
results with the other equation.
This shows that
$\dagger$ swaps 
$E'_i, E''_i$.
Using this and
$\mathbb C = \sum_{i=0}^d E''_{d-i} E'_{i}$ we find
that $\dagger$ fixes $\mathbb C$.
To obtain the action of $\dagger$ on
$C_{\rm out}$, $C_{\rm in}$, we invoke Proposition
\ref{prop:rotatorRaction}. Referring to that proposition,
assume that $r,s$ are nonzero, so that $R$ is invertible,
and consider the equations
$s C_{\rm out} R = r \rho''_0 R A_{\rm out}$ and
$r \rho'_0 B_{\rm in} R = s R C_{\rm in}$. In either equation,
apply $\dagger $ to each side and compare the
results with the other equation.
This shows that $\dagger$ sends
$C_{\rm out} \mapsto (\rho''_0 /\rho'_0)C_{\rm in}$
and
$C_{\rm in} \mapsto (\rho'_0 /\rho''_0)C_{\rm out}$.
\\
\noindent (ii), (iii) Similar to the proof of (i) above.
\end{proof}

\begin{corollary} Assume that $A,B,C$ is equitable, bipartite, and nontrivial.
Then the following {\rm (i)--(iii)} hold:
\begin{enumerate}
\item[\rm (i)] the $(A,B)$-reflector swaps $\overline A, \overline B$
and fixes $\overline C$;
\item[\rm (ii)] the $(B,C)$-reflector swaps $\overline B,
\overline C$
and fixes $\overline A$;
\item[\rm (iii)] the $(C,A)$-reflector swaps $\overline C,
\overline A$
and fixes $\overline B$.
\end{enumerate}
\end{corollary}
\begin{proof} Use
Theorem
\ref{thm:DL}(iii) and the comment
above that theorem, along with
Proposition
\ref{eq:EBRef}.
\end{proof}

\section{Normalized LR triples with diameter at most 2}

\noindent Our next general goal is to classify
up to isomorphism the normalized LR triples.
As a warmup, we consider the
normalized LR triples with diameter
at most 2. For the results in this section the proofs 
are routine, and left as an exercise.
\begin{lemma}
Up to isomorphism, there exists a unique normalized LR triple
over $\mathbb F$ that has diameter $0$. This LR triple
is trivial.
\end{lemma}

\begin{lemma}
\label{lem:1or2}
Up to isomorphism, there exists a unique normalized LR triple
$A,B,C$ over $\mathbb F$ that has diameter 1.
This LR triple is nonbipartite and
$\varphi_1 = 
 -1$.
Moreover
$a_0 =1$ and
$a_1 =  -1$.
With respect to an $(A,B)$-basis the matrices representing
$A,B,C$  and
the standard rotator $\Omega$ are 
\begin{eqnarray*}
A: \left(
\begin{array}{cc}
0 & 1  
\\
 0 & 0 
\\
\end{array}
\right),
 \qquad 
B: \left(
\begin{array}{cc}
0 & 0  
\\
-1 & 0 
\\
\end{array}
\right),
\qquad 
C: \left(
\begin{array}{cc}
1 & 1  
\\
-1 & -1 
\\
\end{array}
\right),
 \qquad 
\Omega: \left(
\begin{array}{cc}
1 & 1  
\\
-1 &  0 
\\
\end{array}
\right).
\end{eqnarray*}
\end{lemma}

\begin{lemma}
\label{lem:dtwoClassNB}
We give a bijection from the set 
$\mathbb F \backslash \lbrace 0,-1\rbrace$ to 
the set of isomorphism classes of normalized
nonbipartite LR triples over $\mathbb F$
that have diameter $2$. 
For $q \in \mathbb F \backslash \lbrace 0,-1\rbrace$  the corresponding
LR triple $A,B,C$ has parameters
\begin{eqnarray*}
&&\varphi_1 = -1-q^{-1}, \qquad \varphi_2 = -1-q,
\\
&&\alpha_2 = \frac{1}{1+q}, \qquad \beta_2 = \frac{q}{1+q},
\\
&&
a_0 = 1+q, \qquad
a_1 = q^{-1}-q, \qquad
a_2 = -1-q^{-1}.
\end{eqnarray*}
With respect to an $(A,B)$-basis the matrices representing
$A,B,C$ and
the standard rotator $\Omega$ 
are
\begin{eqnarray*}
&&
A: \left(
\begin{array}{ccc}
0 & 1  & 0 
\\
 0 & 0 & 1 
\\
0 & 0 & 0
\\
\end{array}
\right),
 \qquad \qquad 
B: \left(
\begin{array}{ccc}
0 & 0  &0 
\\
-1-q^{-1} & 0 & 0 
\\
0 & -1-q & 0 
\end{array}
\right),
\\
&&
C: \left(
\begin{array}{ccc}
1+q &  q & 0  
\\
-1-q & q^{-1}-q & q^{-1} 
\\
0 & -1-q^{-1} & -1-q^{-1}
\end{array}
\right),
\qquad 
\Omega: \left(
\begin{array}{ccc}
1 &  1 & (1+q)^{-1}  
\\
-1-q^{-1} & -1 & 0 
\\
1+q^{-1} & 0  & 0 
\end{array}
\right).
\end{eqnarray*}
\end{lemma}

\begin{lemma}
\label{lem:dtwoClassB}
We give a bijection from the set the 3-tuples 
\begin{eqnarray}
\label{eq:3tupleset}
(\rho_0, \rho'_0, \rho''_0) \in \mathbb F^3,
\qquad \qquad 
\rho_0 \rho'_0 \rho''_0 = -1
\end{eqnarray}
to the set of isomorphism classes of normalized
bipartite LR triples over $\mathbb F$
that have diameter $2$. 
For a 3-tuple
$(\rho_0, \rho'_0, \rho''_0)$ in the set
{\rm (\ref{eq:3tupleset})}, 
 the corresponding
LR triple $A,B,C$ has parameters
\begin{eqnarray*}
&&
\varphi_1 = -1/\rho_0, 
\qquad 
\varphi'_1 = -1/\rho'_0, 
\qquad 
\varphi''_1 = -1/\rho''_0, 
\\
&&
 \varphi_2 = \rho_0,
\qquad
 \varphi'_2 = \rho'_0,
\qquad
 \varphi''_2 = \rho''_0.
\end{eqnarray*}
With respect to an $(A,B)$-basis the matrices representing
$A,B,C$, the projector $J$,  and the standard outer/inner rotators
$\Omega_{\rm out}$,
$\Omega_{\rm in}$
are
\begin{eqnarray*}
&&
A: \left(
\begin{array}{ccc}
0 & 1  & 0 
\\
 0 & 0 & 1 
\\
0 & 0 & 0
\\
\end{array}
\right),
 \qquad 
B: \left(
\begin{array}{ccc}
0 & 0  &0 
\\
-1/\rho_0 & 0 & 0 
\\
0 & \rho_0 & 0 
\end{array}
\right),
\qquad
C: \left(
\begin{array}{ccc}
0 &  1/\rho''_0 & 0  
\\
\rho''_0 & 0 &  \rho''_0 
\\
0 & -1/\rho''_0 & 0 
\end{array}
\right),
\\
&&
J: \left(
\begin{array}{ccc}
1 &  0 & 0  
\\
0 & 0 &  0
\\
0 & 0 & 1 
\end{array}
\right),
\qquad
\Omega_{\rm out}: \left(
\begin{array}{ccc}
1 &  0 & 1  
\\
0 & 0 &  0
\\
-1 & 0 & 0 
\end{array}
\right),
\qquad 
\Omega_{\rm in}: \left(
\begin{array}{ccc}
0 &  0 & 0  
\\
0 & 1 &  0
\\
0 & 0 & 0 
\end{array}
\right).
\end{eqnarray*}
\end{lemma}

\section{The sequence $\lbrace  \rho_i \rbrace_{i=0}^{d-1}$ is constrained}

\noindent 
Throughout this section the following notation is in effect.
Let $V$ denote a vector space over $\mathbb F$ with
dimension $d+1$.
Let $A,B,C$ denote
a nontrivial LR triple on
$V$, with
 parameter array
(\ref{eq:paLRT}),
idempotent data
(\ref{eq:idseq}),
trace data
(\ref{eq:tracedata}), and Toeplitz data
(\ref{eq:ToeplitzData}).
We assume that $A,B,C$ is equitable, so that
$\alpha_i = \alpha'_i = \alpha''_i$ and
$\beta_i = \beta'_i = \beta''_i$ 
for
$0 \leq i \leq d$.
For $A,B,C$ nonbipartite we have the
sequence
$\lbrace \rho_i \rbrace_{i=0}^{d-1}$ from 
Definition
\ref{def:Rhoi}, and for 
$A,B,C$ bipartite we have the
sequences
$\lbrace \rho_i \rbrace_{i=0}^{d-1}$,
$\lbrace \rho'_i \rbrace_{i=0}^{d-1}$,
$\lbrace \rho''_i \rbrace_{i=0}^{d-1}$
from Definition
\ref{def:RHOD}. Our next goal is to 
show that these sequences are constrained, in
the sense of
Definition
\ref{def:constrain}. 

\begin{lemma}
\label{lem:Con1}
Assume that $A,B,C$ is equitable.
Then the following {\rm (i)--(iii)} hold.
\begin{enumerate} 
\item[\rm (i)] 
For $d\geq 2$,
\begin{eqnarray}
\label{eq:L1}
&&\rho_i = \alpha_0 \beta_2 \varphi_i + 
 \alpha_1 \beta_1 \varphi_{i+1} + 
 \alpha_2 \beta_0 \varphi_{i+2}
\qquad \qquad (0 \leq i \leq d-1),
\\
&&0 = \alpha_0 \beta_2 + 
 \alpha_1 \beta_1 + 
 \alpha_2 \beta_0.
\label{eq:L2}
\end{eqnarray}
\item[\rm (ii)] For $d \geq 3$,
\begin{eqnarray}
\label{eq:L3}
&&0 = 
\alpha_0 \beta_3 \varphi_{i-2} +
\alpha_1 \beta_2 \varphi_{i-1} +
\alpha_2 \beta_1 \varphi_{i} +
\alpha_3 \beta_0 \varphi_{i+1}
 \qquad (2 \leq i \leq d),
\\
&&
0 = \alpha_0 \beta_3 
+\alpha_1 \beta_2 +
\alpha_2 \beta_1 
+
\alpha_3 \beta_0.
\label{eq:L4}
\end{eqnarray}
\item[\rm (iii)] For $A,B,C$ bipartite and $d\geq 4$,
\begin{eqnarray}
\label{eq:L5}
&&
0 = 
\alpha_0 \beta_4 \varphi_{i-2} +
\alpha_2 \beta_2 \varphi_{i} +
\alpha_4 \beta_0 \varphi_{i+2}
\qquad  \qquad (2 \leq i \leq d-1),
\\
&&
\label{eq:L6}
0 = 
\alpha_0 \beta_4  +
\alpha_2 \beta_2 +
\alpha_4 \beta_0.
\end{eqnarray}
\end{enumerate}
\end{lemma}
\begin{proof}(i) Line 
(\ref{eq:L1}) is from
the first displayed equation in
Proposition
\ref{prop:goodrec}, along with  
Lemma
\ref{lem:equitBasic}(i).
Line 
(\ref{eq:L2}) is from
(\ref{eq:recursion}).
\\
\noindent (ii) Line
(\ref{eq:L3})
is from the first displayed equation in
Proposition
\ref{prop:longrec} (with $r=3$).
Line
(\ref{eq:L4}) is from
(\ref{eq:recursion}).
\\
\noindent (iii) Similar to (ii) above, but also use
Lemma
\ref{lem:case}.
\end{proof}


\noindent As we proceed, we will consider the bipartite and nonbipartite cases
separately. We begin with the nonbipartite case.

\begin{lemma}
\label{lem:NBCon1}
Assume that $A,B,C$ is nonbipartite, equitable, and $d\geq 2$. Then
for $0 \leq i \leq d-1$,
\begin{eqnarray}
\label{eq:middleElim}
\rho_i = \alpha_0 \beta_2( \varphi_i -\varphi_{i+1})
-
 \alpha_2 \beta_0 (\varphi_{i+1}-\varphi_{i+2}).
\end{eqnarray}
\end{lemma}
\begin{proof} Subtract 
$\varphi_{i+1}$ times
(\ref{eq:L2}) from
(\ref{eq:L1}).
\end{proof}

\begin{definition}
\label{def:NBCon2}
Assume that $A,B,C$ is nonbipartite, equitable, and $d\geq 3$. Define
\begin{eqnarray}
\label{eq:abc}
a = 
\alpha_0 \beta_3, \quad \qquad
b =
\alpha_0 \beta_3 + 
\alpha_1 \beta_2 =  
-
\alpha_2 \beta_1 - 
\alpha_3 \beta_0, 
\qquad \quad 
c =  -\alpha_3 \beta_0.
\end{eqnarray}
\end{definition}

\begin{lemma}
\label{lem:NBCon4}
Assume that $A,B,C$ is nonbipartite, equitable, and $d\geq 3$.
Then for 
 $2 \leq i \leq d$,
\begin{eqnarray}
0 =
a(\varphi_{i-2}-\varphi_{i-1})+
b(\varphi_{i-1}-\varphi_{i})+
c(\varphi_{i}-\varphi_{i+1}),
\label{eq:abcRec}
\end{eqnarray}
where $a,b,c$ are from
{\rm (\ref{eq:abc})}.
\end{lemma}
\begin{proof} 
To verify
(\ref{eq:abcRec}),
eliminate $a,b,c$ 
using
(\ref{eq:abc}), and compare the
result with
(\ref{eq:L3}).
\end{proof}

\begin{lemma}
\label{lem:NBCon5}
Assume that $A,B,C$ is nonbipartite, equitable, and $d\geq 3$.
Then for 
 $1 \leq i \leq d-2$,
\begin{eqnarray}
0 = a \rho_{i-1} + b\rho_i + c \rho_{i+1},
\label{eq:abcRho}
\end{eqnarray}
where $a,b,c$ are from
{\rm (\ref{eq:abc})}.
\end{lemma}
\begin{proof}
To verify
(\ref{eq:abcRho}), eliminate
$\rho_{i-1}, \rho_i,\rho_{i+1}$
 using
Lemma
\ref{lem:NBCon1}, 
and evaluate
the result using
Lemma \ref{lem:NBCon4}.
\end{proof}

\begin{lemma}
\label{lem:NBCon3}
Assume that $A,B,C$ is nonbipartite, equitable, and $d\geq 3$.
Then the scalars $a,b,c$ from Definition
\ref{def:NBCon2} are not all zero.
\end{lemma}
\begin{proof} 
Recall from Lemmas
\ref{lem:alpha0},
\ref{lem:NotB}
that $\alpha_0 = 1 = \beta_0$ and
$\alpha_1 = -\beta_1$ is nonzero.
Suppose 
that each of $a,b,c$ is zero.
Using (\ref{eq:abc}) we obtain
$\alpha_3=0$,
$\beta_3=0$,
$\alpha_2=0$, $\beta_2=0$.
Now in (\ref{eq:middleElim}) the right-hand side
is zero and the left-hand side is nonzero, for a contradiction.
The result follows.
\end{proof}

\begin{proposition}
\label{lem:NBCon6}
Assume that $A,B,C$ is nonbipartite and equitable.
Then the sequence $\lbrace \rho_i \rbrace_{i=0}^{d-1}$ is
constrained.
\end{proposition}
\begin{proof}
We verify that 
$\lbrace \rho_i \rbrace_{i=0}^{d-1}$ satisfies the 
conditions (i), (ii)  
of 
Definition
\ref{def:constrain}. Definition
\ref{def:constrain}(i) holds by
Lemma \ref{def:RhoiCom}.
If $d\leq 2$ then Definition \ref{def:constrain}(ii)  holds vacuosly,
and 
if $d\geq 3$ then
Definition \ref{def:constrain}(ii)  
 holds by
Lemmas  
\ref{lem:NBCon5},
\ref{lem:NBCon3}.
\end{proof}

\begin{lemma}
\label{cor:NotGeomD4}
Assume that $A,B,C$ is nonbipartite and equitable,
but the sequence $\lbrace \rho_{i}\rbrace_{i=0}^{d-1}$
is not geometric. Then $d$ is even and at least
4.
\end{lemma}
\begin{proof}
By Lemma
\ref{lem:NGd4} and
Proposition
\ref{lem:NBCon6}.
\end{proof}

\noindent We turn our attention to bipartite LR triples.

\begin{lemma}
\label{lem:BipRRR2}
Assume that $A,B,C$ is bipartite, equitable, and $d\geq 2$. Then
for $0 \leq i \leq d-1$,
\begin{eqnarray}
\rho_i = \alpha_0 \beta_2 (\varphi_i - \varphi_{i+2}),
\qquad
\rho'_i = \alpha_0 \beta_2 (\varphi'_i - \varphi'_{i+2}),
 \qquad
\rho''_i = \alpha_0 \beta_2 (\varphi''_i - \varphi''_{i+2}).
\label{eq:RRR}
\end{eqnarray}
\end{lemma}
\begin{proof}
To verify the equation on the left in
(\ref{eq:RRR}), set 
 $\alpha_1=0$, $\beta_1=0$ in
 Lemma
\ref{lem:Con1}(i). The other two 
equations in 
(\ref{eq:RRR}) are similarly verified.
\end{proof}

\begin{lemma} 
\label{lem:BipRRR3}
Assume that $A,B,C$ is bipartite, equitable, and $d\geq 4$. 
Then
for 
$2 \leq i \leq d-1$,
\begin{eqnarray}
&&
\alpha_0 \beta_4
(\varphi_{i-2}-\varphi_i) = \alpha_4 \beta_0 (\varphi_i-\varphi_{i+2}),
\label{eq:biprec1}
\\
&&
\alpha_0 \beta_4
(\varphi'_{i-2}-\varphi'_i) = \alpha_4 \beta_0 (\varphi'_i-\varphi'_{i+2}),
\label{eq:biprec2}
\\
&&
\alpha_0 \beta_4
(\varphi''_{i-2}-\varphi''_i) = \alpha_4 \beta_0 (\varphi''_i-\varphi''_{i+2}).
\label{eq:biprec3}
\end{eqnarray}
\end{lemma}
\begin{proof}
To obtain
(\ref{eq:biprec1}),
subtract $\varphi_i$ times
(\ref{eq:L6}) 
from
(\ref{eq:L5}). 
Equations
(\ref{eq:biprec2}),
(\ref{eq:biprec3}) are similarly obtained.
\end{proof}

\begin{lemma}
\label{lem:BipRRR4}
Assume that $A,B,C$ is bipartite, equitable, and $d\geq 4$. 
Then for $1 \leq i \leq d-2$,
\begin{eqnarray}
\alpha_0 \beta_4 \rho_{i-1} = 
\alpha_4 \beta_0 \rho_{i+1},
\qquad
\alpha_0 \beta_4 \rho'_{i-1} = 
\alpha_4 \beta_0 \rho'_{i+1},
\qquad
\alpha_0 \beta_4 \rho''_{i-1} = 
\alpha_4 \beta_0 \rho''_{i+1}.
\label{eq:BipRRR5}
\end{eqnarray}
\end{lemma}
\begin{proof}
Use
Lemmas \ref{lem:BipRRR2},
\ref{lem:BipRRR3}.
\end{proof}

\begin{proposition}
\label{lem:3seqCon}
Assume that $A,B,C$ is bipartite, equitable, and nontrivial.
Then the sequences
$\lbrace \rho_i \rbrace_{i=0}^{d-1}$,
$\lbrace \rho'_i \rbrace_{i=0}^{d-1}$,
$\lbrace \rho''_i \rbrace_{i=0}^{d-1}$
are constrained.
\end{proposition}
\begin{proof}
We verify that
$\lbrace \rho_i \rbrace_{i=0}^{d-1}$,
$\lbrace \rho'_i \rbrace_{i=0}^{d-1}$,
$\lbrace \rho''_i \rbrace_{i=0}^{d-1}$
satisfy the
conditions (i), (ii)  
of 
Definition
\ref{def:constrain}. Definition
\ref{def:constrain}(i) holds by
Lemma \ref{def:BipRhoiCom}.
Recall that $d$ is even.
If $d=2$  then Definition \ref{def:constrain}(ii)  holds vacuosly,
and 
if $d\geq 4$ then
Definition \ref{def:constrain}(ii)  
 holds by
Lemma 
\ref{lem:BipRRR4} and since
$\alpha_4\not=0$,
$\beta_4\not=0$ by
Lemma
\ref{lem:case}.
\end{proof}

\section{The classification of normalized LR triples; an overview}

Throughout this section assume $d\geq 2$.
Our next goal is to classify up to isomorphism 
the normalized LR triples
over $\mathbb F$ that have diameter $d$.
We now describe our strategy.
Consider a normalized LR triple $A,B,C$ over $\mathbb F$
that has 
 parameter array
(\ref{eq:paLRT}).
Recall the sequence
$\lbrace \rho_i\rbrace_{i=0}^{d-1}$ from Definition
\ref{def:Rhoi}.
We place $A,B,C$ into one of four families as follows:
\bigskip

\centerline{
\begin{tabular}[t]{c|c|c}
 {\rm family name} & {\rm family definition}  &{\rm $d$ restriction}
 \\
 \hline
${\rm NBWeyl}_d(\mathbb F)$ & 
over $\mathbb F$; diameter $d$;
nonbipartite; 
normalized;
there exist & 
\\   &scalars
$a,b,c$  in $\mathbb F$
that are not all zero 
such that 
&      \\    &
$a+b+c=0$ 
and
$a \varphi_{i-1} + b\varphi_i + c \varphi_{i+1} = 0$ for
$1 \leq i \leq d$
&
\\
&&
\\
${\rm NBG}_d(\mathbb F)$ & 
over $\mathbb F$; diameter $d$;
nonbipartite;
normalized;
not in \\&
${\rm NBWeyl}_d(\mathbb F)$;
the sequence $\lbrace \rho_i
\rbrace_{i=0}^{d-1}$ is geometric
&
\\
&&
\\
${\rm NBNG}_d(\mathbb F)$ & 
over $\mathbb F$; diameter $d$;
nonbipartite; 
normalized; & $d$ even;
\\&
the sequence 
$\lbrace \rho_i
\rbrace_{i=0}^{d-1}$ 
is not geometric
& $d \geq 4$
\\
&&
\\
${\rm B}_d(\mathbb F)$ &
over $\mathbb F$; diameter $d$;
bipartite;
normalized & $d$ even
     \end{tabular}}
     \bigskip

\noindent As we will show in Lemma
\ref{lem:NBWisGeom},
if $A,B,C$ is contained in
${\rm NBWeyl}_d(\mathbb F)$ then
$\lbrace \rho_i \rbrace_{i=0}^{d-1}$ is geometric.
By this and Lemmas
\ref{lem:case},
\ref{cor:NotGeomD4}
the
 LR triple $A,B,C$ falls into
exactly one of the four families.
\medskip

\noindent 
Over the next four sections, we classify up to isomorphism
the LR triples in each  family.

\section{The classification of LR triples in ${\rm NBWeyl}_d(\mathbb F)$ 
}

\noindent 
In this section we classify up to isomorphism the
LR triples 
in ${\rm NBWeyl}_d(\mathbb F)$, for $d\geq 2$.
We first describe some examples.

\begin{example}  
\label{ex:nbw1}
\rm The LR triple ${\rm NBWeyl}^{+}_d(\mathbb F;j,q)$ is 
over $\mathbb F$, diameter $d$, 
nonbipartite,
normalized,
and satisfies
\begin{eqnarray*}
&&d\geq 2; \qquad \qquad {\mbox{\rm $d$ is even}};
\qquad \qquad 
j \in \mathbb Z, \quad 0 \leq j < d/2;
\qquad \qquad 
0 \not=q \in \mathbb F;
\\
&&
{\mbox{\rm if
${\rm Char}(\mathbb F)\not=2$ then $q$ is a primitive
$(2d+2)$-root of unity;}}
\\
&&
{\mbox{\rm 
if
${\rm Char}(\mathbb F)=2$ then $q$ is a primitive
$(d+1)$-root of unity;}}
\\
&&
\varphi_i = 
\frac{(1+q^{2j+1})^2(1-q^{-2i})}{q^{2j+1}(q-q^{-1})^2}
\qquad \qquad  (1 \leq i \leq d).
\end{eqnarray*}
\end{example}

\begin{example}  
\label{ex:nbw2}
\rm The LR triple ${\rm NBWeyl}^{-}_d(\mathbb F;j,q)$ is
over $\mathbb F$, diameter $d$, 
nonbipartite, 
normalized,
and satisfies
\begin{eqnarray*}
&& 
{\mbox{\rm 
${\rm Char}(\mathbb F)\not=2$;
}}
\qquad \qquad 
d\geq 3; \qquad \qquad {\mbox{\rm $d$ is odd}};
\\
&&
j \in \mathbb Z, \quad 0 \leq j < (d-1)/4;
\qquad \qquad 
0 \not=q \in \mathbb F;
\\
&&
{\mbox{\rm 
$q$ is a primitive
$(2d+2)$-root of unity;}}
\\
&&
\varphi_i = \frac{(1+q^{2j+1})^2(1-q^{-2i})}{q^{2j+1}(q-q^{-1})^2}
\qquad \qquad  (1 \leq i \leq d).
\end{eqnarray*}
\end{example}

\begin{example}  
\label{ex:nbw3}
\rm The LR triple ${\rm NBWeyl}^{-}_d(\mathbb F; t)$ is 
over $\mathbb F$, diameter $d$,
nonbipartite, 
normalized,
and satisfies
\begin{eqnarray*}
&& 
{\mbox{\rm 
${\rm Char}(\mathbb F)\not=2$;
}}
\qquad \qquad 
d\geq 5; \qquad \qquad d \equiv 1  \pmod{4};
\\ &&
0 \not=t \in \mathbb F;
\qquad \qquad 
{\mbox{\rm 
$t$ is a primitive
$(d+1)$-root of unity;}}
\\
&&
\varphi_i = \frac{2t(1-t^{i})}{(1-t)^2}
\qquad \qquad  (1 \leq i \leq d).
\end{eqnarray*}
\end{example}

\begin{lemma}
\label{lem:PreThm} 
For the LR triples in Examples
\ref{ex:nbw1}--\ref{ex:nbw3},
{\rm (i)} they exist;
{\rm (ii)} they are contained in 
${\rm NBWeyl}_d(\mathbb F)$;
{\rm (iii)} they are mutually nonisomorphic.
\end{lemma}
\begin{proof}
(i) 
In Examples
\ref{ex:nbw1},
\ref{ex:nbw2} we see an integer $j$.
For Example \ref{ex:nbw3} define an integer $j=(d-1)/4$.
In Examples
\ref{ex:nbw1},
\ref{ex:nbw2} 
we see a parameter $q \in \mathbb F$.
For Example
\ref{ex:nbw3}, define
$q \in \overline {\mathbb F}$ such that
$t=q^{-2}$.
In each of
 Examples
\ref{ex:nbw1}--\ref{ex:nbw3}
the pair $d,q$ is standard.
For each example we use the data $d,j,q$ 
and Proposition
\ref{prop:AllqWeyl}
to get 
an LR triple over $\overline {\mathbb F}$
that has $q$-Weyl type. This LR triple
is nonbipartite, since its first Toeplitz number is nonzero.
Normalize this LR triple 
and apply Lemma
\ref{lem:SpittingField} to
get the desired LR triple over $\mathbb F$.
\\
\noindent (ii)
Let $A,B,C$ denote an LR triple listed in
Examples
\ref{ex:nbw1}--\ref{ex:nbw3}.
By assumption $A,B,C$ is over $\mathbb F$,
diameter $d$, nonbipartite, and normalized.
Define $a=1$, $b=-1-q^2$, $c=q^2$, where
$t=q^{-2}$ in Example
\ref{ex:nbw3}.
Then $a+b+c=0$, and 
$a\varphi_{i-1}+b\varphi_i + c\varphi_{i+1}=0$
for $1 \leq i \leq d$.
Therefore $A,B,C$ is contained in
 ${\rm NBWeyl}_d(\mathbb F)$.
\\
\noindent
(iii)
Among the
LR triples listed in
Examples
\ref{ex:nbw1}--\ref{ex:nbw3}, no two
have the same parameter array. Therefore
no two are isomorphic.
\end{proof}

\begin{theorem}
\label{thm:ClassW}
For $d\geq 2$,
each LR triple 
 in ${\rm NBWeyl}_d(\mathbb F)$ is isomorphic to a unique
 LR triple listed in
Examples
\ref{ex:nbw1}--\ref{ex:nbw3}.
\end{theorem}
\begin{proof}
 Let $A,B,C$ denote an LR triple in
 ${\rm NBWeyl}_d(\mathbb F)$, with parameter array
(\ref{eq:paLRT}) and Toeplitz data
(\ref{eq:ToeplitzData}).
 By assumption there exist scalars
$a,b,c$ in $\mathbb F$ that are not all zero,
such that $a+b+c=0$ and
$a\varphi_{i-1}+b \varphi_i + c \varphi_{i+1} = 0$
for $1 \leq i \leq d$.
Setting $i=1$ and $\varphi_0=0$ 
we obtain $b\varphi_1 + c\varphi_2=0$.
Setting $i=d$ and $\varphi_{d+1}=0$ 
we obtain $a\varphi_{d-1} + b\varphi_d=0$.
By these comments each of $a,b,c$ is nonzero.
Define a polynomial $g \in \mathbb F\lbrack \lambda \rbrack$ by
$g(\lambda) = a+ b\lambda + c\lambda^2$.
Observe that $g(1)=0$, so there exists $t \in \mathbb F$
such that $g(\lambda) = c(\lambda-1)(\lambda-t)$. We have
$ct=a$, so $t\not=0$. 
Assume for the moment that $t=1$. By construction
there exists $u,v \in \mathbb F$ such that
$\varphi_i = u + vi $ for $0 \leq i \leq d+1$.
setting $i=0$ and $\varphi_0=0$ we obtain $u=0$.
Consequently $\varphi_i = vi$ for $0 \leq i \leq d+1$,
and $v\not=0$.
Fix a square root $v^{1/2}\in \overline {\mathbb F}$.
Define an LR triple  $A^{\vee},B^\vee, C^\vee$ 
over
$\overline {\mathbb F}$ by
\begin{eqnarray*}
A^\vee = A v^{-1/2},
 \qquad 
B^\vee = B v^{-1/2},
 \qquad 
C^\vee = C v^{-1/2}.
\end{eqnarray*}
By Lemma
\ref{lem:albega}
this LR triple has parameter array
\begin{eqnarray*}
\varphi^\vee_i = 
(\varphi'_i)^\vee = 
(\varphi''_i)^\vee = \varphi_i /v=  i \qquad \qquad (1 \leq i \leq d).
\end{eqnarray*}
Now by 
 Definition
\ref{def:WEYL} the LR triple
$A^\vee,
B^\vee,
C^\vee $
has Weyl type.
Consider its
first Toeplitz number 
$\alpha_1^\vee$. On one hand, by 
Lemma
\ref{lem:ToeplitzAdjust}
and the 
construction, 
$\alpha_1^\vee
= \alpha_1 v^{1/2}=v^{1/2}$.
On the other hand,
by 
Lemma
\ref{lem:WEYLalphaZero},
$\alpha_1^\vee = 0 $.
This is a contradiction, so $t \not=1$.
The polynomial $g(\lambda)=c(\lambda-1)(\lambda-t)$ 
has distinct roots. Therefore there exist
$u,v\in \mathbb F$ such that $\varphi_i = u + vt^i$ for
$0 \leq i \leq d+1$. Setting
$i=0$ and $\varphi_0=0$ we obtain $0 = u+v$.
Consequently
$\varphi_i = u(1-t^i)$ for $0 \leq i \leq d+1$, and $u \not=0$.
Fix square roots
$u^{1/2}, t^{1/2} \in 
\overline {\mathbb F}$.
Define
$q=t^{-1/2}$.
By construction $q \not=0$, 
$t=q^{-2}$, and
$q^2 \not=1$.
Define
 an LR triple  $A^{\vee},B^\vee, C^\vee$ 
over
$\overline {\mathbb F}$ by
\begin{eqnarray*}
A^\vee = A u^{-1/2}, \qquad 
B^\vee = B u^{-1/2}, \qquad 
C^\vee = C u^{-1/2}.
\end{eqnarray*}
By Lemma
\ref{lem:albega}
this LR triple has parameter array
\begin{eqnarray*}
\varphi^\vee_i = 
(\varphi'_i)^\vee = 
(\varphi''_i)^\vee = \varphi_i/u= 1-q^{-2i} \qquad \qquad (1 \leq i \leq d).
\end{eqnarray*}
Now by 
 Definition
\ref{def:qExceptional},
the LR triple
$A^\vee,
B^\vee,
C^\vee $
has $q$-Weyl type.
Replacing $q$ by $-q$ if necessary, 
we may assume by Lemma
\ref{lem:Mq}
that the pair $d,q$ is standard in the sense of
Definition
\ref{def:qstand}.
By that definition, if
${\rm Char}(\mathbb F)\not=2$ then $q$ is a primitive
$(2d+2)$-root of unity. Moreover  if
${\rm Char}(\mathbb F)=2$, then $d$ is even and
$q$ is a primitive
$(d+1)$-root of unity.
Consider the first Toeplitz number
$\alpha_1^\vee$.
By Lemma
\ref{lem:alphaOneList} there exists an integer $j$
$(0 \leq j \leq d)$ such that
\begin{eqnarray*}
\alpha_1^\vee = \frac{q^{j+1/2}+ q^{-j-1/2}}{q-q^{-1}}.
\end{eqnarray*}
By Lemma
\ref{lem:ToeplitzAdjust}
and the construction, $\alpha_1^\vee = u^{1/2}$.
Therefore 
$u = (\alpha_1^\vee)^2$. By these comments
\begin{eqnarray}
\label{eq:UVal}
u = \frac{(1+q^{2j+1})^2}{q^{2j+1}(q-q^{-1})^2}.
\end{eqnarray}
Replacing $j$ by $d-j$ corresponds to
replacing  $u^{1/2}$ by 
$-u^{1/2}$, and this move leaves
 $u$ invariant. Replacing $j$ by $d-j$
if necessary, we may assume without loss that $j\leq d/2$.
Note that
$ j \not=d/2$; otherwise
$1+q^{2j+1}=0$ which contradicts
(\ref{eq:UVal}). 
Therefore $j < d/2$.
Assume for the moment that  $d=1+4j$.
Then $q^{2j+1} + q^{-2j-1} = 0$. In this case
(\ref{eq:UVal}) reduces to $u=2/(q-q^{-1})^2$,
or in 
other words $u=2t/(1-t)^2$. Now $d,t$
satisfy  the requirements of
Example
\ref{ex:nbw3}, so 
$A,B,C$ is isomorphic to
${\rm NBWeyl}^{-}_d(\mathbb F; t)$.
For the rest of this proof, assume that
$d \not= 1 + 4j$.
We show $q \in \mathbb F$.
Define $f=q^{2j+1}+ q^{-2j-1}$
and note by 
(\ref{eq:UVal}) that   
$f+2=uq^2(1-q^{-2})^2$.
By this and
$u, q^2 \in \mathbb F$ we find
 $f \in \mathbb F$.
Also, using $q^2 \in \mathbb F$ we obtain
$f q = q^{2j+2} + q^{-2j} \in \mathbb F$.
We have $f \not=0$
since $d \not=1+4j$. By these comments $q = fq/f \in \mathbb F$.
For the moment assume that $d$ is even.
Then  
$d, j,q$ meet the requirements of
Example
\ref{ex:nbw1}, so
$A,B,C$ is isomorphic to
 ${\rm NBWeyl}^{+}_d(\mathbb F;j,q)$.
Next assume that $d$ is odd.
We mentioned earlier that
the pair $d,q$ is standard. The pair $d,q$
remains standard if we replace $q$ by
$-q$. Consider what happens if we replace 
$q$ by $-q$ and
$j$ by $(d-1)/2-j$. By
(\ref{eq:UVal}) 
this replacement
has
no effect on $u$. Making this adjustment if necessary,
we may assume without loss that $j< (d-1)/4$.
Now $d,j,q$ meet the requirements of
Example
\ref{ex:nbw2}, so $A,B,C$ is isomorphic to
 ${\rm NBWeyl}^{-}_d(\mathbb F;j,q)$.
We have shown that $A,B,C$ is isomorphic to at least
one of the LR triples listed in
Examples
\ref{ex:nbw1}--\ref{ex:nbw3}.
The result follows from this and
Lemma
\ref{lem:PreThm}(iii). 
\end{proof}

\begin{lemma}
\label{lem:NBWisGeom}
Assume $d\geq 2$.
Let $A,B,C$ denote an LR triple in
${\rm NBWeyl}_d(\mathbb F)$. Then for $0 \leq i \leq d-1$
the scalar
$\rho_i$ from
Definition
\ref{def:Rhoi} satisfies

     \bigskip

\centerline{
\begin{tabular}[t]{c|ccc}
 {\rm case} 
 & 
 ${\rm NBWeyl}^{+}_d(\mathbb F;j,q)$
 &
 ${\rm NBWeyl}^{-}_d(\mathbb F;j,q)$
 &
 ${\rm NBWeyl}^{-}_d(\mathbb F;t)$
\\
 \hline
 $\rho_i$ &
$-q^{-2i-2}$
&
 $-q^{-2i-2}$
 &
 $-t^{i+1}$ 
   \\
     \end{tabular}}
     \bigskip

\noindent Moreover, the sequence
$\lbrace \rho_i \rbrace_{i=0}^{d-1}$ 
is geometric.
\end{lemma}
\begin{proof}
Compute
$\rho_i=\varphi_{i+1}/\varphi_{d-i}$ using
the data in
Examples
\ref{ex:nbw1}--\ref{ex:nbw3}.
\end{proof}

\section{The classification of LR triples in
${\rm NBG}_d(\mathbb F)$ 
}

\noindent 
In this section we classify up to isomorphism
the LR triples in
${\rm NBG}_d(\mathbb F)$, for $d\geq 2$.
We first describe some examples.

\begin{example}
\label{ex:NBGq}
\rm
The LR triple
${\rm NBG}_d(\mathbb F;q)$
is over $\mathbb F$, diameter $d$, nonbipartite, normalized, and
satisfies
\begin{eqnarray*}
&& d\geq 2; \qquad \qquad 0 \not=q \in \mathbb F;
\\
&&
q^i \not=1 \quad (1 \leq i \leq d);
\qquad \qquad \quad q^{d+1}\not=-1; 
\\
&&\varphi_i = \frac{q(q^i-1)(q^{i-d-1}-1)}{(q-1)^2}
\qquad \qquad (1 \leq i \leq d).
\end{eqnarray*}
\end{example}

\begin{example}
\label{ex:NBGqEQ1}
\rm 
The LR triple
${\rm NBG}_d(\mathbb F;1)$
is over $\mathbb F$, diameter $d$, nonbipartite, normalized, and
satisfies
\begin{eqnarray*}
&& d\geq 2;
\qquad \qquad 
{\mbox {\rm ${\rm Char}(\mathbb F)$ is 0 or greater than $d$}};
\\
&&
\varphi_i = i (i-d-1) \qquad \qquad (1 \leq i \leq d).
\end{eqnarray*}
\end{example}

\begin{lemma}
\label{lem:NBGexist}
For the LR triples in
Examples \ref{ex:NBGq},
\ref{ex:NBGqEQ1},
{\rm (i)} they exist; {\rm (ii)} they are contained in
${\rm NBG}_d(\mathbb F)$;
{\rm (iii)} they are mutually nonisomorphic.
\end{lemma}
\begin{proof} (i) 
In Example
\ref{ex:NBGq} we see a parameter $q \in \mathbb F$.
For Example
\ref{ex:NBGqEQ1} define $q=1$.
Using $q$ we construct an LR triple $A,B,C$ as follows.
For notational convenience define
 $a_i = \varphi_{d-i+1}-\varphi_{d-i}$
for $0 \leq i \leq d$, where $\varphi_0=0$ and
$\varphi_{d+1}=0$.
Let $V$ denote a vector space over $\mathbb F$
with dimension $d+1$.
Let $\lbrace v_i \rbrace_{i=0}^d$ denote a basis for
$V$. Define $A,B,C$ in
${\rm End}(V)$ such that the matrices representing
$A,B,C$ with respect to
 $\lbrace v_i \rbrace_{i=0}^d$
are given by
the first row of the table in
Proposition
\ref{prop:matrixRep}.
Here $\varphi'_i=\varphi_i$ and
$\varphi''_i = \varphi_i $ for $1 \leq i \leq d$.
We show that $A,B,C$ is an LR triple on $V$.
We first show that $A,B$ is an LR pair on $V$.
Let $\lbrace V_i\rbrace_{i=0}^d$ denote the decomposition
of $V$ induced by the basis $\lbrace v_i\rbrace_{i=0}^d$.
Using the matrices defining
$A,B$ we find that 
$\lbrace V_i\rbrace_{i=0}^d$ is lowered by $A$
and raised by $B$. Therefore
$A,B$ is an LR pair on $V$.  
Next we show that $B,C$ is an LR pair on $V$.
Define the scalars $\lbrace \alpha_i \rbrace_{i=0}^d$ by
$\alpha_0=1$ and $\alpha_{i-1}/\alpha_{i} = 
\sum_{k=0}^{i-1} q^k$
for $1 \leq i \leq d$.  Define $\mathbb B = \sum_{i=0}^d \alpha_i B^i$.
Note that $\mathbb B B = B \mathbb B$.
With respect to the basis
$\lbrace v_i\rbrace_{i=0}^d$ the matrix representing
$\mathbb B$ is lower triangular, with each diagonal entry  1.
Therefore
$\mathbb B$ is invertible. 
Observe that  $\lbrace \mathbb B V_{d-i}\rbrace_{i=0}^d$
is a decomposition of $V$ that is lowered by $B$.
Using the matrices defining $B,C$
one checks that
 $\lbrace \mathbb B V_{d-i}\rbrace_{i=0}^d$ is raised by $C$.
 By these comments
$B,C$ is an LR pair on $V$.
Next we show that $C,A$ is an LR pair on $V$.
Define ${\mathbb A}^{\S}  = \sum_{i=0}^d \beta_i A^i$, where
 $\beta_0=1$ and $\beta_{i-1}/\beta_{i} = 
- \sum_{k=0}^{i-1}   q^{-k}$ 
for 
 $1 \leq i \leq d$.
Note that $\mathbb A^{\S} A = A \mathbb A^{\S}$.
With respect to the basis $\lbrace v_i\rbrace_{i=0}^d$
the matrix representing
$\mathbb A^{\S}$ is upper triangular and Toeplitz,
with parameters $\lbrace \beta_i \rbrace_{i=0}^d$.
Therefore $\mathbb A^{\S}$ is invertible.
(In fact $\mathbb A^{\S}$ is the inverse of
$\mathbb A = \sum_{i=0}^d \alpha_i A^i$, although we
do not need this result).
Observe that 
 $\lbrace \mathbb A^{\S}V_{d-i}\rbrace_{i=0}^d$ is a decomposition of
 $V$ that is raised by $A$.
Using the matrices defining $A,C$
one checks that
 $\lbrace \mathbb A^{\S} V_{d-i}\rbrace_{i=0}^d$ is lowered by $C$.
 By these comments
$C,A$ is an LR pair on $V$.
We have shown that $A,B,C$ is an LR triple on $V$.
Using the matrices defining $A,B,C$ 
we find that this LR triple is the desired one.
\\
\noindent (ii)
Let $A,B,C$ denote an LR triple listed in Examples
 \ref{ex:NBGq},
\ref{ex:NBGqEQ1}. By assumption $A,B,C$ is over $\mathbb F$,
diameter $d$, nonbipartite, and normalized.
We check that $A,B,C$ is not in
 ${\rm NBWeyl}_d(\mathbb F)$.
The matrix
\begin{eqnarray*}
 \left(
\begin{array}{ccc}
1 & 1 &  1  
\\
 \varphi_0 & \varphi_1 &\varphi_2 
\\
 \varphi_1 & \varphi_2 &\varphi_3 
\end{array}
\right)
\end{eqnarray*}
has determinant
$-(q+1)(q^{d+1}+1)q^{2-2d}$
for 
${\rm NBG}_d(\mathbb F;q)$, and
$-4$ 
for
${\rm NBG}_d(\mathbb F;1)$.
In each case the determinant is nonzero.  Consequently,
there does not exist $a,b,c$ in $\mathbb F$ that are not all zero 
such that $a+b+c=0$ and $
a\varphi_{i-1} + b\varphi_i + c \varphi_{i+1} = 0$
for $1 \leq i \leq d$.
Therefore  $A,B,C$ is not in
 ${\rm NBWeyl}_d(\mathbb F)$. 
We check that the sequence 
$\lbrace \rho_i\rbrace_{i=0}^{d-1}$ from Definition
\ref{def:Rhoi}
is geometric. 
For $0 \leq i \leq d-1$
the scalar $\rho_i = \varphi_{i+1}/\varphi_{d-i}$ is equal to
$q^{2i-d+1}$ for 
${\rm NBG}_d(\mathbb F;q)$, and 1 for
${\rm NBG}_d(\mathbb F;1)$.
Therefore $\lbrace \rho_i\rbrace_{i=0}^{d-1}$ is 
geometric.
We have shown that 
  $A,B,C$ is contained in
 ${\rm NBG}_d(\mathbb F)$.
\\
\noindent (iii) Among the
LR triples listed in Examples
\ref{ex:NBGq},
\ref{ex:NBGqEQ1} no two have the same parameter array.
Therefore no two are isomorphic.
\end{proof}

\begin{theorem} 
\label{thm:NBG}
For $d\geq 2$,
each LR triple 
 in ${\rm NBG}_d(\mathbb F)$ is isomorphic to a unique
LR triple listed in
Examples
\ref{ex:NBGq},
\ref{ex:NBGqEQ1}.
\end{theorem}
\begin{proof}
Let $A,B,C$ denote 
an LR triple in
 ${\rm NBG}_d(\mathbb F)$,
with parameter array
(\ref{eq:paLRT}) and
 Toeplitz data
(\ref{eq:ToeplitzData}).
Recall the sequence
$\lbrace \rho_i \rbrace_{i=0}^{d-1}$ from
Definition
\ref{def:Rhoi}. This sequence is
constrained by
Proposition 
\ref{lem:NBCon6}, and
geometric by the definition of
 ${\rm NBG}_d(\mathbb F)$.
Therefore there exists $0 \not= r \in \mathbb F$
such that
\begin{eqnarray}
\label{eq:gammaForm}
&&
\rho_i = \rho_0 r^i \qquad \qquad (0 \leq i \leq d-1),
\\
&& \rho^2_0 = r^{1-d}.
\label{eq:rhoSquared}
\end{eqnarray}
By assumption $A,B,C$ is nonbipartite and normalized,
so $\alpha_1=1$
and $\beta_1 = -1$. Also $\alpha_2+\beta_2=1$ from
above
Lemma \ref{lem:reform1}.
Define
\begin{eqnarray}
q = \begin{cases}
 \beta_2/\alpha_2
&  {\mbox{\rm if $\alpha_2\not=0 $}}; \\
\infty & {\mbox{\rm if $\alpha_2 =0$.}}
\end{cases}
\label{eq:qDEF}
\end{eqnarray}
Conceivably
$q=0$ or $q=1$.
 Using 
 $\beta_2 = q \alpha_2$
and 
 $\alpha_2+\beta_2=1$,
we obtain $q \not=-1$ and
\begin{eqnarray}
\alpha_2 = \frac{1}{1+q},
\qquad \qquad 
\beta_2 = \frac{q}{1+q}.
\label{lem:qbasics}
\end{eqnarray}
If $q=\infty$ then $\beta_2=1$.
Evaluating 
(\ref{eq:middleElim}) using
(\ref{lem:qbasics}) we obtain
\begin{eqnarray}
\label{eq:mainrec}
\rho_i = \frac{q \varphi_i-(q+1)\varphi_{i+1} +\varphi_{i+2}}{q+1}
\qquad \qquad (0 \leq i \leq d-1).
\end{eqnarray}
If $q=\infty$ then
(\ref{eq:mainrec}) becomes
$\rho_i = \varphi_i - \varphi_{i+1}$ for $0 \leq i \leq d-1$.
Until further notice, assume that
$q\not=0$, $q\not=\infty$, and 
$1,q,r$ are mutually distinct.
Define
\begin{eqnarray}
L = \frac{(q+1) \rho_0}{(q-r)(1-r)}.
\label{eq:Tform}
\end{eqnarray}
Note that $L\not=0$.
Since $q \not=1$, there exist
$H,K \in \mathbb F$ such that 
$\varphi_i = H + Kq^i + L r^i$  
for $i=0$ and $i=1$.
Using 
(\ref{eq:gammaForm}),
(\ref{eq:mainrec}) and induction on $i$, we obtain
\begin{eqnarray}
\varphi_i = H + K q^i + L r^i \qquad \qquad (0 \leq i \leq d+1).   
\label{eq:vpForm}
\end{eqnarray}
We have $K \not=0$; otherwise
$r\varphi_{i-1} -(r+1)\varphi_i + \varphi_{i+1}=0$
$(1 \leq i \leq d$), 
putting  $A,B,C$ in
${\rm NBWeyl}_d(\mathbb F)$ for a contradiction.
Since $\varphi_0=0$ and $\varphi_{d+1}=0$,
\begin{eqnarray}
0 = H+K+L, \qquad \qquad 0 = H+ Kq^{d+1} + Lr^{d+1}.
\label{eq:RST}
\end{eqnarray}
For $0 \leq i \leq d+1$ define
\begin{eqnarray}
\Delta_i = \varphi_i - \rho_0 r^{i-1} \varphi_{d-i+1}.
\label{eq:del}
\end{eqnarray}
We claim $\Delta_i=0$. This is the case
for $i=0$ and $i=d+1$, since
$\varphi_0=0$ and $\varphi_{d+1}=0$. For $1 \leq i \leq d$
we have $\Delta_i = \varphi_i-\rho_{i-1}\varphi_{d-i+1}$
by
(\ref{eq:gammaForm}), and this is zero by
(\ref{eq:mainrec2}). The claim is proven.
For $0 \leq i \leq d+1$, in the equation $\Delta_i=0$ 
we evaluate the left-hand side using
(\ref{eq:vpForm}),
(\ref{eq:del})
to find that the following linear
combination is zero:

     \bigskip

\centerline{
\begin{tabular}[t]{c|cccc}
 {\rm term} & $1$ & $q^i$ & $r^i$ & $(r/q)^i$
 \\
 \hline
 {\rm coefficient} &
$H-L\rho_0 r^d$
&
 $K$
 &
 $L-H\rho_0 r^{-1}$ 
 &
  $-K \rho_0 q^{d+1}r^{-1}$
   \\
     \end{tabular}}
     \bigskip

\noindent 
By assumption $1,q,r$ are mutually distinct. Also
$K\not=0$ and $d\geq 2$. We show $r=q^2$.
Assume $r\not=q^2$. Then $1,q,r,r/q$ are mutually
distinct.
Setting $i=0,1,2,3$ in the above table,
we obtain a $4 \times 4$ homogeneous linear system
with coefficient matrix 
\begin{eqnarray*}
 \left(
\begin{array}{cccc}
1 & 1  & 1 & 1 
\\
 1 & q & r & r/q
\\
 1 & q^2 & r^2 & (r/q)^2
\\
 1 & q^3 & r^3 & (r/q)^3
\end{array}
\right).
\end{eqnarray*}
This matrix is Vandermonde and hence invertible.
Therefore each coefficient in the table
is zero. The coeffient $K$ is nonzero,
for a contradiction. Consequently $r=q^2$.
From further examination of the coefficients in the table, 
\begin{eqnarray}
\label{eq:threeCoef}
H= L \rho_0 r^d,\qquad \qquad 
K = K \rho_0 q^{d+1}r^{-1},
\qquad \qquad
L = H\rho_0 r^{-1}.
\end{eqnarray} 
By (\ref{eq:threeCoef}) and $r=q^2$ we find
 $\rho_0 = q^{1-d}$ and
$H = L q^{d+1}$. By these comments and
(\ref{eq:Tform}) we obtain
\begin{eqnarray}
H = q(q-1)^{-2},\qquad \qquad 
L = q^{-d}(q-1)^{-2}.
\label{eq:RSTsol}
\end{eqnarray}
Evaluating 
(\ref{eq:vpForm}) using 
(\ref{eq:RSTsol}) and $K=-H-L$,  we obtain
\begin{eqnarray}
\varphi_i = \frac{q(q^i-1)(q^{i-d-1}-1)}{(q-1)^2}
\qquad \qquad (1 \leq i \leq d).
\label{eq:vpsol}
\end{eqnarray}
From 
(\ref{eq:vpsol}) and since the
$\lbrace \varphi_i \rbrace_{i=1}^d$ are nonzero,
we obtain $q^i \not=1$ $(1 \leq i \leq d)$.
Note that $q^{d+1}\not=-1$; otherwise
(\ref{eq:vpsol}) becomes
$\varphi_i = q(1-q^{2i})(q-1)^{-2}$ $(1 \leq i \leq d)$,
forcing $q\varphi_{i-1}-(q+q^{-1})\varphi_i + q^{-1} \varphi_{i+1}=0$
 $(1 \leq i \leq d)$,
putting $A,B,C$ in
${\rm NBWeyl}_d(\mathbb F)$ for a contradiction.
We have met the requirements of
Example
\ref{ex:NBGq},
so $A,B,C$ is isomorphic to 
${\rm NBG}_d(\mathbb F;q)$.
We are done with the case in which $q\not=0$, $q\not=\infty$,
and $1,q,r$ are mutually
distinct. 
Until further notice, assume that $1=q=r$. 
We have 
$1=q\not=-1$, so
${\rm Char}(\mathbb F) \not=2$.
By
(\ref{eq:gammaForm}),
(\ref{eq:rhoSquared})
we obtain $\rho_i = \rho_0$ for $0 \leq i \leq d-1$, 
and $\rho^2_0=1$. By 
(\ref{eq:mainrec}),
\begin{eqnarray}
\label{eq:mainrecOne}
2 \rho_0 = \varphi_i-2\varphi_{i+1}+ \varphi_{i+2}
\qquad \qquad (0 \leq i \leq d-1).
\end{eqnarray}
Define $Q=\varphi_1 - \rho_0$, and note that
$\varphi_i = i (Q+\rho_0 i)$ for $i=0$ and $i=1$.
By 
(\ref{eq:mainrecOne}) and induction on $i$,
\begin{eqnarray}
\label{eq:vpFormOne}
\varphi_i = i (Q+\rho_0 i) \qquad \qquad (0 \leq i \leq d+1).
\end{eqnarray}
Mimicking the argument below
(\ref{eq:del}), we find that for
$0 \leq i \leq d+1$,
\begin{eqnarray}
0 = \varphi_i - \rho_0 \varphi_{d-i+1}.
\label{eq:VarRho}
\end{eqnarray}
Evaluate the right-hand side of (\ref{eq:VarRho})
using
(\ref{eq:vpFormOne}) and $\rho^2_0=1$,
to find that the following linear
combination is zero:

     \bigskip

\centerline{
\begin{tabular}[t]{c|ccc}
 {\rm term} & $1$ & $i$ & $i^2$ 
 \\
 \hline
 {\rm coefficient} &
$-(d+1) (d+1+\rho_0 Q)$
&
 $2(d+1)+Q(1+\rho_0)$
 &
 $\rho_0-1$ 
   \\
     \end{tabular}}
     \bigskip

\noindent 
Since $d\geq 2$, each coefficient in the table is zero.
Therefore $\rho_0=1$ and $Q=-d-1$.
By this and (\ref{eq:vpFormOne}),
\begin{eqnarray}
\label{eq:RQone}
\varphi_i = i(i-d-1) \qquad \qquad (1 \leq i \leq d).
\end{eqnarray}
By 
(\ref{eq:RQone}) and since
$\lbrace \varphi_i \rbrace_{i=1}^d$ are nonzero,
${\rm Char}(\mathbb F)$ is $0$ or greater than $d$.
We have met the requirements of
Example
\ref{ex:NBGqEQ1},
so $A,B,C$ is isomorphic to 
${\rm NBG}_d(\mathbb F;1)$.
We are done with the case $1 = q = r$.
The remaining cases are
(a) $q=0 $ and  $r \not=1$;
(b) $q=0$ and $r =1$;
(c) $q=\infty$ and $r\not=1$;
(d) $q=\infty$ and $r=1$;
(e) $1=q\not=r$;
(f) $1\not=q=r$;
(g) $1=r\not=q$ and $q\not=0, q\not=\infty$.
Each case (a)--(g) is handled in a manner similar to the 
first two.
In each case 
we obtain a contradiction; the details are routine and omitted.
We have shown that $A,B,C$ is isomorphic to at least
one LR triple in Examples
\ref{ex:NBGq},
\ref{ex:NBGqEQ1}.
The result follows in view of Lemma
\ref{lem:NBGexist}(iii).
\end{proof}

\begin{lemma}
\label{lem:NBGrho}
Assume $d\geq 2$.
Let $A,B,C$ denote an LR triple in
${\rm NBG}_d(\mathbb F)$. Then for $0 \leq i \leq d-1$
the scalar
$\rho_i$ from
Definition
\ref{def:Rhoi} satisfies

     \bigskip

\centerline{
\begin{tabular}[t]{c|cc}
 {\rm case} 
 & 
 ${\rm NBG}_d(\mathbb F;q)$
 &
 ${\rm NBG}_d(\mathbb F;1)$
\\
 \hline
 $\rho_i$ &
$q^{2i-d+1}$
&
 $1$
   \\
     \end{tabular}}
     \bigskip

\end{lemma}
\begin{proof}
Compute
$\rho_i=\varphi_{i+1}/\varphi_{d-i}$ using
the data in
Examples
\ref{ex:NBGq},
\ref{ex:NBGqEQ1}.
\end{proof}

\section{The classification of LR triples in
${\rm NBNG}_d(\mathbb F)$ 
}

\noindent 
In this section we classify up to isomorphism the LR triples in
${\rm NBNG}_d(\mathbb F)$, for even $d\geq 4$. We first describe some
examples.

\begin{example}
\label{ex:NBNGdt}
\rm The LR triple 
${\rm NBNG}_d(\mathbb F; t)$
is over $\mathbb F$, diameter $d$, nonbipartite, normalized,
and satisfies
\begin{eqnarray*}
&& d\geq 4; \qquad \qquad 
{\mbox {\rm $d$ is even}};
\qquad \qquad 
0 \not=t \in \mathbb F;
\\
&&
t^i \not= 1 \quad (1 \leq i \leq d/2);
\qquad \qquad \quad 
t^{d+1}\not=1;
\\
&&
\varphi_i = \begin{cases}
 t^{i/2}-1  &  {\mbox{\rm if $i$ is even}}; \\
t^{(i-d-1)/2}-1 & {\mbox{\rm if $i$ is odd}}
\end{cases}
\qquad \qquad (1 \leq i \leq d).
\end{eqnarray*}
\end{example}
 

\begin{lemma}
\label{lem:NBNGpre}
For the LR triples in
Example \ref{ex:NBNGdt}, 
{\rm (i)} they exist;
{\rm (ii)} they are contained in 
${\rm NBNG}_d(\mathbb F)$;
{\rm (iii)} they are mutually nonisomorphic.
\end{lemma}
\begin{proof} (i) Similar to the proof of
Lemma \ref{lem:NBGexist}(i), except that
the sequences
$\lbrace \alpha_i \rbrace_{i=0}^d$,
$\lbrace \beta_i \rbrace_{i=0}^d$ are now
defined as follows:
$\alpha_0 =1$, $\beta_0=1$ and for 
$1 \leq i \leq d$,
\begin{eqnarray*}
&&
\alpha_{i-2}/\alpha_i = 1-t^{i/2}, \qquad \qquad 
\beta_{i-2}/\beta_i = 1-t^{-i/2}
\qquad \qquad {\mbox {\rm (if $i$ is even),}}
\\
&&
\alpha_i = \alpha_{i-1},
\qquad \qquad 
\beta_i = -\beta_{i-1}
\qquad \qquad {\mbox {\rm (if $i$ is odd).}}
\end{eqnarray*}
\noindent (ii)
Let $A,B,C$ denote an LR triple listed in Example
\ref{ex:NBNGdt}.
 By assumption $A,B,C$ is over $\mathbb F$,
diameter $d$, nonbipartite, and normalized.
We check that the sequence 
$\lbrace \rho_i\rbrace_{i=0}^{d-1}$ from Definition
\ref{def:Rhoi}
is not geometric. 
For $0 \leq i \leq d-1$
the scalar $\rho_i = \varphi_{i+1}/\varphi_{d-i}$ is equal to
$-t^{(i-d)/2}$ (if $i$ is even) and
$-t^{(i+1)/2}$  (if $i$ is odd).
By this and since $t^{d+1}\not=1$, the sequence
$\lbrace \rho_i\rbrace_{i=0}^{d-1}$ is not geometric.
 By these comments 
  $A,B,C$ is contained in
${\rm NBNG}_d(\mathbb F)$.
\\
(iii) Similar to the proof of
Lemma \ref{lem:NBGexist}(iii).
\end{proof}

\begin{theorem} 
\label{thm:NBNG}
Assume that $d$ is even and at least 4.
Then each LR triple in
${\rm NBNG}_d(\mathbb F)$ is isomorphic to a unique
 LR triple listed in
Example \ref{ex:NBNGdt}.
\end{theorem}
\begin{proof}
Let $A,B,C$ denote an LR triple in
${\rm NBNG}_d(\mathbb F)$, with
parameter array
(\ref{eq:paLRT}) and Toeplitz data
(\ref{eq:ToeplitzData}).
Recall the sequence
$\lbrace \rho_i \rbrace_{i=0}^{d-1}$ from
Definition
\ref{def:Rhoi}. This sequence is
constrained by
Proposition 
\ref{lem:NBCon6}, so by
Proposition \ref{prop:noddPre} there exists $0 \not=t \in \mathbb F$ 
such that
\begin{eqnarray}
&&
\label{eq:gammaNew}
\rho_i = \begin{cases}
\rho_0 t^{i/2}  &  {\mbox{\rm if $i$ is even}}; \\
\rho_0^{-1}t^{(i-d+1)/2} & {\mbox{\rm if $i$ is odd}}
\end{cases}
\qquad \qquad (0 \leq i \leq d-1).
\end{eqnarray}
By assumption 
$\lbrace \rho_i \rbrace_{i=0}^{d-1}$ is not geometric,
so by 
Lemma
\ref{lem:ConGeo}(iv),
\begin{eqnarray}
 \rho_0^4 \not=t^{1-d}.
\label{eq:gamma4}
\label{eq:x4}
\end{eqnarray}
We claim that
\begin{equation}
t (\varphi_{i-2}-\varphi_{i-1}) = \varphi_i - \varphi_{i+1}
\qquad \qquad (2 \leq i \leq d).
\label{eq:extrarecNG}
\end{equation}
To prove the claim, consider the
scalars
$a,b,c$ from 
Definition
\ref{def:NBCon2}.
By Lemma
\ref{lem:NBCon5}
the 3-tuple $(a,b,c)$ is a linear
constraint for 
$\lbrace \rho_i \rbrace_{i=0}^{d-1}$ in the sense of
Definition
\ref{def:lincon}.
Now using
Definition
\ref{def:LC}
and
Proposition
\ref{prop:LCbasis}(ii) we obtain
$a=-tc$ and $b=0$.
The claim follows from this and
 Lemma
\ref{lem:NBCon4}.
We show that $t\not=1$. Suppose $t=1$. 
By 
(\ref{eq:extrarecNG}), for $2 \leq i \leq d$ we have  
\begin{eqnarray}
\varphi_{i-2}-\varphi_{i-1} = \varphi_i - \varphi_{i+1}.
\label{eq:tOne}
\end{eqnarray}
Sum 
(\ref{eq:tOne}) over $i=2,3,\ldots, d$ and use
$\varphi_0= 0= \varphi_{d+1}$ to obtain
$-\varphi_{d-1} = \varphi_{2}$. Set $i=d-2$ in
(\ref{eq:mainrec2}) and use
(\ref{eq:gammaNew}) with $t=1$ 
to find
$\rho_0 = \varphi_{d-1}/\varphi_{2}=-1$,
which contradicts 
(\ref{eq:gamma4}). 
We have shown $t \not=1$.
There exist $H,K,L \in \mathbb F$ such that
for $i=0,1,2$,
\begin{eqnarray}
\varphi_i = \begin{cases}
H+Kt^{i/2} &  {\mbox{\rm if $i$ is even}}; \\
H+Lt^{(i-d-1)/2} & {\mbox{\rm if $i$ is odd}}.
\end{cases}
\label{eq:gentNG012}
\end{eqnarray}
By 
(\ref{eq:extrarecNG})
and induction on $i$,
(\ref{eq:gentNG012}) holds for $0 \leq i \leq d+1$.
Using $\varphi_0 = 0$ and $\varphi_{d+1}=0$,
we obtain
$H+K=0$ and $H+L=0$.
Now 
(\ref{eq:gentNG012})
becomes
\begin{eqnarray}
\varphi_i = \begin{cases}
H(1-t^{i/2}) &  {\mbox{\rm if $i$ is even}}; \\
H(1-t^{(i-d-1)/2}) & {\mbox{\rm if $i$ is odd}}
\end{cases}
\qquad \qquad (1 \leq i \leq d).
\label{eq:gentNGall}
\end{eqnarray}
The scalars  $\lbrace \varphi_i\rbrace_{i=1}^d$
are nonzero. Consequently $H\not=0$, and
$t^i \not=1$ for $1 \leq i \leq d/2$.
Evaluating $\rho_0 = \varphi_1/\varphi_d$ using
(\ref{eq:gentNGall}) we obtain
$\rho_0 = -t^{-d/2}$. 
Now 
(\ref{eq:gammaNew}) becomes
\begin{eqnarray}
\label{eq:needthisRHO}
\rho_i = \begin{cases}
 -t^{(i-d)/2}  &  {\mbox{\rm if $i$ is even}}; \\
-t^{(i+1)/2} & {\mbox{\rm if $i$ is odd}}
\end{cases}
\qquad \qquad (0 \leq i \leq d-1),
\end{eqnarray}
and 
(\ref{eq:gamma4}) becomes $t^{d+1}\not=1$.
We show $H=-1$.
As in the proof of
Theorem
\ref{thm:NBG}, 
there exists $q \in \mathbb F \cup \lbrace \infty \rbrace $ such that
$q \not=-1$ and
\begin{eqnarray}
\label{eq:mainrecREP}
\rho_i = \frac{q \varphi_i-(q+1)\varphi_{i+1} +\varphi_{i+2}}{q+1}
\qquad \qquad (0 \leq i \leq d-1).
\end{eqnarray}
Evaluate the recursion 
(\ref{eq:mainrecREP}) using
(\ref{eq:gentNGall}),
(\ref{eq:needthisRHO}).
For $i$ even
this 
gives
\begin{eqnarray}
1+H^{-1} = \frac{q+t}{q+1}t^{d/2},
\label{eq:qtRel1}
\end{eqnarray}
and for $i$ odd this gives
\begin{eqnarray}
1+H^{-1} = \frac{q+t}{q+1}t^{-d/2-1}.
\label{eq:qtRel2}
\end{eqnarray}
Combining
(\ref{eq:qtRel1}),
(\ref{eq:qtRel2}) we obtain
\begin{eqnarray*}
\frac{(q+t)(1-t^{d+1})}{q+1} = 0.
\end{eqnarray*}
But $t^{d+1}\not=1$,
so $q=-t$, and therefore 
$H=-1$ in view of
(\ref{eq:qtRel1}).
Setting $H=-1$ in (\ref{eq:gentNGall}) we obtain
\begin{eqnarray*}
\varphi_i = \begin{cases}
 t^{i/2}-1  &  {\mbox{\rm if $i$ is even}}; \\
t^{(i-d-1)/2}-1 & {\mbox{\rm if $i$ is odd}}
\end{cases}
\qquad \qquad (1 \leq i \leq d).
\end{eqnarray*}
We have met the requirements of Example
\ref{ex:NBNGdt}, so
$A,B,C$ is isomorphic to 
${\rm NBNG}_d(\mathbb F; t)$.
We have shown that $A,B,C$ is isomorphic to
at least one LR triple in
Example \ref{ex:NBNGdt}. The result follows
in view of Lemma
\ref{lem:NBNGpre}(iii).
\end{proof}

\noindent We record a fact from the proof of Theorem
\ref{thm:NBNG}.

\begin{lemma}
Assume that $d$ is even and at least $4$. Let $A,B,C$ denote an LR triple in
${\rm NBNG}_d(\mathbb F)$.
Then for $0 \leq i \leq d-1$ the
scalar $\rho_i $ from Definition
\ref{def:Rhoi}
satisfies
\begin{eqnarray*}
\rho_i = \begin{cases}
 -t^{(i-d)/2}  &  {\mbox{\rm if $i$ is even}}; \\
-t^{(i+1)/2} & {\mbox{\rm if $i$ is odd}}.
\end{cases}
\end{eqnarray*}
\end{lemma}

\section{The classification of LR triples in ${\rm B}_d(\mathbb F)$}

\noindent 
In this section we classify up to isomorphism the LR triples in
${\rm B}_d(\mathbb F)$, for even $d\geq 2$.
We first describe some examples.

\begin{example}
\label{ex:BIP}
\rm
The LR triple 
${\rm B}_d(\mathbb F;t,\rho_0, \rho'_0, \rho''_0)$
is over $\mathbb F$, diameter $d$, bipartite,
normalized, and satisfies
\begin{eqnarray*}
&&  d\geq 4; \qquad \quad 
{\mbox {\rm $d$ is even}};
\qquad \quad 
0 \not=t \in\mathbb F; 
\qquad \quad t^i \not= 1 \quad (1 \leq i \leq d/2);
\\
&&
\rho_0, \rho'_0, \rho''_0 \in \mathbb F;
\qquad \qquad  \quad 
\rho_0 \rho'_0 \rho''_0 = -t^{1-d/2};
\\
&&\varphi_i = \begin{cases}
\rho_0 \frac{1-t^{i/2}}{1-t}  &  {\mbox{\rm if $i$ is even}}; \\
\frac{t}{\rho_0} \frac{1-t^{(i-d-1)/2}}{1-t} & {\mbox{\rm if $i$ is odd}}
\end{cases}
\qquad \qquad (1 \leq i \leq d);
\\
&&\varphi'_i = \begin{cases}
\rho'_0 \frac{1-t^{i/2}}{1-t}  &  {\mbox{\rm if $i$ is even}}; \\
\frac{t}{\rho'_0} \frac{1-t^{(i-d-1)/2}}{1-t} & {\mbox{\rm if $i$ is odd}}
\end{cases}
\qquad \qquad (1 \leq i \leq d);
\\
&&\varphi''_i = \begin{cases}
\rho''_0 \frac{1-t^{i/2}}{1-t}  &  {\mbox{\rm if $i$ is even}}; \\
\frac{t}{\rho''_0} \frac{1-t^{(i-d-1)/2}}{1-t} & {\mbox{\rm if $i$ is odd}}
\end{cases}
\qquad \qquad (1 \leq i \leq d).
\end{eqnarray*}
\end{example}

\begin{example}
\label{ex:BipONE}
\rm
The LR triple
${\rm B}_d(\mathbb F;1, \rho_0, \rho'_0, \rho''_0)$
is over $\mathbb F$, diameter $d$, bipartite,
normalized, and satisfies
\begin{eqnarray*}
&&  d\geq 4; \quad \qquad 
{\mbox {\rm $d$ is even}};
\qquad \quad 
{\mbox {\rm ${\rm Char}(\mathbb F)$ is 0 or greater than $d/2$}};
\\
&& 
\rho_0, \rho'_0, \rho''_0 \in \mathbb F;
\qquad \qquad \qquad 
\rho_0 \rho'_0 \rho''_0 = -1;
\\&&
\varphi_i = \begin{cases}
\frac{i \rho_0 }{2}  &  {\mbox{\rm if $i$ is even}}; \\
 \frac{i-d-1}{2\rho_0} & {\mbox{\rm if $i$ is odd}}
\end{cases}
\qquad \qquad (1 \leq i \leq d);
\\
&&\varphi'_i = \begin{cases}
\frac{i  \rho'_0 }{2}  &  {\mbox{\rm if $i$ is even}}; \\
\frac{i-d-1}{2 \rho'_0} & {\mbox{\rm if $i$ is odd}}
\end{cases}
\qquad \qquad (1 \leq i \leq d);
\\
&&\varphi''_i = \begin{cases}
 \frac{i \rho''_0 }{2}  &  {\mbox{\rm if $i$ is even}}; \\
\frac{i-d-1}{2\rho''_0} & {\mbox{\rm if $i$ is odd}}
\end{cases}
\qquad \qquad (1 \leq i \leq d).
\end{eqnarray*}
\end{example}

\begin{example}
\label{ex:BipONEd2}
\rm
The LR triple
${\rm B}_2(\mathbb F;\rho_0, \rho'_0, \rho''_0)$
is over $\mathbb F$, diameter $2$, bipartite,
normalized, and satisfies
\begin{eqnarray*}
&&  
\rho_0, \rho'_0, \rho''_0 \in \mathbb F;
\qquad \qquad \qquad 
\rho_0 \rho'_0 \rho''_0 = -1;
\\&&
\varphi_1=-1/ \rho_0,
\qquad \quad \varphi'_1=-1/ \rho'_0,
\qquad \quad 
\varphi''_1=-1/ \rho''_0,
\\&&
\varphi_2= \rho_0,
\qquad \qquad \quad  \varphi'_2=\rho'_0,
\qquad \qquad  \quad 
\varphi''_2=\rho''_0.
\end{eqnarray*}
\end{example}


\begin{lemma}
\label{lem:BIPpre}
For the LR triples in
Examples
\ref{ex:BIP}--\ref{ex:BipONEd2},
{\rm (i)} they exist;
{\rm (ii)} they are contained in 
${\rm B}_d(\mathbb F)$;
{\rm (iii)} they are mutually nonisomorphic.
\end{lemma}
\begin{proof} Without loss we may assume $d\geq 4$,
since for $d=2$ the result follows
from Lemma
\ref{lem:dtwoClassB}. 
\\
\noindent 
(i) 
In Example 
\ref{ex:BIP} we see a parameter $t \in \mathbb F$.
For Example
\ref{ex:BipONE} define $t=1$.
We proceed as in 
the proof of
Lemma \ref{lem:NBGexist}(i), except that
now $a_i = 0 $ for $0 \leq i \leq d$ and
the sequences
$\lbrace \alpha_i \rbrace_{i=0}^d$,
$\lbrace \beta_i \rbrace_{i=0}^d$ are
defined as follows:
$\alpha_0 =1$, $\beta_0=1$ and for 
$1 \leq i \leq d$,
\begin{eqnarray*}
&&
\alpha_{i-2}/\alpha_i = \sum_{k=0}^{i/2-1} t^k
\qquad \qquad 
\beta_{i-2}/\beta_i =  
-\sum_{k=0}^{i/2-1} t^{-k}
\qquad \qquad {\mbox {\rm (if $i$ is even),}}
\\
&&
\alpha_i = 0,
\qquad \qquad 
\beta_i = 0
\qquad \qquad {\mbox {\rm (if $i$ is odd).}}
\end{eqnarray*}
\noindent (ii) By construction.
\\
(iii) Similar to the proof of
Lemma \ref{lem:NBGexist}(iii).
\end{proof}

\begin{theorem}
\label{thm:BIPclass}
Assume that $d$ is even and at least 2. Then each
LR triple in
${\rm B}_d(\mathbb F)$ is isomorphic to 
a unique LR triple listed in
Examples
\ref{ex:BIP}--\ref{ex:BipONEd2}.
\end{theorem}
\begin{proof}
Without loss we may assume $d\geq 4$, since for
$d=2$ the result follows from
Lemma
\ref{lem:dtwoClassB}.
Let $A,B,C$ denote an LR triple in
${\rm B}_d(\mathbb F)$, with parameter array
(\ref{eq:paLRT})
and Toeplitz data
(\ref{eq:ToeplitzData}).
Recall the sequences
$\lbrace \rho_i \rbrace_{i=0}^{d-1}$,
$\lbrace \rho'_i \rbrace_{i=0}^{d-1}$,
$\lbrace \rho''_i \rbrace_{i=0}^{d-1}$
from Definition
\ref{def:RHOD}. These sequences are
constrained by 
Proposition \ref{lem:3seqCon}, and related to each other by
Lemma \ref{lem:RhoRelBip}(iii).
So by Proposition
\ref{prop:noddPre}, there exists $0 \not=t \in \mathbb F$
such that
for $0 \leq i \leq d-1$,
\begin{eqnarray}
\label{ex:rho3times}
&&\frac{\rho_i}{\rho_0} = 
\frac{\rho'_i}{\rho'_0} = 
\frac{\rho''_i}{\rho''_0} = 
  t^{i/2}  \; \qquad \qquad \quad  \qquad  {\mbox{\rm if $i$ is even}}; 
  \\
 &&
{\rho_i}{\rho_0} = 
{\rho'_i}{\rho'_0} = 
{\rho''_i}{\rho''_0} = 
 t^{(i-d+1)/2} \qquad \quad  {\mbox{\rm if $i$ is odd}}.
\label{ex:rho3timesB}
\end{eqnarray}
We now compute $\lbrace \varphi_i \rbrace_{i=1}^d$.
By Lemma
\ref{lem:BipRRR2}
and since $A,B,C$ is normalized,
\begin{eqnarray*}
\rho_i = \varphi_{i+2}-\varphi_i
\qquad \qquad (0 \leq i \leq d-1).
\end{eqnarray*}
By this and since $\varphi_0 = 0 = \varphi_{d+1}$,
\begin{eqnarray}
\varphi_i = \begin{cases}
\rho_0 + \rho_2 + \rho_4 + \cdots + \rho_{i-2}  & 
{\mbox{\rm if $i$ is even}}; \\
-\rho_i-\rho_{i+2}-\rho_{i+4}- \cdots - \rho_{d-1} & {\mbox{\rm if $i$ is odd}}
\end{cases}
\qquad \qquad (1 \leq i \leq d).
\label{eq:vpFormBIP}
\end{eqnarray}
Evaluate
(\ref{eq:vpFormBIP}) using
(\ref{ex:rho3times}),
(\ref{ex:rho3timesB})
to obtain the formula
for $\lbrace \varphi_i\rbrace_{i=1}^d $ given in Example
\ref{ex:BIP} (if $t \not=1$) or
Example
\ref{ex:BipONE} (if $t=1$).
We similarly obtain 
the formulae for
$\lbrace \varphi'_i\rbrace_{i=1}^d $,
$\lbrace \varphi''_i\rbrace_{i=1}^d $ given in
Examples
\ref{ex:BIP}, 
\ref{ex:BipONE}.
Using these formulae and
$\rho_0 = \varphi'_1/\varphi''_d$ we obtain
the formula for $\rho_0 \rho'_0 \rho''_0$
given in
Examples
\ref{ex:BIP}, 
\ref{ex:BipONE}.
For the moment assume  that $t\not= 1$.
Then for $1 \leq i \leq d/2$, $t^i\not=1$ since
$\varphi_{2i}\not=0$. 
We have met the requirements of
Example
\ref{ex:BIP},
so $A,B,C$ is isomorphic to
${\rm B}_d(\mathbb F;t, \rho_0, \rho'_0, \rho''_0)$.
Next assume that $t= 1$. Then 
for $1 \leq i \leq d/2$, $i\not=0$ in $\mathbb F$ since
$\varphi_{2i}\not=0$. Therefore
${\rm Char}(\mathbb F)$ is 0 or greater than $d/2$.
We have met the requirements of
Example
\ref{ex:BipONE},
so $A,B,C$ is isomorphic to
${\rm B}_d(\mathbb F;1, \rho_0, \rho'_0, \rho''_0)$.
In summary, we have shown that $A,B,C$ is isomorphic to at least one
LR triple in 
Examples \ref{ex:BIP},
\ref{ex:BipONE}.
The result follows in view of Lemma
\ref{lem:BIPpre}(iii).
\end{proof}

\noindent We record a result from the proof of Theorem
\ref{thm:BIPclass}.

\begin{lemma} Assume that $d$ is even and at least 2.
Let $A,B,C$ denote an LR triple in 
 ${\rm B}_d(\mathbb F)$. Then for
 $0 \leq i \leq d-1$ the scalars
$\rho_i,
\rho'_i,
\rho''_i$
from Definition
\ref{def:RHOD} are described as follows.
For ${\rm B}_d(\mathbb F;t, \rho_0, \rho'_0, \rho''_0)$,
\begin{eqnarray*}
 &&\rho_i = \begin{cases}
 \rho_0 t^{i/2}  &  {\mbox{\rm if $i$ is even}}; \\
 \rho_0^{-1}t^{(i-d+1)/2} & {\mbox{\rm if $i$ is odd}};
 \end{cases}
 \\
 &&\rho'_i = \begin{cases}
 \rho'_0 t^{i/2}  &  {\mbox{\rm if $i$ is even}}; \\
 (\rho'_0)^{-1}t^{(i-d+1)/2} & {\mbox{\rm if $i$ is odd}};
 \end{cases}
 \\
 &&\rho''_i = \begin{cases}
 \rho''_0 t^{i/2}  &  {\mbox{\rm if $i$ is even}}; \\
 (\rho''_0)^{-1}t^{(i-d+1)/2} & {\mbox{\rm if $i$ is odd}}.
 \end{cases}
\end{eqnarray*}
For ${\rm B}_d(\mathbb F;1, \rho_0, \rho'_0, \rho''_0)$ and
${\rm B}_2(\mathbb F; \rho_0, \rho'_0, \rho''_0)$,
\begin{eqnarray*}
 &&\rho_i = \begin{cases}
 \rho_0  &  {\mbox{\rm if $i$ is even}}; \\
 \rho_0^{-1} & {\mbox{\rm if $i$ is odd}};
 \end{cases}
 \\
 &&\rho'_i = \begin{cases}
 \rho'_0  &  {\mbox{\rm if $i$ is even}}; \\
 (\rho'_0)^{-1}& {\mbox{\rm if $i$ is odd}};
 \end{cases}
 \\
 &&\rho''_i = \begin{cases}
 \rho''_0   &  {\mbox{\rm if $i$ is even}}; \\
 (\rho''_0)^{-1} & {\mbox{\rm if $i$ is odd}}.
 \end{cases}
\end{eqnarray*}
\end{lemma}
\begin{proof}
For $d\geq 4$ use (\ref{ex:rho3times}), 
(\ref{ex:rho3timesB}). For $d=2$ use
Definition
\ref{def:RHOD} and Example
\ref{ex:BipONEd2}.
\end{proof}

\noindent We have now classified up to isomorphism
the normalized LR triples over $\mathbb F$ that have
diameter $d\geq 2$.  
\medskip

\noindent Recall the similarity relation for LR triples, from
 Definition \ref{def:simiLar}.

\begin{corollary}
Consider the set of LR triples consisting of
the LR triple in
Lemma
\ref{lem:1or2} and the LR triples in
Examples
\ref{ex:nbw1}--\ref{ex:nbw3},
\ref{ex:NBGq},
\ref{ex:NBGqEQ1},
\ref{ex:NBNGdt}.
Each nonbipartite LR triple over $\mathbb F$
is similar to a unique LR triple in this set.
\end{corollary}
\begin{proof}
By Definition \ref{def:simiLar},
Corollary
\ref{prop:normalNBunique},
Lemma \ref{lem:1or2},
and Theorems
\ref{thm:ClassW},
\ref{thm:NBG},
\ref{thm:NBNG}.
\end{proof}

\noindent Recall the bisimilarity relation for bipartite LR triples, from
 Definition
\ref{def:biSim}.

\begin{corollary}
Consider the set of LR triples consisting of
Examples
\ref{ex:BIP}--\ref{ex:BipONEd2}.
Each nontrival bipartite LR triple over $\mathbb F$
is bisimilar to a unique LR triple in this set.
\end{corollary}
\begin{proof}
By Definition
\ref{def:biSim}, Corollary
\ref{lem:biassocUnique}, and Theorem
\ref{thm:BIPclass}.
\end{proof}

\section{The Toeplitz data and unipotent maps}

\noindent 
Throughout this section the following notation is in effect.
Let $V$ denote a vector space over $\mathbb F$ with
dimension $d+1$.
Let $A,B,C$ denote
an equitable LR triple on $V$, with
parameter array
(\ref{eq:paLRT})
and Toeplitz data
(\ref{eq:ToeplitzData}).
By Definition
\ref{def:equitNorm}
we have $\alpha_i = 
\alpha'_i = 
\alpha''_i $ for
$0 \leq i \leq d$,
and by 
Lemma
\ref{lem:equitBasicBeta}
we have
$\beta_i = 
\beta'_i = 
\beta''_i $ for $0 \leq i \leq d$.
Recall the unipotent maps 
$\mathbb A$,
$\mathbb B$,
$\mathbb C$ 
from Definition
\ref{def:Del}. By Proposition
\ref{prop:Apoly},
\begin{eqnarray}
\label{eq:rotator1A}
&&
\mathbb A = \sum_{i=0}^d \alpha_i A^i,
\qquad \qquad 
\mathbb B = \sum_{i=0}^d \alpha_i B^i,
\qquad \qquad 
\mathbb C = \sum_{i=0}^d \alpha_i C^i,
\\
&&
\label{eq:rotator2A}
\mathbb A^{-1} = \sum_{i=0}^d \beta_i A^i,
\qquad \qquad 
\mathbb B^{-1} = \sum_{i=0}^d \beta_i B^i,
\qquad \qquad 
\mathbb C^{-1} = \sum_{i=0}^d \beta_i C^i.
\end{eqnarray}
Since $A,B,C$ are Nil,
\begin{eqnarray}
\label{eq:NilDF}
A^{d+1}=0, \qquad \qquad
B^{d+1}=0, \qquad \qquad
C^{d+1}=0.
\end{eqnarray}
In this section we compute
$\lbrace \alpha_i\rbrace_{i=0}^d$,
$\lbrace \beta_i\rbrace_{i=0}^d$
for the cases
${\rm NBG}_d(\mathbb F)$,
${\rm NBNG}_d(\mathbb F)$,
${\rm B}_d(\mathbb F)$.
In each case, we
relate
$\mathbb A$,
$\mathbb B$,
$\mathbb C$  to the exponential function or quantum exponential
function. We now recall these functions. In what follows,
$\lambda $ denotes an indeterminate. The infinite series
that we will encounter should be viewed as formal sums; their convergence
is not an issue.

\begin{definition}
\label{def:exp}
\rm 
Define 
\begin{eqnarray*}
{\rm exp}(\lambda) = \sum_{i=0}^N \frac{\lambda^i}{i!},
\end{eqnarray*}
where $N=\infty$ if 
${\rm Char}(\mathbb F)=0$,
and $N+1=
{\rm Char}(\mathbb F)$
if 
${\rm Char}(\mathbb F)>0$.
\end{definition}

\begin{definition}
\label{def:qexp}
\rm 
For a nonzero $q \in \mathbb F$ such that $q \not=1$,
define
\begin{eqnarray*}
{\rm exp}_q(\lambda)
= \sum_{i=0}^N \frac{\lambda^i}{(1)(1+q)(1+q+q^2)\cdots 
(1+q+q^2+\cdots + q^{i-1})},
\end{eqnarray*}
where $N=\infty$ if 
 $q$ is not a root of unity, and  otherwise
$q$ is a primitive $(N+1)$-root of unity.
\end{definition}

\begin{proposition}
\label{prop:alphaBeta}
Assume that $A,B,C$ is in
${\rm NBG}_d(\mathbb F)$ or
${\rm NBNG}_d(\mathbb F)$ or
${\rm B}_d(\mathbb F)$.
Then the scalars
$\lbrace \alpha_i \rbrace_{i=0}^d$,
$\lbrace \beta_i \rbrace_{i=0}^d$ and maps
$\mathbb A$,
$\mathbb B$,
$\mathbb C$ are described as follows.
\medskip

\noindent {\bf Case 
${\rm NBG}_d(\mathbb F;q)$}.
For $0 \leq i \leq d$,
\begin{eqnarray*}
\alpha_i &=& \frac{1}{
(1) (1+q)
(1+q+q^2)
\cdots
(1+q+q^2+\cdots + q^{i-1})
},
\\
\beta_i &=&
\frac{(-1)^i q^{\binom{i}{2}}}{(1) (1+q)(1+q+q^2) \cdots (1+q+q^2+\cdots 
+ q^{i-1})}
\\
&=& 
\frac{(-1)^i}{(1) (1+q^{-1})(1+q^{-1}+q^{-2}) \cdots (1+q^{-1}+q^{-2}+\cdots 
+ q^{1-i})}.
\end{eqnarray*}
Moreover,
\begin{eqnarray*}
&&
\mathbb A = 
{\rm exp}_q(A),
\qquad \qquad 
\mathbb B = 
{\rm exp}_q(B),
\qquad \qquad 
\mathbb C = 
{\rm exp}_q(C),
\\
&&
\mathbb A^{-1} = 
{\rm exp}_{q^{-1}}(-A),
\qquad  
\mathbb B^{-1} = 
{\rm exp}_{q^{-1}}(-B),
\qquad  
\mathbb C^{-1} = 
{\rm exp}_{q^{-1}}(-C).
\end{eqnarray*}
\noindent {\bf Case 
${\rm NBG}_d(\mathbb F;1)$}.
For $0 \leq i \leq d$,
\begin{eqnarray*}
&&\alpha_i = \frac{1}{i!},
\qquad \qquad 
\beta_i =
\frac{(-1)^i}{i!}.
\end{eqnarray*}
Moreover,
\begin{eqnarray*}
&&
\mathbb A =
{\rm exp}(A),
\qquad \qquad 
\mathbb B =
{\rm exp}(B),
\qquad \qquad 
\mathbb C =
{\rm exp}(C),
\\
&&
\mathbb A^{-1} = {\rm exp}(-A),
\qquad  \quad
\mathbb B^{-1} =
{\rm exp}(-B),
 \qquad  \quad
\mathbb C^{-1} =
{\rm exp}(-C).
\end{eqnarray*}
\noindent {\bf Case 
${\rm NBNG}_d(\mathbb F;t)$}.
For $0 \leq i \leq d/2$,
\begin{eqnarray*}
\alpha_{2i} &=& \frac{1}{(1-t)(1-t^2)\cdots (1-t^{i})},
\\
\beta_{2i} &=& \frac{(-1)^i t^{\binom{i+1}{2}} }{(1-t)(1-t^2)\cdots (1-t^i)}
\\
 &=& \frac{1}{(1-t^{-1})(1-t^{-2})\cdots (1-t^{-i})}.
\end{eqnarray*}
For $0 \leq i \leq d/2-1$,
\begin{eqnarray*}
\alpha_{2i+1} = \alpha_{2i},
\qquad \qquad 
\beta_{2i+1} = -\beta_{2i}.
\end{eqnarray*}
\noindent 
Moreover,
\begin{eqnarray*}
&&
\mathbb A =
(I+A){\rm exp}_t \Bigl(\frac{A^2}{1-t}\Bigr),
\qquad \qquad
\mathbb A^{-1} = 
(I-A){\rm exp}_{t^{-1}} \Bigl(\frac{A^2}{1-t^{-1}}\Bigr),
\\
&&
\mathbb B =
(I+B){\rm exp}_t \Bigl(\frac{B^2}{1-t}\Bigr),
\qquad \qquad
\mathbb B^{-1} = 
(I-B){\rm exp}_{t^{-1}} \Bigl(\frac{B^2}{1-t^{-1}}\Bigr),
\\
&&
\mathbb C =
(I+C){\rm exp}_t \Bigl(\frac{C^2}{1-t}\Bigr),
\qquad \qquad 
\mathbb C^{-1} = 
(I-C){\rm exp}_{t^{-1}} \Bigl(\frac{C^2}{1-t^{-1}}\Bigr).
\end{eqnarray*}
\noindent {\bf Case 
${\rm B}_d(\mathbb F;t,\rho_0, \rho'_0,\rho''_0)$}.
For $0 \leq i \leq d/2$,
\begin{eqnarray*}
\alpha_{2i} &=& \frac{1}{(1)(1+t)(1+t+t^2)\cdots (1+t+t^2 + \cdots + t^{i-1})},
\\
\beta_{2i} &=&
\frac{(-1)^i t^{\binom{i}{2}}}{(1) (1+t)(1+t+t^2)\cdots 
(1+t+t^2 + \cdots + t^{i-1})}
\\
&=&
\frac{(-1)^i} {(1) (1+t^{-1})(1+t^{-1}+t^{-2})\cdots
(1+t^{-1}+t^{-2} + \cdots + t^{1-i})}.
\end{eqnarray*}
For $0 \leq i \leq d/2-1$,
\begin{eqnarray*}
\alpha_{2i+1} = 0,\qquad \qquad \beta_{2i+1}=0.
\end{eqnarray*}
Moreover,
\begin{eqnarray*}
&&
\mathbb A=
{\rm exp}_t(A^2),
\qquad \qquad 
\mathbb B=
{\rm exp}_t(B^2),
\qquad \qquad 
\mathbb C=
{\rm exp}_t(C^2),
\\
&&
\mathbb A^{-1} =
{\rm exp}_{t^{-1}}(-A^2),
\qquad
\mathbb B^{-1} =
{\rm exp}_{t^{-1}}(-B^2),
\qquad 
\mathbb C^{-1} =
{\rm exp}_{t^{-1}}(-C^2).
\end{eqnarray*}
\noindent
{\bf Case 
${\rm B}_d(\mathbb F; 1,\rho_0,\rho'_0,\rho''_0)$}.
For $0 \leq i \leq d/2$, 
\begin{eqnarray*}
\alpha_{2i} = \frac{1}{i!}, \qquad \qquad \beta_{2i} = \frac{(-1)^i}{i!}.  
\end{eqnarray*}
For $0 \leq i \leq d/2-1$, 
\begin{eqnarray*} 
\alpha_{2i+1} = 0,\qquad \qquad \beta_{2i+1}=0.
\end{eqnarray*}
\noindent Moreover,
\begin{eqnarray*}
&&
\mathbb A =
{\rm exp}(A^2),
\qquad \qquad
\mathbb B =
{\rm exp}(B^2),
\qquad \qquad
\mathbb C =
{\rm exp}(C^2),
\\
&&
\mathbb A^{-1} = 
{\rm exp}(-A^2),
\qquad
\mathbb B^{-1} = 
{\rm exp}(-B^2),
\qquad
\mathbb C^{-1} = 
{\rm exp}(-C^2).
\end{eqnarray*}
\noindent {\bf Case
${\rm B}_2(\mathbb F; \rho_0,\rho'_0,\rho''_0)$}.
Same as 
${\rm B}_d(\mathbb F; 1;\rho_0,\rho'_0,\rho''_0)$
with $d=2$.
\end{proposition}
\begin{proof}
Compute
$\lbrace \alpha_i\rbrace_{i=0}^d$,
$\lbrace \beta_i\rbrace_{i=0}^d$ as follows. 
In each nonbipartite
case, use
Proposition
\ref{prop:AlphaRecursion3} and induction,
together with
\begin{eqnarray*}
\alpha_0=1, \qquad \alpha_1 = 1,\qquad  \beta_0=1, \qquad \beta_1 = -1.
\end{eqnarray*}
In each bipartite case, use
Proposition \ref{prop:AlphaRecursion} and induction, together with
\begin{eqnarray*}
\alpha_0=1, \qquad \alpha_1 = 0, \qquad \alpha_2=1,
\qquad  \beta_0=1, \qquad \beta_1 = 0, \qquad  \beta_2 = -1.
\end{eqnarray*}
Our assertions about 
$\mathbb A$,
$\mathbb B$,
$\mathbb C$ follow from
(\ref{eq:rotator1A})--(\ref{eq:NilDF})
and
Definitions
\ref{def:exp},
\ref{def:qexp}.
\end{proof}

\section{Relations for  LR triples}

In this section we consider the relations
satisfied by an LR triple $A,B,C$.
In order to motivate our results, 
 assume for the moment that $A,B,C$ has
$q$-Weyl type, in the sense of
Definition
\ref{def:qExceptional}.
Then $A,B,C$ satisfy the
relations in
(\ref{eq:WWW})
and Lemma
\ref{lem:sixeqs}.
Next assume that
$A,B,C$ is contained in
${\rm NBG}_d(\mathbb F)$
or ${\rm NBNG}_d(\mathbb F)$
or 
${\rm B}_d(\mathbb F)$.
We show that $A,B,C$ satisfy some analogous relations.
We treat the nonbipartite cases
${\rm NBG}_d(\mathbb F)$,
${\rm NBNG}_d(\mathbb F)$
and the bipartite case
${\rm B}_d(\mathbb F)$
separately.

\begin{proposition}
\label{prop:ABCRel}
Let $A,B,C$ denote an LR triple in
${\rm NBG}_d(\mathbb F)$ or
${\rm NBNG}_d(\mathbb F)$.
Then $A,B,C$ satisfy the following relations.
\medskip

\noindent {\bf Case 
${\rm NBG}_d(\mathbb F;q)$}.
\noindent We have
\begin{eqnarray*}
&&
A^2B-q(1+q)ABA+q^3BA^2=q(1+q)A, 
\\
&&
B^2C-q(1+q)BCB+q^3CB^2=q(1+q)B, 
\\
&&
C^2A-q(1+q)CAC+q^3AC^2=q(1+q)C 
\end{eqnarray*}
and also
\begin{eqnarray*}
&&
AB^2-q(1+q)BAB+q^3B^2A=q(1+q)B,
\\
&&
BC^2-q(1+q)CBC+q^3C^2B=q(1+q)C,
\\
&&
CA^2-q(1+q)ACA+q^3A^2C=q(1+q)A.
\end{eqnarray*}
\noindent We have
\begin{eqnarray*}
&&
A
\bigl(I+(BC-qCB)(1-q^{-1})\bigr)
= qB+q^{-1}C+qCB-q^{-1}BC,
\\
&&
B
\bigl(I+(CA-qAC)(1-q^{-1})\bigr)
= qC+q^{-1}A+qAC-q^{-1}CA,
\\
&&
C
\bigl(I+(AB-qBA)(1-q^{-1})\bigr)
= qA+q^{-1}B+qBA-q^{-1}AB 
\end{eqnarray*}
and also
\begin{eqnarray*}
&&
\bigl(I+(BC-qCB)(1-q^{-1})\bigr)
A = q^{-1}B+qC+qCB-q^{-1}BC,
\\
&&
\bigl(I+(CA-qAC)(1-q^{-1})\bigr)
B = q^{-1}C+qA+qAC-q^{-1}CA,
\\
&&
\bigl(I+(AB-qBA)(1-q^{-1})\bigr)
C = q^{-1}A+qB+qBA-q^{-1}AB.
\end{eqnarray*}
\noindent We have
\begin{eqnarray*}
ABC - BCA + q (CBA-ACB) &=& (1+q)(B-C),
\\
BCA - CAB + q (ACB-BAC) &=& (1+q)(C-A),
\\
CAB - ABC + q (BAC-CBA) &=& (1+q)(A-B)
\end{eqnarray*}
and also
\begin{eqnarray*}
&&(1+2 q^{-1})(ABC+BCA+CAB)-(1+2q)(CBA+ACB+BAC)
\\
&&=(q-q^{-1})(A+B+C)-
\frac{3(q^d-1)(q^{d+2}-1)}{q^d(q-1)^2} I.
\end{eqnarray*}
\noindent {\bf Case 
${\rm NBG}_d(\mathbb F;1)$}.
 We have
\begin{eqnarray*}
&&
\lbrack A, \lbrack A,B\rbrack \rbrack = 2A, \qquad \qquad 
\lbrack B, \lbrack B,A\rbrack \rbrack = 2B,
\\
&&
\lbrack B, \lbrack B,C\rbrack \rbrack = 2B, \qquad \qquad 
\lbrack C, \lbrack C,B\rbrack \rbrack = 2C,
\\
&&
\lbrack C, \lbrack C,A\rbrack \rbrack = 2C, \qquad \qquad 
\lbrack A, \lbrack A,C\rbrack \rbrack = 2A
\end{eqnarray*}
and also
\begin{eqnarray*}
A = B+C-\lbrack B,C\rbrack,
\qquad
B = C+A-\lbrack C,A\rbrack,
\qquad
C = A+B-\lbrack A,B\rbrack.
\end{eqnarray*}
We have
\begin{eqnarray*}
&&
\lbrack A, \lbrack B, C\rbrack \rbrack = 2(B-C),
\\
&&
\lbrack B, \lbrack C, A\rbrack \rbrack = 2(C-A),
\\
&&
\lbrack C, \lbrack A, B\rbrack \rbrack = 2(A-B)
\end{eqnarray*}
and also
\begin{eqnarray*}
ABC+BCA+CAB-CBA-ACB-BAC = -d(d+2)I.
\end{eqnarray*}
\noindent {\bf Case 
${\rm NBNG}_d(\mathbb F;t)$}.
We have
\begin{eqnarray*}
&&
\frac{A^2B-tBA^2}{1-t} = -A,
\qquad 
\frac{B^2C-tCB^2}{1-t} = -B,
\qquad 
\frac{C^2A-tAC^2}{1-t} = -C,
\\
&&
\frac{AB^2-tB^2A}{1-t} = -B,
\qquad 
\frac{BC^2-tC^2B}{1-t} = -C,
\qquad 
\frac{CA^2-tA^2C}{1-t} = -A
\end{eqnarray*}
and also
\begin{eqnarray*}
&&
\frac{ABC-tCBA}{1-t}+ A+C = -\frac{(1-t^{-d/2})(1-t^{1+d/2})}{1-t}\,I,
\\
&&
\frac{BCA-tACB}{1-t}+ B+A = -\frac{(1-t^{-d/2})(1-t^{1+d/2})}{1-t}\,I,
\\
&&
\frac{CAB-tBAC}{1-t}+ C+B = -\frac{(1-t^{-d/2})(1-t^{1+d/2})}{1-t}\,I.
\end{eqnarray*}
\end{proposition}
\begin{proof}
To verify these relations,
represent $A,B,C$ by matrices, using for
example
the first row of the table in Proposition
\ref{prop:matrixRep}.
\end{proof}

\begin{remark}\rm
In 
\cite[p.~308]{benkRoby} G. Benkart and T. Roby introduce the
concept of a down-up algebra.
Consider an LR triple $A,B,C$ from
 Proposition
\ref{prop:ABCRel}. By that proposition,
any two of
$A,B,C$ satisfy the defining relations
for a down-up algebra.
\end{remark}

\noindent Let $A,B,C$ denote an LR triple in
${\rm B}_d(\mathbb F)$, and consider its
projector $J$ from Definition
\ref{def:projector}.
By
Lemma
\ref{lem:EEEJ} we have
$J^2=J$, and 
by 
Lemma \ref{lem:ABCJ},
\begin{eqnarray*}
A = JA+AJ, \qquad \qquad
B = JB+BJ, \qquad \qquad
C = JC+CJ.
\end{eqnarray*}

\begin{proposition}
\label{prop:ABCRelBIP}
Let $A,B,C$ denote an LR triple in
${\rm B}_d(\mathbb F)$.
Then $A,B,C$ and its projector $J$ satisfy the following
relations.
\medskip

\noindent {\bf Case 
${\rm B}_d(\mathbb F;t,\rho_0, \rho'_0,\rho''_0)$}.
We have
\begin{eqnarray*}
&&
\Bigl(\rho_0 AB+\rho_0' \rho_0 ''  B A -\frac{1-t^{-d/2}}{1-t} tI
\Bigr)J=0,
\\
&&
\Bigl(\rho_0' BC+\rho_0'' \rho_0  CB -\frac{1-t^{-d/2}}{1-t} tI
\Bigr)J=0,
\\
&&
\Bigl(\rho_0'' CA+\rho_0 \rho_0'  AC -\frac{1-t^{-d/2}}{1-t} tI
\Bigr)J=0
\end{eqnarray*}
and also
\begin{eqnarray*}
&&
\Bigl(\rho_0'\rho_0'' AB+\rho_0 t B A -\frac{1-t^{-1-d/2}}{1-t} t^2I
\Bigr)(I-J)=0,
\\
&&
\Bigl(\rho_0''\rho_0 BC+\rho_0' t CB -\frac{1-t^{-1-d/2}}{1-t} t^2I
\Bigr)(I-J)=0,
\\
&&
\Bigl(\rho_0\rho_0' CA+\rho_0'' t AC -\frac{1-t^{-1-d/2}}{1-t} t^2I
\Bigr)(I-J)=0.
\end{eqnarray*}
\noindent We have
\begin{eqnarray*}
&&
(A^2B-tBA^2-(t/\rho_0) A)J=0, \qquad \qquad
J(A^2B-tBA^2-\rho_0 A)=0,
\\
&&
(B^2C-tCB^2-(t/\rho'_0) B)J=0, \qquad \qquad
J(B^2C-tCB^2-\rho'_0 B)=0,
\\
&&
(C^2A-tAC^2-(t/\rho''_0) C)J=0, \qquad \qquad
J(C^2A-tAC^2-\rho''_0 C)=0
\end{eqnarray*}
and also
\begin{eqnarray*}
&&
J(AB^2-tB^2A-(t/\rho_0) B)=0, \qquad \qquad
(AB^2-tB^2A-\rho_0 B)J=0,
\\
&&
J(BC^2-tC^2B-(t/\rho'_0) C)=0, \qquad \qquad
(BC^2-tC^2B-\rho'_0 C)J=0,
\\
&&
J(CA^2-tA^2C-(t/\rho''_0) A)=0, \qquad \qquad
(CA^2-tA^2C-\rho''_0 A)J=0.
\end{eqnarray*}
\noindent We have
\begin{eqnarray*}
&&A^3 B + A^2BA-tABA^2 -tBA^3 = (\rho_0 + t/\rho_0)A^2,
\\
&&
B^3 C + B^2CB-tBCB^2 -tCB^3 = (\rho'_0 + t/\rho'_0)B^2,
\\
&&
C^3 A + C^2AC-tCAC^2 -tAC^3 = (\rho''_0 + t/\rho''_0)C^2
\end{eqnarray*}
and also
\begin{eqnarray*}
&&
A B^3 + BAB^2-tB^2AB -tB^3A = (\rho_0 + t/\rho_0)B^2,
\\
&&
B C^3 + CBC^2-tC^2BC -tC^3B = (\rho'_0 + t/\rho'_0)C^2,
\\
&&
C A^3 + ACA^2-tA^2CA -tA^3C = (\rho''_0 + t/\rho''_0)A^2.
\end{eqnarray*}
We have
\begin{eqnarray*}
&&
\Bigl(ABC-\frac{tA-\rho_0 t B + \rho_0 \rho_0' C}{\rho_0'(1-t)}\Bigr)J=0,
\qquad \quad
\Bigl(CBA-\frac{tA-\rho_0 B + \rho_0 \rho_0' C}{\rho_0'(1-t)}\Bigr)J=0,
\\
&&
\Bigl(BCA-\frac{tB-\rho_0' t C + \rho_0' \rho_0'' A}{\rho_0''(1-t)}\Bigr)J=0,
\qquad \quad
\Bigl(ACB-\frac{tB-\rho_0' C + \rho_0' \rho_0'' A}{\rho_0''(1-t)}\Bigr)J=0,
\\
&&
\Bigl(CAB-\frac{tC-\rho_0'' t A + \rho_0'' \rho_0 B}{\rho_0(1-t)}\Bigr)J=0,
\qquad \quad
\Bigl(BAC-\frac{tC-\rho_0'' A  + \rho_0'' \rho_0 B}{\rho_0(1-t)}\Bigr)J=0
\end{eqnarray*}
and also
\begin{eqnarray*}
&&
J\Bigl(ABC-\frac{\rho_0 \rho'_0 A-\rho_0' t B + t C}{\rho_0(1-t)}\Bigr)=0,
\qquad \quad
J\Bigl(CBA-\frac{\rho_0 \rho'_0 A-\rho_0'  B + t C}{\rho_0(1-t)}\Bigr)=0,
\\
&&
J\Bigl(BCA-\frac{\rho_0' \rho''_0 B-\rho_0'' t C + t A}{\rho'_0(1-t)}\Bigr)=0,
\qquad \quad
J\Bigl(ACB-\frac{\rho_0' \rho''_0 B-\rho_0''  C + t A}{\rho'_0(1-t)}\Bigr)=0,
\\
&&
J\Bigl(CAB-\frac{\rho_0'' \rho_0 C-\rho_0 t A + t B}{\rho''_0(1-t)}\Bigr)=0,
\qquad \quad
J\Bigl(BAC-\frac{\rho_0'' \rho_0 C-\rho_0  A + t B}{\rho''_0(1-t)}\Bigr)=0.
\end{eqnarray*}
\noindent 
We have
\begin{eqnarray*}
&&
(ABC-CBA-(\rho_0/\rho_0')B)J=0,
\qquad \qquad 
J(ABC-CBA-(\rho_0'/\rho_0)B)=0,
\\
&&
(BCA-ACB-(\rho_0'/\rho_0'')C)J=0,
\qquad \qquad 
J(BCA-ACB-(\rho_0''/\rho_0')C)=0,
\\
&&
(CAB-BAC-(\rho_0''/\rho_0)A)J=0,
\qquad \qquad 
J(CAB-BAC-(\rho_0/\rho_0'')A)=0.
\end{eqnarray*}

\noindent {\bf Case 
${\rm B}_d(\mathbb F;1,\rho_0, \rho'_0,\rho''_0)$}.
We have
\begin{eqnarray*}
&&
(\rho_0 AB + \rho_0' \rho_0'' BA +  (d/2)I )J=0,
\\
&&
(\rho_0' BC + \rho_0'' \rho_0 CB + (d/2)I )J=0,
\\
&&
(\rho_0'' CA + \rho_0 \rho_0' AC +  (d/2)I )J=0
\end{eqnarray*}
and also
\begin{eqnarray*}
&&
\Bigl(\rho_0' \rho_0'' AB + \rho_0 BA +  \frac{d+2}{2} I\Bigr)(I-J)=0,
\\
&&
\Bigl(\rho_0'' \rho_0 BC + \rho_0' CB + \frac{d+2}{2} I \Bigr)(I-J)=0,
\\
&&
\Bigl(\rho_0 \rho_0' CA + \rho_0'' AC + \frac{d+2}{2} I \Bigr) (I-J)=0.
\end{eqnarray*}
\noindent We have
\begin{eqnarray*}
&&
(A^2 B - B A^2 -A/\rho_0)J=0,
\qquad \qquad
J(A^2 B - B A^2 -\rho_0 A)=0,
\\
&&
(B^2 C - C B^2 -B/\rho_0')J=0,
\qquad \qquad
J(B^2 C - C B^2 -\rho_0' B)=0,
\\
&&
(C^2 A - A C^2 -C/\rho_0'')J=0,
\qquad \qquad
J(C^2 A - A C^2 -\rho_0'' C)=0
\end{eqnarray*} 
and also
\begin{eqnarray*}
&&
J(A B^2 - B^2 A -B/\rho_0)=0,
\qquad \qquad
(A B^2 - B^2 A -\rho_0 B)J=0,
\\
&&
J(B C^2 - C^2 B -C/\rho_0')=0,
\qquad \qquad
(B C^2 - C^2 B -\rho_0' C)J=0,
\\
&&
J(C A^2 - A^2 C -A/\rho_0'')=0,
\qquad \qquad
(C A^2 - A^2 C -\rho_0'' A)J=0.
\end{eqnarray*} 
We have
\begin{eqnarray*}
&&A^3 B + A^2BA-ABA^2 -BA^3 = (\rho_0 + 1/\rho_0)A^2,
\\
&&
B^3 C + B^2CB-BCB^2 -CB^3 = (\rho'_0 + 1/\rho'_0)B^2,
\\
&&
C^3 A + C^2AC-CAC^2 -AC^3 = (\rho''_0 + 1/\rho''_0)C^2
\end{eqnarray*}
and also
\begin{eqnarray*}
&&
A B^3 + BAB^2-B^2AB -B^3A = (\rho_0 + 1/\rho_0)B^2,
\\
&&
B C^3 + CBC^2-C^2BC -C^3B = (\rho'_0 + 1/\rho'_0)C^2,
\\
&&
C A^3 + ACA^2-A^2CA -A^3C = (\rho''_0 + 1/\rho''_0)A^2.
\end{eqnarray*}

\noindent We have
\begin{eqnarray*}
&&
(A - B\rho_0   - C/\rho_0'')J=0,
\qquad \qquad
J(A - B/\rho_0 - C\rho_0'' )=0,
\\
&&
(B - C\rho_0'   - A/\rho_0)J=0,
\qquad \qquad
J(B - C/\rho_0' - A\rho_0 )=0,
\\
&&
(C - A\rho_0 ''  - B/\rho_0')J=0,
\qquad \qquad
J(C - A/\rho_0 '' - B\rho_0')=0.
\end{eqnarray*}
\noindent {\bf Case 
${\rm B}_2(\mathbb F;\rho_0, \rho'_0,\rho''_0)$}.
Same as 
${\rm B}_d(\mathbb F;1,\rho_0, \rho'_0,\rho''_0)$
with $d=2$, where we interpret $d/2=1$ and $(d+2)/2=0$ if
${\rm Char}(\mathbb F)=2$.
\end{proposition}
\begin{proof}
To verify these relations, 
represent $A,B,C,J$ by matrices, using for
example
the first row of the table in Proposition
\ref{prop:matrixRep}, along with
 Lemma
\ref{lem:JMat}.
\end{proof}

\section{The quantum algebra 
$U_q(\mathfrak{sl}_2)$ and the Lie algebra 
$\mathfrak{sl}_2$}

\noindent In this section, we discuss how
 LR triples are related to
the quantum algebra $U_q(\mathfrak{sl}_2)$ and the Lie
algebra
 $\mathfrak{sl}_2$.
\medskip

\noindent Until further notice, assume that
the field $\mathbb F$
is arbitrary, and fix a nonzero $q \in \mathbb F$
such that $q^2 \not=1$. We recall the algebra
$U_q(\mathfrak{sl}_2)$.
We will use the equitable presentation, which was introduced in
\cite{equit}.

\begin{definition}
\label{def:uqA}
\rm 
(See \cite[Theorem~2.1]{equit}.)
Let 
$U_q(\mathfrak{sl}_2)$ denote the $\mathbb F$-algebra
with generators $x,y^{\pm 1},z$ and  relations
$yy^{-1}=1$, $y^{-1}y=1$,
\begin{equation}
\frac{qxy-q^{-1}yx}{q-q^{-1}} = 1, \quad \qquad
\frac{qyz-q^{-1}zy}{q-q^{-1}} = 1, \quad \qquad
\frac{qzx-q^{-1}xz}{q-q^{-1}} = 1.
\label{eq:uqrels}
\end{equation}
\end{definition}

\noindent The following subalgebra of
$U_q(\mathfrak{sl}_2)$ is of interest.

\begin{definition}
\rm 
Let $U^R_q(\mathfrak{sl}_2)$ denote the subalgebra of
 $U_q(\mathfrak{sl}_2)$ generated by $x,y,z$.
We call 
$U^R_q(\mathfrak{sl}_2)$ the {\it reduced
$U_q(\mathfrak{sl}_2)$} algebra.
\end{definition}

\begin{lemma}
\label{def:uqARED}
\rm 
(See \cite[Definition~10.6, Lemma~10.9]{uawe}.)
The algebra $U^R_q(\mathfrak{sl}_2)$ has a presentation
by generators $x,y,z$ and  relations
\begin{equation}
\frac{qxy-q^{-1}yx}{q-q^{-1}} = 1,  \quad \qquad
\frac{qyz-q^{-1}zy}{q-q^{-1}} = 1, \quad \qquad
\frac{qzx-q^{-1}xz}{q-q^{-1}} = 1.
\label{eq:uqrelsRED}
\end{equation}
\end{lemma}

\noindent There is a central element
$\Lambda $ 
in $U_q(\mathfrak{sl}_2)$ that is often 
called the normalized Casimir element
\cite[Definition~2.11]{uawe}. 
The element $\Lambda$ is equal to 
$(q-q^{-1})^2$ times the Casimir element
given in 
\cite[Section~2.7]{jantzen}.
By \cite[Lemma~2.15]{uawe}, $\Lambda$ is equal to
each of the following:
\begin{eqnarray}
&& qx+q^{-1}y + qz-qxyz, \qquad \qquad q^{-1}x+qy+q^{-1}z-q^{-1}zyx,
\label{eq:Lam1}
\\
&& qy+q^{-1}z + qx-qyzx, \qquad \qquad q^{-1}y+qz+q^{-1}x-q^{-1}xzy,
\label{eq:Lam2}
\\
&& qz+q^{-1}x + qy-qzxy, \qquad \qquad q^{-1}z+qx+q^{-1}y-q^{-1}yxz.
\label{eq:Lam3}
\end{eqnarray}
Note that $\Lambda $ is contained in
 $U^R_q(\mathfrak{sl}_2)$.
\medskip

\noindent Recall from Definition
\ref{def:qExceptional}
the LR triples of $q$-Weyl type.

\begin{proposition}
\label{thm:WeylUq}
Let $A,B,C$ denote an LR triple over $\mathbb F$
that has
$q$-Weyl type. Then 
the underlying vector space 
$V$ becomes a $U^R_q(\mathfrak{sl}_2)$-module
on which
\begin{eqnarray}
A = x, \qquad \qquad
B = y, \qquad \qquad
C = z.
\label{lem:xyzABC}
\end{eqnarray}
The 
$U^R_q(\mathfrak{sl}_2)$-module $V$ is irreducible.
On the
$U^R_q(\mathfrak{sl}_2)$-module $V$,
\begin{eqnarray}
\Lambda = \alpha_1 (q-q^{-1}) I,
\label{eq:LambdaVal}
\end{eqnarray}
where $\alpha_1$
is the first Toeplitz number for $A,B,C$.
\end{proposition}
\begin{proof}
Compare
(\ref{eq:WWW}),
(\ref{eq:uqrelsRED}) to obtain the first assertion.
The $U^R_q(\mathfrak{sl}_2)$-module $V$ is irreducible
by 
(\ref{lem:xyzABC}) and
Lemma
\ref{lem:ABCGEN}.
To get (\ref{eq:LambdaVal}),
compare
Lemma
\ref{lem:sixeqs}
and
(\ref{eq:Lam1})--(\ref{eq:Lam3}) using
(\ref{lem:xyzABC}).
\end{proof}

\noindent We are done discussing the LR triples of
$q$-Weyl type.
\medskip

\noindent 
We return our attention to
$U_q(\mathfrak{sl}_2)$.
By 
\cite[Lemma~5.1]{equit} 
we find that in
$U_q(\mathfrak{sl}_2)$,
\begin{eqnarray*}
q(1-xy)= q^{-1}(1-yx),
\qquad
q(1-yz)= q^{-1}(1-zy),
\qquad
q(1-zx)= q^{-1}(1-xz).
\end{eqnarray*}

\begin{definition}
\label{def:nxnynz}
\rm
(See \cite[Definition~5.2]{equit}.)
Let $n_x, n_y, n_z$ denote the following elements in 
$U_q(\mathfrak{sl}_2)$:
\begin{eqnarray*}
&&
n_x = \frac{q(1-yz)}{q-q^{-1}} = \frac{q^{-1}(1-zy)}{q-q^{-1}},
\\
&&
n_y = \frac{q(1-zx)}{q-q^{-1}} = \frac{q^{-1}(1-xz)}{q-q^{-1}},
\\
&&
n_z = \frac{q(1-xy)}{q-q^{-1}} = \frac{q^{-1}(1-yx)}{q-q^{-1}}.
\end{eqnarray*}
\end{definition}

\begin{lemma}
{\rm (See \cite[Lemma~5.4]{equit}.)}
The following relations hold in 
$U_q(\mathfrak{sl}_2)$:
\begin{eqnarray*}
&&
x n_y =q^2 n_y x, \qquad \qquad x n_z = q^{-2} n_z x,
\\
&&
y n_z =q^2 n_z y, \qquad \qquad y n_x = q^{-2} n_x y,
\\
&&
z n_x =q^2 n_x z, \qquad \qquad z n_y = q^{-2} n_y z.
\end{eqnarray*}
\end{lemma}

\noindent Until further notice, let $A,B,C$ 
denote an LR triple that is contained in
${\rm NBG}_d(\mathbb F;q^{-2})$.  Let $V$
denote the underlying vector space.

\begin{definition}
\label{def:xyzUQ}
\rm
Define $X,Y,Z$ in
${\rm End}(V)$ such that
for $0 \leq i \leq d$,
$X-q^{d-2i}I$
(resp. $Y-q^{d-2i}I$)
(resp. $Z-q^{d-2i}I$) vanishes on component
$i$ of the
$(B,C)$-decomposition
(resp. $(C,A)$-decomposition)
(resp. $(A,B)$-decomposition) of $V$. Note that
each of $X,Y,Z$ is invertible.
\end{definition}

\noindent The next result is meant to clarify
Definition \ref{def:xyzUQ}.
Recall the idempotent data 
(\ref{eq:idseq}) for
$A,B,C$.

\begin{lemma}
The elements $X,Y,Z$ from Definition 
\ref{def:xyzUQ} satisfy
\begin{eqnarray*}
X = \sum_{i=0}^d q^{d-2i} E'_i,
\qquad \qquad
Y = \sum_{i=0}^d q^{d-2i} E''_i,
\qquad \qquad
Z = \sum_{i=0}^d q^{d-2i} E_i.
\end{eqnarray*}
\end{lemma}
\begin{proof}
By Definition
\ref{def:xyzUQ} and the meaning of the idempotent
data.
\end{proof}

\begin{lemma} 
\label{eq:XYZform}
The elements $X,Y,Z$ from Definition 
\ref{def:xyzUQ} satisfy
\begin{eqnarray*}
&&
X =
\frac{(q+q^{-1})I - (q-q^{-1})(q^2BC-q^{-2}CB)}{q^{d+1}+q^{-d-1}},
\\
&&
Y =
\frac{(q+q^{-1})I - (q-q^{-1})(q^2CA-q^{-2}AC)}{q^{d+1}+q^{-d-1}},
\\
&&
Z =
\frac{(q+q^{-1})I - (q-q^{-1})(q^2AB-q^{-2}BA)}{q^{d+1}+q^{-d-1}}.
\end{eqnarray*}
\end{lemma}
\begin{proof}
To verify the first equation,
work with the matrices in
${\rm Mat}_{d+1}(\mathbb F)$
that represent
$B,C,X$ with respect to a
$(B,C)$-basis for $V$.
For $B,C$ these matrices are given in Proposition
\ref{prop:matrixRep}. For $X$ this matrix is diagonal,
with $(i,i)$-entry $q^{d-2i}$ for $0 \leq i \leq d$.
The other two equations are similarly verified.
\end{proof}

\begin{lemma}
\label{lem:XYZUQ}
The elements $X,Y,Z$ from Definition 
\ref{def:xyzUQ} satisfy
\begin{eqnarray*}
\frac{qXY-q^{-1}YX}{q-q^{-1}} = I,
\quad \qquad 
\frac{qYZ-q^{-1}ZY}{q-q^{-1}} = I,
\quad \qquad 
\frac{qZX-q^{-1}XZ}{q-q^{-1}} = I.
\end{eqnarray*}
\end{lemma}
\begin{proof}
To verify these equations, eliminate
$X,Y,Z$ using
Lemma
\ref{eq:XYZform}, and evaluate the
result using
the relations for
${\rm NBG}_d(\mathbb F;q^{-2})$ given in
Proposition
\ref{prop:ABCRel}.
\end{proof}

\begin{proposition}
\label{prop:NBGUq}
Let $A,B,C$ denote an LR triple contained in
${\rm NBG}_d(\mathbb F;q^{-2})$.
Then there exists a unique
$U_q(\mathfrak{sl}_2)$-module structure on 
the underlying vector space $V$, such that
for $0 \leq i \leq d$,
$x-q^{d-2i}1$
(resp. $y-q^{d-2i}1$)
(resp. $z-q^{d-2i}1$)
vanishes on component $i$ of the 
$(B,C)$-decomposition 
(resp. $(C,A)$-decomposition) 
(resp. $(A,B)$-decomposition) of $V$.
The
$U_q(\mathfrak{sl}_2)$-module  $V$ is irreducible.
On the 
$U_q(\mathfrak{sl}_2)$-module  $V$,
\begin{eqnarray}
\label{eq:ABCnuxyz}
A = n_x, \qquad \qquad
B = n_y, \qquad \qquad
C = n_z.
\end{eqnarray}
\end{proposition}
\begin{proof}
The 
$U_q(\mathfrak{sl}_2)$-module structure exists,
by
Lemma
\ref{lem:XYZUQ}  and since $Y$ is invertible.
The 
$U_q(\mathfrak{sl}_2)$-module structure is unique
by construction.
On the $U_q(\mathfrak{sl}_2)$-module $V$ we have
$x=X$,
$y=Y$,
$z=Z$.
To verify 
(\ref{eq:ABCnuxyz}),  eliminate 
$n_x$, $n_y$, $n_z$ using
Definition \ref{def:nxnynz}, and
evaluate the
result using
 Lemma
\ref{eq:XYZform}
along with the relations for
${\rm NBG}_d(\mathbb F;q^{-2})$ given in
Proposition
\ref{prop:ABCRel}. The 
$U_q(\mathfrak{sl}_2)$-module $V$ is irreducible
by
(\ref{eq:ABCnuxyz}) and Lemma
\ref{lem:ABCGEN}.
\end{proof}

\noindent We are done discussing
the LR triples contained in ${\rm NBG}_d(\mathbb F;q^{-2})$. 
\medskip

\noindent 
In a moment we will discuss the LR triples contained in
${\rm NBNG}_d(\mathbb F;q^{-2})$. To prepare for this,
we have some comments about 
$U_q(\mathfrak{sl}_2)$.

\begin{lemma} 
\label{lem:x2y} Assume that $q^4 \not=1$. Then
in 
 $U_q({\mathfrak{sl}_2})$,
 \begin{equation}
 \label{eq:x2y}
 \frac{q^2x^2y-q^{-2}yx^2}{q^2-q^{-2}}=x,
 \qquad
 \frac{q^2y^2z-q^{-2}zy^2}{q^2-q^{-2}}=y,
 \qquad
 \frac{q^2z^2x-q^{-2}xz^2}{q^2-q^{-2}}=z.
 \end{equation}
 \end{lemma}
 \begin{proof}
 We verify the equation on the left in
 (\ref{eq:x2y}).
 In the  equation on the left in
 (\ref{eq:uqrels}),  multiply each term on the left by $x$ to get
 \begin{equation}
 \label{eq:xxy}
 \frac{qx^2y-q^{-1}xyx}{q-q^{-1}}=x.
 \end{equation}
 Also, in the equation on the left in
 (\ref{eq:uqrels}),  multiply each term on the right by $x$ to get
 \begin{equation} 
 \label{eq:yxx} 
 \frac{qxyx-q^{-1}yx^2}{q-q^{-1}}=x.
 \end{equation} 
 Now in 
 (\ref{eq:xxy}),  eliminate $xyx$ using
 (\ref{eq:yxx}) to obtain the equation on the left in
 (\ref{eq:x2y}). The remaining equations in
 (\ref{eq:x2y}) are similarly verified.
 \end{proof}

\begin{lemma}
\label{lem:xy2}
Assume that $q^4 \not=1$. Then in
 $U_q({\mathfrak{sl}_2})$,
 \begin{equation}
 \label{eq:xy2}
 \frac{q^2xy^2-q^{-2}y^2x}{q^2-q^{-2}}=y,
 \qquad
 \frac{q^2yz^2-q^{-2}z^2y}{q^2-q^{-2}}=z,
 \qquad
 \frac{q^2zx^2-q^{-2}x^2z}{q^2-q^{-2}}=x.
 \end{equation}
 \end{lemma}
 \begin{proof}
 Similar to the proof of
 Lemma
 \ref{lem:x2y}.
 \end{proof}

 \begin{lemma}
\label{lem:xyzUQ} 
Assume that $q^4\not=1$. Then
in 
 $U_q({\mathfrak{sl}_2})$,
 \begin{eqnarray}
 -\frac{\Lambda}{q+q^{-1}}
 &=& \frac{q^2xyz-q^{-2}zyx}{q^2-q^{-2}}-x-z
 \label{eq:casxyz}
 \\ 
 &=& \frac{q^2yzx-q^{-2}xzy}{q^2-q^{-2}}-y-x
 \label{eq:casyzx}
 \\
 &=& \frac{q^2zxy-q^{-2}yxz}{q^2-q^{-2}}-z-y.
 \label{eq:caszxy}
 \end{eqnarray}
 \end{lemma}
 \begin{proof} We verify
 (\ref{eq:casxyz}).
 We have 
 \begin{eqnarray*}
 \Lambda = qx+q^{-1}y+qz-qxyz,
 \end{eqnarray*}
 so 
 \begin{eqnarray}
 q\Lambda = q^2x+y+q^2z-q^2xyz.
 \label{eq:cas1}
 \end{eqnarray}
 We have
 \begin{eqnarray*}
 \Lambda= q^{-1}x+qy+q^{-1}z-q^{-1}zyx,
 \end{eqnarray*}
 so 
 \begin{eqnarray}
 \label{eq:cas2}
 q^{-1}\Lambda= q^{-2}x+y+q^{-2}z-q^{-2}zyx.
 \end{eqnarray}
 Subtract 
 (\ref{eq:cas2}) from 
 (\ref{eq:cas1}) and simplify  to get
 (\ref{eq:casxyz}).
 The equations
 (\ref{eq:casyzx}), 
 (\ref{eq:caszxy}) 
 are
 similarly verified.
 \end{proof}

\begin{definition}
\label{def:UqMock} 
\rm
 Let 
$U^E_q(\mathfrak{sl}_2)$ denote the $\mathbb F$-algebra
with generators $x,y,z$ and relations
\begin{eqnarray}
&&\frac{qx^2y-q^{-1}yx^2}{q-q^{-1}}=-x,
\qquad
\frac{qy^2z-q^{-1}zy^2}{q-q^{-1}}=-y,
\qquad
\frac{qz^2x-q^{-1}xz^2}{q-q^{-1}}=-z,
\nonumber
\\
&&\frac{qxy^2-q^{-1}y^2x}{q-q^{-1}}=-y,
\qquad
\frac{qyz^2-q^{-1}z^2y}{q-q^{-1}}=-z,
\qquad
\frac{qzx^2-q^{-1}x^2z}{q-q^{-1}}=-x,
\nonumber
\\
&&
\frac{qxyz-q^{-1}zyx}{q-q^{-1}}+x+z =
\frac{qyzx-q^{-1}xzy}{q-q^{-1}}+y+x =
\frac{qzxy-q^{-1}yxz}{q-q^{-1}}+z+y.
\qquad \quad 
\label{eq:omeg}
\end{eqnarray}
We call
$U^E_q(\mathfrak{sl}_2)$ the {\it extended
$U_q(\mathfrak{sl}_2)$} algebra.
Let 
$\Omega$ denote the common value
of 
(\ref{eq:omeg}).
\end{definition}

\begin{lemma}
\label{lem:omeg}
The element
$\Omega$ 
from
 Definition
\ref{def:UqMock} is
central in $U^E_q(\mathfrak{sl}_2)$.
\end{lemma}
\begin{proof}
Using the relations from
 Definition
\ref{def:UqMock}, one checks that
$\Omega$ commutes with each generator $x,y,z$
of 
$U^E_q(\mathfrak{sl}_2)$.
\end{proof}

\begin{lemma}
Assume that $q^4 \not=1$ 
and there exists $i \in \mathbb F$ such that
$i^2=-1$. Then there exists an $\mathbb F$-algebra
homomorhism 
$U^E_{q^2}(\mathfrak{sl}_2) \to 
U_q(\mathfrak{sl}_2)$ that sends
\begin{eqnarray*}
x \mapsto ix, \qquad \quad
y \mapsto iy, \qquad \quad
z \mapsto iz, \qquad \quad
\Omega \mapsto \frac{i \Lambda}{q+q^{-1}}.
\end{eqnarray*}
\end{lemma}
\begin{proof}
Compare 
the relations in
Lemmas
\ref{lem:x2y}--\ref{lem:xyzUQ}
with the relations in
Definition
\ref{def:UqMock}.
\end{proof}

\begin{proposition}
Let $A,B,C$ denote an LR triple contained in
${\rm NBNG}_d(\mathbb F; q^{-2})$. Then the underlying
vector space $V$ becomes a 
$U^E_{q}(\mathfrak{sl}_2)$-module on
 which
\begin{eqnarray}
A = x, \qquad \qquad
B = y, \qquad \qquad
C = z.
\label{eq:NBNGxyz}
\end{eqnarray}
The 
$U^E_{q}(\mathfrak{sl}_2)$-module $V$ is irreducible.
On the 
$U^E_{q}(\mathfrak{sl}_2)$-module $V$,
\begin{eqnarray}
\label{eq:Omegaval}
\Omega = 
\frac{(q^{d/2}-q^{-d/2})(q^{1+d/2}-q^{-1-d/2})}{q-q^{-1}} I.
\end{eqnarray}
\end{proposition}
\begin{proof} 
To get the first assertion
and
(\ref{eq:Omegaval}),
compare the relations in
Definition
\ref{def:UqMock} 
with
the relations for 
${\rm NBNG}_d(\mathbb F; q^{-2})$ given in
Proposition
\ref{prop:ABCRel}.
The 
$U^E_{q}(\mathfrak{sl}_2)$-module $V$
is irreducible by
(\ref{eq:NBNGxyz})
and
Lemma
\ref{lem:ABCGEN}.
\end{proof}

\noindent We are done discussing the 
LR triples contained in
${\rm NBNG}_d(\mathbb F; q^{-2})$.
\medskip

\noindent Until further notice let $A,B,C$
denote an LR triple that is contained in
${\rm B}_d(\mathbb F;q^{-2}, \rho_0, \rho'_0, \rho''_0)$.
Let $V$ denote the underlying vector space, and let $J$
denote the projector.

\begin{definition}
\label{def:xyzBq}
\rm
Define $X,Y,Z$ in
${\rm End}(V)$ such that
for $0 \leq i \leq d$,
$X-q^{d/2-i}I$
(resp. $Y-q^{d/2-i}I$)
(resp. $Z-q^{d/2-i}I$) vanishes on component
$i$ of the
$(B,C)$-decomposition
(resp. $(C,A)$-decomposition)
(resp. $(A,B)$-decomposition) of $V$. Note that
each of $X,Y,Z$ is invertible.
\end{definition}

\noindent 
Recall the idempotent data 
(\ref{eq:idseq}) for
$A,B,C$.

\begin{lemma}
The elements $X,Y,Z$ from Definition 
\ref{def:xyzBq} satisfy
\begin{eqnarray*}
X = \sum_{i=0}^d q^{d/2-i} E'_i,
\qquad \qquad
Y = \sum_{i=0}^d q^{d/2-i} E''_i,
\qquad \qquad
Z = \sum_{i=0}^d q^{d/2-i} E_i.
\end{eqnarray*}
\end{lemma}
\begin{proof}
By Definition
\ref{def:xyzBq} and the meaning of the idempotent
data.
\end{proof}

\begin{lemma} 
\label{eq:XYZformBq}
The elements $X,Y,Z$ from Definition 
\ref{def:xyzBq} satisfy
\begin{eqnarray*}
&&
X=(q^{-d/2}I-B C q^{1-d/2}(q-q^{-1}) \rho_0')J
+
(q^{1+d/2}I-B C q^{d/2}(q-q^{-1})/\rho_0')(I-J),
\\
&&
Y=(q^{-d/2}I-C A q^{1-d/2}(q-q^{-1}) \rho_0'')J
+
(q^{1+d/2}I-C A q^{d/2}(q-q^{-1})/\rho_0'')(I-J),
\\
&&
Z=(q^{-d/2}I-A B q^{1-d/2}(q-q^{-1}) \rho_0)J
+
(q^{1+d/2}I-A B q^{d/2}(q-q^{-1})/\rho_0)(I-J).
\end{eqnarray*}
\end{lemma}
\begin{proof}
To verify the first equation,
work with the matrices in
${\rm Mat}_{d+1}(\mathbb F)$
that represent
$B,C,J,X$ with respect to a
$(B,C)$-basis for $V$.
For $B,C$ these matrices are given in Proposition
\ref{prop:matrixRep}. 
For $J$ this matrix is given in
Lemma
\ref{lem:JMat}.
For $X$ this matrix is diagonal,
with $(i,i)$-entry $q^{d/2-i}$ for $0 \leq i \leq d$.
The other two equations are similarly verified.
\end{proof}

\begin{lemma}
\label{lem:XYZBq}
The elements $X,Y,Z$ from Definition 
\ref{def:xyzBq} satisfy
\begin{eqnarray*}
\frac{qXY-q^{-1}YX}{q-q^{-1}} = I,
\quad \qquad 
\frac{qYZ-q^{-1}ZY}{q-q^{-1}} = I,
\quad \qquad 
\frac{qZX-q^{-1}XZ}{q-q^{-1}} = I.
\end{eqnarray*}
\end{lemma}
\begin{proof}
To verify these equations, eliminate
$X,Y,Z$ using
Lemma
\ref{eq:XYZformBq},
and evaluate the
result using
the relations for
${\rm B}_d(\mathbb F;q^{-2},\rho_0, \rho'_0, \rho''_0)$ given in
Proposition
\ref{prop:ABCRelBIP}.
\end{proof}

\begin{proposition}
\label{prop:Bq}
Let $A,B,C$ denote an LR triple contained in
${\rm B}_d(\mathbb F;q^{-2},\rho_0,\rho'_0,\rho''_0)$.
Then there exists a unique
$U_q(\mathfrak{sl}_2)$-module structure on 
the underlying vector space $V$, such that
for $0 \leq i \leq d$,
$x-q^{d/2-i}1$
(resp. $y-q^{d/2-i}1$)
(resp. $z-q^{d/2-i}1$)
vanishes on component $i$ of the 
$(B,C)$-decomposition 
(resp. $(C,A)$-decomposition) 
(resp. $(A,B)$-decomposition) of $V$.
For the 
$U_q(\mathfrak{sl}_2)$-module $V$
the subspaces $V_{\rm out}$ and 
$V_{\rm in}$ are irreducible 
$U_q(\mathfrak{sl}_2)$-submodules.
On the 
$U_q(\mathfrak{sl}_2)$-module  $V$,
\begin{eqnarray}
\label{eq:ABCBqnuxyz}
A^2 = n_x, \qquad \qquad
B^2 = n_y, \qquad \qquad
C^2 = n_z.
\end{eqnarray}
\end{proposition}
\begin{proof}
The 
$U_q(\mathfrak{sl}_2)$-module structure exists,
by
Lemma
\ref{lem:XYZBq}
 and since $Y$ is invertible.
The 
$U_q(\mathfrak{sl}_2)$-module structure is unique
by construction.
On the $U_q(\mathfrak{sl}_2)$-module $V$ we have
$x=X$,
$y=Y$,
$z=Z$.
To verify 
(\ref{eq:ABCBqnuxyz}),  eliminate 
$n_x$, $n_y$, $n_z$ using
Definition \ref{def:nxnynz}, and
evaluate the
result using
 Lemma
\ref{eq:XYZformBq}
along with the relations for
${\rm B}_d(\mathbb F;q^{-2},\rho_0,\rho'_0,\rho''_0)$ given in
Proposition
\ref{prop:ABCRelBIP}. 
By construction
$V_{\rm out}$ and $V_{\rm in}$ are 
$U_q(\mathfrak{sl}_2)$-submodules of $V$.
By Lemmas 
\ref{lem:A2B2C2Out},
\ref{lem:A2B2C2In}
the 3-tuple 
$A^2,B^2,C^2$ acts on
$V_{\rm out}$ and $V_{\rm in}$ as an LR triple.
By these comments along with 
 (\ref{eq:ABCBqnuxyz}) 
and Lemma \ref{lem:ABCGEN}, 
we find that the 
$U_q(\mathfrak{sl}_2)$-submodules 
$V_{\rm out}$ and $V_{\rm in}$
are irreducible.
\end{proof}

\noindent We mention some additional relations that
hold on the 
$U_q(\mathfrak{sl}_2)$-module $V$ from Proposition
\ref{prop:Bq}. These relations may be of independent interest.

\begin{lemma}
\label{lem:MoreBp}
For the 
$U_q(\mathfrak{sl}_2)$-module $V$ from 
Proposition
\ref{prop:Bq},
we have
\begin{eqnarray*}
&&
xB = qBx, \qquad \qquad  yC=qCy, \qquad \qquad zA = qAz,
\\
&&
yA  = q^{-1} A y, \qquad \quad 
zB = q^{-1}Bz, \qquad \quad xC=q^{-1}Cx
\end{eqnarray*}
and also 
\begin{eqnarray*}
Jx=xJ, \qquad \qquad Jy=yJ, \qquad \qquad Jz = zJ.
\end{eqnarray*}
We have
\begin{eqnarray*}
&&
\Bigl(
AB-\frac{I - q^{d/2}z}{\rho_0 q(q-q^{-1})}\Bigr)J=0,
\qquad \qquad
\Bigl(
BA-\frac{\rho_0 q I - \rho_0 q^{1-d/2}z}{q-q^{-1}}\Bigr)J=0,
\\
&&
\Bigl(
BC-\frac{I - q^{d/2}x}{\rho'_0 q(q-q^{-1})}\Bigr)J=0,
\qquad \qquad
\Bigl(
CB-\frac{\rho'_0 q I - \rho'_0 q^{1-d/2}x}{q-q^{-1}}\Bigr)J=0,
\\
&&
\Bigl(
CA-\frac{I - q^{d/2}y}{\rho''_0 q(q-q^{-1})}\Bigr)J=0,
\qquad \qquad
\Bigl(
AC-\frac{\rho''_0 q I - \rho''_0 q^{1-d/2}y}{q-q^{-1}}\Bigr)J=0
\end{eqnarray*}
and also
\begin{eqnarray*}
&&
\Bigl(
AB-\frac{\rho_0 q I - \rho_0 q^{-d/2}z}{q-q^{-1}}\Bigr)(I-J)=0,
\qquad \qquad
\Bigl(
BA-\frac{ I - q^{1+d/2}z}{\rho_0 q (q-q^{-1})}\Bigr)(I-J)=0,
\\
&&
\Bigl(
BC-\frac{\rho'_0 q I - \rho'_0 q^{-d/2}x}{q-q^{-1}}\Bigr)(I-J)=0,
\qquad \qquad
\Bigl(
CB-\frac{ I - q^{1+d/2}x}{\rho'_0 q (q-q^{-1})}\Bigr)(I-J)=0,
\\
&&
\Bigl(
CA-\frac{\rho''_0 q I - \rho''_0 q^{-d/2}y}{q-q^{-1}}\Bigr)(I-J)=0,
\qquad \qquad
\Bigl(
AC-\frac{ I - q^{1+d/2}y}{\rho''_0 q (q-q^{-1})}\Bigr)(I-J)=0.
\end{eqnarray*}
We have
\begin{eqnarray*}
&&
(Ax -Bq^{-d/2}\rho_0 - C q^{d/2}/\rho''_0)J=0,
\qquad \quad
(xA -Bq^{1-d/2}\rho_0 - C q^{d/2-1}/\rho''_0)J=0,
\\
&&
(By -Cq^{-d/2}\rho'_0 - A q^{d/2}/\rho_0)J=0,
\qquad \quad
(yB -Cq^{1-d/2}\rho'_0 - A q^{d/2-1}/\rho_0)J=0,
\\
&&
(Cz -Aq^{-d/2}\rho''_0 - B q^{d/2}/\rho'_0)J=0,
\qquad \quad
(zC -Aq^{1-d/2}\rho''_0 - B q^{d/2-1}/\rho'_0)J=0
\end{eqnarray*}
and also
\begin{eqnarray*}
&&
J(Ax -Bq^{d/2-1}/\rho_0 - C q^{1-d/2}\rho''_0)=0,
\qquad \quad
J(xA -Bq^{d/2}/\rho_0 - C q^{-d/2}\rho''_0)=0,
\\
&&
J(By -Cq^{d/2-1}/\rho'_0 - A q^{1-d/2}\rho_0)=0,
\qquad \quad
J(yB -Cq^{d/2}/\rho'_0 - A q^{-d/2}\rho_0)=0,
\\
&&
J(Cz -Aq^{d/2-1}/\rho''_0 - B q^{1-d/2}\rho'_0)=0,
\qquad \quad
J(zC -Aq^{d/2}/\rho''_0 - B q^{-d/2}\rho'_0)=0.
\end{eqnarray*}
\end{lemma}
\begin{proof}
Similar to the proof of
Lemma 
\ref{lem:XYZBq}.
\end{proof}

\noindent We are done discussing
the LR triples contained in
${\rm B}_d(\mathbb F;q^{-2},\rho_0,\rho'_0,\rho''_0)$.
\medskip

\noindent 
For the rest of this section, assume that 
${\rm Char}(\mathbb F)\not=2$.
We now recall the Lie algebra 
$\mathfrak{sl}_2$ and its equitable basis.

\begin{definition}\rm
\label{def:sl2A}
{\rm (See \cite[Lemma~3.2]{ht}.)}
Let $\mathfrak{sl}_2$ denote the Lie algebra over $\mathbb F$
with basis $x,y,z$ and Lie bracket
\begin{equation}
\label{eq:uqrelsq1}
\lbrack{ x,y}\rbrack = 2 {x}+2{ y}, \qquad 
\lbrack{ y,z}\rbrack = 2{ y}+2 { z}, \qquad 
\lbrack{ z,x}\rbrack = 2{ z}+2{ x}.
\end{equation}
\end{definition}

\noindent Until further notice let $A,B,C$ denote an LR triple
that is contained in
${\rm NBG}_d(\mathbb F;1)$.  Let $V$ denote the underlying vector
space.

\begin{definition}
\label{def:XYZsl2}
\rm
Define $X,Y,Z$ in
${\rm End}(V)$ such that
for $0 \leq i \leq d$,
$X-(2i-d)I$
(resp. $Y-(2i-d)I$)
(resp. $Z-(2i-d)I$) vanishes on component $i$ of
the
$(B,C)$-decomposition 
(resp. $(C,A)$-decomposition)
(resp. $(A,B)$-decomposition) 
of $V$.
\end{definition}

\noindent Recall the idempotent data
(\ref{eq:idseq}) for
 $A,B,C$.

\begin{lemma}
The elements $X,Y,Z$ from Definition 
\ref{def:XYZsl2} satisfy
\begin{eqnarray*}
X = \sum_{i=0}^d (2i-d) E'_i,
\qquad \qquad
Y = \sum_{i=0}^d (2i-d) E''_i,
\qquad \qquad
Z = \sum_{i=0}^d (2i-d) E_i.
\end{eqnarray*}
\end{lemma}
\begin{proof}
By Definition
\ref{def:XYZsl2}
and the meaning of the idempotent
data.
\end{proof}

\begin{lemma}
\label{lem:XYZsl2A}
The elements $X,Y,Z$ from
 Definition
\ref{def:XYZsl2} satisfy
\begin{eqnarray*}
X = B+C -A,
\qquad \quad 
Y = C+A-B,  
\qquad \quad 
Z = A+B-C.
\end{eqnarray*}
\end{lemma}
\begin{proof}
To verify the first equation,
work with the matrices in
${\rm Mat}_{d+1}(\mathbb F)$
that represent
$A,B,C,X$ with respect to a
$(B,C)$-basis for $V$.
For $A,B,C$ these matrices are given in Proposition
\ref{prop:matrixRep}. For $X$ this matrix is diagonal,
with $(i,i)$-entry $2i-d$ for $0 \leq i \leq d$.
The other two equations are similarly verified.
\end{proof}

\begin{lemma}
\label{lem:sl2moduleExists}
The elements $X,Y,Z$ from
 Definition
\ref{def:XYZsl2} satisfy
\begin{eqnarray*}
\lbrack{ X,Y}\rbrack = 2 {X}+2{Y}, \qquad \quad  
\lbrack{ Y,Z}\rbrack = 2{ Y}+2 {Z}, \qquad  \quad
\lbrack{ Z,X}\rbrack = 2{ Z}+2{ X}.
\end{eqnarray*}
\end{lemma}
\begin{proof}
To verify these equations, eliminate
$X,Y,Z$ using
Lemma \ref{lem:XYZsl2A},
and evaluate the
result using
the relations for
${\rm NBG}_d(\mathbb F;1)$ given in
Proposition
\ref{prop:ABCRel}.
\end{proof}

\begin{proposition}
\label{prop:mainRESsl2}
Let $A,B,C$ denote an LR triple contained in
${\rm NBG}_d(\mathbb F;1)$.
Then there exists a unique
$\mathfrak{sl}_2$-module structure on 
the underlying vector space $V$, such that
for $0 \leq i \leq d$,
$x-(2i-d)1$
(resp. $y-(2i-d)1$)
(resp. $z-(2i-d)1$)
vanishes on component $i$ of the 
$(B,C)$-decomposition 
(resp. $(C,A)$-decomposition) 
(resp. $(A,B)$-decomposition) of $V$.
The
$\mathfrak{sl}_2$-module  $V$ is irreducible.
On the 
$\mathfrak{sl}_2$-module  $V$,
\begin{eqnarray}
\label{eq:ABCnuxyzsl2}
A = (y+z)/2, \qquad \qquad
B = (z+x)/2, \qquad \qquad
C = (x+y)/2.
\end{eqnarray}
\end{proposition}
\begin{proof}
The 
$\mathfrak{sl}_2$-module structure exists by
Lemma \ref{lem:sl2moduleExists}.
The 
$\mathfrak{sl}_2$-module structure is unique by
construction.
On the 
$\mathfrak{sl}_2$-module $V$ we have
$x=X$,
$y=Y$,
$z=Z$. To verify
(\ref{eq:ABCnuxyzsl2}), eliminate
$x,y,z $ using
Lemma \ref{lem:XYZsl2A}, and 
evaluate the result using
the relations for
${\rm NBG}_d(\mathbb F;1)$
given in
Proposition
\ref{prop:ABCRel}.
The 
$\mathfrak{sl}_2$-module $V$ is irreducible by
(\ref{eq:ABCnuxyzsl2}) and Lemma
\ref{lem:ABCGEN}.
\end{proof}

\noindent We are done discussing the LR triples contained in
${\rm NBG}_d(\mathbb F;1)$. 
\medskip

\noindent 
For the rest of this section
let $A,B,C$ denote an LR triple
that is contained in
${\rm B}_d(\mathbb F;1,\rho_0,\rho'_0,\rho''_0)$.
Let $V$ denote the underlying vector
space, and let $J$ denote the projector.

\begin{definition}
\label{def:XYZsl2Bip}
\rm
Define $X,Y,Z$ in
${\rm End}(V)$ such that
for $0 \leq i \leq d$,
$X-(i-d/2)I$
(resp. $Y-(i-d/2)I$)
(resp. $Z-(i-d/2)I$) vanishes on component $i$ of
the
$(B,C)$-decomposition 
(resp. $(C,A)$-decomposition)
(resp. $(A,B)$-decomposition) 
of $V$.
\end{definition}

\noindent Recall the idempotent data
(\ref{eq:idseq}) for
 $A,B,C$.

\begin{lemma}
The elements $X,Y,Z$ from Definition 
\ref{def:XYZsl2Bip} satisfy
\begin{eqnarray*}
X = \sum_{i=0}^d (i-d/2) E'_i,
\qquad \qquad
Y = \sum_{i=0}^d (i-d/2) E''_i,
\qquad \qquad
Z = \sum_{i=0}^d (i-d/2) E_i.
\end{eqnarray*}
\end{lemma}
\begin{proof}
By Definition
\ref{def:XYZsl2Bip}
and the meaning of the idempotent
data.
\end{proof}

\begin{lemma}
\label{lem:XYZsl2ABip}
The elements $X,Y,Z$ from
 Definition
\ref{def:XYZsl2Bip} satisfy
\begin{eqnarray*}
&&
X = \Bigl(2 \rho_0' B C +  \frac{d}{2}I \Bigr)J
+ \Bigl(   \frac{2}{\rho_0'} B C - \frac{d+2}{2} I \Bigr) (I-J),
\\
&&
Y = \Bigl(2 \rho_0'' C A +  \frac{d}{2}I\Bigr)J
+ \Bigl( \frac{2}{\rho_0''} C A - \frac{d+2}{2} I \Bigr) (I-J),
\\
&&
Z = \Bigl(2 \rho_0 A B +  \frac{d}{2}I \Bigr)J
+ \Bigl(   \frac{2}{\rho_0} A B- \frac{d+2}{2} I \Bigr) (I-J).
\end{eqnarray*}
\end{lemma}
\begin{proof}
To verify the first equation,
work with the matrices in
${\rm Mat}_{d+1}(\mathbb F)$
that represent
$B,C,J,X$ with respect to a
$(B,C)$-basis for $V$.
For $B,C$ these matrices are given in Proposition
\ref{prop:matrixRep}. 
For $J$ this matrix is given in
Lemma
\ref{lem:JMat}.
For $X$ this matrix is diagonal,
with $(i,i)$-entry $i-d/2$ for $0 \leq i \leq d$.
The other two equations are similarly verified.
\end{proof}

\begin{lemma}
\label{lem:sl2moduleExistsBip}
The elements $X,Y,Z$ from
 Definition
\ref{def:XYZsl2Bip} satisfy
\begin{eqnarray*}
\lbrack{ X,Y}\rbrack = 2 {X}+2{Y}, \qquad \quad  
\lbrack{ Y,Z}\rbrack = 2{ Y}+2 {Z}, \qquad  \quad
\lbrack{ Z,X}\rbrack = 2{ Z}+2{ X}.
\end{eqnarray*}
\end{lemma}
\begin{proof}
To verify these equations, eliminate
$X,Y,Z$ using
Lemma \ref{lem:XYZsl2ABip},
and evaluate the
result using
the relations for
${\rm B}_d(\mathbb F;1,\rho_0,\rho'_0,\rho''_0)$ given in
Proposition
\ref{prop:ABCRelBIP}.
\end{proof}

\begin{proposition}
\label{prop:mainRESsl2Bip}
Let $A,B,C$ denote an LR triple contained in
${\rm B}_d(\mathbb F;1,\rho_0,\rho'_0,\rho''_0)$.
Then there exists a unique
$\mathfrak{sl}_2$-module structure on 
the underlying vector space $V$, such that
for $0 \leq i \leq d$,
$x-(i-d/2)1$
(resp. $y-(i-d/2)1$)
(resp. $z-(i-d/2)1$)
vanishes on component $i$ of the 
$(B,C)$-decomposition 
(resp. $(C,A)$-decomposition) 
(resp. $(A,B)$-decomposition) of $V$.
For the 
$\mathfrak{sl}_2$-module  $V$ the subspaces
$V_{\rm out}$ and
$V_{\rm in}$ are irreducible
$\mathfrak{sl}_2$-submodules.
On the 
$\mathfrak{sl}_2$-module  $V$,
\begin{eqnarray}
\label{eq:ABCnuxyzsl2Bip}
A^2 = (y+z)/2, \qquad \qquad
B^2 = (z+x)/2, \qquad \qquad
C^2 = (x+y)/2.
\end{eqnarray}
\end{proposition}
\begin{proof}
Similar to the proof of
Proposition
\ref{prop:Bq}.
\end{proof}

\noindent We mention some additional relations that
hold on the 
$\mathfrak{sl}_2$-module $V$ from Proposition
\ref{prop:mainRESsl2Bip}.
These relations may be of independent interest.

\begin{lemma}
\label{lem:BipIndep}
For the 
$\mathfrak{sl}_2$-module $V$ from Proposition
\ref{prop:mainRESsl2Bip}, we have
\begin{eqnarray*}
\lbrack A,z\rbrack = A, 
\qquad \qquad 
\lbrack B,x\rbrack = B, 
\qquad \qquad 
\lbrack C,y\rbrack = C
\end{eqnarray*}
and also
\begin{eqnarray*}
\lbrack J,x\rbrack=0, \qquad \qquad
\lbrack J,y\rbrack=0, \qquad \qquad
\lbrack J,z\rbrack=0.
\end{eqnarray*}
\noindent We have
\begin{eqnarray*}
&&
\Bigl(
AB - \frac{2 z - d  }{4\rho_0} \Bigr)
J=0,
\qquad \qquad
\Bigl(
B A -  \frac{\rho_0(2 z + d  )}{4}\Bigr)J=0,
\\
&&
\Bigl(
BC - \frac{2 x - d  }{4\rho'_0}\Bigr)
J=0,
\qquad \qquad
\Bigl(
C B - \frac{\rho'_0(2 x + d )}{4} \Bigr)J=0,
\\
&&
\Bigl(
CA - \frac{2 y - d }{4\rho''_0} \Bigr)
J=0,
\qquad \qquad
\Bigl(
A C - \frac{\rho''_0(2 y + d )}{4} \Bigr)J=0
\end{eqnarray*} 
and also
\begin{eqnarray*}
&&
\Bigl(A B - \frac{\rho_0(2 z + d+2)}{4} \Bigr)(I-J)=0,
\qquad \qquad
\Bigl( B A - \frac{2 z - d-2}{4 \rho_0} \Bigr)(I-J)=0,
\\
&&
\Bigl(B C - \frac{\rho'_0(2 x + d+2)}{4} \Bigr)(I-J)=0,
\qquad \qquad
\Bigl( C B -  \frac{2 x - d-2}{4 \rho'_0} \Bigr)(I-J)=0,
\\
&&
\Bigl(C A - \frac{\rho''_0(2 y + d+2)}{4} \Bigr)(I-J)=0,
\qquad \qquad
\Bigl( A C -  \frac{2 y - d-2}{4 \rho''_0} \Bigr)(I-J)=0.
\end{eqnarray*}
\end{lemma}
\begin{proof} Similar to the proof of
Lemma 
\ref{lem:sl2moduleExistsBip}.
\end{proof}

\section{Three characterizations of an LR triple}

\noindent Throughout this section
let $V$ denote a vector space over $\mathbb F$
with dimension $d+1$. We characterize the LR triples on $V$
in three ways.
\medskip

\noindent A matrix $M \in 
{\rm Mat}_{d+1}(\mathbb F)$ will be called
{\it antidiagonal} whenever the entry
$M_{i,j}=0$ if $i+j \not=d$ $(0 \leq i,j\leq d)$.
Note that the following are equivalent:
(i) $M$ is antidiagonal; (ii) ${\bf Z}M$ is diagonal;
(iii) $M{\bf Z}$ is diagonal.

\begin{theorem} 
\label{thm:6bases}
Suppose we are given six bases for
$V$, denoted
\begin{eqnarray}
\label{eq:6Bases}
\mbox{\rm 
basis 1, $\quad$
basis 2,$\quad$
basis 3, $\quad$
basis 4, $\quad$
basis 5, $\quad$
basis 6}.
\end{eqnarray}
\noindent Then the following are equivalent:
\begin{enumerate}
\item[\rm (i)] The transition matrix 

     \bigskip

\centerline{
\begin{tabular}[t]{c|c|c}
 {\rm from basis} &
{\rm to basis} 
& {\rm is }
 \\
 \hline
 {\rm  1} &  {\rm  2} &
{\rm upper triangular Toeplitz}
   \\
 {\rm  2} &  {\rm  3} &
{\rm antidiagonal}
   \\
 {\rm  3} &  {\rm 4} &
{\rm upper triangular Toeplitz}
   \\
 {\rm  4} &  {\rm 5} &
{\rm antidiagonal}
   \\
 {\rm  5} &  {\rm  6} &
{\rm upper triangular Toeplitz}
   \\
 {\rm  6} &  {\rm 1} &
{\rm antidiagonal}
   \\
   \end{tabular}}
     \bigskip
\item[\rm (ii)] There exists an LR triple $A,B,C$ on
$V$ for which
     \bigskip

\centerline{
\begin{tabular}[t]{c|c}
 {\rm basis} & {\rm has type}
 \\
 \hline
{\rm 1}
&
$(A,C)$
\\
{\rm 2}
&
$(A,B)$
\\
{\rm 3}
&
$(B,A)$
\\
{\rm 4}
&
$(B,C)$
\\
{\rm 5}
&
$(C,B)$
\\
{\rm 6}
&
$(C,A)$
\\
     \end{tabular}}
     \bigskip
\end{enumerate}
\noindent Suppose {\rm (i), (ii)} hold. Then
the LR triple $A,B,C$ is uniquely determined by the
sequence
{\rm (\ref{eq:6Bases})}.
\end{theorem}
\begin{proof}
${\rm (i)}\Rightarrow {\rm (ii)}$ 
For $1 \leq k \leq 6$ let
$\mathcal D_k$
denote the decomposition of $V$
induced by basis $k$.
By assumption,
the transition matrix from
basis 1 to basis 2
is upper triangular and Toeplitz. 
For notational convenience  denote basis 1 by
$\lbrace u_i \rbrace_{i=0}^d$ and
basis 2 by $\lbrace v_i \rbrace_{i=0}^d$.
By Proposition
\ref{prop:Toe}
there exists $A \in
{\rm End}(V)$ such that
\begin{eqnarray}
\nonumber
&&Au_i = u_{i-1} \qquad (1 \leq i \leq d),
\qquad 
Au_0 =0,
\\
\label{eq:whatNeed2}
&&
Av_i = v_{i-1} \qquad (1 \leq i \leq d),
\qquad
Av_0 =0.
\end{eqnarray}
Consequently $A$ lowers 
$\mathcal D_1$ and $\mathcal D_2$.
Similarly there exists
 $B \in
{\rm End}(V)$ that lowers
$\mathcal D_3$ and $\mathcal D_4$. Also
there exists
 $C \in
{\rm End}(V)$ that lowers
$\mathcal D_5$ and $\mathcal D_6$.
By our assumption concerning
the three antidiagonal transition
matrices, the decompositions
 $\mathcal D_2$, $\mathcal D_4$, $\mathcal D_6$ 
are the inversions of 
$\mathcal D_3$, $\mathcal D_5$, $\mathcal D_1$, respectively.
By these comments $A$, $B$, $C$ raise
$\mathcal D_6$, $\mathcal D_2$, $\mathcal D_4$ respectively.
Observe that
$\mathcal D_2$
is lowered by $A$ and raised by $B$;
therefore $A,B$ form an LR pair on $V$.
Similarly $\mathcal D_4$
is lowered by $B$ and raised by $C$;
therefore $B,C$ form an LR pair on $V$.
Also $\mathcal D_6$
is lowered by $C$ and raised by $A$;
therefore $C,A$ form an LR pair on $V$.
By these comments $A,B,C$ form an LR triple on $V$.
We now show that basis 2 is an $(A,B)$-basis of $V$.
We just mentioned that
$\mathcal D_2$ is lowered by $A$ and raised by $B$.
Therefore 
 $\mathcal D_2$ is the $(A,B)$-decomposition
of $V$.
Now using (\ref{eq:whatNeed2}) and
Definition
\ref{def:ABbasis},
we see that basis 2 is an $(A,B)$-basis of $V$.
We have shown that basis 2 meets the requirements
of the table in (ii).
One similarly shows that bases 1, 3, 4, 5, 6
meet these requirements.
\\
\noindent
${\rm (ii)}\Rightarrow {\rm (i)}$ 
By assumption basis 1
is an $(A,C)$-basis of $V$, and
basis 2
is an $(A,B)$-basis of $V$.
By Lemma
\ref{lem:xyz},
the transition matrix from basis 1 to basis 2
is upper triangular and Toeplitz.
By assumption basis 3
is a $(B,A)$-basis of $V$. So the inversion 
of basis 3 
is an
inverted $(B,A)$-basis of $V$. 
By Lemma
\ref{lem:Dmeaning},
the transition
matrix from  
basis 2 to the inversion of basis 3
is diagonal.
Recall that $\bf Z$ is the transition
matrix between basis 3 and its inversion.
By these comments the transition matrix
from
basis 2 to basis 3 is
 antidiagonal.
The remaining assertions of part (i) are similarly
obtained. 
\\
\noindent Assume that (i), (ii) hold. From the
construction of $A$ in the proof of 
${\rm (i)}\Rightarrow {\rm (ii)}$ above,
we find that $A$ is uniquely determined by the sequence
(\ref{eq:6Bases}).
Similarly $B$ and $C$ are uniquely determined by the
the sequence (\ref{eq:6Bases}).
\end{proof}


\begin{theorem}
Suppose we are given six invertible matrices in
${\rm Mat}_{d+1}(\mathbb F)$:
\begin{eqnarray}
&&D_1,\quad D_2, \quad D_3 \qquad \qquad \mbox{\rm (diagonal}),
\label{eq:D123}
\\
&&T_1,\quad T_2, \quad T_3 \qquad \qquad \mbox{\rm (upper triangular
Toeplitz).}
\label{eq:T123}
\end{eqnarray}
Then the following {\rm (i), (ii)} are equivalent.
\begin{enumerate}
\item[\rm (i)] $T_1 D_1 {\bf Z} T_2 D_2 {\bf Z} T_3 D_3 {\bf Z} 
\in \mathbb F I$.
\item[\rm (ii)] There exists an LR triple $A,B,C$ over $\mathbb F$
for which the matrices
{\rm (\ref{eq:D123}), (\ref{eq:T123})} are transition
matrices of the following kind:

     \bigskip

\centerline{
\begin{tabular}[t]{c|c|c}
 {\rm transition matrix} & {\rm from a basis of type}  & {\rm to a basis of type}
 \\
 \hline
  $T_1$ & $(A,C)$ & $(A,B)$
\\
  $D_1$ & $(A,B)$ & {\rm inv. $(B,A)$}
\\
  $T_2$ & $(B,A)$ & $(B,C)$
\\
  $D_2$ & $(B,C)$& {\rm inv. $(C,B)$}
\\
  $T_3$ & $(C,B)$ & $(C,A)$
\\
  $D_3$ & $(C,A)$ & {\rm inv. $(A,C)$}
   \\
   \end{tabular}}
     \bigskip

\end{enumerate}
\noindent Suppose {\rm (i), (ii)} hold.
Then the LR triple $A,B,C$ is uniquely determined up to isomorphism
by the sequence
$D_1,D_2,D_3,
T_1,T_2,T_3$.
\end{theorem}
\begin{proof} 
${\rm (i)}\Rightarrow {\rm (ii)}$
We invoke Theorem
\ref{thm:6bases}. Multiplying $D_1$ by a nonzero
scalar in $\mathbb F$ if necessary,
we may assume without loss 
that 
$T_1 D_1 {\bf Z} T_2 D_2 {\bf Z} T_3 D_3 {\bf Z} =I$.
By linear algebra there exist six bases
of $V$ as in (\ref{eq:6Bases}) 
 such that the transition matrix

     \bigskip

\centerline{
\begin{tabular}[t]{c|c|c}
 {\rm from basis} &
{\rm to basis} 
& {\rm is }
 \\
 \hline
 {\rm  1} &  {\rm  2} &
$T_1$
   \\
 {\rm  2} &  {\rm  3} &
$D_1{\bf Z}$
   \\
 {\rm  3} &  {\rm 4} &
$T_2$
   \\
 {\rm  4} &  {\rm 5} &
$D_2{\bf Z}$
   \\
 {\rm  5} &  {\rm  6} &
$T_3$
   \\
 {\rm  6} &  {\rm 1} &
$D_3{\bf Z}$
   \\
   \end{tabular}}
     \bigskip

\noindent By construction, these six bases satisfy
Theorem \ref{thm:6bases}(i). 
Therefore they satisfy Theorem \ref{thm:6bases}(ii). The
LR triple $A,B,C$ mentioned in
 Theorem \ref{thm:6bases}(ii) 
satisfies
condition (ii) of the present theorem.
\\
\noindent 
${\rm (ii)}\Rightarrow {\rm (i)}$
Associated with the LR triple $A,B,C$ are
the upper triangular Toeplitz matrices
$T$, $T'$, $T''$ from Definition
\ref{def:TTT}, 
and the diagonal matrices $D$, $D'$, $D''$ from
Definition
\ref{def:DDD}.
Using
Lemma
\ref{lem:BAbasisU}
and
Definition
\ref{def:TTT}(ii) we find 
$T_1 \in \mathbb F T'$. 
Similarly 
 $T_2 \in \mathbb F T''$  and
 $T_3 \in \mathbb F T$.
Using 
Lemma \ref{lem:Dmeaning} we find
$D_1 \in \mathbb F D$. 
Similarly $D_2 \in \mathbb F D'$ 
and $D_3 \in \mathbb F D''$. 
By these comments and
Proposition
\ref{prop:12cycle} we obtain
$T_1 D_1 {\bf Z} T_2D_2{\bf Z}T_3D_3{\bf Z}\in \mathbb F I$.
\\
\noindent Assume that (i), (ii) hold.
Consider the matrices
 $D,D',D''$ and $T,T',T''$ from the proof of
${\rm (ii)}\Rightarrow {\rm (i)}$ above.
By Proposition
\ref{prop:extra}
the LR triple $A,B,C$ is uniquely determined
up to isomorphism by
the sequence $D,D',D'',T,T',T''$.
The matrix $D$ is obtained from $D_1$
by dividing $D_1$ by its $(0,0)$-entry.
So $D$ is determined by $D_1$.
The matrices 
$D',D'',T,T',T''$ are similarly
determined by $D_2, D_3,
T_3,T_1,T_2$, respectively.
By these comments
the LR triple $A,B,C$ is uniquely determined
up to isomorphism by the sequence
$D_1,D_2,D_3,
T_1,T_2,T_3$.
\end{proof}

\begin{theorem}
Let $A,B,C
\in {\rm End}(V)$. Then $A,B,C$ form an LR triple on $V$ if
and only if the following {\rm (i)--(iv)} hold:
\begin{enumerate}
\item[\rm (i)] each of $A$, $B$, $C$ is Nil;
\item[\rm (ii)] the flag
$\lbrace A^{d-i}V\rbrace_{i=0}^d$ is raised by $B,C$;
\item[\rm (iii)] the flag
$\lbrace B^{d-i}V\rbrace_{i=0}^d$ is raised by $C,A$;
\item[\rm (iv)] the flag
$\lbrace C^{d-i}V\rbrace_{i=0}^d$ is raised by $A,B$.
\end{enumerate}
\end{theorem}
\begin{proof} By Proposition
\ref{prop:LRchar} and
Definition
	     \ref{def:LRT}.
\end{proof}

\section{Appendix I: The nonbipartite LR triples in matrix form}

\noindent 
In this section we display the nonbipartite equitable LR triples
in matrix form.

\medskip
\noindent Let $V$ denote a vector space over $\mathbb F$
with dimension $d+1$. Let $A,B,C$ denote a nonbipartite 
equitable
LR triple on
$V$, with
 parameter array
(\ref{eq:paLRT}),
trace data
(\ref{eq:tracedata}), and Toeplitz data
(\ref{eq:ToeplitzData}).
By Definition
\ref{def:equitNorm}
we have $\alpha_i = 
\alpha'_i = 
\alpha''_i $ for
$0 \leq i \leq d$,
and by 
Lemma
\ref{lem:equitBasicBeta}
we have
$\beta_i = 
\beta'_i = 
\beta''_i $ for $0 \leq i \leq d$.
By 
(\ref{eq:list3}) we have
$\beta_1=-\alpha_1$.
By Lemma
\ref{lem:equitBasic} we have
$\varphi_i = \varphi'_i = \varphi''_i$
for $1 \leq i \leq d$, and
$a_i = a'_i = a''_i = \alpha_1 (\varphi_{d-i+1}-\varphi_{d-i})$
for $0 \leq i \leq d$.
Recall from Definition 
\ref{def:NBNorm} 
that $A,B,C$ is normalized if and only if $\alpha_1=1$.
By Proposition
\ref{prop:matrixRep}
we have the following.
\medskip

\noindent 
With respect to an $(A,B)$-basis for $V$
the matrices representing $A,B,C$ are
\begin{eqnarray*}
&&A:\; 
\left(
\begin{array}{ c c cc c c }
 0 & 1   &   &&   & \bf 0  \\
  & 0  &  1  &&  &      \\ 
   &   &  0   & \cdot &&  \\
     &   &  & \cdot  & \cdot & \\
       &  & &  & \cdot & 1 \\
        {\bf 0}  &&  & &   &  0  \\
	\end{array}
	\right),
	\qquad \qquad
	B:\;
	\left(
	\begin{array}{ c c cc c c }
	0 &   &   &&   & \bf 0  \\
	\varphi_1 & 0  &   &&  &      \\
	 &  \varphi_2 & 0   & &&  \\
	   &   & \cdot & \cdot  & & \\
	     &  & & \cdot & \cdot & \\
	      {\bf 0}  &&  & &  \varphi_d & 0  \\
	      \end{array}
	      \right),
\\
&&C:\; 
\left(
\begin{array}{ c c c c c c }
 a_0 & \varphi_d/\varphi_1   &   &&   & \bf 0  \\
 \varphi_d & a_1 &  \varphi_{d-1}/\varphi_2  &&  &      \\ 
   & \varphi_{d-1}  & a_2  
   & \cdot &&  \\
     &   & \cdot & \cdot  & \cdot & \\
       &  & & \cdot & \cdot & \varphi_1/\varphi_d \\
        {\bf 0}  &&  & &  \varphi_1 &  a_d \\
	\end{array}
	\right).
\end{eqnarray*}
With respect to an inverted $(A,B)$ basis for $V$ the matrices representing
$A,B,C$ are
\begin{eqnarray*}
 &&A:\;
	\left(
	\begin{array}{ c c cc c c }
	0 &   &   &&   & \bf 0  \\
     1 & 0  &   &&  &      \\
	 &  1& 0   & &&  \\
	   &   & \cdot & \cdot  & & \\
	     &  & & \cdot & \cdot & \\
	      {\bf 0}  &&  & &  1 & 0  \\
	      \end{array}
	      \right),
\qquad \qquad 
B:\; 
\left(
\begin{array}{ c c cc c c }
 0 & \varphi_d   &   &&   & \bf 0  \\
  & 0  &  \varphi_{d-1}  &&  &      \\ 
   &   &  0   & \cdot &&  \\
     &   &  & \cdot  & \cdot & \\
       &  & &  & \cdot & \varphi_1 \\
        {\bf 0}  &&  & &   &  0  \\
	\end{array}
	\right),
\\
&&C:\; 
\left(
\begin{array}{ c c c c c c }
a_d & \varphi_1   &   &&   & \bf 0  \\
 \varphi_1/\varphi_d & a_{d-1}  &  \varphi_{2} &&  &      \\ 
   & \varphi_{2}/\varphi_{d-1}  & a_{d-2}  
   & \cdot &&  \\
     &   & \cdot & \cdot  & \cdot & \\
       &  & & \cdot & \cdot & \varphi_d \\
        {\bf 0}  &&  & &  \varphi_d/\varphi_1 &  a_0  \\
	\end{array}
	\right).
\end{eqnarray*}
With respect to a $(B,A)$ basis for $V$ the matrices representing
$A,B,C$ are
\begin{eqnarray*}
   &&     A:\;
	\left(
	\begin{array}{ c c cc c c }
	0 &   &   &&   & \bf 0  \\
	\varphi_d & 0  &   &&  &      \\
	 &  \varphi_{d-1} & 0   & &&  \\
	   &   & \cdot & \cdot  & & \\
	     &  & & \cdot & \cdot & \\
	      {\bf 0}  &&  & &  \varphi_1 & 0  \\
	      \end{array}
	      \right),
\qquad \qquad 
B:\; 
\left(
\begin{array}{ c c cc c c }
 0 & 1   &   &&   & \bf 0  \\
  & 0  &  1  &&  &      \\ 
   &   &  0   & \cdot &&  \\
     &   &  & \cdot  & \cdot & \\
       &  & &  & \cdot & 1 \\
        {\bf 0}  &&  & &   &  0  \\
	\end{array}
	\right),
\\
&&C:\; 
\left(
\begin{array}{ c c c c c c }
a_d  & \varphi_1/\varphi_d   &   &&   & \bf 0  \\
 \varphi_1 & a_{d-1}  &  \varphi_{2}/\varphi_{d-1}  &&  &      \\ 
   & \varphi_{2}  & a_{d-2}  
   & \cdot &&  \\
     &   & \cdot & \cdot  & \cdot & \\
       &  & & \cdot & \cdot & \varphi_d/\varphi_1 \\
        {\bf 0}  &&  & &  \varphi_d &  a_0  \\
	\end{array}
	\right).
\end{eqnarray*}
With respect to an inverted $(B,A)$ basis for $V$ the matrices representing
$A,B,C$ are
\begin{eqnarray*}
&&A:\; 
\left(
\begin{array}{ c c cc c c }
 0 & \varphi_1   &   &&   & \bf 0  \\
  & 0  &  \varphi_2  &&  &      \\ 
   &   &  0   & \cdot &&  \\
     &   &  & \cdot  & \cdot & \\
       &  & &  & \cdot & \varphi_d \\
        {\bf 0}  &&  & &   &  0  \\
	\end{array}
	\right),
	\qquad \qquad
	B:\;
	\left(
	\begin{array}{ c c cc c c }
	0 &   &   &&   & \bf 0  \\
	1 & 0  &   &&  &      \\
	 &  1 & 0   & &&  \\
	   &   & \cdot & \cdot  & & \\
	     &  & & \cdot & \cdot & \\
	      {\bf 0}  &&  & &   1& 0  \\
	      \end{array}
	      \right),
\\
&&C:\; 
\left(
\begin{array}{ c c c c c c }
a_0 & \varphi_d   &   &&   & \bf 0  \\
 \varphi_d/\varphi_1 & a_1  &  \varphi_{d-1}  &&  &      \\ 
   & \varphi_{d-1}/\varphi_2  & a_2  
   & \cdot &&  \\
     &   & \cdot & \cdot  & \cdot & \\
       &  & & \cdot & \cdot & \varphi_1\\
        {\bf 0}  &&  & &  \varphi_1/\varphi_d &  a_d  \\
	\end{array}
	\right).
\end{eqnarray*}

\section{Appendix II: The bipartite LR triples in matrix form}

\noindent 
In this section we display the bipartite LR triples
in matrix form.

\medskip
\noindent Let $V$ denote a vector space over $\mathbb F$
with dimension $d+1$. Let $A,B,C$ denote a 
bipartite LR triple on $V$,
with
 parameter array
(\ref{eq:paLRT}).
By Proposition
\ref{prop:matrixRep}
we have the following.
\medskip

\noindent
With respect to an $(A,B)$-basis for $V$
the matrices representing $A,B,C$ are
\begin{eqnarray*}
&&A:\; 
\left(
\begin{array}{ c c cc c c }
 0 & 1   &   &&   & \bf 0  \\
  & 0  &  1  &&  &      \\ 
   &   &  0   & \cdot &&  \\
     &   &  & \cdot  & \cdot & \\
       &  & &  & \cdot & 1 \\
        {\bf 0}  &&  & &   &  0  \\
	\end{array}
	\right),
	\qquad \qquad
	B:\;
	\left(
	\begin{array}{ c c cc c c }
	0 &   &   &&   & \bf 0  \\
	\varphi_1 & 0  &   &&  &      \\
	 &  \varphi_2 & 0   & &&  \\
	   &   & \cdot & \cdot  & & \\
	     &  & & \cdot & \cdot & \\
	      {\bf 0}  &&  & &  \varphi_d & 0  \\
	      \end{array}
	      \right),
\\
&&C:\; 
\left(
\begin{array}{ c c c c c c }
 0  & \varphi'_d/\varphi_1   &   &&   & \bf 0  \\
 \varphi''_d & 0 &  \varphi'_{d-1}/\varphi_2  &&  &      \\ 
   & \varphi''_{d-1}  & 0  
   & \cdot &&  \\
     &   & \cdot & \cdot  & \cdot & \\
       &  & & \cdot & \cdot & \varphi'_1/\varphi_d \\
        {\bf 0}  &&  & &  \varphi''_1 &  0  \\
	\end{array}
	\right).
\end{eqnarray*}
With respect to an inverted $(A,B)$ basis for $V$ the matrices representing
$A,B,C$ are
\begin{eqnarray*}
 &&A:\;
	\left(
	\begin{array}{ c c cc c c }
	0 &   &   &&   & \bf 0  \\
     1 & 0  &   &&  &      \\
	 &  1& 0   & &&  \\
	   &   & \cdot & \cdot  & & \\
	     &  & & \cdot & \cdot & \\
	      {\bf 0}  &&  & &  1 & 0  \\
	      \end{array}
	      \right),
\qquad \qquad 
B:\; 
\left(
\begin{array}{ c c cc c c }
 0 & \varphi_d   &   &&   & \bf 0  \\
  & 0  &  \varphi_{d-1}  &&  &      \\ 
   &   &  0   & \cdot &&  \\
     &   &  & \cdot  & \cdot & \\
       &  & &  & \cdot & \varphi_1 \\
        {\bf 0}  &&  & &   &  0  \\
	\end{array}
	\right),
\\
&&C:\; 
\left(
\begin{array}{ c c c c c c }
0  & \varphi''_1   &   &&   & \bf 0  \\
 \varphi'_1/\varphi_d & 0  &  \varphi''_{2} &&  &      \\ 
   & \varphi'_{2}/\varphi_{d-1}  & 0  
   & \cdot &&  \\
     &   & \cdot & \cdot  & \cdot & \\
       &  & & \cdot & \cdot & \varphi''_d \\
        {\bf 0}  &&  & &  \varphi'_d/\varphi_1 &  0\\
	\end{array}
	\right).
\end{eqnarray*}
With respect to a $(B,A)$ basis for $V$ the matrices representing
$A,B,C$ are
\begin{eqnarray*}
   &&     A:\;
	\left(
	\begin{array}{ c c cc c c }
	0 &   &   &&   & \bf 0  \\
	\varphi_d & 0  &   &&  &      \\
	 &  \varphi_{d-1} & 0   & &&  \\
	   &   & \cdot & \cdot  & & \\
	     &  & & \cdot & \cdot & \\
	      {\bf 0}  &&  & &  \varphi_1 & 0  \\
	      \end{array}
	      \right),
\qquad \qquad 
B:\; 
\left(
\begin{array}{ c c cc c c }
 0 & 1   &   &&   & \bf 0  \\
  & 0  &  1  &&  &      \\ 
   &   &  0   & \cdot &&  \\
     &   &  & \cdot  & \cdot & \\
       &  & &  & \cdot & 1 \\
        {\bf 0}  &&  & &   &  0  \\
	\end{array}
	\right),
\\
&&C:\; 
\left(
\begin{array}{ c c c c c c }
 0  & \varphi''_1/\varphi_d   &   &&   & \bf 0  \\
 \varphi'_1 & 0 &  \varphi''_{2}/\varphi_{d-1}  &&  &      \\ 
   & \varphi'_{2}  & 0  
   & \cdot &&  \\
     &   & \cdot & \cdot  & \cdot & \\
       &  & & \cdot & \cdot & \varphi''_d/\varphi_1 \\
        {\bf 0}  &&  & &  \varphi'_d &  0 \\
	\end{array}
	\right).
\end{eqnarray*}
With respect to an inverted $(B,A)$ basis for $V$ the matrices representing
$A,B,C$ are
\begin{eqnarray*}
&&A:\; 
\left(
\begin{array}{ c c cc c c }
 0 & \varphi_1   &   &&   & \bf 0  \\
  & 0  &  \varphi_2  &&  &      \\ 
   &   &  0   & \cdot &&  \\
     &   &  & \cdot  & \cdot & \\
       &  & &  & \cdot & \varphi_d \\
        {\bf 0}  &&  & &   &  0  \\
	\end{array}
	\right),
	\qquad \qquad
	B:\;
	\left(
	\begin{array}{ c c cc c c }
	0 &   &   &&   & \bf 0  \\
	1 & 0  &   &&  &      \\
	 &  1 & 0   & &&  \\
	   &   & \cdot & \cdot  & & \\
	     &  & & \cdot & \cdot & \\
	      {\bf 0}  &&  & &   1& 0  \\
	      \end{array}
	      \right),
\\
&&C:\; 
\left(
\begin{array}{ c c c c c c }
 0 & \varphi'_d   &   &&   & \bf 0  \\
 \varphi''_d/\varphi_1 & 0  &  \varphi'_{d-1}  &&  &      \\ 
   & \varphi''_{d-1}/\varphi_2  & 0  
   & \cdot &&  \\
     &   & \cdot & \cdot  & \cdot & \\
       &  & & \cdot & \cdot & \varphi'_1\\
        {\bf 0}  &&  & &  \varphi''_1/\varphi_d &  0 \\
	\end{array}
	\right).
\end{eqnarray*}


\section{Acknowledgments}
The author thanks 
Kazumasa Nomura 
for giving this paper a close reading and offering  valuable
suggestions.
Also, part of this paper was written during the author's
visit to the Graduate School of Information Sciences, at Tohoku U. in Japan.
The visit was from December 20, 2014 to January 15, 2015.
The author thanks his hosts Hajime Tanaka and Jae-ho Lee for
their kind hospitality and insightful conversations.

\noindent Paul Terwilliger \hfil\break
\noindent Department of Mathematics \hfil\break
\noindent University of Wisconsin \hfil\break
\noindent 480 Lincoln Drive \hfil\break
\noindent Madison, WI 53706 USA \hfil\break
\noindent email: {\tt terwilli@math.wisc.edu }\hfil\break

\end{document}